\documentclass[11pt]{article}

\usepackage{amsmath,amssymb,amsfonts,amsthm,enumitem,pstricks,float,xy,mathrsfs,accents}
\xyoption{all}

\numberwithin{equation}{section} 
\newtheorem{thm}{Theorem}[section]
\newtheorem{prp}[thm]{Proposition} 
\newtheorem{lmm}[thm]{Lemma}  
\newtheorem{crl}[thm]{Corollary}
\newtheorem*{prp*}{Proposition}
\theoremstyle{definition}
\newtheorem{dfn}[thm]{Definition}

\newtheorem{rmk}[thm]{Remark}

\newtheorem{cnj}[thm]{Conjecture}

\DeclareFontFamily{U}{mathx}{\hyphenchar\font45}
\DeclareFontShape{U}{mathx}{m}{n}{
      <5> <6> <7> <8> <9> <10>
      <10.95> <12> 
      mathx10
      }{}
\DeclareSymbolFont{mathx}{U}{mathx}{m}{n}
\DeclareFontSubstitution{U}{mathx}{m}{n}
\DeclareMathAccent{\widecheck}{0}{mathx}{"71}

\def\BE#1{\begin{equation}\label{#1}}
\def\EE{\end{equation}}
\def\eref#1{(\ref{#1})}

\def\lra{\longrightarrow}
\def\llra{\longleftrightarrow}

\def\xra#1{\xrightarrow{\hspace*{#1cm}}}

\def\lr#1{\langle{#1}\rangle}
\def\blr#1{\big\langle{#1}\big\rangle}
\def\ov#1{\overline{#1}}
\def\ti#1{\tilde{#1}}
\def\wc#1{\widecheck{#1}}
\def\wt#1{\widetilde{#1}}
\def\wh#1{\widehat{#1}}
\def\tn#1{\textnormal{#1}}
\def\smsize#1{\begin{small}#1\end{small}}
\def\ri#1{\accentset{\circ}{#1}}

\def\al{\alpha}
\def\be{\beta}
\def\de{\delta}
\def\ep{\epsilon}
\def\ga{\gamma}
\def\io{\iota}
\def\ka{\kappa}
\def\la{\lambda}
\def\na{\nabla}
\def\om{\omega}

\def\th{\theta}
\def\vph{\varphi}

\def\vr{\varrho}
\def\vt{\vartheta}
\def\uvt{\underline{\vt}}
\def\vp{\varpi}
\def\ve{\varepsilon}
\def\ze{\zeta}

\def\Ga{\Gamma}
\def\De{\Delta}
\def\Om{\Omega}
\def\Th{\Theta}

\def\fa{\mathfrak a}
\def\ufa{\underline{\fa}} 
\def\cA{\mathcal A}
\def\C{\mathbb C}
\def\cC{\mathcal C}
\def\bfC{\mathbf C}

\def\fD{\mathfrak D}
\def\bD{\mathbb D}
\def\bE{\mathbb E}
\def\tE{\textnormal{E}}
\def\bF{\mathbb F}
\def\cF{\mathcal F}
\def\cH{\mathcal H}
\def\bI{\mathbb I}
\def\fj{\mathfrak j} 
\def\cK{\mathcal K}

\def\cL{\mathcal L}
\def\fL{\mathfrak L}
\def\M{\mathcal M}
\def\chM{\wc\M}
\def\bN{\mathbb N}
\def\cN{\mathcal N}
\def\cO{\mathcal O}
\def\P{\mathbb P}
\def\cP{\mathcal P}
\def\cR{\mathcal R}
\def\Q{\mathbb Q}
\def\R{\mathbb R}

\def\T{\mathbb T}
\def\cU{\mathcal U}
\def\chU{\wc\cU}
\def\tV{\textnormal{V}}

\def\cZ{\mathcal Z}
\def\Z{\mathbb Z}

\def\bm{\mathbf m}
\def\bs{\mathbf s}

\def\bz{\mathbf z}
\def\bfI{\mathbf I}

\def\ne{\textnormal{e}}
\def\fI{\mathfrak i}

\def\0{\mathbf 0}

\def\Si{\Sigma}

\def\coker{\textnormal{coker}}
\def\cyl{\textnormal{cyl}}
\def\Deck{\tn{Deck}}
\def\deg{\textnormal{deg}}
\def\ev{\textnormal{ev}}

\def\GT{\textnormal{GT}}
\def\GW{\textnormal{GW}}
\def\Hom{\textnormal{Hom}}
\def\hor{\textnormal{hor}}
\def\id{\textnormal{id}}

\def\neck{\textnormal{neck}}
\def\ord{\textnormal{ord}}

\def\PD{\textnormal{PD}}
\def\Re{\textnormal{Re}}
\def\RA{\textnormal{RA}}

\def\st{\textnormal{st}}

\def\vrt{\textnormal{vrt}}
\def\vir{\textnormal{vir}}

\def\to{~$\lra$~}
\def\i{\infty}
\def\w{\wedge}
\def\prt{\partial}
\def\dbar{\bar\partial}
\def\eset{\emptyset}
\def\nd{\textnormal{d}}
\def\bu{\bullet}

\def\II{\mathrm{II}}

\def\XVm{X_V^m}

\def\XYm{X\!\cup_V^m\!Y}
\def\tXV{\wh{X}_V}
\def\tYV{\wh{Y}_V}
\def\tXVm{\wh{X}_V^m}
\def\tYVm{\wh{Y}_V^m}

\def\tXYm{X\!\wh\cup_V^m\!Y}
\def\oXV{\ri{X}_V}
\def\oYV{\ri{Y}_V}
\def\oXVm{\ri{X}_V^m}
\def\oYVm{\ri{Y}_V^m}

\def\oXYm{X\!\ri\cup_V^m\!Y}
\def\oSi{\ri\Si}
\def\hSi{\wh\Si}
\def\oTh{\ri\Th}
\def\tTh{\wh\Th}
\def\oJ{\ri{J}}
\def\oJm{\ri{J}_m}

\def\ou{\ri{u}}
\def\tu{\wh{u}}

\begin{document}

\title{On Symplectic Sum Formulas in\\ Gromov-Witten Theory}
\author{Mohammad F.~Tehrani and Aleksey Zinger\thanks{Partially 
supported by NSF grant 0846978}}
\date{\small March~14, 2014. Updated: \today}
\maketitle

\begin{abstract}

\noindent
This manuscript describes in detail the symplectic sum formulas in 
Gromov-Witten theory and related topological and analytic issues.
In particular, we analyze and compare two analytic approaches
to these formulas.
The Ionel-Parker formula contains two unique features, 
rim tori refinements of relative invariants and the so-called $S$-matrix, 
which   
have been a mystery in Gromov-Witten theory over the past decade.
We explain why the latter, which appears due to imprecise reasoning, 
should not be present and how the former should be interpreted.
While the key gluing argument in the Ionel-Parker work attempts to address
all of the issues relevant to certain ``semi-positive" cases,
it contains several highly technical, but crucial, mistakes, which 
invalidate it and thus the whole paper almost completely.
The idea behind the Li-Ruan approach is to adapt the SFT stretching
of the target.
This has the potential of avoiding many issues with the degeneration
of the metric on the target occurring in the Ionel-Parker approach,
which we expect to realize in a forthcoming paper.
Unfortunately, the Li-Ruan paper is vague about the key notions and 
aspects of the setup, including the definition of relative stable maps, and
does not contain even a description of the local structure of the relative moduli space
or an attempt at a complete proof of any major statement, 
such as the compactness and Hausdorffness of the relative moduli space or
the bijectivity of the gluing construction.
The only technical arguments in this paper concern fairly minor points 
and are either incorrect or unnecessary.
Neither of the two papers even considers gluing stable maps with extra rubber structure,
which is necessary to do for defining the relevant invariants outside of 
a relatively narrow collection of ``semi-positive" cases
and is fundamentally different from the gluing situation in the absolute Gromov-Witten theory. 
In this manuscript, we 
re-formulate the (numerical) symplectic sum formula, 
describe the issues arising in both approaches,  
and explain how the Li-Ruan SFT type idea can be used to address~them.
\end{abstract}

\tableofcontents

\section{Introduction}
\label{intro_sec}

\noindent
Gromov-Witten invariants of symplectic manifolds, which include nonsingular projective varieties, 
are certain counts of pseudo-holomorphic curves that play 
prominent roles in symplectic topology, algebraic geometry, and string theory.
The decomposition formulas, known as symplectic sum formulas 
in symplectic topology and degeneration formulas in algebraic geometry, 
are one of the main tools used to compute Gromov-Witten invariants.
Such formulas are suggested in~\cite{T} and described fully in \cite{LR,Jun2,IPsum}.
They relate Gromov-Witten invariants of a target symplectic manifold 
to Gromov-Witten invariants of simpler symplectic manifolds;
in many cases, these formulas completely determine the former in terms of the latter.\\

\noindent
The main formula of~\cite{IPsum} contains two features not present in 
the formulas of~\cite{LR} and~\cite{Jun2}:
a rim tori refinement of relative invariants and the so-called $S$-matrix.
In Section~\ref{Smat_subs}, we explain why the latter should not have appeared
in the first place and acts by the identity anyway for essentially the same reason.
The situation with the former is explored in detail in~\cite{GWrelIP,GWsumIP}.
While \cite{IPrel,IPsum} only suggest how to construct this refinement,
making incorrect statements on key aspects and in simple examples,
it is possible to implement the general idea behind this refinement
and to extract some qualitative implications from~it.
In this manuscript, we  point out several highly technical, but crucial, mistakes
in the key gluing argument of~\cite{IPsum}.\\

\noindent
The symplectic sum formula of~\cite{LR}, 
which is spread out across multiple statements, is not formulated entirely correctly.
The idea of~\cite{LR} to adapt the SFT  stretching of the target
beautifully captures the degeneration of both the domain and
the target and has a great potential of avoiding many analytic difficulties
caused by the degeneration of the latter arising in~\cite{IPsum}.
Unfortunately, the implementation of this idea does not contain even an attempt at 
a complete proof of any major statement, 
such the compactness of the moduli space of relative maps or  
the bijectivity of the gluing construction.
There is not even a reasonably precise definition of relative stable map in~\cite{LR};
the definition of morphism into the singular fiber is simply wrong for the intended purposes,
as it does not describe limits of maps into smooth fibers.
The only technical arguments in this paper concern fairly minor points 
and are either incorrect or add unnecessary complications.  
Neither~\cite{LR} nor~\cite{IPrel,IPsum} even considers gluing stable maps with 
extra rubber structure,
which is necessary to do for defining the relevant invariants outside of 
the relatively narrow collection of ``semi-positive" cases.\\

\noindent
Section~\ref{issues_sec} summarizes our understanding of the issues 
with \cite{IPrel,IPsum} and~\cite{LR} and
directs to places in this manuscript where they are described in more detail;
considerations related to \cite{Jun1,Jun2} appear in~\cite{AF,Chen,GS}.
Throughout this manuscript, we generally follow the reasoning and notation in~\cite{IPsum} closely, 
but also adapt some of the statements from~\cite{LR} and~\cite{Jun2}.

\subsection{Symplectic sums}
\label{SympSum_subs0}

\noindent
We denote by $(X,\om_X)$ and $(Y,\om_Y)$ compact symplectic manifolds,
of the same dimension and without boundary.
A compact submanifold~$V$ of $(X,\om_X)$ is a \textsf{symplectic hypersurface}
if the real codimension of~$V$ in~$X$ is~2 and $\om_X|_V$ is a nondegenerate two-form on~$V$.
The normal bundle of a symplectic hypersurface~$V$ in~$X$,
\BE{cNXVsymp_e}\cN_XV\equiv \frac{TX|_V}{TV}\approx TV^{\om_X}
\equiv \big\{v\!\in\!T_xV\!:\,x\!\in\!V,\,\om_X(v,w)\!=\!0~\forall\,w\!\in\!T_xV\big\},\EE
then inherits a symplectic structure~$\om_X|_{\cN_XV}$ from~$\om_X$ and 
thus a complex structure up to homotopy.
If
\BE{cNVcond_e}e(\cN_XV)=-e(\cN_YV)\in H^2(V;\Z),\EE
there exists an isomorphism
\BE{cNpair_e}\Phi\!:\cN_XV\otimes_{\C}\cN_YV\approx V\!\times\!\C\EE
of complex line bundles.\\

\noindent
As recalled in Section~\ref{SympSum_subs},
a \textsf{symplectic sum} of symplectic manifolds $(X,\om_X)$ and $(Y,\om_Y)$ 
with a common symplectic divisor~$V$ such that~\eref{cNVcond_e} holds
is a symplectic manifold $(Z,\om_Z)\!=\!(X\!\#_V\!Y,\om_{\#})$ obtained from~$X$ and~$Y$ 
by gluing the complements of tubular neighborhoods of~$V$ in~$X$ and~$Y$ 
along their common boundary as directed by the isomorphism~\eref{cNpair_e}.
In fact, the symplectic sum construction of~\cite{Gf, MW} produces 
a \textsf{symplectic fibration} $\pi\!:\cZ\!\lra\!\De$ 
with central fiber \hbox{$\cZ_0\!=\!X\!\cup_V\!Y$}, where
$\De\!\subset\!\C$ is a disk centered at the origin and 
$\cZ$ is a symplectic manifold with symplectic form~$\om_{\cZ}$
such~that 
\begin{enumerate}[label=$\bullet$,leftmargin=*]
\item  $\pi$ is surjective and is a submersion outside of $V\!\subset\!\cZ_0$,
\item the restriction~$\om_{\la}$ of~$\om_{\cZ}$ to $\cZ_{\la}\equiv\pi^{-1}(\la)$ 
is nondegenerate for every $\la\!\in\!\De^*$,
\item $\om_{\cZ}|_X\!=\!\om_X$, $\om_{\cZ}|_Y\!=\!\om_Y$.
\end{enumerate}
The symplectic manifolds $(\cZ_{\la},\om_{\la})$ with $\la\!\in\!\De^*$
are then symplectically deformation equivalent to each other and denoted~$(X\!\#_V\!Y,\om_{\#})$.
However, different homotopy classes of the isomorphisms~\eref{cNpair_e} give rise
to generally different topological manifolds; see~\cite{Gf0}.\\

\noindent
There is also a retraction $q\!:\cZ\!\lra\!\cZ_0$ such that $q_{\la}\!\equiv\!q|_{\cZ_{\la}}$ 
restricts to a diffeomorphism 
$$\cZ_{\la}-q_{\la}^{-1}(V)\lra \cZ_0-V$$
and to an $S^1$-fiber bundle $q_{\la}^{-1}(V)\!\lra\!V$, whenever $\la\!\in\De^*$.
We denote by $q_{\#}\!:X\!\#_V\!Y\!\lra\!X\!\cup_V\!Y$ a typical collapsing map~$q_{\la}$.\\

\noindent
In the algebraic setting of~\cite{Jun2}, 
$\pi\!:\cZ\!\lra\!\De$ is a holomorphic map from a Kahler manifold~$\cZ$
with an ample line bundle $\cL\!\lra\!\cZ$;
the curvature form of a suitably chosen metric on~$\cL$ gives rise to a symplectic form~$\om_{\cZ}$
on~$\cZ$, as in \cite[Section 1.2]{GH}.

\subsection{Absolute and relative GW-invariants}
\label{AbsRelGWs_subs}

\noindent
If $g,k\!\in\!\Z^{\ge0}$, $\chi\!\in\!\Z$, $A\!\in\!H_2(X;\Z)$, and 
$J$ is an $\om_X$-compatible almost complex structure on~$X$, 
let $\ov\M_{g,k}(X,A)$ and $\wt\M_{\chi,k}(X,A)$ denote the moduli spaces of 
stable $J$-holomorphic $k$-marked maps from connected nodal curves of genus~$g$ 
and from (possibly) disconnected nodal curves of euler characteristic~$\chi$, respectively;
the latter moduli spaces are quotients of disjoint unions of products 
of the former moduli spaces. 
If $V\!\subset\!X$ is a symplectic divisor, $\bs\!\equiv\!(s_1,\ldots,s_{\ell})$
is an $\ell$-tuple of positive integers such that 
\BE{bsumcond_e} s_1+\ldots+s_{\ell}=A\cdot V,\EE
and $J$ restricts to a complex structure on $V$, let 
$\ov\M_{g,k;\bs}^V(X,A)$ and $\wt\M_{\chi,k;\bs}^V(X,A)$ denote the moduli spaces of 
stable $J$-holomorphic $(k\!+\!\ell)$-marked maps from connected nodal curves of genus~$g$ 
and from (possibly) disconnected nodal curves of euler characteristic~$\chi$, respectively,
that have contact with~$V$ at the last $\ell$ marked points of orders $s_1,\ldots,s_{\ell}$.
These moduli spaces are introduced in \cite{LR,IPrel,Jun1} under certain assumptions on~$J$
and reviewed in Section~\ref{RelInv_sub0}.\\

\noindent
There are natural evaluation morphisms
\BE{evdfn_e}\begin{split}
\ev\!\equiv\!\ev_1\!\times\!\ldots\!\times\!\ev_k\!: 
\wt\M_{\chi,k}(X,A),\wt\M_{\chi,k;\bs}^V(X,A)&\lra X^k,\\
\ev^V\!\equiv\!\ev_{k+1}\!\times\!\ldots\!\times\!\ev_{k+\ell}:
\wt\M_{\chi,k;\bs}^V(X,A)&\lra V_{\bs}\equiv V^{\ell},
\end{split}\EE 
sending each element to the values of the map at the marked points.
We denote the restrictions of these maps~to
$$\ov\M_{g,k}(X,A)\subset\wt\M_{2-2g,k}(X,A)
\qquad\hbox{and}\qquad
\ov\M_{g,k;\bs}^V(X,A)\subset\wt\M_{2-2g,k;\bs}^V(X,A)$$
by the same symbols.
Along with the virtual class for $\ov\M_{g,k}(X,A)$,
constructed in \cite{RT2} in the semi-positive case,
in \cite{BF} in the algebraic case, 
and in~\cite{FO,LT} in the general case, the morphisms~\eref{evdfn_e} with $V\!=\!\eset$ give rise
to the (\textsf{absolute}) \textsf{Gromov-Witten} 
and \textsf{Gromov-Taubes} invariants of~$(X,\om_X)$,
\begin{alignat*}{2}
\GW_{g,A}^X\!: \T^*(X)&\lra\Q, &\quad
\GW_{g,A}^X(\al)&=\sum_{k=0}^{\i}\blr{\ev^*\al,[\ov\M_{g,k}(X,A)]^{\vir}},\\
\GT_{\chi,A}^X\!: \T^*(X)&\lra\Q,  &\quad
\GT_{\chi,A}^X(\al)&=\sum_{k=0}^{\i}\blr{\ev^*\al,[\wt\M_{\chi,k}(X,A)]^{\vir}},
\end{alignat*}
where
$$\T^*(X)\equiv \bigoplus_{k=0}^{\i} H^{2*}(X)^{\otimes k} 
\subset\bigoplus_{k=0}^{\i} H^{2*}(X^k)$$
is the tensor algebra of $H^{2*}(X)\!\equiv\!H^{2*}(X;\Q)$.\footnote{Odd cohomology
classes can be considered as well, but at the cost of introducing suitable
signs into the symplectic sum formulas.}
Along with the virtual class for $\ov\M_{g,k;\bs}^V(X,A)$, the morphisms~\eref{evdfn_e} give rise
to the \textsf{relative Gromov-Witten} 
and \textsf{Gromov-Taubes} invariants of~$(X,V,\om_X)$,
\begin{alignat*}{2}
\GW_{g,A;\bs}^{X,V}\!: \T^*(X)&\lra H_*(V_{\bs}), \quad
\GW_{g,A;\bs}^{X,V}(\al)&=\sum_{k=0}^{\i}
\ev^V_*\big(\ev^*\al\!\cap\!\big[\ov\M_{g,k;\bs}^V(X,A)\big]^{\vir}\big),\\
\GT_{\chi,A;\bs}^{X,V}\!: \T^*(X)&\lra H_*(V_{\bs}), \quad
\GT_{\chi,A;\bs}^{X,V}(\al)&=\sum_{k=0}^{\i}\ev^V_*
\big(\ev^*\al\!\cap\!\big[\wt\M_{\chi,k;\bs}^V(X,A)\big]^{\vir}\big),
\end{alignat*}
where $H_*(V_{\bs})\!\equiv\!H_*(V_{\bs};\Q)$.
Such a virtual class is constructed in \cite{IPrel} in ``semi-positive" cases
and in \cite{Jun1} in the algebraic case and is used in~\cite{LR} in the general case;
see Section~\ref{RelInv_subs} for more details.
While the homomorphisms $\GW_{g,A}^X$ and $\GW_{g,A;\bs}^{X,V}$ completely determine
the homomorphisms $\GT_{\chi,A}^X$ and $\GT_{\chi,A;\bs}^{X,V}$, 
the latter lead to more streamlined decomposition formulas for 
(primary) GW-invariants, as noticed in~\cite{IPsum}.

\subsection{A splitting formula for GW-invariants}
\label{GWsplit_subs}

\noindent
The symplectic sum formulas for GW-invariants relate the absolute GW-invariants of~$X\!\#_V\!Y$
to the relative GW-invariants of the pairs $(X,V)$ and $(Y,V)$.
Let 
\BE{H2XYV_e}H_2(X;\Z)\!\times_V\!H_2(Y;\Z)=
\big\{(A_X,A_Y)\!\in\!H_2(X;\Z)\!\times\!H_2(Y;\Z)\!:~A_X\!\cdot_X\!V=A_Y\!\cdot_Y\!V\big\},\EE
where $\cdot_X$ and $\cdot_Y$ denote the homology intersection pairings in~$X$ and~$Y$, e.g.
$$A_X\!\cdot_X\!V=\blr{\PD_XA_X\cup \PD_X[V],[X]}\in\Z.$$
As described in \cite[Section~2.1]{GWrelIP}, there is a natural homomorphism
\BE{homsumdfn_e}
H_2(X;\Z)\!\times_V\!H_2(Y;\Z)\lra H_2(X\!\#_V\!Y;\Z)/\cR_{X,Y}^V, \qquad
(A_X,A_Y)\lra A_X\!\#_V\!A_Y,\EE
where
\BE{cRXYVprop_e}
\cR_{X,Y}^V=\ker\big\{q_{\#*}\!:H_2(X\!\#_V\!Y;\Z)\lra H_2(X\!\cup_V\!Y;\Z)\big\}.\EE
We arrange the GT-invariants of $X\!\#_V\!Y$ into the formal power series
\BE{GTabsser_e} \GT^{X\#_VY}=\sum_{\chi\in\Z}\sum_{\eta\in H_2(X\!\#_V\!Y;\Z)/\cR_{X,Y}^V}
\sum_{C\in\eta}\GT^{X\#_VY}_{\chi,C}\,t_{\eta}\la^{\chi}.\EE
By Gromov's Compactness Theorem for $J$-holomorphic curves, only 
finitely many distinct elements $C\!\in\!\eta$ can be
represented by $J$-holomorphic curves of a given genus, 
since $\om_{\#}$ vanishes on~$\cR_{X,Y}^V$.
Thus, the coefficient of each $t_{\eta}\la^{\chi}$ in $\GT^{X\#_VY}$ is finite.\\

\noindent
For a tuple $\bs\!=\!(s_1,\ldots,s_{\ell})\in(\Z^+)^{\ell}$, let 
$$\ell(\bs)=\ell, \qquad  |\bs|= s_1+\ldots+s_{\ell},   \qquad 
\lr\bs= s_1\cdot\ldots\cdot s_{\ell}.$$
We arrange the GW-invariants of $(X,V)$ and $(Y,V)$ into the formal power series
\BE{GTrelser_e}
\GT^{M,V}=\sum_{\chi\in\Z}\sum_{A\in H_2(M;\Z)}
\sum_{\ell=0}^{\i}\sum_{\begin{subarray}{c}\bs\in(\Z^+)^{\ell}\\ |\bs|=A\cdot_MV\end{subarray}}
\!\!\!\!\!\GT_{\chi,A;\bs}^{M,V}\,t_A\la^{\chi},\EE
where $M\!=\!X,Y$.\\

\noindent
Let
$$V_{\i}=\bigsqcup_{\ell=0}^{\i}\bigsqcup_{\bs\in(\Z^+)^{\ell}}\!\!\!\!\!V_{\bs}\,.$$
We define a pairing $\star\!: H_*(V_{\i})\otimes H_*(V_{\i})\lra\Q[\la^{-1}]$ by
\BE{bspairing_e}
Z_X\star Z_Y=
\begin{cases}
\frac{\lr\bs}{\ell(\bs)!}\la^{-2\ell(\bs)}
Z_X\cdot_{V_{\bs}}\!Z_Y,&\hbox{if}~Z_X,Z_Y\in H_*(V_{\bs});\\
0,&\hbox{if}~Z_X\in H_*(V_{\bs}),~Z_Y\in H_*(V_{\bs'}),~\bs\neq\bs'.
\end{cases}\EE
For homomorphisms $L_X\!:\T^*(X)\!\lra\!H_*(V_{\i})$ and $L_Y\!:\T^*(Y)\!\lra\!H_*(V_{\i})$,
define
\BE{bspairing_e1}
L_X\!\star\!L_Y\!: \T^*(X)\otimes\T^*(Y)\lra\Q[\la^{-1}] \quad\hbox{by}\quad
\{L_X\!\star\!L_Y\}(\al_X\!\otimes\!\al_Y)=L_X(\al_X)\star L_Y(\al_Y).\EE
If in addition  $(A_X,A_Y)\!\in\!H_2(X;\Z)\!\times_V\!H_2(Y;\Z)$ and $\chi_X,\chi_Y\!\in\!\Z$, 
let 
\BE{bspairing_e2}
L_Xt_{A_X}\la^{\chi_X}\star L_Yt_{A_Y}\la^{\chi_Y}
=L_X\!\star\!L_Y\, t_{A_X\#_V A_Y}\la^{\chi_X+\chi_Y}.\EE

\begin{thm}\label{main_thm}
Let $(X,\om_X)$ and $(Y,\om_Y)$ be symplectic manifolds and 
$V\!\subset\!X,Y$ be a symplectic hypersurface satisfying~\eref{cNVcond_e}.
If $q_{\#}\!:X\!\#_V\!Y\!\lra\!X\!\cup_V\!Y$ is a collapsing map 
for an associated symplectic sum fibration
and $q_{\sqcup}\!:X\!\sqcup\!Y\!\lra\!X\!\cup\!_VY$ is the quotient map, then
\BE{SympSumForm_e} \GT^{X\#_VY}(q_{\#}^*\al)=\{\GT^{X,V}\star\GT^{Y,V}\}(q_{\sqcup}^*\al)\EE
for all $\al\!\in\!\T^*(X\!\cup_V\!Y)$.
\end{thm}

\noindent
The motivation behind~\eref{SympSumForm_e}, as well as all other symplectic sum
formulas for GW-invariants, is the following.
The curves in the smooth fibers $\cZ_{\la}\!=\!X\!\#_V\!Y$ of the fibration 
$\pi\!:\cZ\!\lra\!\De$ that contribute to the left-hand side of~\eref{SympSumForm_e}
degenerate, as $\la\!\lra\!0$, to curves in the singular fiber $\cZ_0\!=\!X\!\cup_V\!Y$.
Each of the irreducible components of a limiting curve lies in either $X$ or~$Y$.
Furthermore, the union of the irreducible components of each limiting curve 
that map to~$X$ meets $V\!\subset\!X$ at the same points with the same multiplicity
as  the union of the irreducible components of each limiting curve 
that map to~$Y$; see Figure~\ref{limitcurve_fig}.
Such curves contribute to the right-hand side of~\eref{SympSumForm_e}.
The contact conditions with~$V$ are encoded by a tuple~$\bs$ as above. 
For the reasons outlined in \cite[p938]{IPsum},
each limiting curve of type~$\bs$ arises as a limit of $\lr\bs$ distinct families
of curves into smooth fibers, requiring the factor of~$\lr\bs$ in~\eref{bspairing_e};
see also Section~\ref{IPconv_subs}.
The factor of~$\ell(\bs)!$ in~\eref{bspairing_e} arises due to the fact that the contact 
points with~$V$ are not \'a priori ordered, while the factor of $\la^{-2\ell(\bs)}$ accounts
for the difference between the geometric and algebraic euler characteristics 
of the limiting curve.
Since connected curves can limit to disconnected curves on one side, it is more natural
to formulate decomposition formulas for GW-invariants in terms of counts of disconnected curves,
i.e.~the GT-invariants, as done in~\cite{IPsum}.

\begin{figure}
\begin{pspicture}(38,-2.3)(11,1)
\psset{unit=.3cm}
\psline[linewidth=.1](55,-2)(73,-2)
\rput(54,-2){$V$}\rput(55,-5){$X$}\rput(55,1){$Y$}
\psarc[linewidth=.04](60,-5){3}{0}{180}\pscircle*(60,-2){.3}
\psarc[linewidth=.04](64,-5){1}{180}{0}
\psarc[linewidth=.04](68,-5){3}{0}{180}\pscircle*(68,-2){.3}
\psarc[linewidth=.04](60,1){3}{180}{0}\psarc[linewidth=.04](68,1){3}{180}{0}
\end{pspicture}
\caption{A possible limit of connected curves.}
\label{limitcurve_fig}
\end{figure}
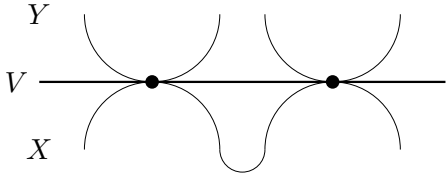

\subsection{Background and alternative formulations}
\label{OtherSSF_subs}

\noindent
A symplectic sum formula for GW-invariants is suggested in \cite[Theorem~10.2]{T},
which is stated without a proof (\cite{T} is expository notes for a conference talk).
The statement of this theorem is limited to the genus~0 GW-invariants 
with primary insertions (i.e.~as in Theorem~\ref{main_thm})
in a ``semi-positive" setting.
The contacts in \cite[Section~10]{T} are assumed to be transverse
(i.e.~only the tuples $\bs\!=\!(1,\ldots,1)$ are considered relevant);
this is not generally the case even in a ``semi-positive" setting, 
as illustrated in Sections~\ref{CH_subs} and~\ref{Hurwitz_subs}.
The formula of \cite[Theorem~10.2]{T} is roughly the specialization of 
\cite[(5.7),(5.4)]{LR} to this simplified case and does not include \cite[(5.9)]{LR},
i.e.~the final third of the symplectic sum formula in~\cite{LR},
which explicitly splits the absolute GW-invariants of~$X\!\#_V\!Y$ into
the relative GW-invariants of~$(X,V)$ and~$(Y,V)$;
the latter had not been defined at the time of~\cite{T}.
The correct multiplicities for non-transverse contacts are suggested 
by elementary algebraic considerations, as in \cite[p938]{IPsum},
which are applied in~\cite{CaH};
the main recursion of~\cite{CaH} is recovered from~\eref{SympSumForm_e} 
in Section~\ref{CH_subs}.\\

\noindent
Theorem~\ref{main_thm} is a basic decomposition formula for GW-invariants, 
presented in the succinct style of~\cite{IPsum}.
However, it is not in any of the three standard symplectic sum papers
and is not directly implied by any formula in these papers.
The primary inputs $q_{\#}^*\al$ on the $X\#_VY$ side of~\eref{SympSumForm_e}
are of the same type as in \cite{LR,Jun2,IPsum}.
A characterization of which cohomology classes on $X\#_VY$ are of the form $q_{\#}^*\al$
is provided in~\cite{IPsum}; see \cite[Lemma~4.11]{GWrelIP}.
The identity~\eref{SympSumForm_e} is equivalent to the intended symplectic sum formula
in~\cite{LR}; unfortunately, it is spread out across several statements in~\cite{LR}
and contains some misstatements, as described in Section~\ref{DfnStatComp_subs}.
The symplectic sum formula in~\cite{IPsum}
 contains two distinct features, the $S$-matrix and rim tori refinements
of relative invariants;
the former should not be present, while the latter is never properly constructed.
Even ignoring these two features, the main symplectic sum statements in~\cite{IPsum},
(0.2) and~(10.14), do not reduce to~\eref{SympSumForm_e},
in part because of definitions that do not make sense; see Section~\ref{DfnStatComp_subs}.
The only one of the three standard symplectic sum papers which contains a correct 
version of the symplectic sum formula (even in the
basic case of primary inputs) is~\cite{Jun2}.
Unfortunately, the main decomposition formulas in~\cite{Jun2},
the two formulas at the bottom of page~201, often yield less sharp
versions of~\eref{SympSumForm_e}, as 
their left-hand sides combine GW-invariants in
the homology classes whose difference lies in a submodule of $H_2(X;\Z)$
containing (often strictly)~$\cR_{X,Y}^V$.\\

\noindent
The general symplectic sum formulas, considered in~\cite{Jun2,IPsum}
and mentioned in~\cite{LR}, involve descendant classes.
These classes effectively impose an order on the combined set of marked points 
of the limiting curve, which has to be taken into account by the 
pairing~\eref{bspairing_e2}.
This is done in~\cite{Jun2} by summing over rules of assignment~$I$
($\vt$ in the notation of Section~\ref{DfnStat_subs}).
It is stated in~\cite{LR} that the symplectic sum formula extends to descendant
invariants, without any mention of some kind of rule of assignment.
In Theorem~\ref{main2_thm}, we give a general symplectic sum
formula summing the GT-type formulas in the style of~\cite{IPsum}
over the rules of assignments of~\cite{Jun2}.
It seems impossible to condense the general symplectic sum formula into 
the format of the formulas~(0.2) and~(10.14) in~\cite{IPsum},
i.e.~the attempted formulation of the symplectic sum formulas in~\cite{IPsum}
is a beautiful idea which
unfortunately does not work as well beyond the case of primary invariants.\\

\noindent
A deficiency of the decomposition formula~\eref{SympSumForm_e} is that it expresses
sums of GW-invariants of $X\!\#_V\!Y$ over homology classes differing by elements of~$\cR_{X,Y}^V$
in terms of relative GW-invariants of $(X,V)$ and $(Y,V)$;
it would of course be preferable to express GW-invariants of $X\!\#_V\!Y$ in
each homology class in terms of relative GW-invariants of $(X,V)$ and~$(Y,V)$.
Rim tori are introduced in \cite[Section~5]{IPrel} with the aim of defining sufficiently fine
relative GW-invariants to rectify this deficiency;
they also provide a concrete way of understanding this deficiency.
Unfortunately, the construction of the refined relative GW-invariants in~\cite{IPrel} 
is only sketched and its description contains incorrect material statements;
its application in the simple cases of \cite[Lemmas~14.5,14.8]{IPsum} is also wrong.
In~\cite{GWrelIP}, we describe the intended construction of~\cite{IPrel},
explain the dependence of the refined ``invariants"
on the choices involved, and obtain some qualitative implications.
The usual relative GW-invariants, as in \cite{LR} and \cite{Jun2}, 
factor through the relative invariants of~\cite{IPrel}
and so the latter are thus indeed refinements (though not necessarily strict refinements)
of the former; see Section~\ref{RelInv_subs2}.
As explained in \cite[Section~1.2]{GWsumIP}, 
these refinements make it possible to express the GW-invariants of $X\!\#_V\!Y$
in terms of the GW-invariants $X\!\cup_V\!Y$, but generally not in terms of 
the (refined) GW-invariants of $(X,V)$ and~$(Y,V)$.
The use of the refined relative invariants in the statement of the symplectic sum formula 
in~\cite{IPsum} causes further problems, including with the definitions of 
the GT power series in \cite[Section~1]{IPsum}; see Section~\ref{DfnStatComp_subs}.\\

\noindent
As explained in \cite[Section~1.5]{GWrelIP},  
the deficiency in question is at most minor in the Kahler category 
and in many other cases.
The speculative extension of this fact to the symplectic category is 
stated~below.

\begin{cnj}\label{RimTori_cnj}
Let $(X,\om_X)$ and $(Y,\om_X)$ be symplectic manifolds and 
$(Z,\om_Z)\!=\!(X\!\#_V\!Y,\om_{\#})$ be their symplectic sum along
a symplectic hypersurface $V\!\subset\!X,Y$ satisfying~\eref{cNVcond_e}.
If $C_1,C_2\!\in\!H_2(Z;\Z)$ are such that $C_1\!-\!C_2\!\in\!\cR_{X,Y}^V$
and $\GW_{g_1,C_1}^Z,\GW_{g_2,C_2}^Z\!\neq\!0$ for some $g_1,g_2\!\in\!\Z^{\ge0}$,
then $C_1\!-\!C_2$ is a torsion class.
\end{cnj}

\noindent
Families of curves in the smooth fibers $\cZ_{\la}\!=\!X\!\#_V\!Y$ of the fibration 
$\pi\!:\cZ\!\lra\!\De$ can limit, as $\la\!\lra\!0$, to a curve in
the singular fiber $\cZ_0\!=\!X\!\cup_V\!Y$ with some components contained 
in the divisor~$V$.
The $S$-matrix in the symplectic sum formula of~\cite{IPsum} is intended to account
for such components of the limiting curves by viewing them as curves in the \textsf{rubber},
a union of a finite number of copies of 
\BE{PVdfn_e}\P_XV\equiv \P(\cN_XV\!\oplus\!\cO_V)
\approx\P(\cO_V\!\oplus\!\cN_YV)\equiv \P_YV,\EE
where $\cO_V\!\lra\!V$ is the trivial complex line bundle.
Such curves also appear as limits of relative maps into $(X,V)$ in \cite{IPrel},
but only up to the natural action of $\C^*$ on each~$\P_XV$.
For this reason, moduli spaces of such limits have lower (virtual) dimensions than 
the corresponding moduli space of smooth relative maps, after a suitable regularization,
and thus do not contribute to the relative invariants of~$(X,V)$.
By the same reasoning as in~\cite{IPrel}, the components of limits of curves in~$\cZ$
that map to~$V$ should be viewed as $\C^*$-equivalence classes of curves in~$\P_XV$; 
moduli spaces of such limits  have
lower (virtual) dimensions than the corresponding moduli space of maps without 
irreducible components contained in~$V$, after a suitable regularization, 
and thus have no effect on the symplectic sum formula.
Even without taking the $\C^*$-equivalence classes, 
the effect of the spaces of maps with non-trivial rubber components on 
the action of the $S$-matrix in the main decomposition formulas in~\cite{IPsum}, 
(0.2) and~(10.4), is to produce 0-dimensional sets (after cutting down by all possible
constraints) on which $\C^*$ acts non-trivially; these sets are thus empty.
It follows that the maps with rubber components have no effect on the action
of the $S$-matrix in the symplectic sum formulas in~\cite{IPsum} and so 
the $S$-matrix acts as if it were the identity; this is not observed in~\cite{IPsum} either.
We discuss the situation with the $S$-matrix in more detail in Section~\ref{Smat_subs}.

\subsection{Outline of the paper}
\label{outline_subs}

\noindent
Sections~\ref{IP_subs} and~\ref{LR_subs} summarize the key issues with \cite{IPrel,IPsum}
and~\cite{LR}, respectively, and direct the reader to the portions of this manuscript 
where they are discussed in more detail.
Section~\ref{SympSum_subs} reviews the symplectic sum construction of \cite{Gf,MW}
from the point of view of~\cite{IPsum};
Section~\ref{SympCut_subs} translates this description into the symplectic cut 
perspective of~\cite{Ler} used in~\cite{LR}.
In Sections~\ref{RelInv_sub0} and~\ref{RelInv_sub0b},
we recall the now-standard notions of relative stable maps and morphisms to $X\!\cup_V\!Y$
and compare them with the notions used in~\cite{IPrel,IPsum} and~\cite{LR}.
The geometric constructions of the absolute and relative GW-invariants in ``semi-positive" 
cases are the subject of Section~\ref{RelInv_subs}.
In Section~\ref{RelInv_subs2}, we summarize the substance of the rim tori refinement to the standard relative 
GW-invariants suggested in~\cite{IPrel}.
A more general version of Theorem~\ref{main_thm} is stated in Section~\ref{DfnStat_subs};
a comparison of the versions of this formula appearing in \cite{LR,Jun2,IPsum}
is presented in Section~\ref{DfnStatComp_subs}.
The topological refinement to these formulas suggested in~\cite{IPsum} is discussed
in Section~\ref{RefSymSum_subs}.
Section~\ref{sumpf_subs} reviews the arguments of~\cite{IPsum} and~\cite{LR}
that are intended to establish symplectic sum formulas and 
outlines how to complete~them.
The power of these formulas for GW-invariants is illustrated in Section~\ref{appl_sec}, 
based on the applications described in~\cite{IPsum} and~\cite{LR}.
For the reader's convenience, 
we include detailed lists of typos/misstatements in \cite{IPrel,IPsum} and~\cite{LR}.
The references in this manuscript are labeled as in~\cite{IPsum},
whenever possible.\\

\noindent
The authors would like to thank the many people in the GW-theory community
with whom they had related discussions over the past decade,
including K.~Fukaya, E.~Ionel, D.~McDuff, J.~Nelson, Y.~Ruan, G.~Tian, and R.~Wang
over the past year.
The second author is also grateful to the IAS School of Mathematics for its hospitality 
during the early stages of this~project.

\section{Summary of issues with \cite{IPrel,IPsum} and~\cite{LR}}
\label{issues_sec}

\noindent
This section summarizes our understanding of the key problems with 
\cite{IPrel,IPsum} and~\cite{LR} and directs the reader to 
the portions of this manuscript where they are analyzed in detail.
We hope that the detailed list of specific points below will make it easier 
for others to gain some mathematical understanding of the issues involved,
instead of judging this manuscript or the related papers based on feelings and hearsay.\\

\noindent
The problems in \cite{IPrel,IPsum} and~\cite{LR} are of very different nature.
The arguments in \cite{IPrel,IPsum} are generally very concrete, often highly technical, 
and aim to completely address all relevant issues, but go wrong in several crucial places
and in particular do not deal correctly with the key gluing issues
(see \ref{IPadj_it}-\ref{quad_it} below), which were the main problems 
that needed to be addressed.
In contrast, \cite{LR} attempts to adapt the beautiful idea of stretching the target
in the normal direction to the divisor~$V$, 
which had been previously used in contact geometry by others
and fits naturally with the relevant gluing issues in the symplectic sum setting.
Unfortunately, \cite{LR} makes hardly any reasonably precise statement, 
either when defining the key objects, specifying the questions to be addressed, or
proving the key claims, even in special cases (to which many sketches of the arguments
in~\cite{LR} are restricted), and does not even mention many of the issues
that need to be addressed.

\begin{rmk}[by A.~Zinger]\label{IPRL_rmk}
A link to the first version of this manuscript, which is still available~at
$$\hbox{http://www.math.sunysb.edu/$\sim$azinger/research/SympSum031414.pdf}\,,$$
was e-mailed to the authors on March~14, 2014.
This was done earlier than we would have liked in order to enable discussions
of these issues during the Simons Center workshop the following week;
unfortunately such discussions hardly happened.
A second e-mail was sent  on March~31, stating that 
I intended to post this manuscript by the following Monday.
This e-mail also stated my belief that the authors' decision to stand by their papers 
would indicate their views on the standards for the {\it Annals} and {\it Inventiones}
backed by their current standing.
The full content of both e-mails is available~at
$$\hbox{http://www.math.sunysb.edu/$\sim$azinger/research/SympSumEmails.pdf}\,.$$ 
If the issues with these papers concerned a specific gap, I would have contacted
the authors individually so that they could fix it (as I have done with 
E.~Ionel and Y.~Ruan before).
However, in the given case, I believe the issues raised leave very little of 
the argument needed to address the main problem (proof of the symplectic sum formula)
and prevented others from doing so 15~years ago;
I realize that the authors' views may be different from~mine.
I~also believe that on some fundamental level the authors had been
at least vaguely aware of the general nature of the main issues in their papers before 
publication or at least had uneasy feelings about some aspects of their arguments.
Some of the reasons for this belief are indicated in Remarks~\ref{IP_rmk} and~\ref{LR_rmk}
and the specific points listed below.
\end{rmk}

\subsection{Comments on~\cite{IPrel,IPsum}}
\label{IP_subs}

\noindent
The approach in \cite{IPrel,IPsum} to the symplectic sum formula for GW-invariants 
follows a clear, logical order.
The aim of~\cite{IPrel} is to define a notion of relative stable map into~$(X,V)$
and a topology on the space of such maps, to show that the resulting moduli space is compact,
and to construct (relative) GW-invariants of $(X,V)$, at least in ``semi-positive" cases.
The main technical part of~\cite{IPrel} is Section~6, which studies limits of sequences
of $(J,\nu)$-holomorphic maps and is the key to the compactness property of the moduli space.
The Hausdorffness of the moduli space is never considered, but it is not necessary to 
define primary GW-invariants in ``semi-positive" cases.
The aim of~\cite{IPsum} is to express the GW-invariants of $X\!\#_V\!Y$ 
in terms of the GW-invariants of~$(X,V)$ and~$(Y,V)$ as defined in~\cite{IPrel}.
This involves determining which maps into the singular fiber $\cZ_0\!=\!X\!\cup_V\!Y$
are limits of $(J,\nu)$-maps into the smooth fibers $\cZ_{\la}\!\approx\!X\!\#_V\!Y$
and in precisely how many ways.
The former is fairly straightforward, though the automorphisms of the rubber components
are not taken into account in~\cite{IPsum}.  
The hard part is the latter, which involves constructing approximately $(J,\nu)$-holomorphic maps,
obtaining uniform Fredholm estimates for their linearizations and 
uniform bounds for the associated quadratic error term,
and verifying the injectivity and surjectivity of the resulting gluing map.
Unfortunately, \cite{IPrel} and especially~\cite{IPsum} contain very little which is 
both correct and of much substance.\\

\noindent
We begin with problems of descriptional and topological flavor in~\cite{IPrel,IPsum}.
\begin{enumerate}[label=(IPt\arabic*),leftmargin=*]

\item\label{IPrelS_it} According to the abstract and summary in \cite{IPrel},
 relative GW-invariants are defined for arbitrary~$(X,\om,V)$  and
more generally than the relative GW-invariants of~\cite{LR}.
While the relative moduli spaces in~\cite{IPrel} are defined
for a wider class of almost complex structures on $(X,\om,V)$ than in~\cite{LR},
relative GW-invariants for~$(X,\om,V)$ are defined in~\cite{IPrel} only in 
a narrow range of ``semi-positive" settings, which are not specified quite correctly;
see Sections~\ref{RelInv_sub0} and~\ref{RelInv_subs}.

\item\label{IPrelVFC_it} According to the last paragraph of \cite[Section~1]{IPrel}, 
the main construction of relative GW-invariants in~\cite{IPrel} applies 
to arbitrary~$(X,\om,V)$ because of a VFC construction
{\it in a separate paper~[IP5]}, listed as {\it in preparation} 
(not {\it work in progress}) in the references.
This citation  first appeared in the 2001 arXiv version;
it replaced Remark~1.8 in the 1999 arXiv version, which claimed that
the semi-positive restriction can be removed because of the VFC construction of~\cite{LT}.
However, applying this construction would have required gluing maps with rubber components,
which is not done in~\cite{IPrel}.
The VFC construction advertised in~\cite{IPrel} is claimed in~\cite{IPvfc} by building on~\cite{CM}. 
However, \cite{CM} first appeared on arXiv almost 5.5~years after the 2001 version 
of~\cite{IPrel}.
Furthermore, for two of the most crucial analytic points, 
\cite[Lemma~7.4]{IPvfc} and \cite[(11.4)]{IPvfc},  
which require rubber gluing (see \ref{rubbglue_it} below),
the authors cite~\cite{IPrel} and~\cite{IPsum};
these two papers restrict to ``semi-positive" cases precisely to avoid such gluing.

\item According to the abstract, the long summary, and the main theorems 
in~\cite{IPsum}, i.e.~{\it Symplectic Sum Theorem} and Theorems~10.6  and~12.3,
the symplectic sum formulas in~\cite{IPsum} are proved without any restrictions 
on $X,Y,V$, but the arguments are clearly restricted
to ``semi-positive" cases;
see Section~\ref{RelInv_subs} and in particular the paragraph before Remark~\ref{SemiPos_rmk}.

\item\label{RefRel_it} Refined relative GW-invariants of $(X,V)$ are obtained in~\cite{IPrel} by 
lifting the relative evaluation morphism~$\ev^V$ in~\eref{evdfn_e} over 
a covering~$\cH_{X;\bs}^V$ of~$V_{\bs}$.
Such a covering is described set-theoretically in \cite[Section~5]{IPrel} without 
ever specifying a topology on~$\cH_{X;\bs}^V$, 
especially when the contact points come together, or showing that~$\ev^V$ actually lifts.
The description of this cover is wrong about the group of its deck transformations
and about the resulting GW-invariants in the simple cases of 
\cite[Lemmas~14.5,14.8]{IPsum};
see Section~\ref{RelInv_subs2} and \cite[Remarks~6.5,6.8]{GWsumIP}.
Furthermore, the lifts to these covers are not unique and 
the refined relative GW-``invariants" generally depend on the choice of such a lift;
see \cite[Sections~1.1,1.2]{GWrelIP}.
A standard way to specify a covering is to specify a subgroup of
the fundamental group of the~base, as is done in \cite[Sections~5.1,6.1]{GWrelIP} 
based on the informal sketch at the end of \cite[Section~5]{IPrel}.
A standard way to show that a continuous map~$\ev^V$ lifts to such a cover is to show
that the image of the fundamental group of the domain under~$\ev^V$ is contained
in the chosen subgroup, as is done in \cite[Lemma~6.3]{GWrelIP}.

\item\label{Noncpt_it} 
The refined symplectic sum formula of~\cite{IPsum} for the GW-invariants of $X\!\#_V\!Y$
involves cohomology classes on products of the covers~$\cH_{X;\bs}^V$ and~$\cH_{Y;\bs}^V$ 
that are Poincare dual to components of the fiber product of the two covers over
the diagonal in $V_{\bs}\!\times\!V_{\bs}$; see Section~\ref{RefSymSum_subs}.
As these covers are often not finite, these cohomology classes need not admit a Kunneth decomposition
into cohomology classes from the two factors; see \cite[Section~1.2]{GWsumIP}.
In such a case, the refined symplectic sum formula of~\cite{IPsum} does not express
the GW-invariants of $X\!\#_V\!Y$ in terms of any kind of numbers arising from
$(X,V)$ and~$(Y,V)$.

\item\label{Smar_it} As explained in the summary and in Section~12 in~\cite{IPsum},
the $S$-matrix appears in the main formulas~(0.2) and~(12.7) of~\cite{IPsum}
due to components of limiting maps sinking into~$V$.
As we explain in Section~\ref{Smat_subs}, such components correspond to maps
into~$\P_XV\!=\!\P_YV$ only up to the~$\C^*$-action on the target, just as happens
in the relative maps setting of \cite[Section~7]{IPrel}.
This action, which is forgotten in the imprecise limiting argument of
\cite[Section~12]{IPsum}, implies that such limits do not contribute 
to the GW-invariants of $X\!\#_VY$ for dimensional reasons,
and so the $S$-matrix should not appear in any symplectic sum formula of~\cite{IPsum}.
As we also show in Section~\ref{Smat_subs}, the $S$-matrix does not matter anyway
because it {\it acts as} the identity in all cases and not just in the cases considered in
\cite[Sections~14,15]{IPsum}, when the $S$-matrix {\it is} the identity.

\item The main symplectic sum formulas in~\cite{IPsum} involve
generating series defined by exponentiating homology classes 
on $\ov\M_{g,n}\!\times\!\cH_{X;\bs}^V$ without an explanation
of how these exponentials are defined.
The use of $\cH_{X;\bs}^V$ in place of~$V_{\bs}$ makes defining such exponentials
particularly difficult, even in the case of primary insertions
(as in Theorem~\ref{main_thm}).
If descendant insertions are also used (as in Theorem~\ref{main2_thm}),
a symplectic sum formula must incorporate some version
of rules of assignment of~\cite{Jun2}.
Finally, the normalizations of the generating series for the absolute and relative
GW-invariants in~\cite{IPsum} are not the same, which makes them incompatible with 
the stated symplectic sum formulas.
These issues are discussed in detail in Section~\ref{DfnStatComp_subs}.

\item\label{arCoh_it} The extension of the symplectic sum formula to arbitrary cohomology
insertions in \cite[Section~13]{IPsum} is not well-defined;
see Section~\ref{RefSymSum_subs}.\\

\end{enumerate}

\noindent
We next list problems of analytic nature in~\cite{IPrel,IPsum};
these concern fairly technical, but at the same time very specific, points.
\begin{enumerate}[label=(IPa\arabic*),leftmargin=*]

\item The index of the linearization of the $\dbar$-operator at a $V$-regular map~$u$
described below \cite[(6.2)]{IPrel} is lower than 
the desired index, given by  \cite[(6.2)]{IPrel},
while the index of the linearization described at the beginning 
of \cite[Section~7]{IPsum} is typically higher than 
the desired one; see Remark~\ref{SemiPos_rmk} and Section~\ref{IPunif_subs}.
As a result, a transverse claim is made about a wrong bundle section
in \cite[Section~6]{IPrel}.

\item\label{2rubb_it} The rescaling arguments of \cite[Sections~6,7]{IPrel}  
do not involve adding new components to the domain of a map to~$X$.
They cannot lead to limiting maps such that some of the component maps 
into a rubber level are stable and some are unstable;
see Remark~\ref{IPrel67_rmk}.

\item\label{Haus_it} It is neither shown nor claimed that the relative moduli 
$\ov\M_{g,k}^V(X,A)$ constructed in \cite[Section~7]{IPrel} is Hausdorff.
This is not relevant for the pseudocycle construction of GW-invariants 
in the ``semi-positive" cases considered in~\cite{IPrel}, but 
is a useful property of $\ov\M_{g,k}^V(X,A)$ for other applications.
With the notion of relative stable map described by \cite[Definitions~7.1,7.2]{IPrel},
this space is not even Hausdorff; see Remark~\ref{IPrel67_rmk}.

\item The gluing constructions of \cite[Sections~6-9]{IPsum} claim uniform  
estimates along each stratum, which are not established even when 
restricting to $\de$-flat maps.
The first failure of uniformity occurs on the level of curves,
essentially because the construction above \cite[Remark~4.1]{IPsum}
need not extend outside of the open strata~$\cN_{\ell}$;
see Remark~\ref{IPsumSec4_rmk} for more details.
The second failure occurs on the level of maps because
the extra bubbling can occur away from the nodes on the divisor
and because the construction requires stabilizing the domains  as in 
\cite[Remark~1.1]{IPsum}, which can be done only locally.
The statement about the linearized operator being Fredholm for a generic~$\de$
in the second paragraph of page~976 in~\cite{IPsum}
pretty much rules out any possibility for uniform estimates
across whole strata.
However, such uniform estimates along entire strata are not necessary and seem 
unrealistic especially in situations requiring a virtual fundamental class construction,
while uniform estimates along compact subsets of open strata are much easier to establish.
This implies that the top arrow in \cite[(10.3)]{IPsum} is defined only 
after restricting to the preimage of a compact subset~$K$ and for $\la$ sufficiently
small (depending on~$K$);
see Remarks~\ref{IPsumSec4_rmk} and~\ref{IPsumSec6_rmk}.

\item\label{C0bnd_it} The uniform control of the $C^0$-norm by the $L^p_1$-norm 
claimed in \cite[Remark~6.6]{IPsum} requires a justification
because the domains~$C_{\mu}$ change (which is not an issue) and 
the metric on the targets~$\cZ_{\la}$ degenerates;
see Remark~\ref{IPsumSec6_rmk}.

\item  The proof of \cite[Lemma 6.9]{IPsum} ignores two of the three components 
of the map $F\!-\!f$ as in \cite[(6.14)]{IPsum}.
The actual estimate is weaker, but good enough;
see~\eref{nonJnu_e} in Section~\ref{preglue_sec}.

\item\label{IPadj_it} The operator in \cite[(7.5)]{IPsum} is not the adjoint of 
the operator in \cite[(7.4)]{IPsum} with respect to any inner-product,
because the first component of its image does not satisfy the average condition.
This  ruins the argument regarding the linearized operators being uniformly 
invertible, which is the main point of the analytic part of~\cite{IPsum},
at the start; see Section~\ref{IPunif_subs}.

\item\label{Gauss_it} Gauss's relation for curvatures, \cite[(8.7)]{IPsum}, is written
in a rather peculiar way, resulting in a sign error.
This appears to be what is referred to as a Bochner formula on
page~939 of~\cite{IPsum}.
The sign error in \cite[(8.7)]{IPsum}
is crucial to establishing a uniform bound on the incorrect adjoint 
operator in \cite[(7.5)]{IPsum};
see Section~\ref{IPunif_subs}.

\item The argument at the bottom of page~984 in \cite{IPsum}
implicitly presupposes that the limiting element~$\eta$ lies in the Sobolev 
space~$L_{\bs}^{1,2}$; see Section~\ref{IPunif_subs}.

\item The justification for the uniform elliptic estimate in \cite[Lemma~8.5]{IPsum}
indicates why the degeneration of the domains does not cause a problem,
but makes no comment about the degeneration of the target.
It is unclear that it is in fact uniform; 
see Section~\ref{IPunif_subs}.

\item The map $\Phi_{\la}$ in \cite[Proposition~9.1]{IPsum} appears to be
non-injective  because the metrics on the target~$\cZ_{\la}$ collapse 
in the normal direction to the divisor~$V$ as $\la\!\lra\!0$.
The wording of the second-to-last paragraph on page~938 suggests that 
the norms are weighted to account for this collapse and 
the convergence estimate of \cite[Lemma~5.4]{IPsum}  
could accommodate norms weighted heavier in the vertical direction,
but the rather light weights in the norms of \cite[Definition~6.5]{IPsum}
appear far from sufficient.
We discuss this issue in Section~\ref{glue_sec}.

\item\label{quad_it} Neither the summary of~\cite{IPsum} nor the proof of \cite[Proposition~9.4]{IPsum} 
makes any mention of whether the quadratic error term in the expansion
\cite[(9.10)]{IPsum} of the $\dbar$-operator is uniformly bounded.
The latter mentions only the need for the 0-th and 1-st order terms to be uniform
(in (a) and (b) on page~939).

\item\label{rubbglue_it} 
In order to define relative invariants and prove a symplectic sum formula
without any semi-positivity restrictions via known techniques,
it is necessary to describe a gluing procedure for maps involving rubber components;
see Section~\ref{RelInv_sub0b}.
This involves two issues not encountered in gluing rubber-free maps into~$X\!\cup_V\!Y$:
\begin{enumerate}[label=(RG\arabic*),leftmargin=*]
\item\label{RG1_it} the component maps into each rubber level are defined only up to~$\C^*$-action;
\item\label{RG2_it} the natural generalization of the gluing construction for maps to~$X\!\cup_V\!Y$
would send maps with rubber to an isomorphic, but not identical, space
(see Section~\ref{glue_sec}).
\end{enumerate}
Such a gluing would be much harder to carry out with the almost complex structures
in \cite{IPrel,IPsum} than with the more restricted ones in~\cite{LR};
\cite{IPsum} fails to do so even in the much simpler case of maps to~$X\!\cup_V\!Y$
with no components mapped to~$V$. 
Since \cite{IPrel} and \cite{IPsum} are restricted to the semi-positive case, 
the issues~\ref{RG1_it} and~\ref{RG2_it} do not need to arise.
However, because of~\ref{Smar_it}, gluing of  maps to rubber still needs
to be considered, and so the second issue  still arises.
\end{enumerate}

\begin{rmk}[by A.~Zinger]\label{IPt_rmk}
Regarding~\ref{Smar_it}, E.~Ionel feels that the limiting
argument in \cite[Section~12]{IPsum} is correct;
she also feels that she can renormalize
the collapse  so that the maps converge to a slice of the $\C^*$-action on the rubber.
I~believe these two statements, which were made during a long discussion 
in D.~McDuff's office on 03/26/14, are contradictory as the dimension of
the slice is smaller than the dimension of all maps.
\end{rmk}

\begin{rmk}[by A.~Zinger]\label{IP_rmk}
The first version of~\cite{LR} appeared on arXiv almost 3~months before
the first version of~\cite{IPmrl},
which is a brief announcement of relative GW-``invariants" and a symplectic sum formula, 
almost 1.5~years before the first version of~\cite{IPrel}, and 
over 2.5~years before the first version  of~\cite{IPsum}.
In particular, the announcement~\cite{IPmrl} appears to have been very premature 
(many people believe that a complete proof of a claimed result must 
appear within 6~months).
Furthermore, \cite[Section~2]{IPmrl} does not impose the last two conditions 
of \cite[Definition~3.2]{IPrel} on~$(J,\nu)$;
it only imposes the obvious conditions~$J(TV)\!=\!TV$ and~\eref{nuVrestr_E}.
This results in a relative moduli space with a codimension~1 boundary   and
the chamber dependence of the relative GW-``invariants" appearing in 
\cite[Theorem~2.5]{IPmrl};
this dependence is accidentally mentioned even in \cite[Definition~11.3]{IPsum}.
Unlike E.~Ionel and T.~Parker's claim to have a proof of 
the Gopakumar-Vafa super-rigidity a decade ago,
a proof of the symplectic sum formula was a purely technical, 
even if non-trivial, problem with all of the necessary tools available and 
gathered in~\cite{LR} 
(the clever argument in~\cite{IPgv}, which may be basically correct, bypasses 
the super-rigidity problem, which is yet to be established).
Another related example of a premature claim is the citation~[IP5] in~\cite{IPrel}
(including the 2001 version),
which appeared as~\cite{IPvfc}, also prematurely and almost preempting a
PhD thesis.
The applications appearing in~\cite{IPsum} are very nice, but
are not new (as stated in the paper).
The exposition in \cite{IPrel,IPsum} is more geometric and easier to follow
than in~\cite{LR}.
As for the content, the rescaling procedure in \cite[Section~6]{IPrel} is 
a (more geometric) reformulation of the stretching procedure in \cite[Section~3.2]{LR}
for a $(J,\nu)$-approach in the style of \cite{RT1,RT2}
with suitable restrictions on~$(J,\nu)$.
Just as in~\cite{LR}, the relative moduli space is not shown to be Hausdorff
(see~\ref{Haus_it}).
Unlike~\cite{LR}, \cite{IPsum} attempts to address all of the crucial gluing issues,
but does not deal successfully with any of the key ones.
The more general pairs~$(J,\nu)$ allowed in \cite{IPrel,IPsum} cause major problems 
in this regard, especially if rubber components are involved
(see~\ref{rubbglue_it} above);
it is unclear to me that they can be overcome with a reasonable effort
in the setting of~\cite{IPrel,IPsum}.
The \'a priori estimates in \cite[Sections~3-5]{IPsum} appear essentially correct,
but these are pretty minor statements in themselves.
The formulations and implications of two of the distinguishing topological features in~\cite{IPrel,IPsum},
the refined relative GW-invariants and distinguishing GW-invariants 
in classes differing by vanishing cycles, are stated very vaguely and often incorrectly;
see~\ref{RefRel_it} and~\ref{Noncpt_it}.
The other two distinguishing topological features,
extensions to arbitrary primary insertions and the $S$-matrix, are simply wrong;
see~\ref{Smar_it} and~\ref{arCoh_it}.
The common problem behind the most crucial analytic errors, such as~\ref{IPadj_it}
and~\ref{Gauss_it}, seems to be  the confusing way in which 
\cite[Sections~7,8]{IPsum} are written.
In particular, the equation \cite[(8.7)]{IPsum} containing the crucial 
sign error (see \ref{Gauss_it}) is written in a very complicated way;
while the sign error very well might not have been intentional,
\cite[(8.6)]{IPsum} should have raised concerns (its right-hand side appears to have
the potential to go negative if \cite[(8.8)]{IPsum} were correct).
It is my understanding from a conversion with E.~Ionel on 03/28/14 
that they are now trying to redo the gluing with the almost complex structures 
of~\cite{LR} based on my explanations on 03/21/14 and in the first version of
this manuscript.
This would only reaffirm my point that the correct parts of \cite{IPrel,IPsum} 
contain little beyond what is in~\cite{LR}, even in  ``semi-positive" cases.
It seems clear that the referee for~\cite{IPsum} did not even read past the long summary,
which should have been obvious to the handling editor, 
and both should publicly acknowledge if this was roughly the case.
While the referee for~\cite{IPrel} apparently noticed the issue with
Remark~1.8 in the 1999 version (see~\ref{IPrelVFC_it}),
he clearly missed a number of crucial, more technical issues,
was misled into believing that the construction applied 
outside of semi-positive cases (see~\ref{IPrelS_it} and~\ref{IPrelVFC_it}),
and was apparently swayed by~\cite{IPrel} being part of a package with~\cite{IPsum}.
\end{rmk}

\subsection{Comments on~\cite{LR}}
\label{LR_subs}

\noindent
About half of~\cite{LR}, i.e.~Sections 1,2, and~6, is devoted to applications 
and the general setting.
While this part could have been written a lot more efficiently,
it appears to be solid content-wise. 
The remainder of~\cite{LR}, just 43 lightly written journal pages,
is organized in a rather haphazard way, in contrast to \cite{IPrel,IPsum}, and 
purports to establish the compactness and Hausdorffness of the relative moduli space,
define relative GW-invariants via a new virtual cycle construction,
and address all of the gluing issues needed to prove a symplectic sum formula for GW-invariants;
see the beginning of Section~\ref{IP_subs}.
It begins with the symplectic cut construction of~\cite{Ler} and introduces 
relative GW-invariants as being associated to such a cut
(instead of a pair $(X,V)$ as in~\cite{IPrel}).
It then discusses fairly straightforward points concerning convergence to periodic orbits
in an overly complicated way and then barely touches on the main analytic issues.
Crucially, the notion of stable morphism into $X\!\cup_V\!Y$ introduced in
\cite[Definition~3.18]{LR} does not describe limits of maps into smooth fibers,
as needed for the purposes of establishing a symplectic sum formula. 
In summary, \cite{LR} does not contain anything resembling a proof of a symplectic sum
formula for GW-invariants.\\

\noindent
The issues in \cite{LR} include the following.
\begin{enumerate}[label=(LR\arabic*),leftmargin=*]

\item The symplectic sum formula (for primary invariants only)
in~\cite{LR} is spread out between three formulas in Section~5,
one of which is incorrect as stated;  
see Section~\ref{DfnStatComp_subs}.

\item Definition~3.14 in~\cite{LR} of the key notion of relative stable map
is not remotely precise. 
For example, it is not specific about the relation between the three different domains
of the map or the equivalence relation; see Section~\ref{RelInv_sub0}.

\item In addition to being imprecise, 
Definition~3.18 in~\cite{LR} of the key notion of stable map to 
$X\!\cup_V\!Y$ ($\ov{M}^+\!\cup_D\!\ov{M}^-$ in the notation of~\cite{LR}) 
is not suitable for the intended purposes, 
as it separates  the rubber components into $X$ and~$Y$-parts;
see Section~\ref{RelInv_sub0b}.

\item\label{Mor_it} The proof of \cite[Proposition~3.4]{LR} is based on an infinite-dimensional version 
of the Morse lemma, for which no justification or citation is provided.
The desired conclusion of this Morse lemma involves the inner-product \cite[(3.14)]{LR}
with respect to which the domain $W^2_r(S^1,SV)$ is not even complete. 

\item\label{Thm37_it} The statement of \cite[Theorem~3.7]{LR} is incorrect.
It describes the asymptotic behavior of $J$-holomorphic maps from~$\C$,
but what is needed to establish compactness in \cite[Section~3.2]{LR} 
and pregluing estimates in \cite[Section~4.1]{LR} is its analogue for maps from 
the punctured disk.
The 4-5 page justification of \cite[Theorem~3.7]{LR}, 
which is one of only three somewhat technical arguments in the paper,
includes \cite[Proposition~3.4]{LR} and circular reasoning.
The correct, required version can be justified in a few lines
and the elaborate sup energy of~\cite{H} can be avoided in the present situation;
see Section~\ref{IPconv_subs}.

\item The compactness argument of \cite[Section~3.2]{LR} is vague on the targets
of the relevant sequences of maps and does not even consider marked maps.
It also involves one node at a time and thus does not lead to 
the kinds of maps described in~\ref{2rubb_it} either.

\item\label{peri_it} In~(3) of the proof of \cite[Lemma~3.11]{LR}, 
the horizontal distance bound \cite[(3.55)]{LR} is used (incorrectly)
to draw a conclusion about the vertical distance in the last equation;
in contrast to the setting in \cite{H,HWZ1}, 
the horizontal and vertical directions in the setting of~\cite{LR}
are not tied together.

\item\label{nonco_it} The statement of \cite[Lemma~3.12]{LR} explicitly rules out 
``contracted" rubber maps from stable domains with only one puncture/node 
at one of the divisors.

\item\label{LRHaus_it} The moduli spaces of relative maps and of maps to $X\!\cup_V\!Y$
are implicitly claimed to be Hausdorff in \cite[Lemmas~4.2,4.4]{LR}.
For a proof, the reader is referred to~\cite{R5},
which does not deal with maps to varying targets.

\item\label{RG_it} The rubber gluing issues, (RG1) and~(RG2) above, 
are not addressed in~\cite{LR} either, even in 
the special, one-node, case considered in \cite[Section~4.1]{LR}. 
The gluing construction of \cite[Section~4.1]{LR} for relative maps
involves a specific representative of a map to the rubber 
(not up to the $\C^*$-action on the target) and 
defines the target of the glued map in a way which depends on the gluing parameter.
These issues are fundamental to~\cite{LR}, in contrast to~\cite{IPsum},
because the former does not impose any semi-positivity conditions.
We discuss them in more detail in Sections~\ref{RelInv_subs} and~\ref{preglue_sec}.

\item\label{InjSur_it} Neither the injectivity nor surjectivity of the gluing construction
of \cite[Section~4.1]{LR} is even mentioned;
in light of~\ref{RG_it}, this would be impossible to do.
Some version of \cite[Sections~4,5]{IPsum} is a necessary preliminary 
to handle these issues.
Both properties are implicitly used in the proof of 
\cite[Proposition~4.10]{LR}.

\item The proof of \cite[Proposition~4.10]{LR} applies the Implicit Function Theorem
in an infinite-dimensional setting without any mention of the needed bounds
on the 0-th and 1-st order terms and the quadratic correction term.
The first two are the subject of the preceding section, but there is no mention
of uniform estimates on the last one anywhere in~\cite{LR};
see Section~\ref{glue_sec}.

\item The VFC approach of \cite{LR} is based on a global regularization of 
the moduli space using the twisted dualizing  sheaf  introduced after 
\cite[Lemma~4.4]{LR}.
It is treated as a line bundle over the entire moduli space with Sobolev norms
on its sections, without any explanation.
The 3-4 pages dedicated to this line bundle in \cite[Sections~4.1,4.2]{LR}
could be avoided by using the local VFC approach of~\cite{FO} or~\cite{LT}.

\item The regularization of maps in \cite[Sections~4.1,4.2]{LR} needs to respect
the $\C^*$-action on maps to the rubber;
this issue is not even mentioned in~\cite{LR}.

\item\label{LRsum_it} The discussion of gluing for maps to $X\!\cup_V\!Y$,
which is needed to establish a symplectic sum formula,
consists of a few lines after \cite[Lemma~5.4]{LR}.
There is no explanation of the crucial multiplicity coefficient~$k$
($\lr\bs$ in our notation) appearing in \cite[Theorem~5.7]{LR}.
The domain and target gluing formulas \cite[(4.12)-(4.15)]{LR} hint
at this coefficient, but barely so even in the case of one node.
If the rubber components are present, these multiplicities no longer show up directly;
the argument in~\cite{Jun2} obtaining them on the level of homology classes
(rather than numbers) is pretty delicate and involves passing to a desingularization.
Because of the much more limited scope of~\cite{IPsum}, this issue is not relevant for~\cite{IPsum}. 
In contrast to~\cite{IPsum}, \cite{LR} does not even clearly describe 
the general setup.
In particular, the one-node case considered in~\cite{LR} as supposedly capturing
all the issues in the general case cannot be representative of the general
case because the target of the glued maps, described by \cite[(4.12),(4.13)]{LR}, 
depends on the gluing parameter associated with each node.
Thus, these parameters must be chosen systematically, 
which is done for rubber-free maps in~\cite{IPsum} and becomes  more complicated
for general maps; see Section~\ref{preglue_sec}.

\item\label{pscyc_it} The most technical part of~\cite{LR}, roughly 4 pages, concerns 
the variation of various operators in Section~4.1 with respect to the norm~$r$ of
the gluing parameter~$(r)$,
which is considered without explicitly identifying the domains and targets of these operators.
This part is used only to show that the integrals \cite[(4.50)]{LR} defining
relative invariants converge.
However, this is not necessary, since the relevant evaluation morphisms had supposedly 
been shown to be rational pseudocycles before then
(and thus define invariants by intersection as in \cite[Section~7.1]{MS2}
and \cite[Section~1]{RT1}).

\end{enumerate}

\begin{rmk}[by A.~Zinger]\label{LR_rmk}
The applications in~\cite{LR} appear well justified and are new (in contrast to~\cite{IPsum}).
However, they are relatively minor, can be handled without a symplectic sum formula
(as can be seen from the clever geometric argument in~\cite{LR}),
and are special cases of~\cite{HLR}.
I~believe the SFT type idea behind the main argument in~\cite{LR} can be used to prove
the symplectic sum formula in an efficient manner.
If such a radical idea had actually been introduced in~\cite{LR}, it would have been
a very clear contribution appearing in the originally published version;
the issues listed above could then have been viewed as some gaps to be filled
(though still very significant ones).
However, this idea already appears in~\cite{H}, \cite{HWZ1}, and perhaps other works
from that time.
Essentially the only contributions of~\cite{LR} in terms of formulating and proving 
a symplectic sum formula are the notion of relative stable map and
an attempt to adapt an SFT idea to symplectic cuts,
and even this is done pretty poorly.
Based on \cite[Section~3.1]{LR}, the authors appear to have only vague understanding
of what is actually needed in the symplectic cut setting and follow~\cite{H} and~\cite{HWZ1}
very closely; see~\ref{Mor_it}, \ref{Thm37_it}, and~\ref{peri_it} above.
With some imagination and knowledgeable help, the intended stretching construction of
\cite[Section~3.2]{LR} can be understood and 
the desired definition of relative stable map of \cite[Definition~3.14]{LR}
could then be deduced.
However, the Hausdorffness of the resulting relative moduli space is not even mentioned
and the justification of compactness has multiple issues;
see \ref{Mor_it}-\ref{nonco_it} above.
The most serious problem with~\cite{LR} in my view is that
many crucial issues arising in 
the most important part, i.e.~gluing, which is the subject of \cite[Section~4]{LR},
are not even mentioned.
It seems unlikely to me that neither of the authors saw that 
some of these issues, such as \ref{LRHaus_it}, \ref{RG_it},
\ref{InjSur_it}, and~\ref{LRsum_it}, needed to be addressed (or at least commented~on);
if this is indeed the case, then they did not have a reasonable understanding of the problem 
at hand.
The entire content of~\cite{LR} could be compressed down to 25-30 journal pages
(at the density similar to~\cite{RT1} or \cite{IPrel,IPsum}, for example);
in contrast \cite{Jun1} and~\cite{Jun2} are 165~pages together and build
on a hundred years of algebraic geometry.
In summary, it does not appear to~me that \cite{LR} either introduces a fundamentally new idea
or comes remotely close to technically justifying a symplectic sum formula.
It seems that the referee for~\cite{LR} only complained about the length 
of the original version, read through Sections~1,2,6, and barely looked at 
the crucial Sections~3,4 (which are actually fairly easy to read).
\end{rmk}

\begin{rmk}[by A.~Zinger]\label{LRresp_rmk}
The response~\cite{AMLi} to the first arXiv version of this manuscript contains
little of mathematical substance and indicates that the author is still unfamiliar
with  GW-theory.
For example, \ref{LRHaus_it}, \ref{RG_it}, and~\ref{InjSur_it}
are supposedly non-issues because they are completely standard.
The same claims are made about the relative moduli spaces, either explicitly or implicitly,
in the first three arXiv versions of~\cite{LR}.
However, in these three versions, 
the equivalence relation on the relative maps does not involve
 the $\C^*$-action on the rubber components (only~$\R$-action, corresponding to the log of the norm);
see the middle of page~63 and Definition~4.16 in the third version, for example. 
These properties cannot be satisfied by both versions of the relative moduli space at the same time. 
Crucially, A.-M.~Li explicitly acknowledges that the notion of stable morphism to 
the singular fiber $\cZ_0\!=\!X\!\cup_V\!Y$ introduced in \cite[Definition~3.18]{LR}
was not an unintentional misstatement, but sees no fundamental problem with~it.
Detailed comments on his response are available on my website.
\end{rmk}

\begin{rmk}[by A.~Zinger]
The primary purpose of this manuscript is an exposition on the symplectic sum formula
and the literature regarding this topic;
it is only natural for such an exposition to include a thorough review of 
the relevant literature.
Unfortunately, it has become normal in mathematics to criticize papers and 
to undermine their authors behind their backs, often without even reading their papers;
this is wrong and creates lingering tension. 
I hope to minimize such lingering tension by listing specific issues with
specific papers in a way that makes it relatively easy for others to judge 
the substance of the concerns raised and for the authors to dispute them.
If any factual statement (as opposed to an opinion)
made in this manuscript is pointed out to me as incorrect,
I will change it (updated versions will be posted on my website).
I also believe that papers, especially for the {\it Annals} and {\it Inventiones},
should be judged on their significance and correctness,
not their authors' status or likability.
While the significance criterion will always remain subjective, it is clear to me that 
the correctness criterion was not applied diligently to either~\cite{LR} 
or~\cite{IPrel,IPsum}.
\end{rmk}

\section{Preliminaries}
\label{RimTori_sec}

\noindent
We review the symplectic sum construction of~\cite{Gf,MW} from
the point of view of~\cite{IPsum} in Section~\ref{SympSum_subs}.
In Section~\ref{SympCut_subs}, we describe the symplectic cut perspective of~\cite{LR}.
The former is more geometric and leads to a simpler description of 
the key notions of relative stable map and relative moduli space.
On the other hand, the latter fits better with the analytic issues 
that need to be addressed in proving a symplectic sum formula for GW-invariants;
unfortunately, \cite[Sections~2,3.0]{LR} do not actually specify a symplectic sum,
but instead provide plenty of related examples of symplectic quotients.
The symplectic manifolds $(\ov{M}_-,\om_-)$, $(\ov{M}_+,\om_+)$, and $(M,\om)$
in~\cite{LR} correspond to $(X,\om_X)$, $(Y,\om_Y)$, and  $(Z,\om_{\#})$,
respectively, in our notation (which is similar to that in~\cite{IPsum});
the hypersurface $\wt{M}\!\subset\!M$ along which~$M$ is split into its 
parts is denoted by~$SV$ below.

\subsection{The symplectic sum}
\label{SympSum_subs}

\noindent
Suppose $V$ is a manifold and $\pi_{\cN}\!:(\cN,\fI_{\cN})\!\lra\!V$ is a complex line bundle.
Let $(g_{\cN},\na^{\cN})$ be a \textsf{Hermitian structure} on  $(\cN,\fI_{\cN})$,
i.e.~a metric and a connection on~$\cN$ such~that 
\begin{gather*}
g_{\cN}(\fI_{\cN}v,w)=\fI\, g_{\cN}(v,w)=-g_{\cN}(v,\fI_{\cN}w)\quad 
\forall~v,w\!\in\!\cN_x,~x\!\in\!V, \\
\na^{\cN}\big(\fI_{\cN}\xi\big)=\fI_{\cN}\na^{\cN}\xi,~~
\nd\big\{g_{\cN}(\xi,\ze)\big\}=g_{\cN}\big(\na^{\cN}\xi,\ze\big)
+g_{\cN}\big(\xi,\na^{\cN}\ze)\quad\forall~\xi,\ze\!\in\!\Ga(V;\cN).
\end{gather*}
Let 
$$\rho_{\cN}\!:\cN\lra\R, \qquad \rho_{\cN}(v)=g_{\cN}(v,v)=|v|^2,$$
be the square of the norm function, 
$q_{\cN}\!:S\cN\!\lra\!V$ be the sphere (circle) bundle of~$\cN$, and 
$$T^{\vrt}(S\cN)\equiv\ker\nd q_{\cN}\subset T(S\cN)$$
be  its \textsf{vertical tangent bundle}.
The connection~$\na^{\cN}$ in~$\cN$ induces a splitting of the exact sequence
$$0\lra T^{\vrt}(S\cN)\lra T(S\cN)\stackrel{\nd q_{\cN}}{\lra} q_{\cN}^*TV\lra0 $$
of vector bundles over~$S\cN$; see \cite[Lemma~1.1]{anal}.
Denote by $\al_{\cN}$ the 1-form on $S\cN$ vanishing~on 
the image of $q_{\cN}^*TV$ in $T(S\cN)$ 
corresponding to this splitting such~that 
$$\al_{\cN}\bigg(\frac{\nd}{\nd\th}\ne^{\fI\th}v\bigg|_{\th=0}\bigg)=1
\quad\forall\,v\in S\cN.$$
We extend it to a 1-form on $\cN\!-\!V$ via the radial retraction
$$\cN\!-\!V\lra S\cN, \quad v\lra\frac{v}{|v|}\,.$$
The 1-form $\rho_{\cN}\al_{\cN}$ is then well-defined and smooth on the total space of $\cN\!\lra\!V$.\\

\noindent
If in addition  $\om_V$ is a symplectic form on~$V$ and $\ep\!\in\!\R$, the 2-form 
\BE{omepsdfn_e}\om_{\cN,V}^{(\ep)}\equiv \pi_{\cN}^*\om_V
+\frac{\ep^2}2 \nd\big(\rho_{\cN}\al_{\cN}\big)\EE
on  the total space of $\cN$ is closed and $\om_{\cN,V}^{(\ep)}|_{TV}\!=\!\om_V$.
If $V$ is compact, there exists $\ep_{\cN}\!\in\!\R^+$ such that the restriction~of
$\om_{\cN,V}^{(\ep)}$ to
$$\cN(\de)\equiv \big\{v\!\in\!\cN\!:\,|v|\!<\!\de\big\}$$
is symplectic whenever $\de,\ep\!\!\in\!\R^+$ and  $\de\ep\!<\!\ep_{\cN}$.\\

\noindent
Suppose $(X,\om_X)$ is a symplectic manifold and $V\!\subset\!X$ is a  symplectic hypersurface. 
Let $\om_X|_{\cN_XV}$ be the induced symplectic form on the normal bundle~$\cN_XV$  of~$V$ in~$X$
as in~\eref{cNXVsymp_e}.
We will call a (fiberwise) complex structure~$\fI_X$ on~$\cN_XV$ \textsf{$\om_X$-compatible}
if $\fI_X$ is compatible with~$\om_X|_{\cN_XV}$, i.e.
$$ \om_X|_{\cN_XV}\big(\fI_Xv,\fI_Xw\big)=\om_X|_{\cN_XV}\big(v,w\big)
\quad \forall~v,w\!\in\!\cN_XV|_x,~x\!\in\!V.$$
For an $\om_X$-compatible complex structure~$\fI_X$ on~$\cN_XV$,
we will call a Hermitian structure $(g_X,\na^X)$ on $(\cN_XV,\fI_X)$ \textsf{$\om_X$-compatible}
if $g_X$ is  compatible with~$\om_X|_{\cN_XV}$ and~$\fI_X$, i.e.
\BE{gXomX_e}g_X(v,w)=\om_X|_{\cN_XV}\big(v,\fI_Xw\big) \quad \forall~v,w\!\in\!\cN_XV|_x,~x\!\in\!V;\EE
this requirement specifies~$g_X$.
The spaces of (fiberwise) $\om_X$-compatible complex structures on~$\cN_XV$ and 
of $\om_X$-compatible Hermitian structures on $(\cN_XV,\fI_X)$ are non-empty and contractible.\\

\noindent
For example, $V$ is a symplectic hypersurface in a neighborhood~$X$ of $V$ in a Hermitian line bundle
$\cN\!\lra\!V$ with respect to the symplectic form~\eref{omepsdfn_e},
$$TV^{\om_{\cN,V}^{(\ep)}}=  T^{\vrt}\cN\big|_V \approx\cN_XV,
\qquad\hbox{and}\qquad
\om_{\cN,V}^{(\ep)}\big|_{\cN_XV}= \frac{\ep^2}2 \nd\big(\rho_{\cN}\al_{\cN}\big)\big|_{\cN_XV}.$$
The original complex structure~$\fI_{\cN}$ on~$\cN$ is $\om_{\cN,V}^{(\ep)}$-compatible,
while the original Hermitian structure $(g_{\cN},\na^{\cN})$ is $\om_{\cN,V}^{(1)}$-compatible.\\

\noindent
For the remainder of this section, let $(X,\om_X)$ and $(Y,\om_Y)$ be compact symplectic manifolds
and $V\!\subset\!X,Y$ be a symplectic hypersurface so that~\eref{cNVcond_e} holds.
Denote  by~$\om_V$ the symplectic form $\om_X|_V\!=\!\om_Y|_V$ on~$V$.
Fix (fiberwise) complex structures~$\fI_X$ and~$\fI_Y$ on the normal bundles
$$\pi_{X,V}\!:\cN_XV\lra V \qquad\hbox{and}\qquad \pi_{Y,V}\!:\cN_YV\lra V$$
of $V$ in~$X$ and $V$ in~$Y$ that are compatible with $\om_X$ and $\om_Y$, respectively.
Choose an isomorphism~$\Phi$ as in~\eref{cNpair_e} compatible with $\om_X$ and~$\om_Y$,
i.e.~so that 
\BE{omPhicomp_e} 
\big|\Phi_2(v\!\otimes_{\C}\!w)\big|^2=\om_X|_{\cN_XV}(v,\fI_Xv)\cdot \om_Y|_{\cN_YV}(w,\fI_Yw)
\quad\forall~v\!\in\!\cN_XV|_x,~w\!\in\!\cN_YV|_x,~x\!\in\!V,\EE
where $\Phi_2$ is the composition of $\Phi$ with the projection $V\!\times\!\C\!\lra\!\C$.
Since $(\cN_XV,\fI_X)$ and~$(\cN_YV,\fI_Y)$ are of rank~1,
\eref{omPhicomp_e} can be achieved by scaling any given isomorphism~$\Phi$ in~\eref{cNpair_e};
this does not change the homotopy class of~$\Phi$.\\

\noindent
Choose Hermitian structures $(g_X,\na^X)$ on $(\cN_XV,\fI_X)$ and 
$(g_Y,\na^Y)$ on $(\cN_YV,\fI_Y)$ that are compatible with $\om_X$ and $\om_Y$, 
in the sense described above, and with~$\Phi$, in the sense that 
\begin{alignat}{2}
\label{omPhicomp_e2a}
\big|\Phi_2(v\!\otimes_{\C}\!w)\big|^2 &=\rho_X(v)\cdot \rho_Y(w)
&\quad &\forall~v\!\in\!\cN_XV|_x,~w\!\in\!\cN_YV|_x,~x\!\in\!V,\\
\label{omPhicomp_e2b}
\nd\big\{\Phi_2(\xi\!\otimes_{\C}\!\ze)\big\} &
=\Phi_2\big((\na^X\xi)\!\otimes_{\C}\!\ze\big)+\Phi_2\big(\xi\!\otimes_{\C}\!(\na^Y\ze)\big)
&\quad &\forall~\xi\!\in\!\Ga(V;\cN_XV),~\ze\!\in\!\Ga(V;\cN_YV).
\end{alignat}
The metrics~$g_X$ and~$g_Y$ are determined by $(\om_X,\fI_X)$ and $(\om_Y,\fI_Y)$
via~\eref{gXomX_e};
they satisfy~\eref{omPhicomp_e2a} by the assumption~\eref{omPhicomp_e}.
The choice of~$\na^X$ and \eref{omPhicomp_e2b} determine~$\na^Y$,
which is compatible with~$\fI_Y$ and~$g_Y$.\\

\noindent
Denote by $\al_X$ and $\al_Y$ the connection 1-forms on $\cN_XV\!-\!V$ and $\cN_YV\!-\!V$
corresponding to $(g_X,\na^X)$ and~$(g_Y,\na^Y)$, respectively.
For $\ep\!\in\!\R$, define
$$\om_{X,V}^{(\ep)}=\pi_{X,V}^*\om_V+\frac{\ep^2}2\nd(\rho_X\al_X)
\qquad\hbox{and}\qquad
 \om_{Y,V}^{(\ep)}=\pi_{Y,V}^*\om_V+\frac{\ep^2}2\nd(\rho_Y\al_Y)\,.$$
The 2-forms $\om_{X,V}^{(1)}$ and $\om_{Y,V}^{(1)}$
restrict to $\om_X$ and $\om_Y$ on $T(\cN_XV)|_V$ and $T(\cN_YV)|_V$
under the isomorphisms as in~\eref{cNXVsymp_e}.
By the Symplectic Neighborhood Theorem \cite[Theorem~3.30]{MS1},
there thus exist $\de_V\!\in\!\R^+$ and smooth injective open~maps
$$\Psi_X\!: \big(\cN_XV(\de_V),V\big)\lra (X,V) \qquad\hbox{and}\qquad
\Psi_Y\!: \big(\cN_YV(\de_V),V\big)\lra (Y,V)$$
such that 
$$\nd_x\Psi_X,\nd_x\Psi_Y=\id~~~\forall~x\!\in\!V, \qquad
\Psi_X^*\om_X=\om_{X,V}^{(1)}\big|_{\cN_XV(\de_V)}, \qquad
\Psi_Y^*\om_Y=\om_{Y,V}^{(1)}\big|_{\cN_YV(\de_V)}.$$
For $\ep\!\in\!\R^+$, define
\begin{alignat*}{2}
\Psi_{X;\ep}\!: \cN_XV\big(\ep^{-1}\de_V\big)&\lra X, &\qquad 
\Psi_{X;\ep}(v)&=\Psi_X(\ep v),\\
\Psi_{Y;\ep}\!: \cN_YV\big(\ep^{-1}\de_V\big)&\lra Y, &\qquad 
\Psi_{Y;\ep}(v)&=\Psi_Y(\ep v).
\end{alignat*}
These smooth injective open maps satisfy 
\BE{omends_eq} 
\Psi_{X;\ep}^*\om_X=\om_{X,V}^{(\ep)}\big|_{\cN_XV(\ep^{-1}\de_V)}
\qquad\hbox{and}\qquad
\Psi_{Y;\ep}^*\om_Y=\om_{Y,V}^{(\ep)}\big|_{\cN_YV(\ep^{-1}\de_V)}\EE
and restrict to the identity on~$V$.\\

\noindent
Let 
$$\pi_V,\pi_X,\pi_Y\!: \cN_XV\!\oplus\!\cN_YV\lra V,\cN_XV,\cN_YV$$
be the natural projections and
$$\wh\Phi_2\!: \cN_XV\!\oplus\!\cN_YV\lra \cN_XV\!\otimes_{\C}\!\cN_YV
\stackrel{\Phi_2}{\lra} \C$$
be the composition of $\Phi_2$ with the natural product map.
Let $\ga(t)$ be a path in~$V$,
$\wt\ga_X(t)$ be a $\na^X$-horizontal lift of~$\ga$ to the sphere bundle $S_XV\!\subset\!\cN_XV$,
and $\wt\ga_Y(t)$ be a $\na^Y$-horizontal lift of~$\ga$ to the sphere bundle $S_YV\!\subset\!\cN_YV$.
By~\eref{omPhicomp_e2b}, $\wh\Phi_2(\wt\ga_X(t),\wt\ga_Y(t))$ is a constant function.
Thus,
\BE{alXalY_e} \wh\Phi_2^*\nd\th=\pi_X^*\al_X+\pi_Y^*\al_Y
\qquad\hbox{on}\quad (\cN_XV\!-\!V)\!\times_V\!(\cN_YV\!-\!V);\EE
this identity can also be verified using local coordinates.
By~\eref{omPhicomp_e2a} and~\eref{alXalY_e},
\BE{piomC_e}
\wh\Phi_2^*\om_{\C}=\frac12\nd\big(\rho_X\rho_Y\wh\Phi_2^*\nd\th\big)
=\frac12\nd\big(\rho_X\rho_Y(\pi_X^*\al_X\!+\!\pi_Y^*\al_Y)\big),\EE
where $\om_{\C}\equiv\frac12 \nd r^2\!\w\!\nd\th$ is the standard symplectic form on~$\C$.\\

\noindent
For $\de,\ep\!\in\!\R^+$ to be chosen later, let 
\begin{gather}
\notag
\cZ_X=\big(X\!-\!\Psi_{X;\ep}(\ov{\cN_XV(1)})\big)\!\times\!\C, \quad 
\cZ_Y=\big(Y\!-\!\Psi_{Y;\ep}(\ov{\cN_YV(1)})\big)\!\times\!\C,\\
\label{cZdfn_e}
\cZ_V=\big\{(v,w)\!\in\!\cN_XV\!\oplus\!\cN_YV\!:~|v|,|w|\!<\!2,~
~\ep|\wh\Phi_2(v,w)|\!<\!\de\big\},\\
\notag
\cZ_{V;X}=\big\{(v,w)\!\in\!\cZ_V\!:\,|v|\!>\!1\big\}, \quad
\cZ_{V;Y}=\big\{(v,w)\!\in\!\cZ_V\!:\,|w|\!>\!1\big\}.
\end{gather}
With $\ep\!\in\!\R^+$ to be chosen first, we assume that
\BE{epdecond_e} 2\ep\!<\!\de_V, \qquad 2\de<\ep.\EE
Let $\cZ$ be the smooth manifold obtained by gluing $\cZ_X$, $\cZ_Y$, and~$\cZ_V$
by the open maps
\begin{alignat*}{2}
\psi_X\!:\,\cZ_{V;X}&\lra\cZ_X,  &\qquad (v,w)&\lra \big(\Psi_{X;\ep}(v),\ep\wh\Phi_2(v,w)\big), \\
\psi_Y\!:\,\cZ_{V;Y}&\lra\cZ_Y,  &\qquad (v,w)&\lra \big(\Psi_{Y;\ep}(w),\ep\wh\Phi_2(v,w)\big);
\end{alignat*}
by the first assumption in~\eref{epdecond_e},  $\psi_X$ and $\psi_Y$ are well-defined 
diffeomorphisms between open subsets of their domains and targets.
Since the~maps 
$$\cZ_V\lra\C, ~ (v,w)\lra \ep\wh\Phi_2(v,w), \quad
\cZ_X\lra\C,~ (v,\la)\lra\la, \quad
\cZ_Y\lra\C,~ (w,\la)\lra\la, $$
are intertwined by $\psi_X$ and $\psi_Y$, they induce a smooth map $\pi_{\ep}\!:\cZ\!\lra\!\C$.
By the second assumption in~\eref{epdecond_e}, every fiber 
$\cZ_{\la}\!\equiv\!\pi_{\ep}^{-1}(\la)$ of~$\pi_{\ep}$ with $|\la|<\!\de$ is compact
($\de\!<\!2\ep$ would have sufficed here).\\

\noindent
We next define a closed 2-form~$\om_{\cZ}^{(\ep)}$ on~$\cZ$.
Let $\eta\!:\R\!\lra\![0,1]$ be a smooth function such~that 
$$\eta(r)=\begin{cases}0,&\hbox{if}~r\le \frac12;\\
1,&\hbox{if}~r\ge1.\end{cases}$$
By the second assumption in~\eref{epdecond_e},
\BE{etarhoXY_e} (\eta\!\circ\!\rho_X)\cdot (\eta\!\circ\!\rho_Y)=0
\qquad\hbox{on}~\cZ_V.\EE
Define
\begin{equation*}\begin{split}
\wt\om_V^{(\ep)}\equiv\pi_V^*\om_V
+\frac{\ep^2}{2}\nd\Big( (1\!-\!\eta\!\circ\!\rho_Y)\pi_X^*(\rho_X\al_X)
&+\big(1\!-\!\eta\!\circ\!\rho_X\big)\pi_Y^*(\rho_Y\al_Y)\\
&+\big(\eta\!\circ\!\rho_X\!+\!\eta\!\circ\!\rho_Y\big) 
\rho_X\rho_Y\big(\pi_X^*\al_X\!+\!\pi_Y^*\al_Y\big)\Big).
\end{split}\end{equation*}
By~\eref{etarhoXY_e}, \eref{omends_eq}, and~\eref{piomC_e}, the restrictions of this closed 2-form 
on $\cZ_V$ to~$\cZ_{V;X}$ and~$\cZ_{V;Y}$ are 
\begin{equation*}\begin{split}
\pi_V^*\om_V+\frac{\ep^2}2\nd\Big(\pi_X^*(\rho_X\al_X)+
 \rho_X\rho_Y\big(\pi_X^*\al_X\!+\!\pi_Y^*\al_Y\big)\Big)
&=\psi_X^*\om_X+\ep^2\wh\Phi_2^*\om_{\C}=\psi_X^*(\om_X\!+\!\pi_{\ep}^*\om_{\C}),\\
\pi_V^*\om_V+\frac{\ep^2}2\nd\Big(\pi_Y^*(\rho_Y\al_Y)+
 \rho_X\rho_Y\big(\pi_X^*\al_X\!+\!\pi_Y^*\al_Y\big)\Big)
&=\psi_Y^*\om_Y+\ep^2\wh\Phi_2^*\om_{\C}=\psi_Y^*(\om_Y\!+\!\pi_{\ep}^*\om_{\C}),
\end{split}\end{equation*}
respectively.
Thus, along with the 2-forms 
\BE{wtomXY_e}\wt\om_X\equiv \om_X +\pi_{\ep}^*\om_{\C} \qquad\hbox{and}\qquad
\wt\om_Y\equiv \om_Y+\pi_{\ep}^*\om_{\C} \EE
on~$\cZ_X$ and on~$\cZ_Y$, $\wt\om_V^{(\ep)}$ induces 
a closed 2-form~$\om_{\cZ}^{(\ep)}$ on~$\cZ$.\\

\noindent
Let $D\al_X,D\al_Y\!\in\!\Om^2(V)$ denote the curvature forms of~$\al_X$ and~$\al_Y$.
Define 
\begin{gather*}
f_X,f_Y\!:\cN_XV\!\oplus\!\cN_YV\lra\R \qquad\hbox{by}\\
f_X=(1\!-\!\eta\!\circ\!\rho_Y)+(\eta\!\circ\!\rho_X\!+\!\eta\!\circ\!\rho_Y)\rho_Y,\quad
f_Y=(1\!-\!\eta\!\circ\!\rho_X)+(\eta\!\circ\!\rho_X\!+\!\eta\!\circ\!\rho_Y)\rho_X.
\end{gather*}
By~\eref{alXalY_e} and~\eref{etarhoXY_e}, 
\BE{Cdebnd_e}
D\al_X=-D\al_Y \qquad\hbox{and}\qquad \frac12< f_X(v,w),f_Y(v,w)<5
~~~\forall~(v,w)\!\in\!\cZ_V,\EE
respectively.
Let
\begin{equation*}\begin{split}
\wt\om_{V;\bu}^{(\ep)}&=
\pi_V^*\bigg(\om_V+\frac{\ep^2}{2}\Big(f_X\rho_XD\al_X\!+\!f_Y\rho_YD\al_Y\Big)\bigg)
+\frac{\ep^2}{2}f_X\pi_X^*(\nd\rho_X\!\w\!\al_X) 
+\frac{\ep^2}{2}f_Y\pi_Y^*(\nd\rho_Y\!\w\!\al_Y),\\
\vp_V^{(\ep)}&=\frac{\ep^2}2\Big((\rho_Y\!-\!1)\nd(\eta\!\circ\!\rho_Y)
+\rho_Y\nd(\eta\!\circ\!\rho_X)
+(\eta\!\circ\!\rho_X\!+\!\eta\!\circ\!\rho_Y)\nd\rho_Y\Big)
\w\pi_X^*(\rho_X\al_X)\\
&\hspace{.5in}
+\frac{\ep^2}2\Big((\rho_X\!-\!1)\nd(\eta\!\circ\!\rho_X)
+\rho_X\nd(\eta\!\circ\!\rho_Y)
+(\eta\!\circ\!\rho_X\!+\!\eta\!\circ\!\rho_Y)\nd\rho_X\Big)
\w\pi_Y^*(\rho_Y\al_Y).
\end{split}\end{equation*}
Thus, $\wt\om_V^{(\ep)}\!=\!\wt\om_{V;\bu}^{(\ep)}\!+\!\vp_V^{(\ep)}$.\\

\noindent
For $\ep\!\in\!\R^+$ sufficiently small (dependent only on $\al_X$), 
\eref{Cdebnd_e} ensures the existence of $\ep_{\al}\!\in\!\R^+$ such that 
the 2-form  $\wt\om_{V;\bu}^{(\ep)}\!+\!\vp$ is nondegenerate on~$\cZ_V$ for any 2-form~$\vp$ 
on~$\cZ_V$ with  $\|\vp\|_{C^0}\!<\!\ep_{\al}\ep^2$.
On the other hand, there exists $C_{\eta}\!\in\!\R^+$ such that 
\BE{omVest_e}\big|\vp_V^{(\ep)}\big|_{(v,w)} \le C_{\eta}\ep^2\big|\wh\Phi_2(v,w)\big|
\le \big(C_{\eta}\de/\ep\big)\ep^2
\qquad\forall~(v,w)\!\in\!\cZ_V.\EE
Thus, the closed 2-form $\wt\om_{\cZ}^{(\ep)}$ on~$\cZ$ is nondegenerate if $C_{\eta}\de\!<\!\ep_{\al}\ep$.\\

\noindent
We now define an $\om_{\cZ}^{(\ep)}$-tame almost complex structure~$J_{\cZ}'$ on~$\cZ$
which preserves (the tangent spaces to) the fibers of the fibration $\pi_{\ep}\!:\cZ\!\lra\!\De$.
The connections~$\na^X$ and~$\na^Y$  induce a splitting of the exact sequence
\BE{TZsplit_e}\begin{split}
0&\lra \pi_V^*(\cN_XV\!\oplus\!\cN_YV) \lra T\cZ_V
\stackrel{\nd\pi_V}{\lra} \pi_V^*TV\lra0
\end{split}\EE
of vector bundles over~$\cZ_V$.
The image of $\pi_V^*TV$ corresponding to this splitting~is
$$\ker\nd\rho_X \cap \ker\pi_X^*\al_X\cap \ker\nd\rho_Y \cap \ker\pi_Y^*\al_Y
\subset T\cZ_V\,$$
outside of~$V$,
as can be seen from \cite[Lemma~1.1]{anal}, for example.
By \cite[Appendix]{IPrel}, there exist $C\!>\!0$ and
a smooth family  $J_{V;\rho}$ with $\rho\!\in\!(-5\ep^2,5\ep^2)$
of almost complex structures on~$V$ such that $J_{V;\rho}$ is compatible
with the symplectic form $\om_V\!+\!(\rho/2) D\al_X$ and 
\BE{Jrho_e}\big\|J_{V;\rho}-J_{V;0}\big\|\le C\rho \qquad\forall~\rho\in(-5\ep^2,5\ep^2).\EE
Let $\wt{J}_V|_{(v,w)}$ be the  complex structure on $T_{(v,w)}\cZ_V$
induced by the complex structure $\fI_X\!\oplus\!\fI_Y$ in the fibers of~$\pi_V$ and
the almost complex structure
$$J_{V;\ep^2(\rho_X(v)f_X(v,w)-\rho_Y(w)f_Y(v,w))}$$
on~$V$ via the splitting~\eref{TZsplit_e}
and $\wt{J}_{V;0}|_{(v,w)}$ be the complex structure induced by $\fI_X\!\oplus\!\fI_Y$ and~$\wt{J}_{V;0}$. 
The almost complex structure~$\wt{J}_V$ on~$\cZ_V$ is
$\wt\om_{V;\bu}^{(\ep)}$-compatible.
By~\eref{omPhicomp_e2b}, it preserves the fibers
\BE{Zlaneck_e}\pi_{\ep}^{-1}(\la)\cap\cZ_V=
\big\{(v,w)\!\in\!\cN_XV\!\oplus\!\cN_YV\!:~\ep\wh\Phi_2(v,w)\!=\!\la\big\}.\EE
Along with~\eref{omVest_e}, this implies that $\wt\om_V^{(\ep)}|_{\cZ_{\la}}$
is a symplectic form taming~$\wt{J}_V$ for all $\la\!\in\!\De$,
provided $\De\!\subset\!\C$ is a sufficiently small neighborhood of the origin.\\

\noindent
Since $\wt{J}_V$ is tamed by $\wt\om_V^{(\ve)}$ and preserves the fibers~\eref{omPhicomp_e2b},
it can be extended to an almost complex structure~$J_{\cZ}'$ on~$\cZ$
which is tamed by~$\wt\om_V^{(\ep)}$ and preserves the fibers of~$\pi_{\ve}$.
The restriction of~$\wt{J}_{V;0}$ to $\pi_V^{-1}(\Si)$
is Kahler for every (real) surface $\Si\!\subset\!V$ preserved by~$J_{V;0}$;
see \cite[Lemma~2.4]{anal}.
Along with~\eref{Jrho_e}, this implies~that 
$$N_{J_{\cZ}'}(v,w)\in T_xV \qquad\forall~v,w\!\in\!T_x\cZ,~x\!\in\!V,$$
where $N_{J_{\cZ}'}$ is the Nijenhuis tensor of~$J_{\cZ}'$.\\

\noindent
In the region $|v|,|w|\!<\!\frac12$, $\wt\om_{V}^{(\ep)}\!=\!\wt\om_{V;\bu}^{(\ep)}$.
Thus, $\wt{J}_V$ is $\wt\om_{V}^{(\ep)}$-compatible and
$$\wt{g}_V^{(\ep)}\equiv \wt\om_V^{(\ep)}(\cdot,\wt{J}_V\cdot)$$
is a metric on this neighborhood of~$V$ in~$\cZ$.
It agrees with the product metric 
$$\wt{g}_{V;0}^{(\ep)}\equiv \om_V(\cdot,J_{V;0}\cdot)\oplus \ep^2g_X \oplus \ep^2g_Y$$
to the second order in~$(v,w)$,
since the splitting of~\eref{TZsplit_e} is $\wt\om_{V;\bu}^{(\ep)}$-orthogonal.
Thus, the second fundamental form $\II_V$ of~$V$ with respect to the metric 
$$g_{\cZ}'^{(\ep)}(\cdot,\cdot)\equiv
\frac12\big(\om_{\cZ}^{(\ve)}(\cdot,J_{\cZ}'\cdot)-\om_{\cZ}^{(\ve)}(J_{\cZ}'\cdot,\cdot)\big)$$
determined by $\om_{\cZ}^{(\ve)}$ and $J_{\cZ}'$ vanishes.\\

\noindent
The  $\wt\om_V^{(\ep)}$-tame almost complex structure~$J_{\cZ}'$ can be replaced
by an $\wt\om_V^{(\ep)}$-compatible almost complex structure~$J_{\cZ}$
by deforming it outside of a neighborhood of~$V$ in~$\cZ$.
Let $J_X$ be an $\om_X$-compatible almost complex structure on~$X$,
$J_Y$ be an $\om_Y$-compatible almost complex structure on~$Y$,
and $\fj_{\C}$ be the standard complex structure on~$\C$.
The almost complex structures
$$\wt{J}_X\equiv J_X\oplus\fj_{\C} \qquad\hbox{and}\qquad
\wt{J}_Y\equiv J_Y\oplus\fj_{\C}$$
on~$\cZ_X$ and~$\cZ_Y$ are then compatible with the symplectic forms~\eref{wtomXY_e}
and preserve the fibers of the projection~$\pi_{\ep}$.
Let
$$\wt{g}_{V;\bu}^{(\ep)}(\cdot,\cdot)\equiv \wt\om_{V;\bu}^{(\ep)}(\cdot,\wt{J}_V\cdot),\qquad
\wt{g}_X(\cdot,\cdot)=\wt\om_X(\cdot,\wt{J}_X\cdot), 
\qquad \wt{g}_Y(\cdot,\cdot)=\wt\om_X(\cdot,\wt{J}_X\cdot).$$
By \cite[Lemma~A.1]{IPrel}, an $\om_Z$-compatible almost complex structure~$J_Z$
on a symplectic manifold $(Z,\om_Z)$ preserves
a symplectic submanifold $W\!\subset\!Z$ if and only if 
the $\om_Z$-orthogonal complement of~$TW$ is also orthogonal to~$TW$
with respect to the metric~$\om_Z(\cdot,J_Z\cdot)$.
Thus, the metrics $\wt{g}_{V;\bu}^{(\ep)}$, $\wt{g}_X$, and $\wt{g}_Y$ can
be patched together over the regions 
\BE{neckends_e}\frac14\le |v| \le 1 \qquad\hbox{and}\qquad 
\frac14\le |w|\le 1\EE
in~$\cZ_V$ into a metric~$g_{\cZ}^{(\ve)}$ compatible with~$\om_{\cZ}^{(\ep)}$
so that the corresponding $\om_{\cZ}^{(\ep)}$-compatible almost complex 
structure~$J_{\cZ}$ preserves~$\cZ_{\la}$ for every $\la\!\in\!\De$.\\

\noindent
For $\la\!\in\!\De^*$, let $\om_{\la}\!=\!\om_{\cZ}^{(\ep)}|_{\cZ_{\la}}$.
By the symplectic sum construction above, the complex normal bundles of 
$X,Y\!\subset\!\cZ$ are given~by
\begin{alignat*}{2}
\cN_{\cZ}X&\approx\big(\pi_{X,V}^*\cN_YV\!\sqcup\!(X\!-\!V)\!\times\!\C\big)\big/\!\!\sim,
&~~ (v,w)&\sim \big(\Psi_{X;\ep}(v),\ep^2\wh\Phi_2(v,w)\big)
~\forall\,(v,w)\!\in\!\pi_{X,V}^*\cN_YV,\,v\!\neq\!0, \\
\cN_{\cZ}Y&\approx\big(\pi_{Y,V}^*\cN_XV\!\sqcup\!(Y\!-\!V)\!\times\!\C\big)\big/\!\!\sim,
&~~ (w,v)&\sim\big(\Psi_{Y;\ep}(w),\ep^2\wh\Phi_2(v,w)\big)
~\forall\,(w,v)\!\in\!\pi_{Y,V}^*\cN_XV,\,w\!\neq\!0,
\end{alignat*}
where
$$\pi_{X,V}\!: \cN_XV(2)\lra V \qquad\hbox{and}\qquad 
\pi_{Y,V}\!: \cN_YV(2)\lra V$$
are the restrictions of the bundle projections.
Since the canonical  meromorphic sections of these line bundles 
are nowhere zero and have polar divisors~$V$,
$$ \blr{c_1(\cN_{\cZ}X),A}=-A\!\cdot_X\!V~~~\forall\,A\!\in\!H_2(X;\Z), \quad
\blr{c_1(\cN_{\cZ}Y),B}=-B\!\cdot_X\!V~~~\forall\,B\!\in\!H_2(Y;\Z).$$
On the other hand, the normal bundle of~$\cZ_{\la}$ with $\la\!\in\!\De^*$ is trivial.\\

\noindent
We now compare the first chern class of $(\cZ_{\la},\om_{\la})$ 
with the first chern classes of~$(X,\om_X)$ and~$(Y,\om_Y)$.
Two 2-pseudocycles 
$$f_X\!:(Z_X,x_1,\ldots,x_{\ell})\lra(X,V) \qquad\hbox{and}\qquad  
f_Y\!:(Z_Y,y_1,\ldots,y_{\ell})\lra(Y,V)$$ 
with boundary disjoint  from~$V$ such that 
\begin{gather*}
f_X^{-1}(V)=\{x_1,\ldots,x_{\ell}\},\qquad f_Y^{-1}(V)=\{y_1,\ldots,y_{\ell}\},\\
f_X(x_i)=f_Y(y_i),\quad
\ord_{x_i}^Vf_X=\ord_{y_i}^Vf_Y\qquad\forall\,i\!=\!1,2,\ldots,\ell,
\end{gather*}
determine a 2-pseudocycle $f_X\!\#_{\la}\!f_Y\!:Z_X\!\#\!Z_Y\!\lra\!\cZ_{\la}$;
see \cite[Section~2.2]{GWrelIP}.
Since the homology class of $f_X\!\#_{\la}\!f_Y$ in~$\cZ$ is the sum of
the homology classes of~$f_X$ and~$f_Y$, it follows~that 
\begin{equation*}\begin{split}
\blr{c_1(T\cZ_{\la}),[f_X\!\#_{\la}\!f_Y]}
&=\blr{c_1(T\cZ),[f_X]}+\blr{c_1(T\cZ),[f_Y]}\\
&=\big(\blr{c_1(TX),[f_X]}-[f_X]\!\cdot_X\!V\big)
+\big(\blr{c_1(TY),[f_Y]}-[f_Y]\!\cdot_Y\!V\big).
\end{split}\end{equation*}
In particular, the left-hand side of this expression depends only on 
the homology classes of~$[f_X]$ in~$X$ and~$[f_Y]$ in~$Y$.
Thus,
$$\blr{c_1(T\cZ_{\la}),A\!\#_{\la}\!B}\in\Z$$
is well-defined for all $(A,B)\!\in\!H_2(X;\Z)\!\times_V\!H_2(Y;\Z)$.
We have thus established the following.

\begin{prp}[Gompf's Symplectic Sum]\label{SymSum_prp}
Let $(X,\om_X)$ and $(Y,\om_Y)$ be compact symplectic manifolds
and $V\!\subset\!X,Y$ be a symplectic hypersurface satisfying~\eref{cNVcond_e}.
For each choice of homotopy class of isomorphisms~\eref{cNpair_e}, 
there exist a symplectic manifold $(\cZ,\om_{\cZ})$, 
a smooth map $\pi\!:\cZ\!\lra\!\De$, and 
an $\om_{\cZ}$-compatible almost complex structure~$J_{\cZ}$ on~$\cZ$ such~that 
\begin{enumerate}[label=$\bullet$,leftmargin=*]
\item  $\pi$ is surjective and $\cZ_0\!=\!X\!\cup_V\!Y$,
\item $\pi$ is a submersion outside of $V\!\subset\!\cZ_0$, 
\item the restriction~$\om_{\la}$ of~$\om_{\cZ}$ to $\cZ_{\la}\equiv\pi^{-1}(\la)$ 
is nondegenerate for every $\la\!\in\!\De^*$,
\item  $\om_{\cZ}|_X\!=\!\om_X$, $\om_{\cZ}|_Y\!=\!\om_Y$, 
\item $J_{\cZ}$ preserves $T\cZ_{\la}$ for every $\la\!\in\!\De^*$,
\item  $N_{J_{\cZ}}(v,w)\in T_xV$ for all $v,w\!\in\!T_x\cZ$, $x\!\in\!V$, and
\item the second fundamental form $\II_V$ of~$V$ with respect to the metric 
$\om_{\cZ}(\cdot,J_{\cZ}\cdot)$ vanishes.
\end{enumerate}
Furthermore, 
\BE{c1Zla_e}
\blr{c_1(T\cZ_{\la}),A\!\#_{\la}\!B}=\blr{c_1(TX),A}+\blr{c_1(TY),B}-2A\!\cdot_X\!V\EE
for all $\la\!\in\!\De^*$ and $(A,B)\!\in\!H_2(X;\Z)\!\times_V\!H_2(Y;\Z)$.
\end{prp}

\begin{rmk}\label{SympSumConstr_rmk}
In \cite[Section~2]{IPsum},  the above $\ep$-rescaling of the symplectic forms
on~$\cN_XV$ and~$\cN_YV$ is absorbed into 
the starting Hermitian structures on~$\cN_XV$ and~$\cN_YV$.
The second identity in~\eref{c1Zla_e} is \cite[Lemma~2.4]{IPsum}.
Our proof of this identity adds details to the proof in~\cite{IPsum}
and in particular formally extends over~$X$ and~$Y$ the key bundles that
are described only over neighborhoods of~$V$ in~\cite{IPsum}. 
In other aspects, the review of the symplectic sum construction in \cite[Section~2]{IPsum} is mostly wrong.
The ``ends'' $\cZ_X$ and~$\cZ_Y$ are described incorrectly:
the specification in~\cite{IPsum} leads to non-compact fibers~$\cZ_{\la}$.
The verification of the nondegeneracy of~$\om_{\cZ}$ in the overlap region is wrong:
this region should not be tied to the parameter~$\de$ appearing in the relevant bounds.
The nondegeneracy of~$\om_{\cZ}|_{\cZ_{\la}}$ is taken for granted.
The justification of~\eref{piomC_e} in \cite[Section~2]{IPsum} is incomplete
and refers to $\pi_X^*\al_X\!+\!\pi_Y^*\al_Y$ as a connection 1-form 
on $\cN_XV\!\oplus\!\cN_YV$, which differs from the standard usage.
The moment maps in \cite[(2.4)]{IPsum} play no role in the symplectic sum construction
described there.
\end{rmk}

\subsection{A symplectic cut perspective: \cite[Sections~2,3.0]{LR}}
\label{SympCut_subs}

\noindent
The symplectic sum formula for GW-invariants is approached in \cite[Section~3.0]{LR}
from the opposite direction by cutting~$(M,\om)\!=\!(X\!\#_V\!Y,\om_{\#})$
into two pieces~$M^-$ and~$M^+$ along a compact hypersurface~$\wt{M}$.
This hypersurface is the preimage of a regular value of a Hamiltonian~$H$ 
on a neighborhood~$U$ of~$\wt{M}$ generating a free $S^1$-action on~$\wt{M}$.
By the Mardsen-Weinstein construction \cite[Section~5.4]{MS1},
the quotient $V\!=\!\wt{M}/S^1$ 
is then a smooth manifold with a symplectic form~$\om_V$
such that $\pi^*\om_V\!=\!\om|_{\wt{M}}$, where $\pi\!:\wt{M}\!\lra\!V$ is the projection
(in~\cite{LR}, $(V,\om_V)$ is denoted by~$(Z,\tau_0)$).
The symplectic cutting construction of~\cite{Ler} collapses the ends of~$M^-$ and~$M^+$
and produces symplectic manifolds $(\ov{M}^-,\om_-)$
and $(\ov{M}^+,\om_+)$ containing $(V,\om_V)$ as a symplectic hypersurface
with dual normal bundles.\\

\noindent
In the description of  Section~\ref{SympSum_subs}, $\wt{M}$ corresponds
to the hypersurface
$$SV_{\la}\equiv \big\{(v,w)\!\in\!\cZ_V\!:~\ep^2\wh\Phi_2(v,w)\!=\!\la,\,|v|\!=\!|w|\big\}
\subset\cZ_{\la}\,,$$
with $\la\!\in\!\C^*$ small.
The symplectic manifolds $(\ov{M}^-,\om_-)$ and $(\ov{M}^+,\om_+)$ 
obtained in this way are symplectically deformation equivalent~to 
$(X,\om_X)$ and~$(Y,\om_Y)$.
We will identify $SV_{\la}$ with the sphere (circle) bundle $SV\!\equiv\!S_XV$
of~$\cN_XV$ and use the isomorphism~\eref{cNpair_e} to identify~$S_YV$ with~$SV$, i.e.  
\BE{SVident_e} S_YV\ni w \llra v\in SV=S_XV\qquad\hbox{if}\quad \wh\Phi_2(v,w)=1\in\C.\EE
In particular, we use the complex structure on~$\cN_XV$ to induce an $S^1$-action on~$SV$
for the purposes of the approach in~\cite{LR}; 
the complex structure on~$\cN_YV$ would induce the inverse $S^1$-action on~$SV$.
The restriction of the Hamiltonian vector field~$\ze_H$, denoted~$X_H$ in \cite[Section~3.0]{LR},
to~$\wt{M}$ then corresponds to the characteristic vector field of the $S^1$-action
on~$SV$, i.e. $\left.\frac{\nd}{\nd\th}(\ne^{\fI\th}v)\right|_{\th=0}$ at each $v\!\in\!SV$.
Let $\al\!=\!\al_X$ be a connection 1-form on~$SV$ as before (denoted by~$\la$ in~\cite{LR}).\\

\noindent
The family of almost complex structures~$\wh{J}_{\la}$ on~$\cZ_{\la}$ used in~\cite{LR}
is more restrictive than in~\cite{IPsum} on the necks.
Given $\de'\!\in\!(0,\frac14)$, let $\wt\eta\!:\R\!\lra\![0,1]$ be a smooth function such~that 
$$\wt\eta(r)=\begin{cases}0,&\hbox{if}~r\le\de';\\
1,&\hbox{if}~r\ge2\de'. \end{cases} $$
With $J_{V;\rho}$ as in~\eref{Jrho_e}, let
$$J_{V;(v,w)}=J_{V;\wt\eta(\rho_X(v)+\rho_Y(w))\ep^2(\rho_X(v)f_X(v,w)-\rho_Y(v)f_Y(v,w))}
\qquad\forall~(v,w)\!\in\!\cZ_V.$$
If $\de'$ is sufficiently small, this  complex structure  on $T_{\pi_V(v,w)}V$ is tamed~by
the symplectic form 
$$\om_V+\frac{\ep^2}{2}\wt\eta\big(\rho_X(v)\!+\!\rho_Y(w)\big)
\big(\rho_X(v)f_X(v,w)\!-\!\rho_Y(w)f_Y(v,w)\big)D\al.$$ 
Let $\wt{J}_V|_{(v,w)}$ be the complex structure on $T_{(v,w)}\cZ_V$
induced by the complex structure $\fI_X\!\oplus\!\fI_Y$ in the fibers of
$\pi_V\!: \cZ_V\!\lra\!V$ and the complex structure
$J_{V;(v,w)}$ on~$T_{\pi_V(v,w)}V$ via the splitting~\eref{TZsplit_e}. 
It again preserves the fibers~\eref{Zlaneck_e} and
is tamed by the symplectic form~$\wt\om_V^{(\ep)}$ of Section~\ref{SympSum_subs} everywhere on~$\cZ_V$.
Thus, we can again extend it to an $\om_{\cZ}^{(\ep)}$-tame almost complex structure~$\wt{J}_{\cZ}$
on~$\cZ$ which preserves the fibers~$\cZ_{\la}$.
Let $\wh{g}_{\cZ}$ be the metric on~$\cZ$ determined by~$\om_{\cZ}^{(\ep)}$ and~$\wt{J}_{\cZ}$.
We denote the restrictions of $(\wt{J}_{\cZ},\wh{g}_{\cZ})$ to~$X$, $Y$, and $\cZ_{\la}$
with $\la\!\in\!\De^*$ small by $(\wt{J}_X,\wh{g}_X)$, $(\wt{J}_Y,\wh{g}_Y)$, 
and $(\wt{J}_{\la},\wh{g}_{\la})$, respectively.\\

\noindent
The stretching construction of~\cite{LR} presents the complements of~$V$ in
tubular neighborhoods in~$X$ and~$Y$ as bundles over~$V$ whose fibers are
infinite half-cylinders.
In the notation of~\cite{IPsum}, the ``height" coordinates can be taken to~be
$$a_X\big(\Psi_{X;\ep}(v)\big)=\ln|v| \qquad\hbox{and}\qquad 
a_Y\big(\Psi_{Y;\ep}(w)\big)=-\ln|w|$$
on~$X$ and~$Y$, respectively. 
Thus, 
\BE{oXVdfn_e}\R^-\!\times\!SV\subset \oXV\equiv X\!-\!V,\quad 
\R^+\!\times\!SV\subset \oYV\equiv\!Y\!-\!V,\EE
and $X$ and $Y$ are quotients of the manifolds with boundary
\BE{tXVdfn_e}
\tXV\equiv \big(\oXV\cup_{\R^-\times SV} [-\i,0)\!\times\!V\big)/\sim
\quad\hbox{and}\quad
\tYV\equiv \big(\oYV\cup_{\R^+\times SV} (0,\i]\!\times\!V\big)/\sim
\,,\EE
respectively; $V\!\subset\!X,Y$ is the quotient of $\{\mp\i\}\!\times\!SV$ by the $S^1$-action.
For each $a\!\in\!(0,\i]$, let 
\begin{equation*}\begin{split}
X_a=\oXV-\big\{(a_X,v)\!\in\!\R^-\!\times\!SV\!:~a_X\!\le\!-\frac34a\big\},\qquad
Y_a=\oYV-\big\{(a_Y,w)\!\in\!\R^+\!\times\!SV\!:~a_Y\!\ge\!\frac34a\big\}.
\end{split}\end{equation*}
In the approach of~\cite{LR}, the symplectic sum~$\cZ_{\la}$ of~\cite{IPsum} is viewed~as
\BE{cZlaLR_e}
\cZ_{a,\vt}=\big(X_a\sqcup Y_a\big)\big/\!\sim,~~
X_a\!-\!\ov{X}_{\frac{a}{3}}\ni(a_X,v)\sim(a_X\!+\!a,\ne^{\fI\vt}v)\in Y_a\!-\!\ov{Y}_{\frac{a}{3}}
\EE
if $\ep^2\la^{-1}\!=\!\ne^{a+\fI\vt}$;
in the notation of \cite[(4.11,4.12)]{LR}, 
$(a,\vt)\!=\!(4kr,\th_0)$ and $(a_X,a_Y)\!=\!(a_2,a_1)$.\\

\noindent
For any $\ve\!\in\!(0,1]$, let 
\BE{cZepdfn_e}\begin{split}
\cZ_{a,\vt;\ve}&=\big\{(v,w)\!\in\!\cZ_V\!\cap\!\cZ_{\la}\!:\,|v|,|w|\!\le\!\ve^{1/2}\big\}\\
&=\big\{(a_X,v)\!\in\!\R^-\!\times\!SV\!:\,\frac{\ln\ve}{2}\!\ge\!a_X\!\ge\!-\frac{3a}4\big\}
\cup
\big\{(a_Y,w)\!\in\!\R^+\!\times\!SV\!:-\frac{\ln\ve}{2}\!\le\!a_Y\!\le\!\frac{3a}4\big\},
\end{split}\EE
with the union on the second line taken inside of~$\cZ_{a,\vt}$.
Denote by $\frac{\prt}{\prt a_{\la}}$ the vector field on~$\cZ_{a,\vt;1}$ restricting
to $\frac{\prt}{\prt a_X}$ on the intersection with~$X_a$
and to $\frac{\prt}{\prt a_Y}$ on the intersection with~$Y_a$.
The almost complex structure~$\ti{J}$ of the previous paragraph satisfies
$$\wt{J}_X\frac\prt{\prt a_X}=\ze_H ~~\hbox{on}~X_a\!-\!X_{-\frac23\ln\de'}, \quad
\wt{J}_Y\frac\prt{\prt a_Y}=\ze_H ~~\hbox{on}~Y_a\!-\!Y_{-\frac23\ln\de'}, \quad
\wt{J}_{\la}\frac\prt{\prt a_{\la}}=\ze_H ~~\hbox{on}~\cZ_{a,\vt;\de'}\,.$$
It restricts to the pull-back of~$J_V$ on $\ker\al\!\subset\!T(SV)\!\subset\!T\cZ_{a,\vt}$ and
differs slightly from the initially fixed almost complex structures~$J_X$ and~$J_Y$
over
$$\Big\{(a_X,v)\!\in\!\R^-\!\times\!SV\!:\,a_X\ge\frac{\ln\de'}{2}\Big\}\subset X
\qquad\hbox{and}\qquad
\Big\{(a_Y,w)\!\in\!\R^+\!\times\!SV\!:\,a_Y\le-\frac{\ln\de'}{2}\Big\}\subset Y\,,$$
respectively, in a way depending on~$\la$.\\

\noindent
Finally, we specify complete metrics~$\wt{g}_X$, $\wt{g}_Y$, and $\wt{g}_{a,\vt}$ 
on $\oXV$, $\oYV$, and~$\cZ_{a,\vt}$, respectively.
Let $\wh\eta(r)\!=\!\wt\eta(16r)$.
Denote by~$g_{\cyl}$ the metric on~$\R\!\times\!SV$ given~by
$$g_{\cyl}\big((a_1,v_1),(a_2,v_2)\big)=a_1a_2+\al(v_1)\al(v_2)+
q_V^*g_V(v_1,v_2),$$
where $g_V(\cdot,\cdot)\!=\!\om_V(\cdot,J_V\cdot)$ is the metric on~$V$ 
induced by~$J_V$.
Following \cite[(3.7),(3.8)]{LR}, we define
the metrics~$\wt{g}_X$ on~$X\!-\!V$ and~$\wt{g}_Y$ on~$Y\!-\!V$ by
\begin{equation*}\begin{split}
\wt{g}_X|_x&=\begin{cases}\wh{g}_X|_x,&\hbox{if}~x\!\in\!\oXV\!-\!(-\i,-1)\!\times\!SV;\\
\wh\eta(\rho_X(x))\wh{g}_X|_x\!+\!(1\!-\!\wh\eta(\rho_X(x)))g_{\cyl}|_x
,&\hbox{if}~x\!\in\!\R^-\!\times\!SV;
\end{cases}\\
\wt{g}_Y|_y&=\begin{cases}\wh{g}_Y|_y,&~~\hbox{if}~y\!\in\!\oYV\!-\!(1,\i)\!\times\!SV;\\
\wh\eta(\rho_Y(y))\wh{g}_Y|_y\!+\!(1\!-\!\wh\eta(\rho_Y(y)))g_{\cyl}|_y
,&~~\hbox{if}~y\!\in\!\R^+\!\times\!SV.
\end{cases}
\end{split}\end{equation*}
For each $a\!\in\!\R^+$ sufficiently large, we similarly define 
$$\wh{g}_{a,\vt}|_x=
\begin{cases}\wh{g}_{\ne^{-(a+\fI\vt)}}|_x,&\hbox{if}~x\!\in\!\cZ_{a,\vt}\!-\!\cZ_{a,\vt;1};\\
\wh\eta(\rho_X(x)\!+\!\rho_Y(x))\wh{g}_{\ne^{-(a+\fI\vt)}}|_x\!+\!
(1\!-\!\wh\eta(\rho_X(x)\!+\!\rho_Y(x)))g_{\cyl}|_x
,&\hbox{if}~x\!\in\!\cZ_{a,\vt;2}.
\end{cases}$$
This metric agrees with the cylindrical metric on $\cZ_{a,\vt;\de'/16}$.
Its injectivity radius is uniformly (independently of~$(a,\vt)$) bounded below
and the norm of its Riemannian curvature tensor is uniformly bounded above.
 
\begin{rmk}\label{LR230_rmk}
The review of the symplectic sum and cutting constructions in~\cite{LR}
consists of \cite[Examples~2.6-2.8]{LR}.
In particular, the symplectic form~$\om_{\cZ}^{(\ep)}|_{\cZ_{\la}}$ 
on the glued manifold in \cite[(4.11,4.12)]{LR}  is not specified;
as indicated in Section~\ref{SympSum_subs}, constructing such a form is not trivial.
The symplectic form~$\om_0$ on~$\C^n$ in \cite[Example~2.6]{LR} is not specified.
The second set on the RHS of \cite[(2.6)]{LR} can be easily absorbed into the first;
it would perhaps be clearer to describe $\mu^{-1}(0)$ as $|z|^2\!+\!|w|^2\!=\!\ep$.
Since~$z$ is a vector, the expression $z\nd\bar{z}$ in \cite[(2.8)]{LR} does not make sense;
the intended meaning is presumably as in \cite[(2.16)]{LR}.
The formula \cite[(2.8)]{LR} does not seem to appear in~[MS1].
The $S^1$-action for the Hamiltonian in \cite[(2.10)]{LR} with respect to 
the symplectic form in \cite[(2.9)]{LR}
is given by the multiplication by~$\ne^{-\fI t/\ep}$, not as in \cite[(2.11)]{LR}.
The wording of \cite[Lemma~2.5]{LR} is incorrect; 
there should be a homotopy of the maps~$\vph$ as well.
The third sentence on page~165 in~\cite{LR} is vague.
The wording of the paragraph in~\cite{LR} containing this sentence suggests that 
the symplectic blowup construction 
involves almost complex structures, which is not the case;
it is described explicitly on pages 239-250 in~\cite{MS1}.
A direct connection of this paragraph to \cite[Proposition~2.10]{LR} is also unclear.
Other, fairly minor misstatements in \cite[Sections~2,3.0]{LR} include
\begin{enumerate}[label={},topsep=-5pt,itemsep=-5pt,leftmargin=*]
\item\textsf{p161, Ex~2.2:} $\ne^{\fI\th}$ \to $\ne^{\fI\th/2}$;
\item\textsf{p161, Dfn~2.3:} concave dfn is correct only if $N$ is connected;
\item\textsf{p161, bottom:} not by (1.10); {\it The map (1.10) induces a homomorphism...}
\item\textsf{p162, top:} [FO] and [LT] do not require integrality;
\item\textsf{p162, line 7:} this equality does not hold, as LHS is degenerate along $\pi^{-1}(Z)$;
\item\textsf{p162, after (2.9):} on the whole total space, as used above Prop.~2.10;
\item\textsf{p163, lines 4,13:} Example 1 \to Example 2.6;
\item\textsf{p164, line -17:} $\vph$ is not specified in (2.13);
\item\textsf{p164, line -6:} the antipodal \to a conjugation;
\item\textsf{p165, line 12:} Example 2 \to Example 2.7;
\item\textsf{p165, Lemma 2.11, proof:} $\ov{M}^-$ is not a subset of~$M$;
\item\textsf{p166, line 14:} $M_t$ has not been defined;
\item\textsf{p168, Section~3:} need to require $H^{-1}(0)$ to be compact and the $S^1$-action
to be free; the relation of~$\nd\la$ with the Chern class is irrelevant;
\item\textsf{p168, lines -9,-8,-3:} $\{$ and $\}$ should not be here;
identifications along $\{\pm\ell\}\!\times\!\wt{M}$;
\item\textsf{p169, line 1:} this sentence does not make sense and is not used here;
\item\textsf{p169, (3.7),(3.8):} these metrics need to be patched together;
\item\textsf{p170, line 1:} $\Pi$ can be taken to be $\nd\pi$.
\end{enumerate}
\end{rmk}

\section{Relative stable maps}
\label{RelMaps_sec}

\noindent
In Sections~\ref{RelInv_sub0} and~\ref{RelInv_sub0b}, 
we recall the notions of stable (relative) morphism into $(X,V)$
and of stable (predeformable) morphism into the singular space $X\!\cup_V\!Y$,
respectively.
In Section~\ref{RelInv_subs}, we review the geometric construction of the absolute GW-invariants,
due to \cite{RT1, RT2}, and its adaption to the relative setting, due to~\cite{IPrel};
we also comment on the general case considered in~\cite{LR}.
The rim tori refinement for the standard relative GW-invariants suggested in
\cite[Section~5]{IPrel} is discussed in Section~\ref{RelInv_subs2}.\\

\noindent
The notion of stable morphism into $(X,V)$ arises from \cite[Definition~3.14]{LR}.
The formulation we describe first is based on \cite[Definition~4.7]{Jun1}, but 
is adapted to the almost Kahler setting of~\cite{IPrel}. 
The notion of stable (predeformable) morphism into the singular space $X\!\cup_V\!Y$
is due to \cite[Definition~2.5]{Jun1};
it is crucial for establishing the symplectic sum formula.
In \cite[Definition~3.18]{LR}, a pair of relative stable morphisms
into  $(X,V)$ and $(Y,V)$ with matching conditions at the contact points is used instead.
In \cite[Section~12]{IPsum}, the analogue of the notion of~\cite{Jun1} is used without
quotienting out by the reparametrizations of the rubber components; see~\eref{Thdfn_e2}.
The (virtually) main strata of the moduli spaces of stable morphisms into $X\!\cup_V\!Y$  arising 
from~\cite{LR} and~\cite{IPsum} are the same as those of~\cite{Jun1},  
but other strata (i.e.~maps into the spaces  $X\!\cup_V^m\!Y$ of~\eref{XnVYdfn_e} with $m\!\in\!\Z^+$)
are~not.
In particular, the morphisms into $X\!\cup_VY$ in the sense of either~\cite{LR} or~\cite{IPsum} 
are not limits of morphisms into smoothings of $X\!\cup_VY$
with respect to a Hausdorff topology on the moduli space of morphisms into~$\cZ$
and do not provide the necessary setting for establishing the symplectic sum formula.

\subsection{Moduli spaces for $(X,V)$:  \cite[Section~7]{IPrel}, \cite[Sections~3.2,3.3]{LR}}
\label{RelInv_sub0}

\noindent
Let $(X,\om_X)$ be a compact symplectic manifold, $V\!\subset\!X$ be a closed
symplectic hypersurface, and $J_X$ be an $\om_X$-compatible almost complex structure,
such that $J_X(TV)\!=\!TV$.
If $u\!:(\Si,\fj)\!\lra\!(X,J_X)$ is a smooth map from a Riemann surface, let
$$\dbar_{J_X,\fj}u=\frac12\big(\nd u+\{u^*J_X\}\circ\nd u\circ\fj\big)
\in \Ga^{0,1}_{J_X,\fj}(\Si;u^*TX)\equiv
\Ga\big(\Si;(T^*\Si)^{0,1}\!\otimes_{\C}\!u^*TX\big).$$
We denote by $\na$ the Levi-Civita connection of the metric $\om_X(\cdot,J_X\cdot)$ 
on~$X$ and by $\wt\na$ the corresponding $J_X$-linear connection; 
see \cite[p41]{MS2}.
If $u\!:(\Si,\fj)\!\lra\!(X,J_X)$ is $(J_X,\fj)$-holomorphic, 
i.e.~$\dbar_{J_X,\fj}u\!=\!0$, the linearization of the $\dbar_{J_X,\fj}$-operator at~$u$ 
is given~by
\begin{gather}
D_u\!: \Ga(\Si;u^*TX)\lra \Ga^{0,1}_{J_X,\fj}(\Si;u^*TX),\notag\\
\label{Dudfn_e}
D_u\xi=\frac12\big(\wt\na^u\xi+\{u^*J_X\}\circ\wt\na^u\xi\circ\fj)
+\frac14N_{J_X}^u(\xi,\nd u),
\end{gather}
where $\wt\na^u$ and $N_{J_X}^u$ are the pull-backs of the connection~$\wt\na$  
and of the Nijenhuis tensor~$N_{J_X}$ of~$J_X$ normalized as in \cite[p18]{MS2},
respectively, by~$u$; see \cite[(3.1.6)]{MS2}.
If in addition $u(\Si)\!\subset\!V$, 
$$D_u\big(\Ga(\Si;u^*TV)\big)\subset \Ga^{0,1}_{J_X,\fj}(\Si;u^*TV),$$
because the restriction of $D_u$ to $\Ga(\Si;u^*TV)$ is the linearization
of the $\dbar_{J_X,\fj}$-operator at~$u$  for the space of maps to~$V$.
Thus, $D_u$ descends to a first-order differential operator
\BE{DuNXV_e} D_u^{\cN_XV}\!: \Ga(\Si;u^*\cN_XV)\lra \Ga^{0,1}_{J_X,\fj}(\Si;u^*\cN_XV),\EE
which plays a central role in compactifying the moduli space of relative maps to~$(X,V)$.\\

\noindent
Since $J_X(TV)\!=\!TV$, $J_X$ induces a complex structure~$\fI_{X,V}$ on
(the fibers~of) the normal bundle
$$\pi_{X,V}\!: \cN_XV\equiv TX|_V\big/TV\lra V.$$
A connection~$\na^{\cN_XV}$ in $(\cN_XV,\fI_{X,V})$ induces a splitting of the exact sequence
\BE{NXVsplit_e}\begin{split}
0&\lra \pi_{X,V}^*\cN_XV\lra T(\cN_XV)
\stackrel{\nd\pi_{X,V}}{\lra} \pi_{X,V}^*TV\lra0
\end{split}\EE
of vector bundles over~$\cN_XV$ which restricts to the canonical splitting
over the zero section and is preserved by the multiplication by~$\C^*$;
see \cite[Lemma~1.1]{anal}.
For each trivialization 
$$\cN_XV|_U\approx U\!\times\!\C$$ 
over an open subset~$U$ of~$V$, there exists 
$\al\in \Ga(U;T^*V\!\otimes_{\R}\!\C)$
such that the image of $\pi_{X,V}^*TV$ corresponding to this splitting is given~by
$$T_{(x,w)}^{\hor}(\cN_XV)=\big\{(v,-\al_x(v)w)\!:\,v\!\in\!T_xV\big\}
\qquad\forall~(x,w)\!\in\!U\!\times\!\C.$$
The isomorphism $(x,w)\!\lra\!(x,w^{-1})$ of $U\!\times\!\C^*$ maps this vector space~to
\begin{equation*}\begin{split}
T_{(x,w^{-1})}^{\hor}\big((\cN_XV)^*\big)
&=\big\{(v,w^{-2}\al_x(v)w)\!:\,v\!\in\!T_xV\big\}\\
&=\big\{(v,\al_x(v)w^{-1})\!:\,v\!\in\!T_xV\big\}
\qquad\forall~(x,w)\!\in\!U\!\times\!\C^*.
\end{split}\end{equation*}
Thus, the splitting of \eref{NXVsplit_e} induced by a connection in~$(\cN_XV,\fI_{X,V})$ 
extends to a splitting of the exact sequence 
$$0\lra T^{\vrt}(\P_XV) \lra T(\P_XV)
\stackrel{\nd\pi_{X,V}}{\lra} \pi_{X,V}^*TV\lra0,$$
where  $\P_XV$ is as in~\eref{PVdfn_e} and $\pi_{X,V}\!:\P_XV\!\lra\!V$
is the bundle projection map; this splitting restricts to the canonical splittings over 
$$\P_{X,\i}V\equiv\P(\cN_XV\!\oplus\!0)
\qquad\hbox{and}\qquad \P_{X,0}V\equiv\P(0\!\oplus\!\cO_V)$$
and is preserved by the multiplication by~$\C^*$.
Via this splitting, the almost complex structure $J_V\!\equiv\!J_X|_V$ and 
the complex structure $\fI_{X,V}$ in the fibers of~$\pi_{X,V}$ induce
an almost complex structure~$J_{X,V}$ on~$\P_XV$ which restricts to almost complex
structures on $\P_{X,\i}V$ and~$\P_{X,0}V$
and is preserved by the~$\C^*$-action.
Furthermore, the projection $\pi_{X,V}\!:\P_XV\!\lra\!V$ is $(J_V,J_{X,V})$-holomorphic.
By \cite[Lemma~2.2]{anal}, $\xi\!\in\!\Ga(V;\cN_XV)$ is 
$(J_{X,V},J_X|V)$-holomorphic if and only if~$\xi$ lies in the kernel of 
the $\dbar$-operator on~$(\cN_XV,\fI_{X,V})$ corresponding to the connection used above.\\

\noindent
For each $m\!\in\!\Z^{\ge0}$, let 
\begin{gather}\label{XmVdfn_e} 
X_m^V=\big(X\sqcup\{1\}\!\times\!\P_XV\sqcup\ldots\sqcup
\{m\}\!\times\!\P_XV\big)/\!\!\sim\,,\qquad\hbox{where}\\
\notag
x\sim 1\!\times\!\P_{X,\i}V|_x\,,~~~
r\!\times\!\P_{X,0}V|_x\sim (r\!+\!1)\!\times\!\P_{X,\i}V|_x
\quad  \forall\,x\!\in\!V,~r=1,\ldots,m\!-\!1;
\end{gather}
see Figure~\ref{relcurve_fig}.
We denote by $J_m$ the almost complex structure on~$X_m^V$ so that 
$$J_m|_X=J_X \qquad\hbox{and}\qquad 
J_m|_{\{r\}\times\P_XV}=J_{X,V} \quad \forall~r=1,\ldots,m.$$
For each $(c_1,\ldots,c_m)\!\in\!\C^*$, define
\BE{Thdfn_e}\Th_{c_1,\ldots,c_m}\!:X_m^V\lra X_m^V \qquad\hbox{by}\quad
\Th_{c_1,\ldots,c_m}(x)=\begin{cases}x,&\hbox{if}~x\!\in\!X;\\
(r,[c_rv,w]),&\hbox{if}~x\!=\!(r,[v,w])\!\in\!r\!\times\!\P_XV.
\end{cases}\EE
This diffeomorphism is biholomorphic with respect to~$J_m$ and
preserves the fibers of the projection $\P_XV\!\lra\!V$
and the sections~$\P_{X,0}V$ and~$\P_{X,\i}V$.\\

\noindent
The moduli space of relative stable maps into $(X,V)$ is constructed in~\cite{IPrel}
under the additional assumption that
\BE{NijenCond_e}N_{J_X}(v,w)\in T_xV \qquad\forall~v,w\!\in\!T_xX,~x\!\in\!V.\EE
In light of~\eref{Dudfn_e}, this assumption insures that the operator $D_u^{\cN_XV}$
is $\C$-linear for every $(J_X,\fj)$-holomorphic map $u\!:\Si\!\lra\!V$ and thus
the operator 
$$\Ga(V;TX|_V)\lra\Ga\big(V;(T^*V)^{0,1}\!\otimes_{\C}TX|_V\big), \qquad
\xi\lra\frac12\big(\wt\na\xi+J_X\circ\wt\na\xi\circ J_X\big),$$
induces a $\dbar$-operator on $(\cN_XV,\fI_{X,V})$ corresponding to some connection 
$\na^{\cN_XV}$ in $(\cN_XV,\fI_{X,V})$; see \cite[Section~2.3]{anal}.
Let $J_{X,V}$ be the complex structure on~$\P_XV$ induced by~$J_X$
and~$\na^{\cN_XV}$ as in the paragraph above the previous one;
it depends only on the above $\dbar$-operator and not on the connection~$\na^{\cN_XV}$
realizing~it.
Thus, for every $(J_X,\fj)$-holomorphic map $u\!:\Si\!\lra\!V$
and $\xi\!\in\!\Ga(\Si;u^*\cN_XV)$,
$\xi\!\in\!\ker D_u^{\cN_XV}$ if and only if $\xi\!:\Si\!\lra\!\P_XV$
is a $(J_{X,V},\fj)$-holomorphic map.\\

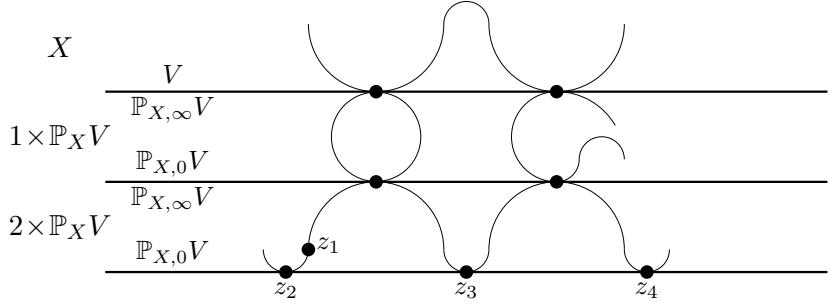
\begin{figure}
\begin{pspicture}(38,-.8)(11,3.2)
\psset{unit=.3cm}
\psline[linewidth=.1](48,-2)(80,-2)
\rput(51,6.8){\smsize{$V$}}\rput(51,5.2){\smsize{$\P_{X,\i}V$}}
\rput(51,2.8){\smsize{$\P_{X,0}V$}}\rput(51,1.2){\smsize{$\P_{X,\i}V$}}
\rput(51,-1.2){\smsize{$\P_{X,0}V$}}
\rput(46,8){$X$}\rput(46,4){$1\!\times\!\P_XV$}\rput(46,0){$2\!\times\!\P_XV$}
\psarc[linewidth=.04](60,9){3}{180}{0}\pscircle*(60,6){.3}
\psarc[linewidth=.04](64,9){1}{0}{180}
\psarc[linewidth=.04](68,9){3}{180}{0}\pscircle*(68,6){.3}
\psarc[linewidth=.04](60,-1){3}{0}{180}
\psarc[linewidth=.04](68,-1){3}{0}{180}
\psarc[linewidth=.04](64,-1){1}{180}{0}
\psarc[linewidth=.04](56,-1){1}{180}{0}
\psline[linewidth=.1](48,2)(80,2)
\pscircle*(57,-1){.3}\rput(57.9,-.8){\smsize{$z_1$}}
\pscircle*(60,2){.3}\rput(56,-2.8){\smsize{$z_2$}}
\pscircle*(64,-2){.3}\rput(64,-2.8){\smsize{$z_3$}}
\pscircle*(68,2){.3}
\pscircle[linewidth=.04](60,4){2}
\psarc[linewidth=.04](68,4){2}{90}{270}\psarc[linewidth=.04](68,3){3}{30}{90}
\psarc[linewidth=.04](68,3){1}{270}{0}\psarc[linewidth=.04](70,3){1}{0}{180}
\psarc[linewidth=.04](72,-1){1}{180}{0}
\psline[linewidth=.1](48,6)(80,6)
\pscircle*(56,-2){.3}\pscircle*(72,-2){.3}
\rput(72,-2.8){\smsize{$z_4$}}
\end{pspicture}
\caption{The image of a relative map with $k\!=\!1$ and $\bs\!=\!(2,2,2)$
to the space $X_2^V$.}
\label{relcurve_fig}
\end{figure}

\noindent
Suppose $k,\ell\!\in\!\Z^{\ge0}$ and $\bs\!=\!(s_1,\ldots,s_{\ell})\!\in\!(\Z^+)^{\ell}$ 
is a tuple satisfying~\eref{bsumcond_e}.
A \textsf{$k$-marked $J_X$-holomorphic map into~$X_0^V$ with contacts of order~$\bs$} is
a $J_X$-holomorphic map $u\!:\Si\!\lra\!X$ 
from a marked connected nodal Riemann surface $(\Si,z_1,\ldots,z_{k+\ell})$
such~that 
$$u^{-1}(V)=\big\{z_{k+1},\ldots,z_{k+\ell}\big\}, \quad
\ord_{z_{k+1}}^V\big(u|_{\Si}\big)=s_1,~\ldots,~
\ord_{z_{k+\ell}}^V\big(u|_{\Si}\big)=s_{\ell}.$$
For $m\!\in\!\Z^+$, 
a \textsf{$k$-marked $J_X$-holomorphic morphism into~$X_m^V$ with contacts of order~$\bs$} is
a continuous map $u\!:\Si\!\lra\!X_m^V$ 
from a marked connected nodal Riemann surface $(\Si,z_1,\ldots,z_{k+\ell})$
such~that 
\BE{MarOnTop_e} 
u^{-1}\big(\{m\}\!\times\!\P_{X,0}V\big)
=\big\{z_{k+1},\ldots,z_{k+\ell}\big\}, \quad
\ord_{z_{k+1}}^{\P_{X,0}V}(u|_{\Si})\!=\!s_1,~\ldots,~
\ord_{z_{k+\ell}}^{\P_{X,0}V}(u|_{\Si})\!=\!s_{\ell},\EE
and the restriction of~$u$ to each irreducible component $\Si_j$ of~$\Si$ is either 
\begin{enumerate}[label=$\bullet$,leftmargin=*]
\item a $J_X$-holomorphic map to $X$ such that the set $u|_{\Si_j}^{\,-1}(V)$ 
consists~of the nodes joining~$\Si_j$ to irreducible components of~$\Si$ mapped to
$\{1\}\!\times\!\P_XV$, or
\item a $J_{X,V}$-holomorphic map to $\{r\}\!\times\!\P_XV$ for some $r\!=\!1,\ldots,m$ such~that 
\begin{enumerate}[label=$\circ$,leftmargin=*]
\item the set $u|_{\Si_j}^{\,-1}(\{r\}\!\times\!\P_{X,\i}V)$ 
consists~of the nodes~$z_{j,i}$ joining~$\Si_j$ to irreducible components of~$\Si$ mapped to
$\{r\!-\!1\}\!\times\!\P_XV$ if $r\!>\!1$ and to~$X$ if $r\!=\!1$
and 
$$\ord_{z_{j,i}}^{\P_{X,\i}V}\big(u|_{\Si_j}\big)=
\begin{cases}
\ord_{z_{i,j}}^{\P_{X,0}V}(u|_{\Si_{i,j}}),&\hbox{if}~r\!>\!1;\\
\ord_{z_{i,j}}^V(u|_{\Si_{i,j}}),&\hbox{if}~r\!=\!1;
\end{cases}$$ 
where $z_{i,j}\!\in\!\Si_{i,j}$ is the point identified with~$z_{j,i}$, 
\item  if $r\!<\!m$, the set $u|_{\Si_j}^{\,-1}(\{r\}\!\times\!\P_{X,0}V)$ 
consists~of the nodes joining~$\Si_j$ to irreducible components of~$\Si$ mapped to
$\{r\!+\!1\}\!\times\!\P_XV$; 
\end{enumerate}
\end{enumerate}
see Figure~\ref{relcurve_fig}.
The \textsf{genus} and the \textsf{degree} of such a map $u\!:\!\Si\!\lra\!X_m^V$
are the arithmetic genus of~$\Si$ and the homology class
\BE{Adeg_e}A=\big[\pi_m\!\circ\!u\big]\in H_2(X;\Z),\EE
where $\pi_m\!:X_m^V\!\lra\!X$ is the natural projection.\\

\noindent
Two tuples $(\Si,z_1,\ldots,z_{k+\ell},u)$ and 
$(\Si',z_1',\ldots,z_{k+\ell}',u')$ as above are \textsf{equivalent}
if there exist a biholomorphic map $\vph\!:\Si'\!\lra\!\Si$ 
and \hbox{$c_1,\ldots,c_m\!\in\!\C^*$} so~that 
$$\vph(z_1')=z_1, \quad\ldots,\quad \vph(z_{k+\ell}')=z_{k+\ell}, \quad\hbox{and}\quad
u'=\Th_{c_1,\ldots,c_m}\circ u\circ\vph.$$
A tuple as above is \textsf{stable} if it has finitely many automorphisms
(self-equivalences).
For each stable tuple $(\Si,z_1,\ldots,z_{k+\ell},u)$ as above and $r\!=\!1,\ldots,m$, 
either 
\begin{enumerate}[label=$\bullet$,leftmargin=*]
\item the degree of the composition of $u|_{u^{-1}(\{r\}\times\P_XV)}$ 
with the projection to~$V$ is not zero, or 
\item the arithmetic genus of some topological component of 
$u^{-1}(\{r\}\times\P_XV)$ is positive, or
\item some topological component of 
$u^{-1}(\{r\}\times\P_XV)$ carries one of the marked points $z_1,\ldots,z_k$, or
\item the restriction of~$u$ to some topological component of 
$u^{-1}(\{r\}\times\P_XV)$ has at least
three special points: nodal, branch, or in the preimage of~$\P_{X,0}V$ or~$\P_{X,\i}V$.\\
\end{enumerate}
If $A\!\in\!H_2(X;\Z)$, $g,k,\ell\!\in\!\Z^{\ge0}$, and $\bs\!=\!(s_1,\ldots,s_{\ell})\!\in\!(\Z^+)^{\ell}$ 
is a tuple satisfying~\eref{bsumcond_e}, let
\BE{relmoddfn_e}\M_{g,k;\bs}^V(X,A)\subset \ov\M_{g,k;\bs}^V(X,A)\EE
denote the set of equivalence classes of stable $k$-marked genus~$g$ degree~$A$ $J_X$-holomorphic maps
into $X_0^V\!\equiv\!X$ and into $X_m^V$ for any $m\!\in\!\Z^{\ge0}$, respectively.
The latter space has a natural compact Hausdorff topology with respect to which 
the former space is an open subspace, but not necessarily a dense one.

\begin{rmk}\label{IPrel67_rmk}
The notion of relative map described by \cite[Definitions~7.1,7.2]{IPrel}
omits the first requirement in~\eref{MarOnTop_e}, allowing the contact marked points
to lie in any layer; this requirement is necessary to get a Hausdorff topology on $\ov\M_{g,k;\bs}^V(X,A)$.
The relative maps are defined in \cite{IPrel} in terms of 
elements of the kernel of the operator~\eref{DuNXV_e}.
It is never mentioned that such elements are $J_{X,V}$-holomorphic maps
to~$\P_XV$ for a certain $\C^*$-invariant almost complex structure~$J_{X,V}$
determined by the operator~\eref{DuNXV_e};
according to E.~Ionel, the authors realized this only recently.
This is necessary to get the multi-layered structure of \cite[Section~7]{IPrel}
by repeatedly rescaling in the normal direction as in \cite[Section~6]{IPrel}.
The rescaling argument in \cite[Section~6]{IPrel} does not ensure that the bubbles
in the different layers connect.
It also does not involve adding new components to the domain of a map to~$X$ 
and thus does not
allow for the appearance of maps as those from domains as in 
Figure~\ref{relcurve_fig} that restrict to a fiber map on the left component mapping 
into $\{1\}\!\times\!\P_XV$.
This means \cite{IPrel} does not even have the necessary ingredients to establish 
that $\ov\M_{g,k;\bs}^V(X,A)$ is compact and Hausdorff
(and the latter issue is never even considered).
\end{rmk}

\noindent
The roles of the ``components" $X$ and $\P_XV$  of the target space
for relative stable maps in the setting of~\cite{LR} are played by 
$\oXV$ and $\R\!\times\!SV$,
respectively, or alternatively by $\tXV$ and $[-\i,\i]\!\times\!SV$;
see \eref{oXVdfn_e} and~\eref{tXVdfn_e}.
For each $m\!\in\!\Z^{\ge0}$, let
\BE{otXVmdfn_e}
\oXVm=\oXV\sqcup\bigsqcup_{r=1}^m\{r\}\!\times\!\R\!\times\!SV, \qquad
\tXVm=\bigg(\tXV\sqcup\bigsqcup_{r=1}^m\{r\}\!\times\![-\i,\i]\!\times\!SV\bigg)\Big/\!\!\sim,\EE
where 
$$(-\i)\!\times\!x\sim 1\!\times\!\i\!\times\!x\,,~~~
r\!\times\!(-\i)\!\times\!x \sim (r\!+\!1)\!\times\!\i\!\times\!x
\qquad \forall\,x\!\in\!SV,~r=1,\ldots,m\!-\!1.$$
The homeomorphism~\eref{Thdfn_e} induces homeomorphisms
\BE{otThdfn_e}\oTh_{c_1,\ldots,c_m}\!: \oXVm\lra\oXVm \qquad\hbox{and}\qquad
\tTh_{c_1,\ldots,c_m}\!: \tXVm\lra\tXVm; \EE
the first is the restriction of the homeomorphism~\eref{Thdfn_e},
while the second is the continuous extension of the first.
As in the setting of~\cite{IPrel} described above, an almost complex structure~$J_X$
on~$X$ such that $J_X(TV)\!=\!TV$ induces an almost complex structures~$\oJm$
on~$\oXVm$ so that the first map in~\eref{otThdfn_e} is biholomorphic.
In the approach of~\cite{LR}, $J_X$ is chosen so that it has an asymptotic behavior near~$V$
as at the end of Section~\ref{SympSum_subs}.
The almost complex structure~$\oJm$ then satisfies
$$\oJm\frac\prt{\prt a_X}=\ze_H ~~\hbox{on}~(-\i,-a_0)\!\times\!SV\subset X, 
\qquad
\oJm\frac\prt{\prt a_{\R}}=\ze_H ~~\hbox{on}~\{r\}\!\times\!\R\!\times\!SV,~r\!=\!1,\ldots,m,$$
for some $a_0\!\in\!\R^+$ sufficiently large, where $\ze_H$ is the characteristic vector
of the $S^1$ action as before and $\frac\prt{\prt a_X}$ and $\frac\prt{\prt a_{\R}}$
are the coordinate vector fields in the $\R$-direction on $(-\i,-a_0)\!\times\!SV$
and $\R\!\times\!SV$, respectively.
It restricts to the pull-back of~$J_V$ on $\ker\al\!\subset\!T(SV)$,
where $\al$ is a connection 1-form on the $S^1$-bundle $SV\!\lra\!V$.\\

\noindent
The roles of the components~$\Si_j$ of the domains of $J_m$-holomorphic maps~$u$ 
into~$\XVm$ with contact with $V\!\subset\!X$ or $\P_{X,0}V,\P_{X,\i}V\!\subset\!\P_XV$
of order~$s_{j,i}$ at $z_{j,i}\!\in\!\Si_j$ are played by the punctured Riemann 
surfaces $\oSi_j\!=\!\Si_j\!-\{z_{j,i}\!\!:i\}$.
The relative maps in the sense of \cite[Definition~3.14]{LR} are $\oJm$-holomorphic maps
\BE{oudfn_e}\ou\!: \oSi\equiv\bigsqcup_j\oSi_j\lra \oXVm\EE
for some $m\!\in\!\Z^{\ge0}$ satisfying certain limiting, stability, and degree conditions.
Let 
$$\pi_{\R},\pi_{SV}\!: \R^-\!\times\!SV\lra\R^-,SV, \qquad 
\pi_{\R},\pi_{SV}\!: \{r\}\!\times\!\R\!\times\!SV\lra\R,SV$$
denote the projection maps.
The punctures of each topological component~$\oSi_j$ are either positive or negative 
with respect to~$\ou$.
If $z\!=\!\ne^{-t+\fI\th}$ is a local coordinate centered at a 
\textsf{positive puncture}~$z_{j,i}$ of~$\Si_j$,
i.e.~$t\!\lra\!\i$ as $z\!\lra\!0$, then 
$\ou(\oSi_j)\!\subset\!\{r\}\!\times\!\R\!\times\!SV$ for some $r\!=\!1,\ldots,m$
and  
$$\lim_{t\lra\i}\pi_{\R}\!\circ\!\ou\big(\ne^{-t+\fI\th}\big)= \i, \qquad
\lim_{t\lra\i}\pi_{SV}\!\circ\!\ou\big(\ne^{-t+\fI\th}\big)= \ga\big(\ne^{\fI k\th}\big) 
\qquad\forall~\th\!\in\!S^1,$$
for some $k\!\in\!\Z^+$ and 1-periodic $S^1$-orbit $\ga\!:S^1\!\lra\!SV$ over 
a point $x\!\in\!V$.
In such a case, we will write 
$$\cP_{z_{j,i}}^+(\ou)=(x,k), \qquad \ord_{z_{j,i}}^+(\ou)=k.$$
If $z\!=\!\ne^{t+\fI\th}$ is a local coordinate centered at a \textsf{negative puncture} of~$\Si_j$,
i.e.~$t\!\lra\!-\i$ as $z\!\lra\!0$, then either
$$\ou(\oSi_j)\subset\{r\}\!\times\!\R\!\times\!SV  ~~\hbox{for some}~~r\!=\!1,\ldots,m
\qquad\hbox{or}\qquad \ou(\ne^{t+\fI\th})\in \R^-\!\times\!SV\subset X\!-\!V$$
and 
$$\lim_{t\lra-\i}\pi_{\R}\!\circ\!\ou\big(\ne^{t+\fI\th}\big)=-\i, \qquad
\lim_{t\lra-\i}\pi_{SV}\!\circ\!\ou\big(\ne^{t+\fI\th}\big)= \ga\big(\ne^{\fI k\th}\big) 
\qquad\forall~\th\!\in\!S^1,$$
for some $k\!\in\!\Z^+$ and 1-periodic $S^1$-orbit $\ga\!:S^1\!\lra\!SV$
over a point~$x$ in~$V$.
In either of the two cases, we will write 
$$\cP_{z_{j,i}}^-(\ou)=(x,k), \qquad \ord_{z_{j,i}}^-(\ou)=k.$$
Any map~\eref{oudfn_e} satisfying these conditions has a well-defined degree $A\!\in\!H_2(X;\Z)$
obtained by composing~$\ou$ with the projection to~$\XVm$ (which sends each limiting orbit
$\ga\!\subset\!SV$
to a single point $x\!\in\!V$) and then with the projection $\pi_m\!:\XVm\!\lra\!X$.\\

\noindent
For any nodal Riemann surface, we denote by $\Si^*\!\subset\!\Si$ the subspace
of smooth points.
Let \hbox{$A\!\in\!H_2(X;\Z)$}, $g,k,\ell\!\in\!\Z^{\ge0}$, and
$\bs\!=\!(s_1,\ldots,s_{\ell})\!\in\!(\Z^+)^{\ell}$ 
be a tuple satisfying~\eref{bsumcond_e}.
The relative moduli space $\ov\M_{g,k;\bs}^V(X,A)$ of~\cite{LR}
consists of stable tuples $(\Si,z_1,\ldots,z_{k+\ell},\ou)$,
where $(\Si,z_1,\ldots,z_{k+\ell})$ is a genus~$g$ marked nodal connected compact 
Riemann surface, 
$$\ou\!: \oSi\equiv\bigsqcup_j\oSi_j\lra \oXVm,
\qquad \Si^*\!-\!\{z_{k+1},\ldots,z_{\ell}\} \subset \oSi \subset \Si\!-\!\{z_{k+1},\ldots,z_{\ell}\},$$
such that 
$\ou$ is a $\oJm$-holomorphic map of degree~$A$,
$$\ou^{-1}\big(\{r\}\!\times\!\R\!\times\!SV\big)\neq\eset \quad\forall~r=1,\ldots,m,
\quad \ord_{z_{k+i}}^-(\ou)=s_i\quad\forall~i\!=\!1,\ldots,\ell,$$
punctured neighborhoods of~$z_{k+1},\ldots,z_{k+\ell}$ are mapped 
to~$\{m\}\!\times\!\R\!\times\!SV$ if $m\!\in\!\Z^+$,
each node in $\Si\!-\!\oSi$ gives rise to one positive and one negative puncture of~$(\oSi,\ou)$,
and the positive punctures~$z_{j,i}$ of any component~$\Si_j$ mapped into 
$\{r\}\!\times\!\R\!\times\!SV$ for some $r\!=\!1,\ldots,m$
correspond to the nodes of~$\Si$ joining~$\Si_j$
to the components mapped into $\{r\!-\!1\}\!\times\!\R\!\times\!SV$ if $r\!>\!1$ and to~$X$ if $r\!=\!1$
outside of the punctures~and 
\BE{cPcond_e}\cP_{z_{j,i}}^+(\ou)\!=\!\cP_{z_{i,j}}^-(\ou),\EE 
where $z_{i,j}\!\in\!\Si_{i,j}$ is the point identified with~$z_{j,i}$.
Two such relative maps $(\Si,z_1,\ldots,z_{k+\ell},\ou)$ and 
$(\Si',z_1',\ldots,z_{k+\ell}',\ou')$ are \textsf{equivalent}
if there are $c_1,\ldots,c_m\!\in\!\C^*$
and a biholomorphic $\vph\!:\Si'\!\lra\!\Si$ so~that 
$$\vph(z_1')=z_1, \quad\ldots,\quad \vph(z_{k+\ell}')=z_{k+\ell}, \quad\hbox{and}\quad
\ou'=\Th_{c_1,\ldots,c_m}\circ \ou\circ\vph.$$
A tuple as above is \textsf{stable} if it has finitely many automorphisms
(self-equivalences).\\

\noindent
By the same construction as in~\eref{tXVdfn_e}, 
the punctured Riemann surfaces~$\oSi_j$ above can be compactified 
to bordered surfaces~$\hSi_j$.
The matching condition~\eref{cPcond_e} insures that the surfaces~$\hSi_j$
can be glued together along pairs of boundary components corresponding
to the same node of~$\Si$ into a surface~$\hSi$ with $\ell$~boundary components
in  such a way that~$\ou$ extends to a continuous~map
$\tu\!:\hSi\!\lra\!\tXVm$.
Composing~$\tu$ with the projection $\tXVm\!\lra\!\XVm$, we obtain a relative map
 $u\!:\Si\!\lra\!\XVm$.
Removing the preimages of $V\!\subset\!X$ and $\P_{X,0}V,\P_{X,\i}V\!\subset\!\P_XV$
under a relative map $u\!:\Si\!\lra\!\XVm$,
we obtain a relative map $\ou\!:\oSi\!\lra\!\oXVm$ in the sense of~\cite{LR}.
Thus, the moduli spaces of relative maps $\ov\M_{g,k;\bs}^V(X,A)$ 
in the two descriptions are canonically identified
when the same almost complex structure~$J_X$ on~$X$ is~used. 
While the space of admissible~$J_X$ is smaller in~\cite{LR},
it is still non-empty and path-connected, possesses 
the same transversality properties as the larger space of~$J_X$ in~\cite{IPrel},
and so is just as good for defining relative invariants.
On the other hand, the stronger restriction on~$J_X$ in~\cite{LR} simplifies
the required gluing constructions; see Section~\ref{sumpf_subs}.

\begin{rmk}\label{LR33_rmk}
The key definition of relative stable maps, \cite[Definition~3.14]{LR},
is not remotely precise.
It involves three different Riemann surfaces, without a clear connection between~them, 
a continuous map into a vaguely described space, and a vaguely specified equivalence relation. 
The signs of the limiting periods are not properly defined either.
\end{rmk}

\subsection{Moduli spaces for $X\!\cup_V\!Y$:  \cite[Section~12]{IPrel}, \cite[Sections~3.2,3.3]{LR}}
\label{RelInv_sub0b}

\noindent
Let $V\!\subset\!X$ be as in Section~\ref{RelInv_sub0}.
Suppose $(Y,\om_Y)$ is another symplectic manifold containing~$V$ as 
a symplectic hypersurface so that~\eref{cNVcond_e} holds, 
$J_Y$ is an $\om_Y$-compatible almost complex structure,
such that $J_Y(TV)\!=\!TV$ and $J_Y|_{TV}\!=\!J_X|_{TV}\!\equiv\!J_V$,
and we have chosen an isomorphism as in~\eref{cNpair_e}.
Such an isomorphism identifies $\P_XV$ with~$\P_YV$.
For each $m\!\in\!\Z^{\ge0}$, let
\BE{XnVYdfn_e}\begin{split}
X\!\cup_V^m\!Y=&\big(X_m^V\sqcup Y_m^V\big)\big/\!\!\sim, \\
 X_m^V\ni(r,x)\sim(m\!+\!1\!-\!r,x)\in Y_m^V &\qquad\forall~r=1,\ldots,m,~x\in\P_XV=\P_YV\,;
\end{split}\EE
see Figure~\ref{sumcurve_fig}.
We extend~\eref{Thdfn_e} to an isomorphism
\BE{Thdfn_e2}\Th_{c_1,\ldots,c_m}\!: X\!\cup_V^m\!Y\lra X\!\cup_V^m\!Y\EE
by taking it to be the identity on~$Y$.
By the discussion following~\eref{NXVsplit_e}, 
the almost complex structures~$J_{X,V}$ and~$J_{Y,V}$ on $\P_XV\!=\!\P_YV$
agree if they are induced from~$J_V$ using dual connections in~$\cN_XV$ and~$\cN_YV$.\\

\noindent
The moduli spaces of stable maps into~$X\!\cup_V\!Y$ are defined under the additional
assumptions that~\eref{NijenCond_e} holds for~$X$ and~$Y$
and  the linearized operator $D_u^{\cN_XV}$ and $D_u^{\cN_YV}$ as in~\eref{DuNXV_e}
are dual to each other.
These assumptions ensure that the almost complex structures~$J_{X,V}$ and~$J_{Y,V}$ 
on $\P_XV\!=\!\P_YV$ agree and so induce a well-defined almost complex structure~$J_m$
on $X\!\cup_V^m\!Y$, which is preserved by~\eref{Thdfn_e2}.
They are satisfied by $J_X\!=\!J_{\cZ}|_X$ and $J_Y\!=\!J_{\cZ}|_Y$ 
if $J_{\cZ}$ is as in Proposition~\ref{SymSum_prp}.\\

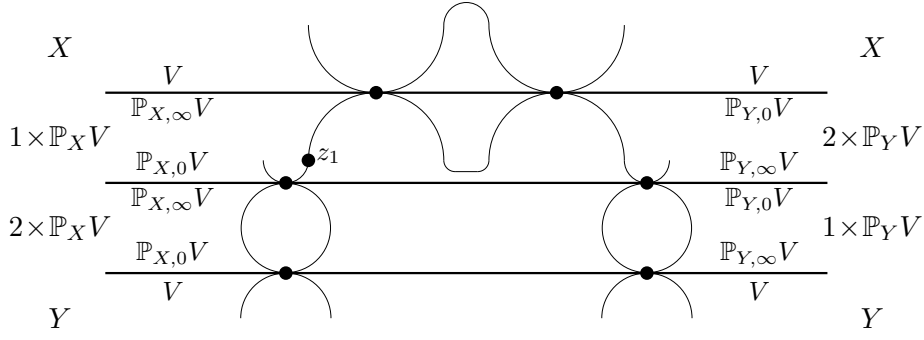
\begin{figure}
\begin{pspicture}(38,-2.3)(11,3)
\psset{unit=.3cm}
\psline[linewidth=.1](48,-2)(80,-2)
\rput(51,6.8){\smsize{$V$}}\rput(51,-2.8){\smsize{$V$}}
\rput(51,5.2){\smsize{$\P_{X,\i}V$}}
\rput(51,2.8){\smsize{$\P_{X,0}V$}}\rput(51,1.2){\smsize{$\P_{X,\i}V$}}
\rput(51,-1.2){\smsize{$\P_{X,0}V$}}\rput(46,8){$X$}\rput(46,-4){$Y$}
\rput(46,4){$1\!\times\!\P_XV$}\rput(46,0){$2\!\times\!\P_XV$}
\rput(77,6.8){\smsize{$V$}}\rput(77,-2.8){\smsize{$V$}}
\rput(77,5.2){\smsize{$\P_{Y,0}V$}}
\rput(77,2.8){\smsize{$\P_{Y,\i}V$}}\rput(77,1.2){\smsize{$\P_{Y,0}V$}}
\rput(77,-1.2){\smsize{$\P_{Y,\i}V$}}\rput(82,8){$X$}\rput(82,-4){$Y$}
\rput(82,4){$2\!\times\!\P_YV$}\rput(82,0){$1\!\times\!\P_YV$}
\psarc[linewidth=.04](60,9){3}{180}{0}\pscircle*(60,6){.3}
\psarc[linewidth=.04](64,9){1}{0}{180}
\psarc[linewidth=.04](68,9){3}{180}{0}\pscircle*(68,6){.3}
\psarc[linewidth=.04](60,3){3}{0}{180}\psarc[linewidth=.04](68,3){3}{0}{180}
\psarc[linewidth=.04](63.5,3){.5}{180}{270}\psarc[linewidth=.04](64.5,3){.5}{270}{0}
\psline[linewidth=.04](63.5,2.5)(64.5,2.5)
\psarc[linewidth=.04](56,3){1}{180}{0}\psarc[linewidth=.04](72,3){1}{180}{0}
\psline[linewidth=.1](48,2)(80,2)
\pscircle*(56,2){.3}\pscircle*(72,2){.3}
\pscircle[linewidth=.04](56,0){2}\pscircle[linewidth=.04](72,0){2}
\psline[linewidth=.1](48,6)(80,6)
\pscircle*(56,-2){.3}\pscircle*(72,-2){.3}
\pscircle*(57,3){.3}\rput(57.9,3.2){\smsize{$z_1$}}
\psarc[linewidth=.04](56,-4){2}{0}{180}\psarc[linewidth=.04](72,-4){2}{0}{180}
\end{pspicture}
\caption{The image of a limit map with $k\!=\!1$ to the space $X\!\cup_V^2Y$.}
\label{sumcurve_fig}
\end{figure}

\noindent
Let $k\!\in\!\Z^{\ge0}$.
A \textsf{$k$-marked $(J_X,J_Y)$-holomorphic map into~$X\!\cup_0^V\!Y$} is
a continuous map \hbox{$u\!:\Si\!\lra\!X\!\cup_V\!Y$} 
from a marked connected nodal Riemann surface $(\Si,z_1,\ldots,z_k)$
such~that the restriction of~$u$ to each irreducible component $\Si_j$ of~$\Si$ is either 
\begin{enumerate}[label=$\bullet$,leftmargin=*]

\item a map to $X$ such that the set $u|_{\Si_j}^{\,-1}(V)$ 
consists~of the nodes~$z_{j,i}$ joining~$\Si_j$ to irreducible components of~$\Si$ mapped to~$Y$
and
$$\ord_{z_{j,i}}^V\big(u|_{\Si_j}\big)=\ord_{z_{i,j}}^V\big(u|_{\Si_{i,j}}\big),$$
where $z_{i,j}\!\in\!\Si_{i,j}$ is the point identified with~$z_{j,i}$, or  

\item a map to $Y$ such that the set $u|_{\Si_j}^{\,-1}(V)$ 
consists~of the nodes~$z_{j,i}$ joining~$\Si_j$ to irreducible components of~$\Si$ mapped to~$X$
and
$$\ord_{z_{j,i}}^V\big(u|_{\Si_j}\big)=\ord_{z_{i,j}}^V\big(u|_{\Si_{i,j}}\big),$$
where $z_{i,j}\!\in\!\Si_{i,j}$ is the point identified with~$z_{j,i}$.\\

\end{enumerate}

\noindent
Suppose in addition that $m\!\in\!\Z^+$.
A \textsf{$k$-marked $(J_X,J_Y)$-holomorphic map into~$X\!\cup_m^V\!Y$} is
a continuous map \hbox{$u\!:\Si\!\lra\!X\!\cup_V^m\!Y$} 
from a marked connected nodal Riemann surface $(\Si,z_1,\ldots,z_k)$
such~that 
the restriction of~$u$ to each irreducible component $\Si_j$ of~$\Si$ is either 
\begin{enumerate}[label=$\bullet$,leftmargin=*]
\item a map to $X$ such that the set $u|_{\Si_j}^{\,-1}(V)$ 
consists~of the nodes~$z_{j,i}$ joining~$\Si_j$ to irreducible components of~$\Si$ mapped to
$\{1\}\!\times\!\P_XV$ and
$$\ord_{z_{j,i}}^V\big(u|_{\Si_j}\big)=
\ord_{z_{i,j}}^{\P_{X,\i}V}\big(u|_{\Si_{i,j}}\big), $$ 
where $z_{i,j}\!\in\!\Si_{i,j}$ is the point identified with~$z_{j,i}$, or  
\item a map to $Y$ such that the set $u|_{\Si_j}^{\,-1}(V)$ 
consists~of the nodes~$z_{j,i}$ joining~$\Si_j$ to irreducible components of~$\Si$ mapped to
$\{1\}\!\times\!\P_YV$ and
$$\ord_{z_{j,i}}^V\big(u|_{\Si_j}\big)=
\ord_{z_{i,j}}^{\P_{Y,\i}V}\big(u|_{\Si_{i,j}}\big), $$ 
where $z_{i,j}\!\in\!\Si_{i,j}$ is the point identified with~$z_{j,i}$, or
\item a map to $\{r\}\!\times\!\P_XV\!=\!\{m\!+\!1\!-\!r\}\!\times\!\P_YV$ 
for some $r\!=\!1,\ldots,m$ such~that 
\begin{enumerate}[label=$\circ$,leftmargin=*]
\item the set $u|_{\Si_j}^{\,-1}(\{r\}\!\times\!\P_{X,\i}V)$ 
consists~of the nodes~$z_{j,i}$ joining~$\Si_j$ to irreducible components of~$\Si$ mapped to
$\{r\!-\!1\}\!\times\!\P_XV$ if $r\!>\!1$ and to~$X$ if $r\!=\!1$
and 
$$\ord_{z_{j,i}}^{\P_{X,\i}V}\big(u|_{\Si_j}\big)=\begin{cases}
\ord_{z_{i,j}}^{\P_{X,0}V}(u|_{\Si_{i,j}}),&\hbox{if}~r\!>\!1;\\
\ord_{z_{i,j}}^V(u|_{\Si_{i,j}}),&\hbox{if}~r\!=\!1;\end{cases}$$ 
where $z_{i,j}\!\in\!\Si_{i,j}$ is the point identified with~$z_{j,i}$, and
\item  the set $u|_{\Si_j}^{\,-1}(\{r\}\!\times\!\P_{Y,\i}V)$ 
consists~of the nodes~$z_{j,i}$ joining~$\Si_j$ to irreducible components of~$\Si$ mapped to
$\{r\!-\!1\}\!\times\!\P_YV$ if $r\!>\!1$ and to~$Y$ if $r\!=\!1$
and 
$$\ord_{z_{j,i}}^{\P_{Y,\i}V}\big(u|_{\Si_j}\big)=\begin{cases}
\ord_{z_{i,j}}^{\P_{Y,0}V}(u|_{\Si_{i,j}}),&\hbox{if}~r\!>\!1;\\
\ord_{z_{i,j}}^V(u|_{\Si_{i,j}}),&\hbox{if}~r\!=\!1;\end{cases}$$ 
where $z_{i,j}\!\in\!\Si_{i,j}$ is the point identified with~$z_{j,i}$;
\end{enumerate}
\end{enumerate}
see Figure~\ref{sumcurve_fig}.
The \textsf{genus} and the \textsf{degree} of such a map $u\!:\!\Si\!\lra\!X\!\cup^m_V\!Y$
are the arithmetic genus of~$\Si$ and the homology class
\BE{Adeg_e2}A=\big[\pi_m\!\circ\!u\big]\in H_2(X\!\cup_V\!Y;\Z),\EE
where $\pi_m\!:X\!\cup^m_V\!Y\!\lra\!X\!\cup_V\!Y$ is the natural projection.\\

\noindent
Two tuples $(\Si,z_1,\ldots,z_k,u)$ and 
$(\Si',z_1',\ldots,z_k',u')$ as above are \textsf{equivalent}
if there exist a biholomorphic map $\vph\!:\Si'\!\lra\!\Si$
and $c_1,\ldots,c_m\!\in\!\C^*$ so~that 
$$\vph(z_1')=z_1, \quad\ldots,\quad \vph(z_k')=z_k, \quad\hbox{and}\quad
u'=\Th_{c_1,\ldots,c_m}\circ u\circ\vph\,.$$
A tuple as above is \textsf{stable} if it has finitely many automorphisms
(self-equivalences).\\

\noindent
If $A\!\in\!H_2(X\!\cup_V\!Y;\Z)$ and $g,k\!\in\!\Z^{\ge0}$, let
\BE{summoddfn_e}\M_{g,k}\big(X\!\cup_V\!Y,A\big)\subset 
\ov\M_{g,k}\big(X\!\cup_V\!Y,A\big)\EE
denote the set of equivalence classes of stable $k$-marked genus~$g$ degree~$A$ 
$(J_X,J_Y)$-holomorphic maps
into $X\!\cup_V^0\!Y\!\equiv\!X\!\cup_V\!Y$ and into $X\!\cup_V^m\!Y$ for any $m\!\in\!\Z^{\ge0}$, 
respectively.
The latter space has a natural compact Hausdorff topology with respect to which 
the former space is an open subspace, but not necessarily a dense one.
We denote by $\wt\M_{\chi,k}(X\!\cup_V\!Y,A)$ the analogue of the space
$\ov\M_{g,k}(X\!\cup_V\!Y,A)$ with disconnected domains~$\Si$.\\

\noindent
While each element of the smaller space in~\eref{summoddfn_e} corresponds 
to a pair of relative maps, from possibly disconnected domains, into $(X,V)$ and $(Y,V)$ 
with matching conditions at the contact points,
there is no canonical splitting of this type for other elements of 
the larger space in~\eref{summoddfn_e}.
Nevertheless, the compact Hausdorff topologies on the relative moduli spaces
described in Section~\ref{RelInv_sub0}
induce the compact Hausdorff topology on $\ov\M_{g,k}(X\!\cup_V\!Y;A)$
of the previous paragraph.

\begin{rmk}\label{IPsum12_rmk}
The moduli space $\ov\M_{g,k}(X\!\cup_V\!Y,A)$ does not appear in~\cite{IPrel} or~\cite{IPsum}.
The space similar to $\ov\M_{g,k}(X\!\cup_V\!Y,A)$ that appears at the top of 
page~1003 in~\cite{IPsum} does not quotient the maps by the $(\C^*)^m$-action on~$X\!\cup_V^m\!Y$
and thus cannot be Hausdorff by \cite[Sections~6,7]{IPrel}.
This space is also not relevant and leads to the mistaken appearance of the $S$-matrix
in the main symplectic sum formulas in~\cite{IPsum}; see Section~\ref{Smat_subs} for more details.
\end{rmk}

\noindent
Continuing with the setup at the end of Section~\ref{SympSum_subs}, let
\BE{otXnVYdfn_e}\begin{split}
\oXYm=\big(\oXVm\sqcup\oYVm\big)\big/\!\!\sim, &\qquad
\tXYm=\big(\tXVm\sqcup\tYVm\big)\big/\!\!\sim,\\
 \tXVm\ni(r,x)\sim(m\!+\!1\!-\!r,x)\in\tYVm &\qquad
 \forall~r=1,\ldots,m,~x\in[-\i,\i]\!\times\!SV.
\end{split}\EE
The homeomorphisms~\eref{otThdfn_e} extend to these spaces by taking them to be 
the identity on~$\oYV$.\\

\noindent
Let $A\!\in\!H_2(X\!\cup_V\!Y;\Z)$ be an element in the image of 
$H_2(X;\Z)\!\times_V\!H_2(Y;\Z)$ under the natural homomorphism
$$H_2(X;\Z)\oplus H_2(Y;\Z)\lra H_2(X\!\cup_V\!Y;\Z)$$
and $g,k\!\in\!\Z^{\ge0}$.
For almost complex structures~$J_X$ on~$X$ and~$J_Y$ on~$Y$ satisfying $J_X|_V\!=\!J_Y|_V$ and 
the asymptotic condition at the end of Section~\ref{SympSum_subs},
the notion of $k$-marked genus~$g$ degree~$A$ stable map to $X\!\cup_V^m\!Y$ 
described above can be re-formulated in the terminology of~\cite{LR} 
similarly to the re-formulation for relative maps to~$(X,V)$
in the second half of Section~\ref{RelInv_sub0}.
Such a map is a tuple $(\Si,z_1,\ldots,z_k,\ou)$,
where $(\Si,z_1,\ldots,z_k)$ is a genus~$g$ marked nodal connected compact 
Riemann surface, 
$$\ou\!: \oSi\equiv\bigsqcup_j\oSi_j\lra \oXYm,
\qquad \Si^*\! \subset \oSi \subset \Si,$$
such that 
\begin{enumerate}[label=$\bullet$,leftmargin=*]
\item $\ou$ is a $\oJm$-holomorphic map of degree~$A$,
\item $\ou^{-1}\big(\{r\}\!\times\!\R\!\times\!SV\big)\neq\eset$ for every $r\!=\!1,\ldots,m$,
\item each node in $\Si\!-\!\oSi$ gives rise to one positive and 
one negative puncture of~$(\oSi,\ou)$,
\item the positive punctures~$z_{j,i}$ of any component~$\Si_j$ mapped into 
$\{r\}\!\times\!\R\!\times\!SV$ for some $r\!=\!1,\ldots,m$
correspond to the nodes of~$\Si$ joining~$\Si_j$
to the components mapped into $\{r\!-\!1\}\!\times\!\R\!\times\!SV$ 
if $r\!>\!1$ and to~$X$ if $r\!=\!1$
outside of the punctures,
\item the negative punctures~$z_{j,i}$ of any component~$\Si_j$ mapped into 
$\{m\!+\!1\!-\!r\}\!\times\!\R\!\times\!SV$ for some $r\!=\!1,\ldots,m$
correspond to the nodes of~$\Si$ joining~$\Si_j$
to the components mapped into $\{m\!-\!r\}\!\times\!\R\!\times\!SV$ 
if $r\!>\!1$ and to~$Y$ if $r\!=\!1$
outside of the punctures, and 
\item $\cP_{z_{j,i}}^+(\ou)\!=\!\cP_{z_{i,j}}^-(\ou)$,
where $z_{i,j}\!\in\!\Si_{i,j}$ is the point identified with~$z_{j,i}$.
\end{enumerate}
The notion of equivalences is defined as before.
The moduli spaces $\ov\M_{g,k}(X\!\cup_V\!Y,A)$ of stable maps 
in the two descriptions are again canonically identified,
by the same procedure as in the relative case at the end of Section~\ref{RelInv_sub0}.

\begin{rmk}\label{LR33b_rmk}
According to \cite[Definition~3.18]{LR}, 
a stable map into $X\!\cup_V\!Y$ is a pair of relative maps into $(X,V)$ and~$(Y,V)$
with matching conditions at the~nodes.
According to \cite[\S3.4.2]{AMLi}, this was the intended meaning and not a mis-wording.
Thus, the version of $\ov\M_{g,k}(X\!\cup_V\!Y,A)$ in~\cite{LR}
does not consist of limits of maps into smoothings of $X\!\cup_V\!Y$.
This means that the setup in~\cite{AMLi} is not even suitable for
comparing invariants of the singular and smooth fibers via a virtual class construction,
since finite-dimensional subspaces of $\Ga_J^{0,1}$ need to be chosen continuously over
a family of moduli spaces. 
\end{rmk}

\begin{rmk}\label{LR33c_rmk}
The analysis related to the compactness of the moduli spaces $\ov\M_{g,k;\bs}^V(X,A)$
and $\ov\M_{g,k}(X\!\cup_V\!Y,A)$ is contained in \cite[Sections~3.1,3.2]{LR}.
Nearly all arguments in \cite[Section~3.1]{LR}, which is primarily concerned 
with rates of convergence for maps to $\R\!\times\!SV$, are either incorrect or incomplete, but 
the only desired claim is easy to establish; see Section~\ref{IPconv_subs} below.
\cite[Section~3.2]{LR} applies this claim to study convergence for sequences
of $J$-holomorphic maps from Riemann surfaces with punctures
into $\oXV$, $\R\!\times\!SV$, and~$X\!\#_V\!Y$, 
though the targets are never specified.
The assumptions $u'(\Si')\!\subset\!D_p(\ep)$ and $u'(\prt\Si')\!\subset\!\prt D_p(\ep)$ 
in \cite[Lemma~3.8]{LR}, which is missing a citation, should be weakened to 
$u'(\prt\Si')\!\cap\!D_p(\ep)\!=\!\eset$.
The bound on the energy of $J$-holomorphic maps to~$\R\!\times\!\wt{V}$
claimed below \cite[(3.44)]{LR} needs a justification;
it follows from the correspondence with maps to~$\P_XV$.
There is no specification of the target of the sequence of maps~$u_i$ 
central to the discussion of \cite[Section~3.2]{LR}.
The sentence containing \cite[(3.48)]{LR} and the next one do not make sense
as stated.
There is no mention of what happens to nodal points of the domain
or if $\wt{m}_0\!=\!\wt{m}(q)$ is zero 
(which can happen, since $\wt{m}_0$ measures only the horizontal energy).
The main argument applies \cite[Theorem~3.7]{LR} to maps from disk,
even though it is stated only for maps from~$\C$ (as done in \cite{H,HWZ1}).
In~(3) of the proof of \cite[Lemma~3.11]{LR}, 
the horizontal distance bound \cite[(3.55)]{LR} is used (incorrectly)
to draw a conclusion about the vertical distance in the last equation;
it would have implied the last claim of~(3) without 
\cite[(3.51),(3.52)]{LR}.
Because of the arbitrary choice of~$t_0$ in \cite[(3.53)]{LR}, 
the claim of \cite[Lemma~3.11(3)]{LR} in fact cannot be possibly true.
The statement of \cite[Lemma~3.12]{LR} even explicitly excludes stable 
ghost bubbles with one puncture going into the rubber, which is incorrect.
As the rescaling argument in \cite[Section~3.2]{LR} concerns one node at a time,
it has the same kind of issue as described in the second-to-last sentence of Remark~\ref{IPrel67_rmk}. 
On the other hand, the approach of~\cite{LR} is better suited to deal with this issue
because it can be readily interpreted as a rescaling on the target.
The proof of \cite[Lemma~3.15]{LR} has basically no content.
Other, fairly minor misstatements in \cite[Sections~3.2,3.3]{LR} include
\begin{enumerate}[label={},topsep=-5pt,itemsep=-5pt,leftmargin=*]
\item\textsf{p178, lines 5-7:} $u$ has finite energy;
\item\textsf{p178, (3.44):} $P$ is not used until Section~3.3;
\item\textsf{p178, below (3.44):} $E_{\phi}(u)$ is fixed, according to (3.43);
\item\textsf{p180, (3.49),(3.50):} follow from \cite[Lemma~4.7.3]{MS2};
\item\textsf{p180, line -1:} $\log\ep\le s\le \log\de_i'$;
\item\textsf{p181, line 2:} $\de_i\!<\!\de_i'$, $\log\ep\le s\le \log\de_i$;
\item\textsf{p181, (3.52):} $log$ \to $\log$;
\item\textsf{p181, Lemma~3.11:} $N$ already denotes a space;
\item\textsf{p181, line -6:} Lemma (3.5) \to Lemma 3.5;
\item\textsf{p182, (3.57):} $lim$ \to $\lim$; $A(\ep,R\de_i)$ \to $A(R\de_i,\ep)$;
\item\textsf{p182, below (3.57):} repeat of first sentence of (3);
\item\textsf{p183, Rmk 3.13:} the collapsed compact manifold is 
$\P(\wt{V}\!\times_{S^1}\!\C\oplus V\!\times\!\C)$;
\item\textsf{p183, Section 3, lines 1,2:} $\Si_1\!\vee\!\Si_2\!\lra\!\R\!\times\!\wt{M}$
be a map;
\item\textsf{p185, line 14:} $\bigoplus$ \to $\bigsqcup$.
\end{enumerate}
\end{rmk}

\subsection{GW-invariants: \cite{IPrel}, \cite[Section~1]{IPsum}, 
\cite[Section~4]{LR}}
\label{RelInv_subs}

\noindent
Let $X$, $V$, $A$, $g$, $k$, and~$\bs$ be as in Section~\ref{RelInv_sub0}.
The moduli space $\ov\M_{g,k;\bs}^V(X,A)$ carries a virtual fundamental class (VFC),
which gives rise to relative GW-invariants of~$(X,\om,V)$ and 
is used in the proof of the symplectic sum formula in~\cite{Jun2, LR}.
The argument in~\cite{IPsum} is restricted to the cases when the relevant
relative and absolute invariants can be realized more geometrically,
but the principles of~\cite{IPsum} apply in the general VFC setting as well,
once the VFC is shown to exist.
In the restricted setting of \cite{IPrel, IPsum}, it is not even necessary
to consider the elaborate rubber structure (maps to~$\P_XV$)
described in Section~\ref{RelInv_sub0}.
Below we review the geometric construction of the absolute GW-invariants,
due to \cite{RT1, RT2}, and its adaption to relative invariants,
due to~\cite{IPrel}.
We then comment on the general case considered in~\cite{LR}.\\

\noindent
We begin with two definitions which are later used to describe
the cases when the absolute and relative invariants can be realized geometrically.

\begin{dfn}\label{semipos_dfn}
A $2n$-dimensional symplectic manifold $(X,\om)$ is 
\begin{enumerate}[label=(\arabic*),leftmargin=*] 
\item \textsf{semi-positive} if $\lr{c_1(X),A}\!\ge\!0$ for all $A\!\in\!\pi_2(M)$ such that  
\BE{semipos_e1} \lr{\om,A}>0 \qquad\hbox{and}\quad c_1(A)\ge 3-n;\EE
\item \textsf{strongly semi-positive} if $\lr{c_1(X),A}\!>\!0$ for all $A\!\in\!\pi_2(M)$ 
such that \eref{semipos_e1} holds.
\end{enumerate}  
\end{dfn}

\begin{dfn}\label{relsemipos_dfn}
Let $(X,\om)$ be a $2n$-dimensional symplectic manifold and 
$V\!\subset\!X$ be a symplectic divisor.
The triple $(X,\om,V)$ is 
\begin{enumerate}[label=(\arabic*),leftmargin=*] 
\item \textsf{semi-positive} if $\lr{c_1(X),A}\!\ge\!A\!\cdot_X\!V\!$ 
for all $A\!\in\!\pi_2(M)$ such that  
\BE{relsemipos_e1} A\!\cdot_X\!V\ge0,\qquad
\lr{\om,A}>0,  \quad\hbox{and}\quad
 \lr{c_1(X),A} \ge \max(3,A\!\cdot_X\!V\!+\!2)-n;\EE
\item  \textsf{strongly semi-positive}
if $\lr{c_1(X),A}\!>\!A\!\cdot_X\!V\!$ 
for all $A\!\in\!\pi_2(M)$ such that \eref{relsemipos_e1} holds.
\end{enumerate}  
\end{dfn}

\noindent
A $2n$-dimensional symplectic manifold $(X,\om)$ is semi-positive if $n\!\le\!3$
and is strictly semi-positive if $n\!\le\!2$.
Similarly, if $V\!\subset\!X$ is a symplectic hypersurface, 
$(X,\om,V)$ is semi-positive if $n\!\le\!2$ and is strictly semi-positive if $n\!=\!1$.\\   
 
\noindent
Let $g,k\!\in\!\Z^{\ge0}$ be such that $2g\!+\!k\!\ge\!3$, 
\BE{PrymCov_e}\chM_{g,k} \lra \ov\M_{g,k} \EE
be the branched cover of the Deligne-Mumford space of stable genus~$g$ $k$-marked
curves by the associated moduli space of Prym structures constructed in~\cite{Lo}, and 
$$\pi_{g,k}\!:\chU_{g,k}\!\lra\!\chM_{g,k}$$ 
be the corresponding universal curve.
A \textsf{genus~$g$ $k$-marked nodal curve with a Prym structure} is 
a connected compact nodal $k$-marked Riemann surface $(\Si,z_1,\ldots,z_k)$ of arithmetic genus~$g$
together with a holomorphic stabilization map $\st_{\Si}\!:\Si\!\lra\!\chU_{g,k}$
which surjects on a fiber of~$\pi_{g,k}$ and takes the marked points of~$\Si$
to the corresponding marked points of the fiber.\\

\noindent
Let $A\!\in\!H_2(X;\Z)$, $J$ be an almost complex structure on~$X$, and 
\BE{nucond0_e}\nu \in \Ga_{g,k}(X,J)\equiv 
\Ga\big(\chU_{g,k}\!\times\!X,\pi_1^*(T^*\chU_{g,k})^{0,1}\!\otimes_{\C}\!\pi_2^*(TX,J)\big).\EE
A \textsf{degree~$A$ genus~$g$ $k$-marked $(J,\nu)$-map} is a tuple $(\Si,z_1,\ldots,z_k,\st_{\Si},u)$
such that $(\Si,z_1,\ldots,z_k,\st_{\Si})$ is a genus~$g$ $k$-marked nodal curve with a Prym structure
and $u\!:\Si\!\lra\!X$ is a smooth (or $L^p_1$, with $p\!>2$) map such that 
$$u_*[\Si]=A \qquad\hbox{and}\qquad
\dbar_{J,\fj}u\big|_z= \nu|_{(\st_{\Si}(z),u(z))}\!\circ\!\nd_z\st_{\Si}\quad\forall\,z\!\in\!\Si,$$
where $\fj$ is the complex structure on~$\Si$.
Two such tuples are \textsf{equivalent} if they differ by a reparametrization
of the domain commuting with the maps to~$\chU_{g,k}$ and~$u$.\\

\noindent
By \cite[Corollary~3.9]{RT2}, the space $\ov\M_{g,k}(X,A;J,\nu)$ of 
equivalence classes~of degree~$A$ genus~$g$ $k$-marked $(J,\nu)$-maps
is Hausdorff and compact (if $X$ is compact) in Gromov's convergence topology.
By \cite[Theorem~3.16]{RT2}, for a generic~$(J,\nu)$
each stratum of $\ov\M_{g,k}(X,A;J,\nu)$ consisting of simple (not multiply covered) 
maps of a fixed combinatorial type is a smooth manifold of the expected even dimension, 
which is less than the expected dimension of 
the subspace of simple maps with smooth domains
(except for this subspace itself).
By \cite[Theorem~3.11]{RT2}, the last stratum has a canonical orientation.
By \cite[Proposition~3.21]{RT2}, 
the images of  the strata of $\ov\M_{g,k}(X,A;J,\nu)$ consisting of multiply covered maps 
under the morphism
\BE{RTcyc_e}
\ev\!\times\!\st\!:\ov\M_{g,k}(X,A;J,\nu)\lra X^k\times\ov\M_{g,k}\EE
are contained in images of maps from smooth even-dimensional manifolds of
 dimension less than this stratum
if $(J,\nu)$ is generic and $(X,\om)$ is semi-positive.
Thus, \eref{RTcyc_e} is a pseudocycle. 
Intersecting it with classes in the target and dividing by
the order of the covering~\eref{PrymCov_e}, we obtain (absolute) GW-invariants 
of a semi-positive symplectic manifold $(X,\om)$ in the stable range, i.e.~with
$(g,k)$ such that $2g\!+\!k\!\ge\!3$.
If $g\!=\!0$, the same reasoning applies with $\nu\!=\!0$ and yields
the same conclusion if $(X,\om)$ is strictly semi-positive.\\

\noindent
Suppose in addition that $V\!\subset\!X$ is a symplectic divisor preserved
by the almost complex structure~$J$ and $\bs\!\in\!(\Z^+)^{\ell}$ is 
a tuple satisfying~\eref{bsumcond_e}.
For
\BE{nuVrestr_E}\nu \in \Ga_{g,k}(X,J) \qquad\hbox{s.t.}\quad
\nu|_{\wc\cU_{g,k}\times V}\in  \Ga_{g,k}(V,J|_V)\,,\EE
we define the moduli space
$$\M_{g,k;\bs}^V(X,A;J,\nu)\subset \ov\M_{g,k+\ell}(X,A;J,\nu)$$
analogously to the smaller moduli space in~\eref{relmoddfn_e}.
If $u\!:\Si\!\lra\!X$ is a $(J,\nu)$-holomorphic map such that $u(\Si)\!\subset\!V$,
the linearization of $\dbar_{J,\fj}\!+\!\nu$ at~$u$ again descends
to a first-order differential operator
$$D_u^{\cN_XV}\!: \Ga(\Si;u^*\cN_XV)\lra \Ga^{0,1}_{J,\fj}(\Si;u^*\cN_XV).$$
If~$J$ satisfies~\eref{NijenCond_e} and
\BE{nanucond_e} \wt\na_w\nu+J\wt\na_{Jw}\nu\in 
(T^*\wc\cU_{g,k})^{0,1}\!\otimes_{\C}\!T_xV 
\qquad\forall~w\!\in\!T_xX,~x\!\in\!V,\EE
then this linearization is $\C$-linear and in fact is the same as 
the corresponding operator with $\nu\!=\!0$.
A compact moduli space $\ov\M_{g,k;\bs}^V(X,A;J,\nu)$  can then be defined 
analogously to the smaller moduli space in~\eref{relmoddfn_e}. 
The component maps into the rubber layers $\{r\}\!\times\!\P_XV$ are then 
$(J_{X,V},\nu_{X,V})$-holomorphic, with
\begin{gather*}
\nu_{X,V}\in\Ga_{g',k'}(\P_XV,J_{X,V}), \\
\{\nu_{X,V}|_w\}(v)=\big(\{\wt\na_w\nu\}(v),\nu(v)\big)\in
T_w^{\vrt}\cN_XV\oplus T_w^{\hor}\cN_XV
\qquad\forall~w\in\cN_XV,~v\in T\wc\cU_{g',k'},
\end{gather*}
with $(g',k')$ determined by each component.

\begin{rmk}\label{LinD_rmk}
There are a number of misstatements in the related parts of~\cite{IPrel,IPsum}.
In the linearization \cite[(3.2)]{IPrel} of the $(J,\nu)$-equation,
$\na_{\xi}\nu$ should be replaced by $\wt\na_{\xi}\nu$,
as can be seen from the proof of  \cite[Proposition~3.1.1]{MS2};
otherwise, it would not even map into the right space.
Thus, $\na$ on the left-hand side in \cite[(3.3c)]{IPrel} should 
be replaced by~$\wt\na$;
the right-hand side of \cite[(3.3c)]{IPrel} is zero by \cite[(C.7.1)]{MS2}.
In \cite[(1.11)]{IPsum}, $+(J\na_{\xi}J)$ should be $-(J\na_{\xi}J)$ to agree with 
\cite[Proposition~3.1.1]{MS2} in the $\nu\!=\!0$ case.
This is also necessary to obtain \cite[(1.14)]{IPsum} with $1/4$ instead of~$1/8$ and
$$N_J(\xi,\ze)=-[\xi,\ze]-J[J\xi,\ze]-J[\xi,J\ze]+[J\xi,J\ze]
\qquad\forall\,\xi,\ze\in\Ga(X,TX),$$
as in~\eref{Dudfn_e} above and in~\cite{MS2}.
Furthermore, $\Phi_f\!=\!0$ if $f$ is $(J,\fj,\nu)$-holomorphic; 
otherwise, there are lots of linearizations of $\dbar_{J,\fj}\!+\!\nu$.
The three-term expression in parenthesis in \cite[(1.11)]{IPsum} reduces 
to $\{\partial f\!-\!\nu\}(w)$,
but should be just $\partial f(w)$ to be consistent with \cite[Proposition~3.1.1]{MS2};
otherwise, this term is not even $(J,\fj)$-antilinear.
In this equation, $\na$ denotes the pull-back connection of the Levi-Civita connection~$\na$ 
for the metric \cite[(1.1)]{IPsum} to a connection in~$u^*TX$ 
 the first two times it appears, but~$\na$ itself the last two
times it appears (contrary to p945, line -4);
the term~$\na_{\xi}\nu$ should be replaced by~$\wt\na_{\xi}\nu$. 
Via the first equation in~\cite[(C.7.5)]{MS2}, the correct version of
\cite[(1.11)]{IPsum} gives
$$\frac14 N_J(\xi,\partial f)-\frac12T_{\nu}(\xi,w),   \qquad\hbox{where}\quad
T_{\nu}(\xi,w)=\{\wt\na_{\xi}\nu\}(w)+J\{\wt\na_{J\xi}\nu\}(w),$$
instead of \cite[(1.14)]{IPsum};
the correct version is consistent with \cite[(3.1.5)]{MS2}.
The reason \cite[(1.15b)]{IPsum}, with the above correction for $T_{\nu}$, is equivalent to
\cite[(3.3bc)]{IPrel}, corrected as above, is the restriction in~\eref{nuVrestr_E}. 
Other related typos in~\cite{IPsum} include
\begin{enumerate}[label={},topsep=-5pt,itemsep=-5pt,leftmargin=*]
\item \textsf{p943, (1.2):} RHS should end with $\circ d\phi$;
\item \textsf{p943, line 11:} $\Hom(\pi_1^*T\P^N,\pi_2^*TX)$ \to $\Hom(\pi_2^*T\P^N,\pi_1^*TX)$;
\item \textsf{p943, line 13:} $J_{\P^n}$ \to $J_{\P^n}v$;
\item \textsf{p943, line 17:} $|(u,v)|^2$ presumably means $|u|^2+|v|^2$, 
in contrast to $|dF|^2$ in \cite[(1.5)]{IPsum};
\item \textsf{p943, (1.5):} second half should read
$\displaystyle\int_BF^*\hat\om=\om([f])+\om_{\P^n}([\phi])$;
\item \textsf{p943, line -3:} {\it smooth} is questionable across the boundary;
\item\textsf{p945, bottom:} since $h\in H^{0,1}(TC(-\sum p_i))$,
which is a quotient, $f_*h$ is not defined;
\item\textsf{p947, lines 14,15:} $\coker D_{\bs}=0$ after restricting the range of~$D$.
\end{enumerate}
\end{rmk}

\noindent
The space $\ov\M_{g,k;\bs}^V(X,A;J,\nu)$ 
is Hausdorff and compact (if $X$ is compact) in Gromov's convergence topology.
By \cite[Lemma~7.5]{IPrel}, for a generic~$(J,\nu)$ each stratum of $\ov\M_{g,k;\bs}^V(X,A;J,\nu)$
consisting of simple maps of a fixed combinatorial type is a smooth manifold of 
the expected even dimension, which is less than the expected dimension of 
the subspace of simple maps with smooth domains (except for this subspace itself).
By \cite[Theorem~7.4]{IPrel}, the last stratum has a canonical orientation.
The multiply covered maps in $\ov\M_{g,k;\bs}^V(X,A;J,\nu)$ fall into
two (overlapping) subsets: those with a multiply covered component mapped
into~$V$ and those with a multiply covered component not contained in~$V$.
By \cite[Proposition~3.21]{RT2}, the images of first type of multiply covered strata
under the morphism
\BE{IPcyc_e}
\ev\!\times\!\ev^V\!\times\!\st\!:
\ov\M_{g,k;\bs}^V(X,A;J,\nu)\lra X^k\times V^{\ell}\times\ov\M_{g,k+\ell}\EE
are contained in images of maps from smooth even-dimensional manifolds of
dimension less than the main stratum
if $(J|_V,\nu|_V)$ is generic and $(V,\om|_V)$ is semi-positive.
By a similar dimension counting, the images of the second type of multiply covered strata
under~\eref{IPcyc_e}
are contained in images of maps from smooth even-dimensional manifolds of
dimension less than the main stratum
if $(J,\nu)$ is generic, subject to the conditions~\eref{NijenCond_e} and~\eref{nanucond_e}, 
and $(X,\om,V)$ is semi-positive.\footnote{If \eref{semipos_e1} fails,
the space of simple degree~$A$ $(J,\nu)$-maps is empty for a generic~$J$.
If \eref{relsemipos_e1} fails, the space of simple relative degree~$A$ $(J,\nu)$-maps 
with one relative marked point is empty.
Irreducible components of the domain of a map in $\ov\M_{g,k;\bs}^V(X,A;J,\nu)$
which carry at least two marked points are stable because they also
carry at least one node;
$(J,\nu)$-maps from stable components are not multiply covered for a generic~$\nu$.}
Thus, \eref{IPcyc_e} is a pseudocycle and gives rise to relative GW-invariants 
of a semi-positive triple~$(X,\om,V)$ with a semi-positive~$(V,\om|_V)$.
In the unstable range, similar reasoning applies with $\nu\!=\!0$ and yields
the same conclusion if $(X,\om,V)$ is strictly semi-positive
and $(V,\om|_V)$ is semi-positive.
One key difference in this case is that the space of multiply covered relative degree~$A$ 
$J$-holomorphic maps from smooth domains with two relative marked points
can be of the same dimension as the space of simple 
degree~$A$ $J$-holomorphic maps from smooth domains,
but is then smooth.

\begin{rmk}\label{SemiPos_rmk}
In the semi-positive case, the relative moduli space described above can be
replaced by a subspace of $\ov\M_{g,k+\ell}(X,A)$; see \cite[Section~7]{IPrel}.
There is some confusion in \cite{IPrel, IPsum} regarding the proper semi-positivity 
requirements in the relative case.
The only requirement stated in \cite[Theorem~1.8]{IPrel} is that $(X,\om)$ 
is semi-positive; \cite[Theorem~8.1]{IPrel} also requires $(V,\om|_V)$
to be semi-positive.
The only condition stated in the bottom half of page~947 in~\cite{IPsum},
in the context of disconnected GT-invariants appearing on the following page, 
is that $\lr{c_1(X),A}\!\ge\!A\!\cdot_X\!V$ whenever 
$$\lr{\om,A}>0 \qquad\hbox{and}\qquad 
\lr{c_1(X),A} \ge \max(3,A\!\cdot_X\!V\!+\!1)-n \,.$$
The domain and the target of the linearized $\dbar$-operator~$D_s^N$
are described incorrectly below \cite[(6.2)]{IPrel};
the index of the described operator is generally too small
(because $s_i(s_i\!+\!1)/2$ contact conditions on the vector fields are imposed
at each contact, but no conditions on the one-forms).
The resulting bundle section in \cite[(6.7)]{IPrel} cannot be transverse unless $s_i\!=\!1$.
However, this issue can be resolved by using the twisting down construction
of \cite[Lemma~2.4.1]{Sh}.
The observation at the end of the preceding paragraph is not made 
in \cite{IPrel, IPsum}, but it is necessary to make sense of the invariants giving rise
to the $S$-matrix in \cite[Section~11]{IPsum}; see Section~\ref{Smat_subs}.
\end{rmk}

\noindent
In order to define relative invariants without a semi-positivity assumption on $(X,\om,V)$,
it is necessary to describe neighborhoods of elements of
the relative moduli space inside of a configuration space
and to construct finite-rank vector bundles over them with certain properties.
Unlike the situation with absolute GW-invariants in~\cite{FO} and~\cite{LT},
describing such a neighborhood requires gluing maps with rubber components
which are defined only up to a $\C^*$-action on the target.
The aim of \cite[Section~4]{LR} is to justify the existence of such invariants.
However, the gluing construction in \cite[Section~4]{LR} is limited to maps with 
a single node.
Even in this very special case, the $\C^*$-action on the maps to the rubber 
($\R\!\times\!SV$ in the approach of~\cite{LR}) is not considered, and 
the target space for the resulting glued maps, described by \cite[(4.12),(4.13)]{LR}, 
is not the original space~$\oXV$,
but a manifold diffeomorphic to~$\oXV$ (and not canonically or biholomorphically).
Neither the injectivity nor surjectivity of the neighborhood 
description is even considered in~\cite{LR}.
Thus, there is not even an attempted construction of a virtual fundamental class
for $\ov\M_{g,k;\bs}^V(X,A)$ in~\cite{LR}.
Nevertheless, the suggested idea of stretching the necks on both the domain and 
the target of the maps fits naturally with the analytic problems involved 
in such a construction; we return to this point in Section~\ref{sumpf_subs}.

\begin{rmk}\label{LR4_rmk}
The formulas \cite[(4.1),(4.2)]{LR} for the linearized $\dbar$-operator are incorrect,
since $J$ is not even tamed by the metric; see \cite[Section~3.1]{MS2}.
The statement above \cite[Remark~4.1]{LR} requires a citation.
The norms on the line bundle $u^*L\!\otimes\!\la$ on page~190 in~\cite{LR} are not specified;
because of the poles at the nodes, specifying a suitable norm is not a triviality.
Furthermore, the 3-4 pages spent on this line bundle are not necessary;
it is used only to construct local finite-rank subbundles of the cokernel bundle~$\cF$.
On the other hand, the deformations constructed from this line bundle need to respect
the~$\C^*$-action on $\R\!\times\!SV$ and thus need to be pulled back from~$V$
as in~\cite{IPrel}, of which no mention is made.
The required bound on the radial component~$a$ in \cite[Lemma~4.6]{LR} and other statements
is not part of any previous statement, such as \cite[Theorem~3.7]{LR}.
In \cite[Section~4.2]{LR}, the Implicit Function Theorem in an infinite-dimensional
setting is invoked twice (middle of page~200 and bottom of page~201) without any care.
While the relevant bounds for the 0-th and 1-st order terms are at least discussed
in \cite[Section~4.1]{LR}, not a word is said about the quadratic term.
The variable~$r$ is used to denote the norm of the gluing parameter $(r)\!=\!(r,\th_0)$
in an ambiguous way.
The issue is further confused by the notation $i_r$ at the bottom of page~193 in~\cite{LR},
$I_r$ at the bottom of page~201, $(\xi_r,h_r)$ in~(4.51);
in all cases, the subscript~$r$ should be replaced by the gluing parameter~$(r)$. 
The most technical part of the paper, roughly 4~pages, concerns the variation
of various operators with respect to the norm~$r$ of~$(r)$,
which is done without explicitly identifying the domains and targets of these spaces.
This part is used only for showing that the integrals \cite[(4.50)]{LR} defining
relative invariants converge.
However, this is not necessary, since the relevant evaluation map
is a rational pseudocycle according to \cite[Proposition~4.10]{LR}.
At the end of the first part of the proof of \cite[Proposition~4.1]{LR},
it is claimed that the overlaps of the gluing maps are smooth;
no one has shown this to be the case along the lower strata.
The wording of \cite[Lemma~4.12]{LR} suggests the existence of a diffeomorphism
between an odd-dimensional manifold and an even-dimensional manifold.
The constant~$C_3$ in \cite[(4.44)]{LR} depends on~$\al$;
thus, it is unclear that $C_3|\al|$ can be made arbitrary small.
The inequality \cite[(4.57)]{LR} is not justified.
The paper does not even touch on the independence claims of \cite[Theorem~4.14]{LR}.
Other, fairly minor misstatements in \cite[Section~4]{LR} include
\begin{enumerate}[label={},topsep=-5pt,itemsep=-5pt,leftmargin=*]
\item\textsf{p188, below Rmk~4.3:} the implication goes the other way;
\item\textsf{p189, lines 10,13:} $\Si_1\!\wedge\!\Si_2$ \to  $\Si_1\!\vee\!\Si_2$;
$h_{10}\!=\!h_{20}$ \to $\hat{h}_{10}\!=\!\hat{h}_{20}$;
\item\textsf{p190, lines -7,-6:} unjustified and irrelevant statement;
\item\textsf{p192, line -2:} $x$ has not been defined; 
\item\textsf{p193, (4.16):} $\de$ as in (4.3);
\item\textsf{p193, (4.17):} $s_2\!+\!4r$ \to $s_2$;
\item\textsf{p194, (4.20):} $P$ has very different meaning in (3.44);
\item\textsf{p203, (4.60)} would be more relevant without $Q$ and $D{\mathcal S}$; 
\item\textsf{p204, (4.62):} the middle term on RHS should be dropped;
\item\textsf{p204, (4.65):} the ``other gluing parameter $v$" is denoted by $\th_0$ on p192;
\item\textsf{p205, Thm 4.14} repeats Thm C on p158 (7 lines).
\end{enumerate}
\end{rmk}

\subsection{Refined GW-invariants: \cite[Section~5]{IPrel}}
\label{RelInv_subs2}

\noindent
As emphasized in \cite[Section~5]{IPrel}, two preimages of the same point 
under the morphism
\BE{evdfn_e2} \ev^V\!: \M_{g,k;\bs}^V(X,A;J,\nu)\lra V_{\bs}\!\equiv\!V^{\ell}\EE
determine an element~of 
\BE{cRXVdfn_e0}\cR_X^V\equiv \ker\big\{\io_{X-V*}^X\!:\,H_2(X\!-\!V;\Z)\lra H_2(X;\Z)\big\},\EE
where $\io_{X-V}^X\!:X\!-\!V\!\lra\!X$ is the inclusion; 
see \cite[Section~2.1]{GWrelIP}.
The elements of~$\cR_X^V$, called \textsf{rim tori} in~\cite{IPrel},
can be represented by circle bundles over loops~$\ga$ in~$V$; 
see \cite[Section~3.1]{GWrelIP}.
By standard topological considerations,
\BE{cRXVvsH1V_e}\cR_X^V\approx H_1(V;\Z)_X\equiv 
\frac{H_1(V;\Z)}{H_X^V},
\qquad\hbox{where}\quad H_X^V\equiv \big\{A\!\cap\!V\!:\,A\!\in\!H_3(X;\Z)\big\} \,;\EE
see \cite[Corollary~3.2]{GWrelIP}.\\

\noindent
The main claim of \cite[Section~5]{IPrel} is that the above observations can be used 
to lift~\eref{evdfn_e2} over some regular (Galois), possibly disconnected (unramified) covering
\BE{IPcov_e}\pi_{X;\bs}^V\!: \cH_{X;\bs}^V\!=\!\wh{V}_{X;\bs}\lra V_{\bs};\EE
the topology of this cover is specified in \cite[Section~6.1]{GWrelIP}.
By \cite[Lemma~6.3]{GWrelIP},
\BE{evfactor_e}\ev_X^V\!=\!\pi_{X;\bs}^V\!\circ\!\wt\ev_X^V\!:\,\ov\M_{g,k;\bs}^V(X,A)\lra V_{\bs}\EE
for some morphism
\BE{evXVlift_e} \wt\ev_X^V\!:\ov\M_{g,k;\bs}^V(X,A)\lra \wh{V}_{X;\bs}\,.\EE
Thus, the numbers obtained by pulling back elements of $H^*(\wh{V}_{X;\bs};\Q)$ by~\eref{evXVlift_e},
instead of elements of $H^*(V_{\bs};\Q)$ by~\eref{evdfn_e2}, and integrating them and 
other natural classes on  $\ov\M_{g,k;\bs}^V(X,A)$ against the virtual class of 
$\ov\M_{g,k;\bs}^V(X,A)$ refine the usual GW-invariants of~$(X,V,\om_X)$.
We will call these numbers the \textsf{IP-counts for~$(X,V,\om_X)$}.
As discussed in \cite[Sections~1.1,1.2]{IPrel},
these numbers generally depend on the choice of the lift~\eref{evXVlift_e}.\\

\noindent
The lift~\eref{evXVlift_e} of~\eref{evdfn_e2} is not unique and 
involves choices of base points in various spaces.
By \cite[Theorem~6.5]{GWrelIP} and \cite[Remark~6.7]{GWrelIP},
these choices can be made in a systematic manner, consistent with 
the perspective of \cite[Section~5]{GWrelIP} and suitable for the intended applications
in the symplectic sum context of \cite[Section~10]{IPsum}; see Section~\ref{RefSymSum_subs}.
Furthermore, \eref{evXVlift_e} extends over the space of stable smooth maps 
(and $L^p_1$-maps with $p\!>\!2$) and is thus compatible with
standard virtual class constructions.
This ensures that the IP-counts for~$(X,V,\om_X)$ are independent of~$J$
and of representative~$\om_X$ in a deformation equivalence class of symplectic forms on~$(X,V)$.
However, their dependence on the choice of the lift~\eref{evXVlift_e} and the fact that
the homology of the cover~\eref{IPcov_e} is often not finitely generated
make the IP-counts of little quantitative use in practice.
On the other hand, it is possible to use them for some qualitative applications;
see \cite[Theorem~1.1]{GWrelIP} and \cite[Theorems~1.1,4.9]{GWsumIP}.\\

\noindent
In the relative ``semi-positive" case described in Section~\ref{RelInv_subs}, the morphism
\BE{IPcyc_e2}
\ev\!\times\!\wt\ev_X^V\!\times\!\st\!:
\ov\M_{g,k;\bs}^V(X,A;J,\nu)\lra X^k\times  \cH_{X;\bs}^V\times\ov\M_{g,k+\ell}\EE
is still a pseudo-cycle for generic $J$ and $\nu$ (but its target may not be compact).
By \cite[Theorem~1.1]{Z}, \eref{IPcyc_e2} determines a homology class in
the target.
It can then be used to define IP-counts for~$(X,V,\om_X)$ by intersecting with
proper immersions from oriented manifolds representing Poincare duals
of cohomology classes, similarly to Section~\ref{RelInv_subs2}.

\begin{rmk}\label{IPsec5_rmk}
Two, essentially identical (not just equivalent), descriptions of the set~$\cH_{X;\bs}^V$ 
are given in \cite[Section~5]{IPrel},
neither of which  specifies a topology on~$\cH_{X;\bs}^V$.
In particular, contrary to the sentence below \cite[(5.6)]{IPrel}, 
the topology on~$\hat{X}$ is not changed, but the inclusion map $S^*\!\lra\!\hat{X}$ 
is still continuous and induces precisely the same inclusion of chain complexes 
as in the first description of~$\cH_{X;\bs}^V$.
A hands-on description of the topology of $\cH_{X;\bs}^V$, 
focusing on the $\bs\!=\!(1)$ case, is given at the end of \cite[Section~5]{IPrel};
our definition of $\wh{V}_{X;\bs}$ in \cite[Section~6.1]{GWrelIP} is based on this description.
The group of deck transformations of the covering~\eref{IPcov_e} is 
$$\Deck\big(\pi_{X;\bs}^V\big) =\frac{\cR_X^V}{\gcd(\bs)\cR_{X;\bs}^{\,V}}
\times\gcd(\bs)\cR_{X;\bs}^{\,V}
\qquad\hbox{if}~|\pi_0(V)|\!=\!1;$$
in particular, it is usually different  from $\cR_X^V$, 
contrary to an explicit statement in \cite[Section~5]{IPrel}.
Furthermore, the IP-counts for $(\wh\P_9^2,F)$ and $(\P^1\!\times\!\T^2,\{0,\i\}\!\times\!\T^2)$
are indexed by the rim tori in \cite[Lemmas~14.5,14.8]{IPsum},
implying that the covers $\cH_{X;(1)}^V$ are $\cR_X^V\!\times\!V\!\approx\!\Z^2\!\times\!V$.
In fact, $\cH_{X;(1)}^V\!\approx\!\C$ in both cases and there is no indexing
of the IP-counts by the rim tori; see \cite[Remarks~6.5,6.8]{GWsumIP}.
The typos in \cite[Section~5]{IPrel} include:
\begin{enumerate}[label={},topsep=-5pt,itemsep=-5pt,leftmargin=*]
\item \textsf{p66, (5.2):} $\M_{g,n}^V(X,A)$ \to $\M_{g,n}^V(X)$; 
\item \textsf{p66, line -2:} $\cH$ \to $\cH_X^V$;
\item \textsf{p67, line 20:} $\cH$ is never used.
\end{enumerate}
\end{rmk}

\section{The symplectic sum formula}
\label{symplsum_sec}

\noindent
We state a version of the standard symplectic sum formula in Section~\ref{DfnStat_subs}
by  combining {\it rules of assignment} of~\cite{Jun2} with
the GT-invariants of~\cite{IPsum}.
In Section~\ref{DfnStatComp_subs}, we compare the variations of this formula appearing
in \cite{IPsum,LR,Jun2}.
In Section~\ref{RefSymSum_subs}, we comment on the refinements to this formula
suggested in~\cite{IPsum}.

\subsection{Main statement: \cite[Sections 0,1,10-13,16]{IPsum}}
\label{DfnStat_subs}

\noindent
For $g,k\!\in\!\Z^{\ge0}$ and $\chi\!\in\!\Z$, denote by $\ov\M_{g,k}$ and $\wt\M_{\chi,k}$ 
the Deligne-Mumford moduli spaces of stable nodal $k$-marked complex curves 
with connected domains of genus~$g$ and with (possibly) disconnected domains of
double holomorphic euler characteristic~$\chi$, respectively;
in the unstable range, $2g\!+\!k\!<\!3$ and $k\!-\!\chi\!<\!1$, 
we define each of these spaces to be a point.
Let
$$\ov\M=\bigsqcup_{g,k\in\Z^{\ge0}}\!\!\!\ov\M_{g,k}, \qquad
\wt\M=\bigsqcup_{\chi\in\Z,k\in\Z^{\ge0}}\!\!\!\!\!\! \wt\M_{\chi,k}\,.$$
A \textsf{rule of assignment} is a bijection 
\BE{RofAdfn_e}\vt\!: \{1\}\!\times\!\{1,\ldots,k_1\}\sqcup\{2\}\!\times\!\{1,\ldots,k_2\}
\lra \{1,\ldots,k_1\!+\!k_2\}\EE
for some $k_1,k_2\!\in\!\Z^{\ge0}$ preserving the ordering of the elements in each of the two 
subsets of the domain.
Let $\RA$ denote the set of all rules of assignment.
If in addition $\ell\!\in\!\Z^{\ge0}$, let
\BE{xivt_e} \xi_{\ell,\vt}\!: \wt\M_{\chi_1,k_1+\ell}\times \wt\M_{\chi_2,k_2+\ell}
\lra  \wt\M_{\chi_1+\chi_2-2\ell,k_1+k_2}\EE
be the morphism obtained by identifying the $(k_1\!+\!i)$-th point
on the first curve with  the $(k_2\!+\!i)$-th point on the second curve for 
$i\!=\!1,\ldots,\ell$ and ordering the remaining points by the bijection~$\vt$.\\

\noindent
Let $(X,\om)$ be a compact symplectic manifold, $V\!\subset\!X$ be a closed
symplectic hypersurface, $J$ be an $\om$-compatible almost complex structure,
such that $J(TV)\!=\!TV$, and $A\!\in\!H_2(X;\Z)$.
There are natural stabilization morphisms
\BE{stdfn_e}\st\!: \wt\M_{\chi,k}(X,A)\lra\wt\M_{\chi,k},\qquad
\st\!: \wt\M_{\chi,k;\bs}^V(X,A)\lra\wt\M_{\chi,k+\ell},\EE
forgetting the map and contracting the unstable components of the domain.
We denote the restrictions of these maps~to
$$\ov\M_{g,k}(X,A)\subset\wt\M_{2-2g,k}(X,A)
\qquad\hbox{and}\qquad
\ov\M_{g,k;\bs}^V(X,A)\subset\wt\M_{2-2g,k;\bs}^V(X,A)$$
by the same symbols.
The morphisms~\eref{evdfn_e} and~\eref{stdfn_e} give rise
to the (\textsf{absolute}) \textsf{Gromov-Witten} and \textsf{Gromov-Taubes} 
invariants of~$(X,\om_X)$ with descendants,
\begin{alignat}{2}
\label{GWgendfn_e}
\GW_{g,A}^X\!: \T^*(X)&\lra H_*(\ov\M), &\quad
\GW_{g,A}^X(\al)&=\sum_{k=0}^{\i}\st_*\big(\ev^*\al\!\cap\!\big[\ov\M_{g,k}(X,A)\big]^{\vir}\big),\\
\GT_{\chi,A}^X\!: \T^*(X)&\lra H_*(\wt\M),  &\quad
\GT_{\chi,A}^X(\al)&=\sum_{k=0}^{\i}\st_*
\big(\ev^*\al\!\cap\!\big[\wt\M_{\chi,k}(X,A)\big]^{\vir}\big),
\notag
\end{alignat}
where $H_*$ denotes the homology with $\Q$-coefficients.
They also give rise to the \textsf{relative Gromov-Witten} 
and \textsf{Gromov-Taubes} invariants of~$(X,V,\om)$,
\begin{alignat*}{2}
\GW_{g,A;\bs}^{X,V}\!: \T^*(X)&\lra H_*(\ov\M\!\times\!V_{\bs}), \quad
\GW_{g,A;\bs}^{X,V}(\al)&=\sum_{k=0}^{\i}\{\st\!\times\!\ev^V\}_*
 \big(\ev^*\al\!\cap\!\big[\ov\M_{g,k;\bs}^V(X,A)\big]^{\vir}\big),\\
\GT_{\chi,A;\bs}^{X,V}\!: \T^*(X)&\lra H_*(\wt\M\!\times\!V_{\bs}), \quad
\GT_{\chi,A;\bs}^{X,V}(\al)&=\sum_{k=0}^{\i}\{\st\!\times\!\ev^V\}_*
\big(\ev^*\al\!\cap\!\big[\wt\M_{\chi,k;\bs}^V(X,A)\big]^{\vir}\big).
\end{alignat*}
We assemble the homomorphisms $\GT_{\chi,C}^{X\#_VY}$ and $\GT_{\chi,A;\bs}^{M,V}$
into generating functions as in~\eref{GTabsser_e} and~\eref{GTrelser_e}:
\begin{alignat}{2}
\label{GTabsser_e2} 
\GT^{X\#_VY}&=\sum_{\chi\in\Z}\sum_{\eta\in H_2(X\!\#_V\!Y;\Z)/\cR_{X,Y}^V}
\sum_{C\in\eta}\GT^{X\#_VY}_{\chi,C}\,t_{\eta}\la^{\chi},\\
\label{GTrelser_e2} 
\GT^{M,V}&=\sum_{\chi\in\Z}\sum_{A\in H_2(M;\Z)}
\sum_{\ell=0}^{\i}\sum_{\begin{subarray}{c}\bs\in(\Z^+)^{\ell}\\ |\bs|=A\cdot_MV\end{subarray}}
\!\!\!\!\!\GT_{\chi,A;\bs}^{M,V}\,t_A\la^{\chi}.
\end{alignat}
The generating functions in~\eref{GTabsser_e} and~\eref{GTrelser_e}
are the sums of the terms in the generating functions 
in~\eref{GTabsser_e2} and~\eref{GTrelser_e2}, respectively, that are of $\wt\M$-degree~0.\\

\noindent
If $\vt$ is a rule of assignment as in~\eref{RofAdfn_e} and
$$\al \equiv (\al_{1;X},\al_{1;Y})\otimes\!\ldots\!\otimes(\al_{k;X},\al_{k;Y})
\in H^{2*}(X\!\sqcup_V\!Y)^{\otimes k},$$ 
we define
\begin{gather*}
\al_{\vt;X}=\al_{\vt(1,1);X}\otimes\!\ldots\!\otimes\al_{\vt(1,k_1);X}\in\T^*(X)\,,
\quad
\al_{\vt;Y}=\al_{\vt(2,1);Y}\otimes\!\ldots\!\otimes\al_{\vt(2,k_2);Y}\in\T^*(Y),\\
\hbox{and}\qquad \al_{\vt}=\al_{\vt;X}\otimes\al_{\vt;Y}\in\T^*(X)\otimes\T^*(Y)
\end{gather*}
if $k_1\!+\!k_2\!=\!k$ and $\al_{\vt}=0$ otherwise.
Using the pairing~$\star$ of~\eref{bspairing_e}, we define the pairing 
$$\star_{\vt}\!: H_*(\wt\M\!\times\!V_{\i})\otimes H_*(\wt\M\!\times\!V_{\i})
\lra  H_*(\wt\M)[\la^{-1}]$$ 
to be given by the composition
\begin{equation*}\begin{split}
H_*(\wt\M_{\chi_1,k_1+\ell(\bs)}\!\times\!V_{\bs})\otimes H_*(\wt\M_{\chi_2,k_2+\ell(\bs)}\!\times\!V_{\bs})
&=H_*\big(\wt\M_{\chi_1,k_1+\ell(\bs)}\!\times\!\wt\M_{\chi_2,k_2+\ell(\bs)}\big)
\otimes H_*(V_{\bs})\!\otimes\!H_*(V_{\bs})\\
&\stackrel{\xi_{\ell(\bs),\vt*}\otimes\,\star}{\xra{1.5}}
H_*(\wt\M)\otimes\!\Q[\la^{-1}]= H_*(\wt\M)[\la^{-1}]
\end{split}\end{equation*}
on the specified summands and be 0 on the remaining summands.
This pairing induces a pairing
\BE{thpair_e}\begin{split}
\star_{\vt}\!:\Hom\big(\T^*(X),H_*(\wt\M\!\times\!V_{\i})\big)
&\otimes \Hom\big(\T^*(Y),H_*(\wt\M\!\times\!V_{\i})\big)\\ 
&\lra
\Hom\big(\T^*(X)\!\otimes\!\T^*(Y),H_*(\wt\M)\big)\big[\la^{-1}\big]
\end{split}\EE
as in~\eref{bspairing_e1}, which we extend as in~\eref{bspairing_e2}, 
replacing~$\star$ by~$\star_{\vt}$.

\begin{thm}\label{main2_thm}
Let $(X,\om_X)$ and $(Y,\om_Y)$ be symplectic manifolds and 
$V\!\subset\!X,Y$ be a symplectic hypersurface satisfying~\eref{cNVcond_e}.
If $q_{\#}\!:X\!\#_V\!Y\!\lra\!X\!\cup_V\!Y$ is a collapsing map 
for an associated symplectic sum fibration
and $q_{\sqcup}\!:X\!\sqcup\!Y\!\lra\!X\!\cup\!_VY$ is the quotient map, then 
\BE{SympSumForm_e2}
\GT^{X\#_VY}(q_{\#}^*\al) =\sum_{\vt\in\RA}
\big\{\GT^{X,V} \star_{\vt}\GT^{Y,V}\big\}\big( (q_{\sqcup}^*\al)_{\vt}\big)\EE
for all $\al\!\in\!\T^*(X\!\cup_V\!Y)$.
\end{thm}

\noindent
The identity \eref{SympSumForm_e2} readily extends to cover descendant invariants
($\psi$-classes).
Furthermore, it is not necessary to assume that $X$ and $Y$ are different manifolds:
the reasoning behind Theorem~\ref{main2_thm} readily applies to symplectic manifolds
obtained by gluing along two disjoint hypersurfaces~$V_1$ and~$V_2$ in~$X$
which have dual normal bundles.

\subsection{Comparison of formulations: \cite[Sections 0,1,10-13,16]{IPsum}, \cite[Section~5]{LR}}
\label{DfnStatComp_subs}

\noindent
In \cite[Section~1]{IPsum}, the absolute GW/GT-invariants of~$X$ are defined 
as cycles in a space involving a Cartesian product of copies of~$X$, 
while the relative GW/GT-invariants of~$(X,V)$ are defined as homomorphisms on~$\T^*(X)$
and~$\T^*(Y)$.
The former is inconsistent with the main symplectic sum formulas in~\cite{IPsum},
i.e.~(0.2), (10.14), and~(12.17).
The GT-invariants are formally defined as exponentials of the GW-invariants.
According to \cite[p944]{IPsum},
$$\GT^X=1+\sum_{m=1}^{\i}\frac{1}{m!}
\sum_{\begin{subarray}{c}A_1,\ldots,A_m\in H_2(X;\Z)\\ g_1,\ldots,g_m\ge0\\
k_1,\ldots,k_m\ge0 \end{subarray}}\!\!\!\!\!\!\!\!\!
\frac{\GW_{g_1,A_1,k_1}^X\!\cdot\!\ldots\!\!\cdot\!\GW_{g_m,A_m,k_m}^X}
{k_1!\ldots k_m!}t_{A_1+\ldots+A_m}\la^{2(m-g_1-\ldots-g_m)}\,,$$
where $\GW_{g,A,k}^X$ is the homomorphism corresponding to the $k$-th summand 
in~\eref{GWgendfn_e} and $\cdot$ is some (unspecified) product on 
$\Hom(\T^*(X),H_*(\wt\M))$.
The wording at the top of page~948 in~\cite{IPsum} is somewhat misleading,
as \cite[(1.24)]{IPsum} is the definition of~$\GW^{X,V}$ in~\cite{IPsum},
not a consequence of another definition.
With this interpretation, \cite[(1.25)]{IPsum} gives
\BE{IPsum124_e}\begin{split}
\GT^{X,V}&=1+\sum_{m=1}^{\i}\frac{1}{m!}
\sum_{\begin{subarray}{c}A_1,\ldots,A_m\in H_2(X;\Z)\\ 
g_1,\ldots,g_m\in\Z^{\ge0}\\ \bs_1,\ldots,\bs_N\end{subarray}}\!\!\!\!\!\!\!\!\!
\frac{\GW_{g_1,A_1;\bs_1}^{X,V}\!\cdot\!\ldots\!\!\cdot\!\GW_{g_m,A_m;\bs_m}^{X,V}}
{\ell(\bs_1)!\ldots\ell(\bs_m)!}t_{A_1+\ldots+A_m}\la^{2(m-g_1-\ldots-g_m)}\,,
\end{split}\EE
where $\cdot$ is some (unspecified) product on $\Hom(\T^*(X),H_*(\wt\M\!\times\!\cH_X^V))$ and
$$\cH_X^V=\bigsqcup_{\ell=0}^{\i}\bigsqcup_{\bs\in(\Z^+)^\ell}\cH_{X;\bs}^V\,.$$
In particular, the normalizations of~$\GT^X$ and~$\GT^{X,V}$ with respect
to the absolute marked points in~\cite{IPsum} are inconsistent.
Thus, the symplectic sum formulas of~\cite{IPsum},
even without the rim tori and the $S$-matrix features,
do not recover~\eref{SympSumForm_e}.\\

\noindent
In \cite[Section~16]{IPsum}, the usual (without rim tori refinement)
relative GT-invariants of~$(X,V)$ are  
described in terms of counts of disconnected curves.
If $\{\ga_i\}$ is a basis for $H^*(V)\!\equiv\!H^*(V;\Q)$
and $\{\ga_i^\vee\}$ is the dual basis for $H_*(V)$, then 
\begin{alignat*}{2}
C_{\bs,\bfI}&\equiv \ga_{i_1}\otimes\!\ldots\!\otimes \ga_{i_{\ell}}
\in H^*(V)^{\otimes\ell}\approx H^*(V_{\bs}), &\qquad\hbox{with}\quad
&\bfI=(i_1,\ldots,i_{\ell}),\\
C_{\bs,\bfI}^{\vee}&\equiv \ga_{i_1}^{\vee}\otimes\!\ldots\!\otimes \ga_{i_{\ell}}^{\vee}
\in H_*(V)^{\otimes\ell}\approx H_*(V_{\bs}), &\qquad\hbox{with}\quad
&\bfI=(i_1,\ldots,i_{\ell}),
\end{alignat*}
are dual bases for $H^*(V_{\bs})$ and $H_*(V_{\bs})$, respectively,
for compatible choices of the above isomorphisms.
According to \cite[(A.3)]{IPsum},
\BE{IPsumA3_e}\begin{split}
\GT^{X,V}(\ka,\al)&\equiv \ka\cap\GT^{X,V}(\al)\\
&=\sum_{A,\chi}\sum_{\ell(\bs)=\ell(\bfI)=\ell}\!\!\!\!\!
\frac{\GT_{\chi,A}^{X,V}(\ka,\al;C_{\bs,\bfI})}{\ell!}
C_{\bs,\bfI}^{\vee}t_A\la^{\chi}
\quad\forall~\ka\!\in\!H^*(\wt\M_{\chi,k+\ell}),\,
\al\!\in\!H^*(X)^{\otimes k},
\end{split}\EE
where $\al\!=\!\al_1\!\otimes\!\ldots\!\otimes\!\al_k$ and
$\GT_{\chi,A}^{X,V}(\ka,\al;C_{\bs,\bfI})$ is the number of $(J,\nu)$-holomorphic
maps~$u$, for a generic~$(J,\nu)$, from a possibly disconnected, marked
curve $(\Si,z_1,\ldots,z_{k+\ell})$ such~that
\begin{enumerate}[label=$\bullet$,leftmargin=*]
\item $(\Si,z_1,\ldots,z_{k+\ell})\!\in\!K$ for a fixed generic 
representative~$K$ for $\PD_{\wt\M_{\chi,k+\ell}}\ka$,
\item for each $i\!=\!1,\ldots,k$, $u(z_i)\!\in\!Z_i$ for a fixed generic
representative~$Z_i$ for~$\PD_X\al_i$, and
\item for each $j\!=\!1,\ldots,\ell(\bs)$, $\ord_{z_{k+j}}^Vu\!=\!s_j$ and
$u(z_{k+j})\!\in\!\Ga_j$ for a fixed generic
representative~$\Ga_j$ for~$\PD_V\ga_{i_j}$.
\end{enumerate}
A comparison of \eref{IPsum124_e} and \eref{IPsumA3_e} suggests that 
the product~$\cdot$ on $\Hom(\T^*(X),H_*(\wt\M\!\times\!V_{\i}))$ not explicitly 
specified in~\cite{IPsum} would have to involve rather elaborate coefficients
in order to obtain \cite[(A.3)]{IPsum} from \cite[(1.24)]{IPsum}.\\

\noindent
The alternative description of the relative GT generating series in
the last paragraph of \cite[Section~16]{IPsum} does not make sense 
on several levels.
Let $N$ be the dimension of $H^*(V)$, i.e.~the number of elements in the set~$\{\ga_i\}$
above.
For each 
$$\bm\equiv(m_{a,i})_{a,i}\!: \Z^+\!\times\!\{1,\ldots,N\}\lra\Z^+$$ 
with finitely many nonzero entries (such a matrix $\bm$ is called a
sequence in~\cite{IPsum}), 
let $(\bs_{\bm},\bfI_{\bm})$ be a pair of tuples
with $m_{a,i}$ entries of the form $(a,\ga_i)$ for each~$(a,i)$
and~set
$$\ell(\bm)\equiv\ell(\bs_{\bm})=\sum_{a,i}m_{a,i}\,, \quad
\bm!\equiv\bs_{\bm}!=\prod_{a,i}m_{a,i}!\,, \quad
\bfC_{\bm}=\prod_{a,i}(a,\ga_i)^{m_{a,i}}\,, \quad
\bz^{\bm}=(a,\ga_i^{\vee})^{m_{a,i}}\,.$$
According to \cite[(A.6)]{IPsum},
\begin{equation*}\begin{split}
\GT^{X,V}(\ka,\al)
=\sum_{A,\chi}\sum_{\ell(\bm)=\ell}\!\!\!
\frac{\GT_{\chi,A}^{X,V}(\ka,\al;\bfC_{\bm})}{\bm!}
\bz^{\bm}t_A\la^{\chi}
\quad\forall~\ka\!\in\!H^*(\wt\M_{\chi,k+\ell}),\,
\al\!\in\!H^*(X)^{\otimes k},
\end{split}\end{equation*}
for some unspecified numbers $\GT_{\chi,A}^{X,V}(\ka,\al;\bfC_{\bm})$.
According to~\cite{IPsum}, the collection $\{\bfC_{\bm}\}$ is 
a basis replacing the above basis~$\{C_{\bs,\bfI}\}$, but these collections are 
subsets of different vector spaces (with the former generating a symmetrization
of the vector space generated by the latter).
The formal variable $z_{a,i}\!=\!(a,\ga_i)$ is described as an element 
of {\it the dual basis}, without specifying of dual to what.
According to~\cite{IPsum}, these formal variables generate a super-commutative
polynomial algebra; presumably the same should apply to~the variables~$(a,\ga_i)$.
This makes $\bz^{\bm}$ and $\bfC_{\bm}$ undefined
if there is more than one class~$\ga_i$ of odd degree.
Even if all~$\bz^{\bm}$ are defined, they generate a symmetrization
of the vector space generated by $\{C_{\bs,\bfI}\}$.
Thus, the right-hand sides of \cite[(A.3)]{IPsum} and \cite[(A.6)]{IPsum}
lie in different vector spaces, even though both are supposed to be~$\GT^{X,V}(\ka,\al)$.
Furthermore, the numbers $\GT_{\chi,A}^{X,V}(\ka,\al;C_{\bs,\bfI})$
are symmetric in the inputs pulled back from~$V$ if $\ka$ is symmetric 
in the relative marked points, but not in general.
If there is at most one odd class~$\ga_i$ and $\ka$ is symmetric 
in the relative marked points,
\cite[(A.6)]{IPsum} can be made sense of by viewing its left-hand side 
as the projection of $\GT^{X,V}(\ka,\al)$ to the symmetrization of $H_*(V_{\i})$
over the permutations of components of each tuple~$\bs$.
Comparing with \cite[(A.3)]{IPsum} and summing over all permutations
of pairs of components of $(\bs_{\bm},\bfI_{\bm})$, we then find that 
$$\GT_{\chi,A}^{X,V}(\ka,\al;\bfC_{\bm})
=\GT_{\chi,A}^{X,V}\big(\ka,\al;C_{\bs_{\bm},\bfI_{\bm}}\big)\,.$$
However, this is inconsistent with \cite[Section~15.2]{IPsum},
in particular the equation after \cite[(15.2)]{IPsum},
in which the relative contacts are unordered.
The number $\GT_{\chi,A}^{X,V}(\ka,\al;\bfC_{\bm})$ obtained as above
from \cite[(A.3),(A.6)]{IPsum} would count curves with unordered relative contacts
if $\ell(\bs)!$ were dropped from \cite[(A.3)]{IPsum},
i.e.~with our choices of the normalizations for the GW/GT generating series.\\

\noindent
With our choices of the normalizations of the GW/GT generating functions, the relationship 
\BE{GTexpGW_e}\GT^{X,V}=\ne^{\GW_{X,V}}\,,\EE
which is not crucial for the symplectic sum formulas, 
holds for a product on the vector space $\Hom(\T^*(X),H_*(\wt\M\!\times\!V_{\i}))$
with the simplest possible coefficients.
Specifically, every pair of tuples $\bs_1$ and $\bs_2$ of nonnegative integers and
every rule of assignment
\begin{gather}
\vt\!:\{1\}\!\times\!\{1,\ldots,k_1\!+\!\ell(\bs_1)\}
\sqcup \{2\}\!\times\!\{1,\ldots,k_2\!+\!\ell(\bs_2)\}
\lra\{1,\ldots,k_1\!+\!k_2\!+\!\ell(\bs_2)\} \qquad\hbox{s.t.}\notag\\
\label{vtcond_e}\vt(i_1),\vt(i_2)\le k_1\!+\!k_2~~\forall~i_1\!\in\!
\{1\}\!\times\!\{1,\ldots,k_1\},
~i_2\!\in\!\{2\}\!\times\!\{1,\ldots,k_2\}
\end{gather}
determine a tuple $\bs_1\!\w_{\vt}\bs_2\!\in\!(\Z^+)^{\ell(\bs_1)+\ell(\bs_2)}$,
assembled from~$\bs_1$ and~$\bs_2$ according to the action of~$\vt$
on the last $\ell(\bs_1)$ points in the first tuple above and 
the last $\ell(\bs_2)$ points in the second tuple. 
Thus, $\vt$ defines an embedding
$$\wt\M_{\chi_1,k_1+\ell(\bs_1)}\!\times\!V_{\bs_1}\times
\wt\M_{\chi_2,k_2+\ell(\bs_2)}\!\times\!V_{\bs_2}
\lra \wt\M_{\chi_1+\chi_2,k_1+k_2+\ell(\bs_1\!\w_{\vt}\bs_2)}\!\times\!
V_{\bs_1\!\w_{\vt}\bs_2}.$$
We denote~by
\begin{equation*}\begin{split}
\vt_*\!: 
H_*\big(\wt\M_{\chi_1,k_1+\ell(\bs_1)}\!\times\!V_{\bs_1}\big)&\otimes\!
H_*\big(\wt\M_{\chi_2,k_2+\ell(\bs_2)}\!\times\!V_{\bs_2}\big)\approx
H_*\big(\wt\M_{\chi_1,k_1+\ell(\bs_1)}\!\times\!V_{\bs_1}\!\times\!
\wt\M_{\chi_2,k_2+\ell(\bs_2)}\!\times\!V_{\bs_2}\big)\\
&\lra
H_*\big(\wt\M_{\chi_1+\chi_2,k_1+k_2+\ell(\bs_1\!\w_{\vt}\bs_2)}\!\times\!
V_{\bs_1\!\w_{\vt}\bs_2}\big)
\subset H_*\big(\wt\M\!\times\!V_{\i}\big)
\end{split}\end{equation*}
the induced homomorphism.
If in addition $\al\!=\!\al_1\!\otimes\!\ldots\!\otimes\!\al_{k_1+k_2}\in
H^*(X)^{\otimes(k_1+k_2)}$, let
$$\al_{\vt;i}=\al_{\vt(i,1)}\!\otimes\!\ldots\!\otimes\!\al_{\vt(i,k_i)}
\in H^*(X)^{\otimes k_i}\, \qquad i=1,2.$$
For $L_i\!:H^*(X)^{\otimes k_i}\lra H_*(\wt\M_{\chi_i,k_i+\ell(\bs_i)}\!\times\!V_{\bs_i})$
with $i\!=\!1,2$, we~define 
\begin{gather*}
L_1\cdot L_2\!:  H^*(X)^{\otimes(k_1+k_2)}\lra H_*(\wt\M\!\times\!V_{\i})
\qquad\hbox{by}\\
\{L_1\cdot L_2\}(\al)
=\sum_{\vt} \vt_*\big(L_1(\al_{\vt;1})\!\otimes\!L_2(\al_{\vt;2})\big)
\quad\forall\,\al\!\in\!H^*(X)^{\otimes(k_1+k_2)},
\end{gather*}
where the sum is taken over all rules of assignment~$\vt$ satisfying~\eref{vtcond_e}.
Combining our definitions of the GW/GT generating functions with this definition,
we obtain~\eref{GTexpGW_e}.\\

\noindent
The relative invariants of \cite{IPrel, IPsum} are refinements of 
the usual relative invariants and take values in the coverings~$\cH_{X;\bs}^V$
and~$\cH_{Y;\bs}^V$ of~$V_{\bs}$ described in \cite[Section~6.1]{GWrelIP},
instead of~$V_{\bs}$.
Their use causes additional difficulty with exponentiating the GW-invariants, 
even in the case of primary constraints, since one must also specify a product
\BE{Hlift_e1}H_*(\cH_{X;\bs_1}^V)\otimes H_*(\cH_{X;\bs_2}^V)\lra H_*(\cH_{X;\bs_1\bs_2}^V)\EE
lifting the Kunneth product
\BE{Hlift_e2} H_*(V_{\bs_1})\otimes H_*(V_{\bs_2})\lra 
H_*(V_{\bs_1\bs_2})=H_*(V_{\bs_1}\!\times\!V_{\bs_2})\,.\EE
It is immediate from the definition of $\cH_{X;\bs}^V\!=\!\wh{V}_{X;\bs}$
in \cite[Section~6.1]{GWrelIP} and \cite[Lemma~79.1]{Mu} that 
the natural map 
$$V_{\bs_1} \times V_{\bs_2} \lra V_{\bs_1\bs_2}$$
lifts to a smooth map on the covering spaces.
Thus, a lift~\eref{Hlift_e1} of~\eref{Hlift_e2} exists, but it is not unique.
Choosing such a lift again requires fixing base points in various spaces; 
see \cite[Remark~6.7]{GWrelIP}.\\

\noindent
Another notable feature of the symplectic sum formulas in~\cite{IPsum} is the presence 
of the so-called $S$-matrix, which is shown to be trivial in many cases.
As we explain in Section~\ref{Smat_subs}, it appears due to an oversight in
\cite[Section~12]{IPsum} and its action is always trivial, essentially due to
the nature of this oversight.\\

\noindent
While it is not stated in the assumptions for \cite[(0.2),(10.14),(12.17)]{IPsum},
the arguments for these formulas in~\cite{IPsum} are restricted to the cases
when $(X\!\#_V\!Y,\om_{\#})$, $(X,\om_X,V)$, and~$(Y,\om_Y,V)$ satisfy suitable 
positivity conditions.
By Section~\ref{RelInv_subs}, these conditions are
\begin{enumerate}[label=(\arabic*),leftmargin=*]
\setcounter{enumi}{-1}
\item\label{sumcond_it} $(X\!\#_V\!Y,\om_{\#})$ is strongly semi-positive;
\item\label{divcond_it} $(V,\om_X|_V)\!=\!(V,\om_Y|_V)$ is semi-positive;
\item\label{relcond_it} $(X,\om_X,V)$ and $(Y,\om_Y,V)$ are strongly semi-positive.
\end{enumerate}
Condition~\ref{sumcond_it} is not implied by the other two conditions in general.
However, it can still be ignored, since it holds when restricted to the classes
$A\!\in\!\pi_2(X\!\#_V\!Y)$ which can be represented by $J_{\#}$-holomorphic curves
for an almost complex structure~$J_{\#}$ induced by generic almost complex 
structures~$J_X$ on~$(X,V)$ and~$J_Y$ on~$(Y,V)$ via the symplectic sum construction
of Section~\ref{SympSum_subs}, i.e.~an almost complex structure~$J_{\#}$ 
of the kind considered in~\cite{IPsum}; see the second identity in~\eref{c1Zla_e}.
In light of~\eref{cNVcond_e}, Condition~\ref{relcond_it} implies Condition~\ref{divcond_it}.
Thus, the setting in~\cite{IPsum} is directly applicable whenever 
Condition~\ref{relcond_it} is satisfied.

\begin{rmk}\label{GWGT3_rmk}
The meaning of $C_{\bs,\bfI}$ in \cite[(A.2)]{IPsum} is not specified.
The entire collection $\{C_{\bs,\bfI}\}$, over all pairs $(\bs,\bfI)$ of tuples
of the same length, is described as a basis for the tensor algebra
on $\bN\!\times\!H^*(V)$, which is not even a vector space over~$\R$,
while the collection $\{C_{\bs,\bfI}^{\vee}\}$ is described as the dual basis.
In fact, $\{C_{\bs,\bfI}\}$ and $\{C_{\bs,\bfI}^{\vee}\}$ are bases~for
$$\bigoplus_{\ell=0}^{\i}\bigoplus_{\bs\in(\Z^+)^{\ell}}\!\!\!H^*(V_{\bs})
\qquad\hbox{and}\qquad
H_*(V_{\i})=\bigoplus_{\ell=0}^{\i}\bigoplus_{\bs\in(\Z^+)^{\ell}}\!\!\!H_*(V_{\bs}),$$
respectively; these two vector spaces are not duals of each other.
The summation indices in \cite[(A.3)]{IPsum} are described incorrectly 
and the two appearances of~$\ov\M$ in this paragraph refer to~$\wt\M$.
The description of the number $\GT_{\chi,A}^{X,V}(\ka,\al;C_{\bs,\bfI})$ is incorrect,
even with the proper normalizations of the relevant power series,
since the $j$-th relative marked point should be mapped to a generic representative
for~$\PD_V\al_{i_j}$, not for~$\PD_V\al_j$, and these representatives should be different
for $j_1\!\neq\!j_2$, even if $i_{j_1}\!=\!i_{j_2}$.
Other, fairly minor misstatements in the related parts of~\cite{IPsum} include:
\begin{enumerate}[label={},topsep=-5pt,itemsep=-5pt,leftmargin=*]
\item \textsf{p935, middle:} the finiteness holds only under ideal circumstances;
\item \textsf{p938, top:} $x(z)$ and $y(w)$ are expansions in the normal directions to~$V$
as explained in Section~\ref{IPconv_subs};
\item \textsf{p940, bottom:} $\T^*(Z)$ is defined only on~p944;
\item \textsf{p946, (1.17),(1.18):} $\M_{\chi,n,s}^V(X,A)$ and $\ov\M_{\chi,n,s}^V(X,A)$ 
refer to disconnected domains here;
\item\textsf{p946, after (1.17):} each unstable~$\P^1$ needs to have at least one
marked point to insure compactness;
$\chi$ is twice the holomorphic Euler characteristic, not the usual EC;
\item\textsf{p994, line 13:} the domain of $g$ is the union of these $\De_s$;
\item\textsf{p994, line 15:} $\cup$ \to $\cap$; this defines LHS;
\item\textsf{p994, line 19:} $Q_{p,q}^V$ needs to be the inverse of the intersection
form for the first equality;
\item\textsf{p994, line 21:} (A.4) is not a basis for $H^*(V_{\i})$; neither is (A.2);
\item\textsf{p994, (10.7):} last product does not make sense with conventions as on~p1023;
\item\textsf{p996, (10.12):} $\oplus$ \to $\otimes$;
\item\textsf{p997, line 9:} $(\al_X,\al_Y)$ \to $\al$;
\item\textsf{p997, (10.15):} $\GT_Z(\al_X,\al_Y)$ \to $\GT_Z(\al)$;
\item\textsf{p997, line -5, and p998, line 2:}
$\GT_{\chi,A,Z}(\al_X,\al_Y)$ \to GT$_{Z,A,\chi}(\al)$;
\item\textsf{p997, line -4:} $\GT_{\chi_2,A_2,Y}^V(C_{\bm^*};\al_Y)$ \to
$\GT_{Y,A_2,\chi_2}^V(\al_Y;C_{\bm^*})$; 
\item\textsf{p998, line 1:} (A.6) also involves $\kappa$;
\item\textsf{p998, line 3:} it is unclear how the relative constraints enter in
the notation;
\item\textsf{p1024, (A.6):} $g$ \to $\chi$; same on line~6 (twice).\\
\end{enumerate}
\end{rmk}

\noindent
The intended symplectic sum formula for primary invariants in~\cite{LR}, 
i.e.~as in Theorem~\ref{main_thm} in this paper,
is split between equations~(5.4), (5.7), and~(5.9).
The first of these is vague on the set~$\cC_{g,m}^{J,[A]}$ indexing the summands,
while the last is vague on the relation between~$\al$ and~$\al^{\pm}$.
The key set~$\cC_{g,m}^{J,[A]}$ is independent of~$J$, but is generally infinite, 
contrary to \cite[Lemma~5.4]{LR}, in part because its elements are not restricted
to the classes that can be represented by $J$-holomorphic maps.
Taken together, the three formulas are at least missing the factor of $\ell({\bf k})!$
in the denominator corresponding to the reorderings of the contact points.\\

\noindent
The symplectic sum formulas in~\cite{Jun2}, in the bottom half of page~201, 
involve triples $\Ga$
{\it consisting of the genus, the number of marked points and
the degree of the stable morphisms}; see \cite[p200, middle]{Jun2}).
This is written as $\Ga=(g,k,A)$ at the bottom of page~200, suggesting that~$A$
is a second homology class. 
The degree becomes~$d$ at the bottom of page~202, suggesting that $d\!\in\!\Z$ is 
the degree with respect to some  ample line bundle,
as at the bottom of page~547 in~\cite{Jun1};
the line bundle is finally mentioned as being implicitly chosen at the bottom of 
page~226 in~\cite{Jun2}.
On the other hand, $A$ becomes $b$ at the top of page~215, 
suggesting again that this is a second homology class,
as in the middle of page~512 in~\cite{Jun1}.
The correct interpretation of~$A$ for the purposes of these formulas is 
that it is the degree with respect to an ample line bundle~$\cL$ over 
the total space $\cZ\!\lra\!\De$. 
Thus, the set~$\cR_{X,Y}^V$ in~\eref{GTabsser_e2} is essentially replaced
in~\cite{Jun2} by the (generally) larger subset of  second homology classes
of $X\!\#_V\!Y$ vanishing on the first chern class of~$\cL$.
Different ample line bundles~$\cL$ give different formulas;
so effectively, the approach of~\cite{Jun2} replaces~$\cR_{X,Y}^V$ in~\eref{GTabsser_e2}
by the set of  second homology classes
of $X\!\#_V\!Y$ vanishing on the first chern classes of all ample line bundles~$\cL$.
The last set can still be larger than~$\cR_{X,Y}^V$,
since the chern class of every ample line bundle vanishes over torsion classes.
Thus, the numerical decomposition formula for primary invariants on 
page~201 of~\cite{Jun2} is weaker than
Theorem~\ref{main_thm}, even when restricting to the algebraic category.
This weakness is fully addressed in~\cite{AF}, according to the authors.\\

\noindent
The analogue of~\eref{SympSumForm_e2} in~\cite{Jun2} is 
an immediate consequence of the decomposition formula for 
virtual fundamental classes (VFCs) at the bottom of page~201 in~\cite{Jun2}.
The latter requires constructing a VFC for 
(absolute) stable maps to the singular target $X\!\cup_V\!Y$, 
showing that it equals to the VFCs for stable maps to~$X\!\#_V\!Y$
in a suitable sense (a priori they lie in homology groups of different spaces),
and decomposing the former into VFCs for relative maps into~$(X,V)$ and~$(Y,V)$.
The last step in particular is not even a priori intuitive because 
the stable maps into $X\!\cup_V\!Y$ generally do not split uniquely into
relative maps to~$(X,V)$ and to~$(Y,V)$; see Section~\ref{RelInv_sub0b}.
As pointed out in \cite[Remark~3.2.11]{AF}, the constructions in \cite{Jun1,Jun2} involve
some delicate issues;
these are further elaborated on in \cite{GS,Chen}.\\

\noindent
The argument in~\cite{LR} considers only primary insertions, as in Theorem~\ref{main_thm}, 
while the argument in~\cite{IPsum} considers only primary insertions and constraints
that are pulled back from the Deligne-Mumford space, as in Theorem~\ref{main2_thm}.
There are brief statements in both papers that the arguments apply to descendants 
($\psi$-classes),
but neither paper contains a symplectic sum formula involving descendants.
As illustrated by the appearance of rules of assignment in the symplectic sum formula
in~\cite{Jun2}, stating such a formula requires  a bit of care.
Furthermore, descendants do not even fit with the approach in \cite{IPrel,IPsum},
as it is based on defining invariants by intersecting with classes in~$X^k$ and 
the Deligne-Mumford space (such intersections do not directly cover the $\psi$-classes).

\subsection{Refining the symplectic sum formula}
\label{RefSymSum_subs}

\noindent
We now describe two refinements to the usual symplectic sum formula suggested 
in~\cite{IPsum}.
The first one concerns differentiating between GW-invariants of $X\!\#_V\!Y$ in classes
differing by vanishing cycles.
It works partially at least on the conceptual level and can sometimes be used to obtain
qualitative information about GW-invariants of $X\!\#_V\!Y$;
see \cite[Theorems~1.1,4.9]{GWsumIP}.
The second suggestion aims to replace $q_{\#}^*\al$ on the left-hand sides of~\eref{SympSumForm_e}
and~\eref{SympSumForm_e2} by arbitrary cohomology insertions from $X\!\#_V\!Y$;
this suggestion does not appear to make sense at~all.\\

\noindent
An unfortunate deficiency of the symplectic sum formulas of 
Theorems~\ref{main_thm} and~\ref{main2_thm} is that generally they describe 
combinations of GW-invariants, rather than individual GW-invariants, 
of a symplectic sum $(X\!\#_V\!Y,\om_{\#})$ of~$(X,\om_X)$ and~$(Y,\om_Y)$
in terms~of relative GW-invariants of~$(X,\om_X,V)$ and~$(Y,\om_Y,V)$. 
The aim of the rim tori refinement of~\cite{IPrel} of the usual relative invariants 
is to resolve this deficiency in~\cite{IPsum}.\\

\noindent
With  $\De_{\bs}^V\!\subset\!V_{\bs}\!\times\!V_{\bs}$ denoting the diagonal, 
let
$$\wh{V}_{X,Y;\bs}= \cH_{X;\bs}^V\times_{V_{\bs}}\cH_{Y;\bs}^V
\equiv
\big\{\pi_{X;\bs}^V\!\times\!\pi_{Y;\bs}^V\big\}^{-1}\big(\De_{\bs}^V\big).$$
Given an element $(A_X,A_Y)$ of~\eref{H2XYV_e}, define 
\BE{Xdiagdfn_e}\begin{split}
\wt\M_{\chi_X,k_X;\bs}^V(X,A_X)\times_{V_{\bs}}
\wt\M_{\chi_Y,k_Y;\bs}^V(Y,A_Y)
&=\big\{\wt\ev_{X;V}\!\times\!\wt\ev_{Y;V}\big\}^{-1}\big(\wh{V}_{X,Y;\bs}\big)\\
&=\big\{\ev_{X;V}\!\times\!\ev_{Y;V}\big\}^{-1}\big(\De_{\bs}^V\big)\,,
\end{split}\EE
with $\wt\ev_X^V$ and $\wt\ev_Y^V$ as in~\eref{evXVlift_e}.
The idea of~\cite{IPsum} is that there is a continuous map
\BE{IPdegmap_e} g_{A_X,A_Y}\!:\wh{V}_{X,Y;\bs}\lra A_X\!\#_V\!A_Y\subset H_2(X\!\#_V\!Y;\Z)\EE
such that its composition with the restriction of
\BE{LfProdEv_e}\wt\ev_X^V\!\times\!\wt\ev_Y^V\!: 
\wt\M_{\chi_X,k_X;\bs}^V(X,A_X)\!\times\!\wt\M_{\chi_Y,k_Y;\bs}^V(Y,A_Y)   
\lra  \wh{V}_{X;\bs}\!\times\! \wh{V}_{Y;\bs}\EE
to the subspace~\eref{Xdiagdfn_e} is the homology degree of the glued map into~$X\!\#_V\!Y$;
see \cite[Figure~2]{GWrelIP}.
By \cite[Proposition~4.2]{GWsumIP}, such a map~\eref{IPdegmap_e} indeed exists 
if the lifted evaluation maps~\eref{evXVlift_e} are chosen systematically
in the sense of \cite[Theorem~6.5]{GWrelIP}.
It again depends on the choices of base points in certain spaces.\\

\noindent
By the previous paragraph, the space of maps into $X\!\cup_V\!\!Y$ contributing to
the GW-invariant of $X\!\#_V\!Y$ of a degree $A\!\in\!A_X\!\#_V\!A_Y$
is the preimage of 
\BE{cHXYVC_e}\wh{V}_{X,Y;\bs}^A\equiv g_{A_X,A_Y}^{-1}(A)\EE
under the morphism~\eref{LfProdEv_e}.
Each $\wh{V}_{X,Y;\bs}^A$ is a closed oriented submanifold of $\wh{V}_{X;\bs}\!\times\!\wh{V}_{Y;\bs}$
and determines an intersection homomorphism and thus a class 
\BE{PDcXY_e} \PD_{X,Y;\bs}^{V,A}\De\in 
H^*\big(\wh{V}_{X;\bs}\!\times\!\wh{V}_{Y;\bs};\Q\big),\EE 
as suggested by \cite[Definition~10.2]{IPsum}; see \cite[Section~3.1]{GWsumIP}.\\

\noindent
The intersection product $\cdot_{V_{\bs}}$ in~\eref{bspairing_e} is equivalent 
to intersecting $Z_X\!\times\!Z_Y$ with $\De_{\bs}^V$ in $V_{\bs}\!\times\!V_{\bs}$.
Replacing this intersection in~\eref{thpair_e}
by the intersection with the closed submanifold~\eref{cHXYVC_e}, we obtain a pairing 
\BE{thpair_e2}\begin{split}
\wt\star_{A;\vt}\!:\Hom\big(\T^*(X),H_*(\wt\M\!\times\!\wh{V}_X)\big)
&\otimes \Hom\big(\T^*(Y),H_*(\wt\M\!\times\!\wh{V}_Y)\big)\\ 
&\lra
\Hom\big(\T^*(X)\!\otimes\!\T^*(Y),H_*(\wt\M)\big)\big[\la^{-1}\big],
\end{split}\EE
where 
$$\wh{V}_X=\bigsqcup_{\ell=0}^{\i}\bigsqcup_{\bs\in(\Z^+)^{\ell}}\!\!\!\!\!\wh{V}_{X;\bs},
\qquad
\wh{V}_Y=\bigsqcup_{\ell=0}^{\i}\bigsqcup_{\bs\in(\Z^+)^{\ell}}\!\!\!\!\!\wh{V}_{Y;\bs}\,.$$
We extend this pairing as in~\eref{bspairing_e2}, replacing $t_{A_X\#_V A_Y}$
by~$t_A$, and denote the result by~$\wt\star_{\vt}$.
The same gluing/deformation arguments that yield~\eref{SympSumForm_e2} then give
\BE{SympSumForm_e3}
\wt\GT^{X\#_VY}(q_{\#}^*\al) =\sum_{\vt\in\RA}
\big\{\wt\GT^{X,V} \wt\star_{\vt}\wt\GT^{Y,V}\big\}\big( (q_{\sqcup}^*\al)_{\vt}\big)
\qquad\forall\,\al\!\in\!\T^*(X\!\cup_V\!Y),\EE
where
$$\wt\GT^{X\#_VY}=\sum_{\chi\in\Z}\sum_{A\in H_2(X\!\#_V\!Y;\Z)} \!\!\!\!\!\!\!\!\!\!\!
\GT^{X\#_VY}_{\chi,A}t_A\la^{\chi}, \quad
\wt\GT^{M,V}=\sum_{\chi\in\Z}\sum_{A\in H_2(M;\Z)}
\sum_{\ell=0}^{\i}\sum_{\begin{subarray}{c}\bs\in(\Z^+)^{\ell}\\ |\bs|=A\cdot_MV\end{subarray}}
\!\!\!\!\!\!\wt\GT_{\chi,A;\bs}^{M,V}\,t_A\la^{\chi},$$
and 
$$\wt\GT_{\chi,A;\bs}^{M,V}\!: \T^*(M)\lra H_*(\wt\M\!\times\!\wh{V}_{M;\bs}), \quad
\GT_{\chi,A;\bs}^{X,V}(\al)=\sum_{k=0}^{\i}\{\st\!\times\!\wt\ev_M^V\}_*
\big(\ev^*\al\!\cap\!\big[\wt\M_{\chi,k;\bs}^V(M,A)\big]^{\vir}\big).$$\\

\noindent
The formulas~\eref{SympSumForm_e2} and~\eref{SympSumForm_e3} appear to
express the GW-invariants of $(X\!\#_V\!Y,\om_{\#})$ in terms
of the GW-invariants of~$(X,V)$ and~$(Y,V)$.
The intersection product in~\eref{SympSumForm_e2} corresponds to pulling back
$\PD_{V_{\bs}^2}\De_{\bs}^V$ by the morphism
$$\ev_X^V\!\times\!\ev_Y^V\!: 
\wt\M_{\chi_X,k_X;\bs}^V(X,A_X)\!\times\!\wt\M_{\chi_Y,k_Y}^V(Y,A_Y)   
\lra  V_{\bs}\!\times\! V_{\bs}$$
and capping the result with the virtual fundamental class of the domain.
Since $H_*(V_{\bs};\Q)$ is finitely generated,
\BE{KunnDecomp_e} \PD_{\bs}^V\De=\sum_{i=1}^N\ka_{X;i}\!\otimes\!\ka_{Y;i}\in H^{(n-1)\ell}(V_{\bs}^2;\Q)\EE
for some $\ka_{X;i},\ka_{Y;i}\!\in\!H^*(V_{\bs};\Q)$;
see \cite[Theorem~60.6]{Mu2}.
Thus, the coefficients on the right-hand side of~\eref{SympSumForm_e2} decompose into
products of the relative GW-invariants of $(X,V)$ and~$(Y,V)$.\\

\noindent
The intersection product in~\eref{SympSumForm_e3} similarly corresponds to pulling back
the class~\eref{PDcXY_e} by the lifted morphism~\eref{LfProdEv_e}.
If the $\Q$-homology of either $\wh{V}_{X;\bs}$ or $\wh{V}_{Y;\bs}$ is finitely generated,
then
\BE{KunnDecomp_e2}
\PD_{X,Y;\bs}^{V,A}\De=\sum_{i=1}^N\wt\ka_{X;i}\!\otimes\!\wt\ka_{Y;i}\in 
H^{(n-1)\ell}(\wh{V}_{X;\bs}\!\times\!\wh{V}_{Y;\bs};\Q)\EE
for some $\wt\ka_{X;i}\!\in\!H^*(\wh{V}_{X;\bs};\Q)$ and $\wt\ka_{Y;i}\!\in\!H^*(\wh{V}_{Y;\bs};\Q)$;
see \cite[Theorem~60.6]{Mu2}.
This is also the case if the submodule~$\cR_{X,Y}^V$ of $H_1(X\!\#_V\!Y;\Z)$ is finite;
see \cite[Corollary~4.3]{GWsumIP}.
In such cases, the approach of~\cite{IPsum} provides a refined 
{\it decomposition} formula for GW-invariants of $X\!\#_V\!Y$ in terms of IP-counts
for $(X,V)$ and~$(Y,V)$.
However, in general  the homologies of $\wh{V}_{X;\bs}$ and $\wh{V}_{Y;\bs}$ are not finitely generated 
and a Kunneth decomposition~\eref{KunnDecomp_e2} need not exist; see 
\cite[Example~3.7]{GWsumIP}.
In these cases, the approach of~\cite{IPsum} does not provide a decomposition formula 
for GW-invariants of $X\!\#_V\!Y$ in terms of any kind of invariants of~$(X,V)$ and~$(Y,V)$.

\begin{rmk}\label{RimTori_rmk}
The map \cite[(3.10)]{IPsum} is not specified.
The typos in the related part of \cite[Section~10]{IPsum} include
\begin{enumerate}[label={},topsep=-5pt,itemsep=-5pt,leftmargin=*]
\item \textsf{p992, after (10.5):} (10.5) already involves disconnected domains;
\item \textsf{p993, line 13:} $\cH_Y^V\times\cH_Y^V$ \to  $\cH_X^V\times\cH_Y^V$;
\item \textsf{p994, line 7:} $\sum\sum$ \to $\bigoplus\bigoplus$.\\
\end{enumerate}
\end{rmk}

\noindent
The stated symplectic sum formulas of \cite{Jun2,IPsum,LR} involve 
cohomology insertions of the form $q_{\#}^*\al$, as in Theorem~\ref{main2_thm}.
Arbitrary cohomology insertions are considered in \cite[Section~13]{IPsum},
as follows.
Let $\wh{X}$ be the compactification of $X\!-\!V$ obtained by removing
an open tubular neighborhood of~$V$ or equivalently by replacing~$V$
with~$SV$ in~$X$; see \cite[Section~5]{IPrel} and 
the end of Section~\ref{SympCut_subs}.
Let
$$q\!:(\wh{X},\prt\wh{X})\lra (X,V)$$
be the natural projection map.
Given a pseudocycle representative \hbox{$\phi\!:(P,\prt P)\!\lra\!(\wh{X},SV)$}
for a class $B\!\in\!H_*(\wh{X},\prt\wh{X})$ and $i\!=\!1,\ldots,k$, let
$$\wt\M_{\chi,k;\bs}(X,A)\!\times_i\!\phi
=\big\{(u,x)\!\in\!\wt\M_{\chi,k;\bs}(X,A)\!\times\!P\!:\,
\ev_i(u)\!=\!q(\phi(x))\big\}\,.$$
Intersecting with other pseudocycle representatives in a similar way,
we obtain a virtual orbifold with boundary and evaluation map~$\ev^V$ to~$V_{\bs}$,
which can then be used to define ``extended" relative counts for~$(X,V)$.\\

\noindent
In general, these counts depend on the choice of the almost complex structure~$J$,
deformation~$\nu$, and the representatives~$\phi$ for classes~$B$.
By \cite[Lemma~13.1]{IPsum}, this dependence disappears whenever
\BE{Bclass_e} \prt B\in\ker\big\{q_{V*}\!: H_{*-1}(SV)\lra H_{*-1}(V)\big\}\,.\EE
By \cite[Corollary~4.12]{GWrelIP}, these are precisely the cases obtained
from cutting a Poincare dual in $X\!\#_V\!Y$ of a cohomology insertion as
in Theorem~\ref{main2_thm}.
The dependence on the representative~$\phi$ for~$B$, but not on~$(J,\nu)$, 
is analyzed in \cite[Lemma~13.2]{IPsum}.
Unfortunately, the intended meaning of \cite[(13.4)]{IPsum} is unclear: 
it involves~$\GT_F^{VV}(\phi')$, which is not defined, as well as 
some convolution product of $\GT_X^V$ and~$\GT_F^{VV}(\phi')$;
no proof of this lemma is provided either.
The intention of \cite[Lemma~13.2]{IPsum}
is to extend the definition of relative invariants 
to homology insertions $B\!\in\!H_*(\wh{X},\prt\wh{X})$ by defining such numbers
for a fixed $J$, $\nu$, and~$\phi$ and then to use them in an extended symplectic 
sum formula, which is not stated.
Even if this were possible to~do, it is not apparent that the resulting relative ``invariants'' 
could be readily computed, especially given their dependence on $J$, $\nu$, and~$\phi$;
so such an extended symplectic sum formula may not be  useful.

\begin{rmk}\label{chap13_rmk}
The definition of $\wt\M_{\chi,k;\bs}(X,A)\!\times_i\!\phi$ in 
the displayed equation above \cite[(13.1)]{IPsum} as an intersection does not make
sense, since the two sets being intersected lie in different spaces.
By \cite[Lemma~3.1]{GWrelIP}, every class~$B$ as in~\eref{Bclass_e}
is the boundary of a pseudocycle into a closed tubular neighborhood of~$V$ in~$X$.
Thus, for such a class~$B$, the cut-down moduli space $\wt\M_{\chi,k;\bs}(X,A)\!\times_i\!\phi$
is a boundary as well.
This implies that the $\GT$-invariants for classes $B\!\in\!H_*(X\!-\!V)$ 
depend only on their images in~$H_*(X)$,
contrary to the suggestion at the top of \cite[p1006]{IPsum}.
The conclusion after \cite[Lemma~13.2]{IPsum} is that extended relative invariants can
be defined by choosing pseudocycle representative~$\phi_{\be}$ as above for each 
$$\be\in\ker\big\{H_{*-1}(SV)\lra H_{*-1}(X)\big\}$$
such that $[\prt\phi]\!=\!\be$; as just indicated, this would not provide any additional
information.
If the intended meaning in~\cite{IPsum} were to fix a representative~$\phi_B$
for each $B\!\in\!H_*(\hat{X},X)$ as in~\eref{Bclass_e},
this would still cover only the insertions of Theorem~\ref{main2_thm}.
Other, fairly minor misstatements in \cite[Section~13]{IPsum} include
\begin{enumerate}[label={},topsep=-5pt,itemsep=-5pt,leftmargin=*]
\item \textsf{p1005, Section 13, line 3:} constraints not of the form 
$q_{\#}^*\al_{\cup}$ as in \cite[Corollary~4.12]{GWrelIP};
\item \textsf{p1005, Section 13, line 15:} $g$ \to $f$;
\item \textsf{p1005, Section 13, line 17:} $\phi\!:P\lra\wh{X}$;
\item \textsf{p1006, line 9:} real codimension one in cut-down of (13.2);
\item \textsf{p1006, Lemma~13.1, line 1:} $\GT_{X,A,\bs}^V(\phi)$ in
\cite[(13.1)]{IPsum} is not a number;
\item \textsf{p1006, Lemma~13.2, line 2:} $\PD$ is not defined;
\item \textsf{p1006, Lemma~13.2, line 5; p1006, line -2}: $H_*(X,V)$ \to $H_*(\wh{X},\prt\wh{X})$.
\end{enumerate}
\end{rmk}

\section{On the proof of Theorem~\ref{main2_thm}}
\label{sumpf_subs}

\noindent
The analytic steps needed to establish Theorem~\ref{main2_thm} 
can be roughly split into four parts:
\'a priori estimates on convergence and on stable maps to~$X\!\cup_V\!Y$,
a pregluing construction, uniform elliptic estimates,  and a gluing construction;
we review them below.
While some statements in~\cite{IPsum} implicitly assume suitable positivity
conditions on~$(X\!\#_V\!Y,\om_{\#})$, $(X,V)$, and~$(Y,V)$,
the approach described in~\cite{IPsum} to comparing numerical GW-invariants 
should fit with all natural VFC constructions, such as in~\cite{FO,LT}, 
once they are shown to apply to relative invariants.
However, the analytic issues required for constructing and comparing 
the relevant VFCs  appear to be much harder to deal with
in the approach of~\cite{IPsum} than of~\cite{LR}.

\subsection{\'A priori estimates: \cite[Sections 3-5]{IPsum}, \cite[Section 3.1]{LR}}
\label{IPconv_subs}

\noindent
Let $V\!\subset\!X$ be a submanifold of real codimension two and
$J$ be an almost complex structure on~$X$ such that $J(TV)\!=\!TV$.
Suppose $(\Si,\fj)$ is a smooth Riemann surface,
$$\nu\in\Ga\big(\Si\!\times\!X,T^*\Si^{0,1}\!\otimes_{\C}\!TX\big)
\qquad\hbox{s.t.}\quad
\nu|_{\Si\times V}\in\Ga\big(\Si\!\times\!V,T^*\Si^{0,1}\!\otimes_{\C}\!TV\big),$$
and $z$ is a complex coordinate on a neighborhood~$\Si_{u;z_0}$ of~$z_0$ with $z(z_0)\!=\!0$.
Let $u\!:\Si\!\lra\!X$ be a smooth map such~that $u^{-1}(V)\!=\!\{z_0\}$ and
$$\dbar_{J,\fj}u|_z\equiv\frac12\big(\nd_z u+J(u(z))\circ\nd_zu\circ\fj_z\big)
=\nu\big(z,u(z)\big) \qquad\forall~z\in\Si.$$
If $\cN_XV|_{W_{u(z_0)}}\!\approx\!W_{u(z_0)}\!\times\!\cN_XV|_{u(z_0)}$ 
is a trivialization of~$\cN_XV$ 
over a neighborhood~$W_{u(z_0)}$ of~$u(z_0)$ in~$V$, then there exist
\begin{enumerate}[label=$\bullet$,leftmargin=*]
\item a neighborhood~$\Si_{u;z_0}'$ of~$z_0$ in~$u^{-1}(\cN_XV|_{W_{u(z_0)}})\!\cap\!\Si_{u;z_0}$ and
\item $\Phi\!\in\!L_1^p(\Si_{u;z_0}';\cN_XV|_{u(z_0)}\!-\!0)$, for any $p\!>\!2$, such that 
\BE{FHS_e} 
\pi_2(u(z))=\Phi(z) z^{\ord_{z_0}^V(u)}~~\forall\,z\!\in\!\Si_{u;z_0}'\,;\EE
\end{enumerate}
see  \cite[Theorem~2.2]{FHS}.\\

\noindent
Let $\pi\!: \cZ\!\lra\!\De$ be a symplectic fibration associated with 
the symplectic sum $(X\!\#_V\!Y,\om_{\#})$ as in Proposition~\ref{SymSum_prp}
and $J_{\cZ}$ be an $\om_{\cZ}$-compatible almost complex structure on~$\cZ$
as before.
By Gromov's Compactness Theorem \cite[Proposition~3.1]{RT1},
a sequence of $(J_{\cZ},\fj_k)$-holomorphic maps 
$u_k\!:\Si\!\lra\!\cZ_{\la_k}$, with $\la_k\!\in\!\De^*$ and $\la_k\!\lra\!0$,
has a subsequence converging to a $(J_{\cZ},\fj)$-holomorphic map
$u\!:\Si'\!\lra\!\cZ_0$.
By~the previous paragraph, 
$$\Si'=\Si_X'\cup\Si_V'\cup\Si_Y'\,,$$
where $\Si_V'$ is the union of irreducible components of~$\Si'$ mapped into~$V$,
$\Si_X'$ is the union of irreducible components mapped into~$X\!-\!V$ outside
of finitely many points $x_1,\ldots,x_{\ell}$, and
$\Si_Y'$ is the union of irreducible components mapped into~$Y\!-\!V$ outside
of finitely many points $x_1',\ldots,x_{\ell'}'$.
By \cite[Lemma~3.3]{IPsum}, which is the main statement of \cite[Section~3]{IPsum},
if $\Si_V'\!=\!\eset$, then $\ell\!=\!\ell'$ and 
\BE{}\big(\ord_{x_1'}^Vu,u(x_1')\big)=\big(\ord_{x_{\tau(1)}}^Vu,u(x_{\tau(1)})\big),
\quad\ldots\quad
\big(\ord_{x_{\ell}'}^Vu,u(x_{\ell}')\big)=\big(\ord_{x_{\tau(\ell)}}^Vu,u(x_{\tau(\ell)})\big)\EE
for some automorphism $\tau\!\in\!S_{\ell}$ of $\{1,\ldots,\ell\}$.
This conclusion also holds for sequences of $(J_{\cZ},\nu)$-holomorphic maps
with 
$$\nu|_V\in \Ga_{g,k}(V,J_{\cZ}|_V),$$
similarly to~\eref{nuVrestr_E}.

\begin{rmk}\label{IPsumSec3_rmk}
The expansion \cite[(5.5)]{IPsum}, based on \cite[Lemma~3.4]{IPrel},
corresponds to $\Phi$ above being differentiable at $z\!=\!0$.
As can be seen from~\eref{FHS_e}, this is indeed the case if $u$ is smooth.
The proof of \cite[Lemma~3.3]{IPsum} is purely topological and applies to 
convergent sequences of continuous maps.
An explicit condition, called \textsf{$\de$-flatness}, insuring that 
$\Si_V'\!=\!\eset$ above is described in \cite[Section~3]{IPsum}.
Contrary to the suggestion after \cite[Definition~3.1]{IPsum},
the $\de$-flatness condition does not prevent the marked points from being sent
into~$V$ and thus a $\de$-flat $J$-holomorphic map into $\cZ_0$
need not be $V$-regular in the sense of \cite[Definition~4.1]{IPrel}.
Other, fairly minor misstatements in \cite[Section~3]{IPsum} include
\begin{enumerate}[label={},topsep=-5pt,itemsep=-5pt,leftmargin=*]
\item \textsf{p954, (3.5):} the limit is over $\la\neq0$;
\item \textsf{p954, line -12:} there is no \cite[Lemma~3.2]{IPrel};
\cite[Lemma~3.4]{IPrel} alone suffices;
\item \textsf{p955, Lemma 3.3, proof:} $f_k$ is an element of a sequence, but
$f_1,f_2$ are parts of a limiting~map;
\item \textsf{p956, lines 7-8}: stabilization does not fit with this map and
there is no need for it, since $\wt\M_{\chi,n}$ consists of curves with finitely 
many components, not necessarily stable ones, according to p946, line~-6;
\item \textsf{p957, (3.11); p957, (3.12):}
the fiber products should be quotiented by $S_{\ell(s)}$.\\
\end{enumerate}
\end{rmk}

\noindent
A node of the limiting map~$u$ as in \cite[Lemma~3.3]{IPsum} corresponds
to special points $z_0\!\in\!\Si_X'$ and $w_0\!\in\!\Si_Y'$ with 
$$u(z_0)=u(w_0)=q \qquad\hbox{and}\qquad \ord_{z_0}^Vu=\ord_{w_0}^Vu=s$$
for some $q\!\in\!V$ and $s\!\in\!\Z^+$.
A neighborhood of this node in the total space of a versal family of deformations of~$\Si'$
is given~by
\BE{VersDef_e} 
U=\big\{(\mu',\mu,z,w)\!\in\!\C^{\ell-1}\!\times\!\C^3\!:\,zw\!=\!\mu\big\},\EE
with $\Si'$ corresponding to $(\mu,\mu')\!=\!0$ and the node at $(z,w)\!=\!0$.
Let
$$U_{\mu',\mu;\ep}=
\big\{(\mu',\mu,z,w)\!\in\!U\!:\,|z|,|w|\!<\!\ep^{1/2}\big\}\,,\qquad
\vr(\mu',\mu,z,w)=\sqrt{|z|^2\!+\!|w|^2}\,.$$
Denote~by
$$x\!:\cN_XV|_{W_q}\lra \cN_XV|_q \qquad\hbox{and}\qquad
y\!:\cN_YV|_{W_q}\lra \cN_YV|_q$$
the projections induced by dual trivializations of~$\cN_X$ and~$\cN_Y$ over 
a neighborhood~$W_q$ of~$q$ in~$V$.
Below we will assume that $W_q$ is identified with a ball in~$\R^{2n}$ using geodesics from~$q$.\\

\noindent
Let $(\mu_k,\mu_k')\!\in\!\C\!\times\!\C^{\ell-1}$ be the parameters corresponding
to~$\Si_k$, the domain of~$u_k$.
For each $\ep\!<\!\ep_k$ such that $u_k(U_{\mu_k',\mu_k;\ep})\!\subset\!\cZ_{\neck}|_{W_q}$, 
let $\bar{u}_{k;\ep}^V\!\in\!W_q$ denote the average value of 
$\pi_V\!\circ\!u_k|_{U_{\mu_k',\mu_k;\ep}}$
with respect to the cylindrical metric on~$U_{\mu',\mu;\ep}$
and
$$\wt{u}_{k;\ep}^V(z)=\pi_V\!\circ\!u_k(z)-\bar{u}_{k;\ep}^V\in W_q
 \qquad\forall~z\!\in\!U_{\mu_k',\mu_k;\ep}.$$
Under the assumptions of the paragraph above Remark~\ref{IPsumSec3_rmk},
$$x(u_k(z))\cdot y(u_k(z))=\la_k \qquad\forall~z\!\in\!\Si_k\,.$$
By \cite[Theorem~2.2]{FHS}, 
\BE{NdExp_e}\lim_{z\lra0}\frac{x(u(z))}{az^s}=1 \qquad\hbox{and}\qquad 
\lim_{z\lra0}\frac{y(u(w))}{bw^s}=1\EE
for some $a\!\in\!\cN_XV|_q\!-\!0$ and $b\!\in\!\cN_YV|_q\!-\!0$.
By \cite[Lemma~5.3]{IPsum},
\BE{mula_e} \lim_{k\lra\i}\frac{\la_k}{ab\mu_k^s}=1.\EE
The factor of~$\lr\bs$ in~\eref{bspairing_e} is a reflection of this statement
and takes into account the number of solutions~$\mu_k$ of 
the equation $ab\mu_k^s\!=\!\la_k$ for a fixed~$\la_k$.
By \cite[Lemma~5.4]{IPsum},
\BE{NdExp_e2}\begin{split}
\int_{U_{\mu_k',\mu_k;\ep}}\!\!\!\!
\Big(|\wt{u}_{k;\ep}|^p\!+\!|\nd\wt{u}_{k;\ep}|^p
&+|\vr^{1-s}x\!\circ\!u_k|^p+|\vr^{1-s}\nd(x\!\circ\!u_k)|^p\\
&+|\vr^{1-s}y\!\circ\!u_k|^p+|\vr^{1-s}\nd(y\!\circ\!u_k)|^p\Big)\vr^{-p\de'}
\le C_p\ep^{p/3}
\end{split}\EE
for all $p\!\ge\!2$, $\de',\ep\!\in\!\R^+$ sufficiently small, $k\!\in\!\Z^+$
sufficiently large,
and for some $C_p\!\in\!\R^+$ (dependent of the sequence~$\{u_k\}$ only);
the norms in~\eref{NdExp_e2} are defined using the cylindrical metric on~$U_{\mu_k',\mu_k;\ep}$
and the metric~$g_{\cZ}$ on~$\cZ$.
Both statements, \eref{mula_e} and~\eref{NdExp_e2}, make use of \cite[Lemma~5.1]{IPsum},
which is a version of the standard exponential decay of the energy of a $J$-holomorphic map
in the ``middle" of a long cylinder; see \cite[Lemma~4.7.3]{MS2}.

\begin{rmk}\label{IPsumSec4_rmk}
The proof of \cite[Lemma~5.3]{IPsum} uses a complete metric 
on the universal curve $\cU_{g,n}$ over the moduli space~$\M_{g,n}$ of 
smooth $n$-marked genus~$g$ curves (with Prym structures) constructed in 
\cite[Section~4]{IPsum} by re-scaling a Kahler metric~$g_{\cU}$ on~$\ov\cU_{g,n}$ 
along the nodal strata~$\cN$.
The apparent, implicit intention is to take the metric~$g_{\cU}$ in \cite[(4.1)]{IPsum}
so that it satisfies \cite[(4.4)]{IPsum}.
As the various local metrics are patched together, 
the resulting global metric is not of the form \cite[(4.10)]{IPsum} everywhere near~$\cN$.
This section also does not yield a compactification of~$\M_{g,n}$ as described
in the last paragraph, because it is unclear how the different tori fit together and 
because \cite[(4.3)]{IPsum} describes the normal bundle to a certain immersion,
not to a submanifold of~$\ov\M_{g,n}$.
Even outside of the singular locus of the immersion, this normal bundle 
may not be biholomorphic to a neighborhood;
in particular, the construction described above \cite[Remark~4.1]{IPsum}
 need not extend outside of the open strata~$\cN_{\ell}$ of curves with 
 precisely~$\ell$ nodes.
For a related reason, 
the construction in this section does not lead to uniform estimates in the following
sections, only fiber-uniform ones,
contrary to a claim at the top of \cite[p960]{IPsum}.
The second sentence of  \cite[Remark~4.1]{IPsum} has no connection with the first.
However, none of these additional statements is necessary for the purposes of~\cite{IPsum}.
Other, fairly minor misstatements in \cite[Section~4]{IPsum} include
\begin{enumerate}[label={},topsep=-5pt,itemsep=-5pt,leftmargin=*]
\item\textsf{p958, line -2:} separated by a minimum distance {\it in each fiber};
\item\textsf{p959, line 1:} distance to the nodal set; this is not a smooth function;
\item\textsf{p959, above Remark 4.1:} all curves have Prym structures by assumption;
\item\textsf{p959, line -1:} $T_k$ \to $T_k(\mu)$;
\item\textsf{p960, line 1:} $T_k=\log(2/2\sqrt{|\mu_k|})$ \to
$T_k(\mu)=\frac12\cosh^{-1}(1/|\mu_k|)$;
\item\textsf{p960, line 10:} the restriction of (4.1) to a fiber agrees with (4.5);
\item\textsf{p960, line 14:} no change of constants needed given conformal invariance;
\item\textsf{p960, last 2 paragraphs:} issues similar to p960, line~1;
\item\textsf{p960, line -9:} the two fractions should be $\mu_k'/\mu_k$; 
the map is defined only near each~$\mu$;
\item\textsf{p960, line -3:} this is for $\ti{v}_k$ on each $B_k(2)$ and there should be no sum;
\item\textsf{p961, line 2:} $j$ restricted to each fiber;
\item\textsf{p961, line -1:} $\Re\,(d\mu_k)^2$ should not be here;
\item\textsf{p962, line 2:} distance between $\mu$ and $\mu'$ is the sum of the logs, but
$\mu_k=e^{t+i\th}$ and $\mu_k'=e^{t'+i\th'}$.
\end{enumerate}
\end{rmk}

\begin{rmk}\label{IPsumSec5_rmk}
The statement of \cite[Lemma~5.4]{IPsum} is not carefully formulated.
In particular, $\bar{v}_n$ ($\bar{u}_{k;\ep}$ in our notation above)
and $\na$ are not defined.
Based on the proof, the latter denotes a connection on~$\cZ$, not on~$\cZ_{\la}$.
Since  $\rho$ ($\vr$ in our notation) is bounded above, 
the intended statement of \cite[Lemma~5.4]{IPsum}
is equivalent to the $\de\!=\!1/3$ case;
this~$\de$ has no connection to the~$\de$ used to construct the symplectic sum~$(\cZ,\om_{\cZ})$
in \cite[Section~2]{IPsum}.
Other, fairly minor misstatements in \cite[Section~5]{IPsum} include
\begin{enumerate}[label={},topsep=-5pt,itemsep=-5pt,leftmargin=*]
\item\textsf{p962, Section 5, line 6:}
$C^{\infty}$-convergence on compact sets implies $L^{1,2}$ and 
$C^0$-convergence on the same sets; $L^{1,2}$ and 
$C^0$-convergence on entire domain does not make sense;
\item\textsf{p962, line -4:} in (3.11), $\cK_{\de}$ is contained in a different space;
\item\textsf{p963, line 1:} (4.4) \to (4.2);
\item\textsf{p963, lines 2-7:} $n$ here is $k$ in Section~4 and different from the subscript
in~$f_n$ on line~9 and the superscripts on line~12;
\item\textsf{p963, lines 7:} $|\log(2/2\sqrt{|\mu_n|})|$ \to
$\frac12\cosh^{-1}(1/|\mu_n|)$;
\item\textsf{p963, line -12:} graph of \to locus;
\item\textsf{pp963-964, Lemma 5.1:} $c_1\!=\!1$; $Z$\to$Z\!\times\!\ov\cU$;
$A_0\!=\![-T,T]\!\times\!S^1$; $\rho(t)^2\!=\!2|\mu_n|\cosh(2t)$, $C\!=\!\rho(T)^{-\frac23}$;
\item\textsf{p964, line 10:} $\dbar F$ \to $2\dbar F$; 
\item\textsf{p964, lines 11,12:} $J$ \to $\hat{J}$;
\item\textsf{p964, line 15:} $\dbar F\equiv F_t+iF_{\th}$ in the rest of the proof;
\item\textsf{p964, line -2:} $E(t)$ \to $E(F,t)$;
\item\textsf{p965, line 1:} $[-T,T]$ \to $[-T/2,T/2]$;
\item\textsf{p965, Definition 5.2:} inconsistency in the definition of $\hat{f}_n$;
$|\mu_n|$ \to $2|\mu_n|$;
\item\textsf{p966, line 6:} near, not along, $V$; $J\!-\!J_0$ is~$O(R)$;
\item\textsf{p966, line 7:} this long undisplayed expression has 3 typos,
and the first inequality need not hold;
\item\textsf{p966, (5.13):} $\nd x_n$ \to $\nd x$, twice;
\item\textsf{p966, line -8:} $\sqrt\mu$ \to $\sqrt{2|\mu|}$;
\item\textsf{p966, line -1:} $|\bar{G}_n(\sqrt{|\mu_n|})|$ \to $2|\bar{G}_n(\sqrt{|\mu_n|})|$;
\item\textsf{p967, line 10:} $[-T,T]$ \to $[-T_n,T_n]$;
\item\textsf{p967, Lemma~5.4, proof:} $G$ is $G_n$ of the proof of Lemma~5.3;
\item\textsf{p967, displayed equation after (5.17)} is not any of
 the CZ inequalities in the 190-page~[IS].
\end{enumerate}
\end{rmk}

\noindent
The analytic approach of~\cite{LR} is motivated by the SFT type constructions
of \cite{H,HWZ1} involving $J$-holomorphic curves on infinite ``cylinders".
Let $SV$, $\al$, and~$\oJ$ be as in Section~\ref{SympCut_subs} and
at the end of Sections~\ref{RelInv_sub0}.
For $\ell_1,\ell_2\!\in\!\R^+$ with $\ell_1\!<\!\ell_2$, 
denote by $\Phi_{\ell_1,\ell_2}$ the set of orientation-preserving 
diffeomorphisms $\phi\!:\R\!\lra\!(\ell_1,\ell_2)$.
For each $\phi\!\in\!\Phi_{\ell_1,\ell_2}$, 
$$\wt\om_{\phi}\equiv \pi^*\om_V+\nd(\phi\al)$$
is a closed two-form on $\R\!\times\!SV$;
it is symplectic and tames~$\oJ$ if $|\ell_1|,|\ell_2|$ are sufficiently small.
With such $\ell_1,\ell_2$ fixed,
for any $(\oJ,\fj)$-holomorphic map $u\!:\Si\!\lra\!\R\!\times\!SV$
from a (not necessarily compact) Riemann surface~$(\Si,\fj)$, let
\BE{LRener_e}E_{\ell_1,\ell_2}(u)=
\sup_{\phi\in\Phi_{\ell_1,\ell_2}}\int_{\Si}u^*\wt\om_{\phi}\,,
\qquad 
E_V(u)=\int_{\Si}u^*\pi^*\om_V;\EE
these numbers may not be finite.
Let $\bD\!\subset\C$ denote the closed unit ball and $\bD^*\!=\!\bD\!-\!\{0\}$.

\begin{lmm}[{\cite[Lemma~3.5]{LR}}]\label{LRconv_lmm1}
\begin{enumerate}[label=(\arabic*),leftmargin=*]
\item Let $u\!:\C\!\lra\!\R\!\times\!SV$ be a $\oJ$-holomorphic map
such that $E_{\ell_1,\ell_2}(u)$ is finite.
If $E_V(u)\!=\!0$, then $u$ is constant.
\item Let $u\!:\R\!\times\!S^1\!\lra\!\R\!\times\!SV$ be a $\oJ$-holomorphic map
such that $E_{\ell_1,\ell_2}(u)$ is finite.
If $E_V(u)\!=\!0$, then there exist $s\!\in\!\Z$, $r_0\!\in\!\R$, and
a  1-periodic orbit $\ga\!:S^1\!\lra\!SV$ of the Hamiltonian~$H$ 
such~that 
$$u\big(r,\ne^{\fI\th}\big)=\big(sr\!+\!r_0,\ga(\ne^{\fI s\th})\big) 
\qquad\forall~\big(r,\ne^{\fI\th}\big)\in \R\!\times\!S^1.$$
\end{enumerate}
\end{lmm}

\begin{crl}\label{LRconv_crl}
If $u\!:\bD^*\!\lra\!\R\!\times\!SV$ is a $\oJ$-holomorphic map
such that $E_{\ell_1,\ell_2}(u)$ is finite, then 
$$\big|\prt_t u(\ne^{t+\fI\th})\big|,
\big|\prt_t u(\ne^{t+\fI\th})\big| \le C_u$$
for some $C_u\!\in\!\R$. 
\end{crl}

\noindent
The justification provided for \cite[Lemma~3.5]{LR} is that it can be obtained
{\it using the same method as in~\cite{H}}, which treats the case 
when $(SV,\al)$ is contact, but the flow of the Reeb vector field~$\ze_H$ 
does not necessarily generate an $S^1$-action.
In fact, the assumption $E_V(u)\!=\!0$ in this case implies that the image 
of~$u$ lies in $\R\!\times\!S_xV$ for some $x\!\in\!V$ and so the situation in~\cite{H}
is directly applicable.
The two statements of Lemma~\ref{LRconv_lmm1} are thus immediately implied by 
the statement of \cite[Lemma~28]{H} and by the proof of \cite[Theorem~31]{H}
in the bottom half of page~538, respectively.\\

\noindent
Let $u$ be as in the statement of Corollary~\ref{LRconv_crl}.
By Gromov's Removable Singularity Theorem \cite[Theorem~4.1.2]{MS2},
the $J_V$-holomorphic map $\pi\!\circ\!u\!:\bD^*\!\lra\!V$ extends
to a $J_V$-holomorphic map $u_V\!:\bD\!\lra\!V$.
The proof of \cite[Proposition~27]{H}, 
which uses the standard rescaling argument to construct 
a $\oJ$-holomorphic map $f\!:\C\!\lra\!\R\!\times\!SV$ out 
of a sequence with increasing derivatives, and 
Lemma~\ref{LRconv_lmm1}(1) then yield Corollary~\ref{LRconv_crl}.

\begin{lmm}\label{LRconv_lmm2}
For every $\oJ$-holomorphic map $u\!=\!(u_{\R},u_{SV})\!:\bD^*\!\lra\!\R\!\times\!SV$ 
such that $E_{\ell_1,\ell_2}(u)$ is finite, there exist $s\!\in\!\Z$ and
a  1-periodic orbit $\ga\!:S^1\!\lra\!SV$ of the Hamiltonian~$H$ 
with the following properties.
If $r_i\!\in\!\R^+$ is a sequence with $r_i\!\lra\!0$,
there exist a subsequence, still denoted by~$r_i$, and $\th_0\!\in\!\R$  such~that 
$$\lim_{i\lra\i}u_{SV}\big(r_i\ne^{\fI t}\big)=
\ga\big(\ne^{\fI s\th+\th_0}\big)$$
in $C^{\i}(S^1,SV)$.
Furthermore, the function $u_{\R}$ is bounded if and only if $s\!=\!0$,
and $u_{\R}(r\ne^{\fI\th})\!\lra\!\mp\i$ as $r\!\lra\!0$ if and only if $s\!\in\!\Z^{\pm}$.
\end{lmm}

\noindent
This lemma corrects, refines, and generalizes the statement of \cite[Lemma~3.6]{LR};
the wording and the usage of the latter suggest that $s\!\in\!\Z^+$.
By Gromov's Removable Singularity Theorem \cite[Theorem~4.1.2]{MS2},
the $J_V$-holomorphic map $\pi\!\circ\!u_{SV}\!:\bD^*\!\lra\!V$ extends
to a $J_V$-holomorphic map $u_V\!:\bD\!\lra\!V$.
Thus, the image~of 
\BE{tivS1_e} S^1\lra SV, \qquad \ne^{\fI\th}\lra u_{SV}(r\ne^{\fI\th}),\EE
approaches $S_{u_V(0)}V$ as $r\!\lra\!0$.
Let $\ga\!:S^1\!\lra\!SV$ be a 1-periodic orbit parametrizing~$S_{u_V(0)}V$.
The claims concerning the sequence, with some choice of~$s$ and~$\th_0$, 
and the relation between the sign of~$s$ and the behavior of~$u_{\R}$
follow from the proof of \cite[Theorem~31]{H},
where the functions~$v$ and~$w$ are used interchangeably and $f\!+\!\fI b$ 
should be replaced by $f\!-\!\fI b$.
However, in the present situation, $\al$ (denoted by~$\la$ in~\cite{LR})
has no relation to~$\pi^*\om_V$.
Thus, the first equation in the second row of \cite[(54)]{H},
the third equation on the first line of \cite[(55)]{H}, and \cite[(56)]{H}
no longer apply, and
the long equation at the end of the proof can no longer be used to relate the period~$s$ 
(denoted by~$k$ in~\cite{LR} and by~$c$ in~\cite{H}) to the energy of~$u_V$.
The independence of~$s$ of the subsequence~$r_i$ follows from the fact that 
$u_{SV}(r\ne^{\fI \th})$ is contained in a tubular neighborhood of 
$S_{u_V(0)}V\!\approx\!S^1$ for all~$r$ sufficiently small
and thus the homology class of~\eref{tivS1_e} is independent of~$s$.

\begin{rmk}\label{LRthm37_rmk}
A completely different approach to the independence of~$\ga$ and~$s$ of 
the subsequence in the statement of Lemma~\ref{LRconv_lmm2}
appears in the proof of \cite[Theorem~3.7]{LR}.
However, the argument in~\cite{LR} is incorrect (or at least incomplete).
In particular, it {\it presupposes} that there exist $r_0\!\in\!\R^+$ and 
a periodic orbit $\ga\!:S^1\!\lra\!SV$  such that 
the images of the maps~\eref{tivS1_e} are contained in a small
neighborhood~$\cO_{\ga,\ep}$ of $\ga$ for all $r\!<\!r_0$; see the top of 
page~175 in~\cite{LR}.
Without this assumption, the key action functional $\cA\!=\!\cA_{\ga}$ is not even defined
in~\cite{LR}.
Most of the remainder of this argument is dedicated to using this~$\cA$ to 
show that such~$\cO_{\ga,\ep}$ can be chosen arbitrarily small, but 
it was arbitrarily small to begin with.
It is actually possible to define~$\cA$ on a neighborhood of the entire space~$\cO_s$
of periodic orbits of period $s\!\in\!\Z$, but 
this cannot be used to show that~$s$ in Lemma~\ref{LRconv_lmm2} is 
independent of the subsequence (as attempted in~\cite{LR}).
The proof of \cite[Theorem~3.7]{LR} also makes use of \cite[Proposition~3.4]{LR};
the proof of the latter is based on an infinite-dimensional version of the Morse lemma,
for which no justification or citation is provided.
The desired conclusion of this Morse lemma involves the inner-product \cite[(3.14)]{LR}
with respect to which the domain $W^1_r(S^1,SV)$ is not even complete.  
The second equality in~\cite[(3.25]{LR} does not appear obvious either.
\end{rmk}

\begin{prp}\label{LRconv_prp}
Let $u\!=\!(u_{\R},u_{SV})\!:\bD^*\!\lra\!\R\!\times\!SV$ be a $\oJ$-holomorphic map.
If $E_{\ell_1,\ell_2}(u)$ is finite, then  there exist $s\!\in\!\Z$,
a 1-periodic orbit $\ga\!:S^1\!\lra\!SV$ of the Hamiltonian~$H$, $r_0\!\in\!\R$,
and $C_u\!\in\!\R^+$  such~that 
\begin{alignat}{2}
\label{LRconv_e1}
\big|u_{\R}\big(\ne^{t+\fI\th}\big)-(st\!+\!r_0)\big|,
d_{SV}\big(u_{SV}(\ne^{t+\fI\th}),\ga(\ne^{\fI s\th})\big)&\le C_u\ne^t\,
&\quad&\forall~(t,\th)\!\in\!(-\i,-1)\!\times\!S^1\,,\\
\label{LRconv_e2}
\big|\nd u_{\R}\big(\ne^{t+\fI\th}\big)-s\,\nd t\big|,
d_{SV}\big(\nd u_{SV}(\ne^{t+\fI\th}),\nd\ga(\ne^{\fI s\th})\big)
&\le C_u\ne^t 
&\quad&\forall~(t,\th)\!\in\!(-\i,-1)\!\times\!S^1\,.
\end{alignat}
Furthermore, the function $u_{\R}$ is bounded if and only if $s\!=\!0$,
and $u_{\R}(r\ne^{\fI\th})\!\lra\!\mp\i$ as $r\!\lra\!0$ if and only if $s\!\in\!\Z^{\pm}$.
\end{prp}

\noindent
This proposition corrects, refines, and generalizes the statement of \cite[Theorem~3.7]{LR},
the main conclusion of \cite[Section~3.1]{LR}.
The contrast of the second bound in~\eref{LRconv_e1}
with the first statement of Lemma~\ref{LRconv_lmm2} is that~$\th_0$
is now independent of the choice of the sequence.
The convergence property for $\pi\!\circ\!u_{SV}$ is standard;
see \cite[Lemmas~4.3.1,4.7.3]{MS2}.
Along with \cite[(3.33),(3.34)]{LR} and the ellipticity of the $\dbar$-operator,
this implies the convergence statements for the vertical direction;
see \cite[Lemma~4.1]{HWZ1}.
The convergence estimates~\eref{LRconv_e1} and~\eref{LRconv_e2},
formulated in the cylindrical metric on the target,
are analogous to the estimates in \cite[Lemma~5.1]{IPsum}
and on $\hat{x}_n,\hat{y}_n$ in the proof of \cite[Lemma~5.3]{IPsum}.\\

\noindent
Proposition~\ref{LRconv_prp} is needed for the convergence arguments of 
\cite[Section~3.2]{LR}; 
\cite[Theorem~3.7]{LR}, which is a similar statement with $\bD^*$ replaced by~$\C$, 
does not suffice for these purposes.
The topological reasoning in the paragraph above Remark~\ref{LRthm37_rmk}
also implies that the ends of the components of broken limits of $J$-holomorphic maps
have matching orders, as described by~\eref{cPcond_e}
and the last bullet above Remark~\ref{LR33b_rmk}.
The proof of this statement in \cite[Lemma~3.11(3)]{LR} is incorrect,
as explained in Remark~\ref{LR33c_rmk}.

\begin{rmk}\label{LR3_rmk}
For \cite[(3.18),(3.20),(3.22)]{LR} to hold, the sign in the definition of the operator~$S$
above \cite[(3.18)]{LR} should be reversed.
The symmetry of \cite[(3.18)]{LR} in~$\ze$ and~$\eta$ is not obvious.
It follows from 
$$\blr{\na_vX_H,w}=\frac12\vp(v,w), \quad
\blr{(\na_{X_H}J)v,w}=\frac12\big(\vp(v,Jw)+\vp(Jv,w)\big)
\qquad\forall~v,w\in\ker\,\la\,,$$
where $\nd\la\!=\!\pi^*\vp$.
For the statement of \cite[Proposition~3.4]{LR} to make sense,
it needs to be shown that~$\cA$ is well-defined on~$\cO$.
Equation~(3.22) should read
$$\|\nd_{\ga}\cA\|_{L^2(S^1)}\ge C|\cA(\ga)|^{\frac12}\qquad \forall\,\ga\!\in\!\cO\,,$$
and $\|\nd_{\ga}\cA\|_{L^2(S^1)}$ needs to be defined.
The second displayed equation in the proof of this proposition should read
$$\|\nd_{\ga}\cA\|_{L^2(S^1)}\ge
\big\|\nd_y\cA(P(x)y/(P(x)y,P(x)y)^{1/2})\big\|_{L^2(S^1)}\,.$$ 
The statement after the proof of \cite[Theorem~3.7]{LR} does not make sense,
because the constants there depend on the map $\C\!\lra\!\R\!\times\!SV$.  
Other, fairly minor misstatements in \cite[Section~3.1]{LR} include
\begin{enumerate}[label={},topsep=-5pt,itemsep=-5pt,leftmargin=*]
\item\textsf{p172, lines 6,-2:} dfn of $T_{\ga}^{\perp}$, $T_{x_k}^{\perp}$ 
should involve pointwise inner-products;
\item\textsf{p172, (3.18):} $\Pi$ not necessary by the previous equation;  
\item\textsf{p172, above Rmk 3.1:} accumulate {\it only} at;
\item\textsf{p172, Rmk 3.1} is meaningless, since (3.18) is derived for
any $\ga$ in a fiber of~$\pi$;
\item\textsf{p173, Rmk 3.3} is irrelevant and debatable;
\item\textsf{p173, Prop~3.4:} $x\!\in\!S_k$; 
\item\textsf{p173, line -10:} no need to introduce $P'$;
\item\textsf{p175, (3.29):}  $\ti{d}(\ti{u}(s,t),\ti{u}(s_i,t))$ \to 
$\ti{d}(\pi(\ti{u}(s,t)),\pi(\ti{u}(s_i,t)))$;
\item\textsf{p175, line -6:} Lemma (3.6) \to Lemma 3.6;
\item\textsf{p176, line 1:} defined just above;
\item\textsf{p177, (3.39)-(3.41)} do not make sense, given the definition of~$r$.\\
\end{enumerate}
\end{rmk}

\noindent
The use of the $\sup$-energy~\eref{LRener_e} introduced in~\cite{H} is not necessary
in the setting of~\cite{LR}.
It can be replaced by the energy with respect to the restriction to $\R\!\times\!SV$
of the symplectic form on
$$\P\big((SV\!\times_{S^1}\!\C)\times\C)\approx\P_XV$$
given~by
$$\wh\om_{\ep}=\pi_V^*\om- \ep\,\nd\bigg(\frac{\al}{1\!+\!\rho^2}\bigg),
\qquad\hbox{where}\quad
\qquad \rho\big(\big[[x,c_1],c_2\big]\big)=\big|c_1/c_2\big|^2\,,$$
with $\ep\!>\!0$ small
(if $\ep$ is not sufficiently small, $\widehat\om_{\ep}$ may be degenerate).
If the target is~$\oXV$ or~$\oYV$, instead of $\R\!\times\!SV$,
the restrictions of the symplectic forms~$\om_X$ and~$\om_Y$ can be used.
This is also related to the reason why the sup energy of the maps appearing
in~\cite{LR} is finite (for which no explanation is provided).\\

\noindent
The convergence topology arising from \cite[Section~3.2]{LR} involves 
pulling the domains of the stable maps apart via long cylinders
on which an $\widehat\om_{\ep}$-type energy disappears.
Along with~\eref{LRconv_e1} and~\eref{LRconv_e2}, this leads to analogues 
of~\eref{mula_e} and~\eref{NdExp_e2}.
The gluing construction on the domains in~\cite{LR} is the same as on the target
in~\eref{cZlaLR_e}
and is parametrized by pairs $(r,\th)\!\in\!\R^+\!\times\!S^1$ at each node
with $r\!\lra\!\i$ with $\mu\!=\!\ne^{-r-\fI\th}$.
In the notation around Remark~\ref{IPsumSec3_rmk}, if
$$x\big(u(\ne^{t+\fI\th'})\big)\approx \ne^{t-r_X+\fI(\th'-\th_X)}
\quad\hbox{as}~~t\lra-\i,\quad
y\big(u(\ne^{-t+\fI\th'})\big)\approx \ne^{-(t-r_Y)+\fI(\th'-\th_Y)}
\quad\hbox{as}~~t\lra\i,$$
then the relation between the gluing parameters for the target~$(a_k,\vt_k)$ 
in~\eref{cZlaLR_e} and the domains of the converging maps is described~by
\BE{GlueCOnd_e}\lim_{k\lra\i}\big((a_k\!+\!\fI\vt_k)-s(r_k\!+\!\fI\th_k)\big)
=r_X+r_Y+\fI(\th_X\!+\!\th_Y)\in \C/2\pi\fI\Z.\EE
This is the analogue of~\eref{mula_e} in the setup of~\cite{LR}.\\

\noindent
In both approaches, it is necessary to consider sequences  $u_k\!:\Si\!\lra\!\cZ_{\la_k}$
that limit to maps $u\!:\Si'\!\lra\!\cZ_0$ with $\Si_V'\!\neq\!\eset$;
see the notation above Remark~\ref{IPsumSec3_rmk}.
Such limits are considered briefly at the top of page~1003 in~\cite{IPsum},
with an incorrect conclusion;
see Section~\ref{Smat_subs} for more details.
On the other hand, the approach of \cite[Section~3.2]{LR} can be corrected 
to show that any such limit lies in a moduli space 
$\ov\M_{g,k}(X\!\cup_V\!Y,A)$ defined in Section~\ref{RelInv_sub0b},
whenever the almost complex structures~$J_{\la}$ satisfy the more restrictive conditions
of~\cite{LR}.
The condition~\eref{GlueCOnd_e} then extends as a relation between smoothing parameters
for the target and the domain at each transition between different levels of 
the target space; see Section~\ref{preglue_sec}.

\subsection{Pregluing: \cite[Section~6]{IPsum}, \cite[Section~4.2]{LR}}
\label{preglue_sec}

\noindent
The pregluing steps of gluing constructions typically involve constructing
approximately $J$-holomorphic maps and defining Sobolev spaces
suitable for studying their deformations.
The former is done in essentially the same way in~\cite{IPsum} and~\cite{LR};
the latter is done very differently.\\

\noindent
For $A\!\in\!H_2(X;\Z)$, $\chi\!\in\!\Z$, $k,\ell\!\in\!\Z^{\ge0}$, and  
a tuple $\bs\!=\!(s_1,\ldots,s_{\ell})\!\in\!(\Z^+)^{\ell}$ 
 satisfying~\eref{bsumcond_e}, let
$$\M_{\chi,k;\bs}^{V*}(X,A)\subset \wt\M_{\chi,k;\bs}^V(X,A)$$
denote the subspace of morphisms  with all components contained in~$X$.
For each $i\!=\!1,\ldots,\ell$, let 
$$L_i\lra \wt\M_{\chi,k;\bs}^V(X,A)$$
be the universal tangent line bundle at the $i$-th relative marked point 
(i.e.~$(k\!+\!i)$-th marked point overall).
By~\eref{FHS_e}, every marked map representing an element of 
$\M_{\chi,k;\bs}^{V*}(X,A)$
has a well-defined $s_i$-th derivative in the normal direction to~$V$
at the $i$-th relative marked point.
By~\eref{NdExp_e}, these derivatives induce a nowhere zero section
of the line bundle
$$ L_i^{*\otimes s_i}\otimes \ev_i^{V*}\cN_XV \lra \M_{\chi,k;\bs}^{V*}(X,A)\,,$$
which we denote by~$\fD_X^{(s_i)}$.\\

\noindent
If \hbox{$u\!:\Si'\!\lra\!\cZ_0$} is the limit of a sequence 
of $(J_{\cZ},\fj)$-holomorphic maps $u_k\!:\Si\!\lra\!\cZ_{\la_k}$, with $\la_k\!\in\!\De^*$,
and has no component mapped into~$V$, then $u$ determines an element~of
\begin{equation*}\begin{split}
\M_{\chi,k;\bs}^{V*}(\cZ_0,C) \equiv
\bigsqcup_{\begin{subarray}{c}\chi_X+\chi_Y=\chi\\ k_X+k_Y=k\end{subarray}} 
\bigsqcup_{A_X\#_VA_Y=C} \hspace{-.2in}
 \big\{(u_X,u_Y)\!\in\!
\M_{\chi_X,k_X;\bs}^{V*}(X,A_X)\!\times\!\M_{\chi_Y,k_Y;\bs}^{V*}(Y,A_Y)\!:\,&\\
\ev^V(u_X)\!=\!\ev^V(u_Y)&\big\}
\end{split}\end{equation*}
for some $C\!\in\!H_2(X\!\#_V\!Y;\Z)/\cR_{X,Y}^V$, 
$\bs\!=\!(s_1,\ldots,s_{\ell})$, and $\ell\!\in\!\Z^{\ge0}$.
Denote by
$$\pi_X,\pi_Y\!: \M_{\chi,k;\bs}^{V*}(\cZ_0,C) \lra
\bigsqcup_{\chi_X,k_X,A_X}\!\!\!\!\!\!\!\M_{\chi_X,k_X;\bs}^{V*}(X,A_X)\,,
\bigsqcup_{\chi_Y,k_Y,A_Y}\!\!\!\!\!\!\!\M_{\chi_Y,k_Y;\bs}^{V*}(X,A_Y)$$
the projection maps.
In \cite[Sections~6-9]{IPsum}, a gluing construction is carried out on 
the $\lr\bs$-fold cover 
\BE{wtMgluedfn_e}\wt\M_{\chi,k;\bs}^{V*}(\cZ_0,C)_{\la}
\equiv\big\{(\mu_{X;i}\!\otimes\!\mu_{Y;i})_{i=1,\ldots,\ell}
\in\bigoplus_{i=1}^{\ell}\pi_X^*L_i\!\otimes\!\pi_Y^*L_i\!:~
\fD_X^{(s_i)}\mu_{X;i}^{\otimes s_i}\otimes 
           \fD_Y^{(s_i)}\mu_{Y;i}^{\otimes s_i}=\la~\forall\,i\big\}\EE
of $\M_{\chi,k;\bs}^{V*}(\cZ_0,C)$, 
with the last equality viewed via the identification~\eref{cNpair_e}.
This cover accounts for the convergence property~\eref{mula_e}.\\ 

\noindent
Fix  a smooth map $\be\!:\R\!\lra\![0,1]$ so that 
$$\be(r)=\begin{cases}1,&\hbox{if}~r\!\le\!1;\\
0,&\hbox{if}~r\!\ge\!2.
\end{cases}$$
For each $\ep\!>\!0$, let $\be_{\ep}(r)\!=\!\be(\ep^{-1}r)$. 
Denote by $\na$ the Levi-Civita connection of the metric 
\hbox{$g_{\cZ}\!=\!\om_{\cZ}(\cdot,J_{\cZ}\cdot)$}
and by~$\na^{\C}$ the associated $J_{\cZ}$-linear connection.
Using the $\na$-geodesics from~$q$, we identify the ball of injectivity 
radius of~$g_{\cZ}|_V$ in~$T_qV$ with a neighborhood~$W_q$ of~$q$ in~$V$. 
Using the parallel transport with respect to~$\na^{\C}$ along the 
$\na$-geodesics from~$q$, we identify $\cN_XV$ and~$\cN_YV$ 
with $W_q\!\times\!\cN_XV|_q$ and $W_q\!\times\!\cN_YV|_q$, respectively.
The proof of \cite[Theorem~2.2]{FHS} then ensures that the map~$\Phi$
in~\eref{FHS_e} can be chosen to depend smoothly on~$u$.\\

\noindent
For $\mu\!\in\!\wt\M_{\chi,k;\bs}^{V*}(\cZ_0,C)_{\la}$, an approximately
$(J_{\cZ},\nu)$-holomorphic map $u_{\mu}\!:\Si_{\mu}\!\lra\!\cZ_{\la}$ 
can be constructed as follows.
Given an element $([u_X,u_Y])$ of $\M_{\chi,k;\bs}^{V*}(\cZ_0,C)$, with 
$u_X\!:\Si_X\!\lra\!X$ and $u_Y\!:\Si_Y\!\lra\!Y$, 
denote by~$\Si_0$ the Riemann surface obtained by identifying 
the $i$-th relative marked point $z_i\!\in\!\Si_X$ with
the $i$-th relative marked point $w_i\!\in\!\Si_Y$ for all $i\!=\!1,\ldots,\ell$
and by $\Si_0^*\!\subset\!\Si_0$ the complement of the nodes.
Define
$$u_0\!:\Si_0\lra \cZ_0 \qquad\hbox{by}\quad
u_0(z)=\begin{cases}u_X(z),&\hbox{if}~z\!\in\!\Si_X;\\
u_Y(z),&\hbox{if}~z\!\in\!\Si_Y\,.
\end{cases}$$
Given $i\!=\!1,\ldots,\ell$,  let~$z$ and~$w$ be coordinates on 
$\Si_{X;i}\!\subset\!\Si_X$ and~$\Si_{Y;i}\!\subset\!\Si_Y$
centered at~$z_i$ and~$w_i$, respectively, and covering the unit ball in~$\C$.
For each sufficiently small $\mu\!\equiv\!(\mu_i)_{i=1,\ldots,\ell}$ in~$\C^{\ell}$, 
we define 
\begin{gather*}
\Si_{\mu;i}\equiv\big\{(z,w)\!\in\!\C^2\!:\,zw\!=\!\mu_i,~|z|,|w|\!<\!1\big\} 
\quad\forall~i=1,\ldots,\ell, \\
\Si_0^*(\mu)=\Si_0-\bigcup_{i=1}^{\ell}\big(
\{z_i\!\in\!\Si_X\!:\,|z|\!\le\!|\mu_i|^{\frac12}\}
\cup\{w_i\!\in\!\Si_Y\!:\,|w|\!\le\!|\mu_i|^{\frac12}\}\big),\\
\Si_{\mu}=\bigg(\Si_0^*(\mu)\sqcup \bigsqcup_{i=1}^{\ell}\Si_{\mu;i}\bigg)\!\!\Big/\!\!\sim\,
\quad
(z,w)\sim\begin{cases}
z\in\Si_X,&\hbox{if}~|z|>|w|;\\
w\in\Si_Y,&\hbox{if}~|z|<|w|;
\end{cases}
~~\begin{aligned}
\forall~&(z,w)\in\Si_{\mu;i},\\
&i=1,\ldots,\ell.
\end{aligned}
\end{gather*}
For each $i\!=\!1,\ldots,\ell$ and $\ep\!>\!0$, we also define
\begin{gather*}
\vr_{\mu;i},\be_{\mu;i}\!:\Si_{\mu;i}\lra\R \quad\hbox{by}\quad
\vr_{\mu;i}(z,w)=\sqrt{|z|^2\!+\!|w|^2}, \quad
\be_{\mu;i}(z,w)=\be_{|\mu_i|^{\frac14}}\big(\vr_{\mu;i}(z,w)\big)\,;\\
\Si_{\mu;i}(\ep)=\big\{(z,w)\!\in\!\Si_{\mu;i}\!:\,
\vr_{\mu;i}(z,w)<\ep\big\}.
\end{gather*}
Let $\ep\!>\!0$ be such that the restrictions of $u_X$ and $u_Y$ to
$$\Si_{X;i}(\ep)\equiv\big\{z\!\in\!\Si_{X;i}\!:\,|z|\!<\!\ep\big\}
\qquad\hbox{and}\qquad
\Si_{Y;i}(\ep)\equiv\big\{w\!\in\!\Si_{Y;i}\!:\,|w|\!<\!\ep\big\}$$
respectively, satisfy~\eref{FHS_e} for some $\Phi\!=\!\Phi_{X;i},\Phi_{Y;i}$.
In particular, $u_X(\Si_{X;i}(\ep))$ and $u_Y(\Si_{Y;i}(\ep))$ are contained
in the open subset~$\cZ_V$ of~$\cZ$ defined in~\eref{cZdfn_e}
and in the total spaces of~$\cN_XV$ and~$\cN_YV$
over the geodesics ball~$W_{q_i}$, where $q_i\!=\!u_X(z_i)\!=\!u_Y(w_i)$.
Thus, there exist smooth functions
\begin{gather*}
u_{X;i}\!:\Si_{X;i}(\ep)\lra T_{q_i}V \quad\hbox{and}\quad 
u_{Y;i}\!:\Si_{Y;i}(\ep)\lra T_{q_i}V \qquad\hbox{s.t.}\\
u_X(z)=\big(u_{X;i}(z),\Phi_{X;i}(z)z^{s_i}\big) ~~~\forall\,z\!\in\!\Si_{X;i}(\ep)\,,
\quad 
u_Y(w)=\big(u_{Y;i}(w),\Phi_{Y;i}(w)w^{s_i}\big)  ~~~\forall\,w\!\in\!\Si_{Y;i}(\ep),
\end{gather*}
under the identifications of the previous paragraph.\\

\noindent
For any $\mu\!\in\!\C^{\ell}$ sufficiently small, let  
\begin{alignat*}{2}
\Phi_{\mu;X;i}\!:\Si_{\mu;i}(\ep)&\lra\cN_XV|_{q_i}\,, &\quad
\Phi_{\mu;X;i}(z)&=\Phi_{X;i}(0)\bigg(\be_{\mu;i}(z,w)
+(1\!-\!\be_{\mu;i}(z,w))
\frac{\Phi_{X;i}(z)}{\Phi_{X;i}(0)}\bigg)z^{s_i}\,,\\
\Phi_{\mu;Y;i}\!:\Si_{\mu;i}(\ep)&\lra\cN_YV|_{q_i}\,, &\quad
\Phi_{\mu;Y;i}(z)&=\Phi_{Y;i}(0)\bigg(\be_{\mu;i}(z,w)
+(1\!-\!\be_{\mu;i}(z,w))
\frac{\Phi_{Y;i}(w)}{\Phi_{Y;i}(0)}\bigg)w^{s_i}\,.
\end{alignat*}
With $\la\!=\!\mu^{s_i}\Phi_{X;i}(0)\Phi_{Y;i}(0)$,
we define $u_{\mu}\!:\Si_{\mu}\!\lra\!\cZ_{\la}$~by requiring that 
\BE{umudfn_e}u_{\mu}(z,w)=\begin{cases}
\big((1\!-\!\be_{\mu;i}(z,w))u_{X;i}(z),\Phi_{\mu;X;i}(z),\frac{\la}{\Phi_{\mu;X;i}(z)}\big),
&\hbox{if}~|z|\!\ge\!|w|;\\
\big((1\!-\!\be_{\mu;i}(z,w))u_{Y;i}(w),\frac{\la}{\Phi_{\mu;Y;i}(w)},\Phi_{\mu;Y;i}(w)\big),
&\hbox{if}~|z|\!\le\!|w|;
\end{cases}\EE
for all $(z,w)\!\in\!\Si_{\mu;i}(\ep)$ and $i\!=\!1,\ldots,\ell$
and extending as~$u$ over the complement of~$\Si_0(\ep/2)$ in~$\Si_0^*$.\\

\noindent
The relevant Sobolev norms for sections of $u_{\mu}^*T\cZ_{\la}$ and for $(0,1)$-forms  
with values in $u_{\mu}^*T\cZ_{\la}$ are defined by 
the $m\!=\!1$ case of \cite[(6.10)]{IPsum} and
the $m\!=\!0$ case of \cite[(6.11)]{IPsum}, respectively,
with $p\!>\!2$ in \cite[(6.9)]{IPsum}.
The failure of the map $u_{\mu}\!:\Si_{\mu}\!\lra\!\cZ_{\la}$ to be 
$(J_{\cZ},\nu)$-holomorphic is described~by
\BE{nonJnu_e} \big\|\{\dbar_{J_{\cZ}}\!-\!\nu\}(u_{\mu})\big\|_{\mu,0}
\le C|\mu|^{\frac16}\le C|\la|^{\frac{1}{6|\bs|}}\,,\EE
with $C$ independent of $\mu$, but depending continuously on the projection of~$\mu$
to $\M_{\chi,k;\bs}^{V*}(\cZ_0,C)$;
this can be deduced from the proof of \cite[Lemma~6.9]{IPsum}.

\begin{rmk}\label{IPsumSec6_rmk}
The pregluing construction done in the first half of \cite[Section~6]{IPsum}
is not needed for the purposes of \cite[Lemma~6.8(a)]{IPsum}, 
which is about properties of moduli spaces of maps into the singular fiber~$\cZ_0$.
Based on the proof, the wording of \cite[Lemma~6.8(a)]{IPsum} is incorrect:
{\it for every $(f,C)\!\in\!\cK_{\de}$} should appear after $\!\le\!\ep$
and again after $\!\ge\!c$
(so that $\rho_0$ in the first part and $c$ in the second part are independent
of~$(f,C)$); there is a similar problem with the wording of \cite[Lemma~6.8(d)]{IPsum}.
\cite[Lemma~6.8(a)]{IPsum} also has nothing to do with $c_i,c_i'$.
The proof of the first part of \cite[Lemma~6.8(b)]{IPsum} is not complete because
\cite[Lemma~5.1]{IPsum} is about finite cylinders, not wedges of disks.
The pregluing setup in \cite[Section~6]{IPsum} implicitly assumes that 
the domains of the nodal maps are stable, since it is based on \cite[Section~4]{IPsum}.
The stability assumption need not hold in general; it is not necessary though.
The domains can be stabilized as in \cite[Remark~1.1]{IPsum}, but not across 
an entire stratum of maps;
in particular, \cite[Observation~6.7]{IPsum} may not always apply.
The definition of the norms in \cite[Section~6]{IPsum} makes no mention that $p\!>\!2$,
which is necessary for the control of the $C^0$-norm.
The statement about uniform $C^0$-bound in \cite[Remark~6.6]{IPsum} needs a justification, 
since the domains~$C_{\mu}$ change (which is standard)
and the metric on the targets~$\cZ_{\la}$ collapses (which is not standard).
Without a local trivialization of the normal bundle, 
the formula \cite[(6.4)]{IPsum} does not make sense.
The crucial bound of \cite[Lemma~6.9]{IPsum} is incorrect.
Its proof neglects to consider the first two components of $F\!-\!f$ 
with respect to the decomposition in \cite[(6.14)]{IPsum};
contrary to the statement immediately after \cite[(6.14)]{IPsum},
these two components are not zero, as $f$ does not involve~$\be$.
However, the weaker bound of~\eref{nonJnu_e} suffices.
Other, fairly minor misstatements in \cite[Section~6]{IPsum} include
\begin{enumerate}[label={},topsep=-5pt,itemsep=-5pt,leftmargin=*]
\item\textsf{p968, above (6.1):} if $f_0$ is in the limit of a sequence, then (6.1) holds;
\item\textsf{p968, line -4:} $C$ \to $C_1$, with the notation as before;
\item\textsf{p969, line 2:} $\fL_k$ is used for $L_k$ in (4.3);
\item\textsf{p969, par.~above Dfn 6.2:} not {\it from} (4.2) and (5.4); {\it as in} (5.11);
\item\textsf{p970, 1st par.:} there are no (a) and (b) in (2.6) or (6.4);
\item\textsf{p970, after (6.5):} were \to where;
\item\textsf{p971, line 11:} $\sqrt{|\mu_k|}$ \to $\sqrt{2|\mu_k|}$;
\item\textsf{p971, above (6.8):} geodesics and parallel transport with respect to what connection?
\item\textsf{p971, below (6.8):}  average value zero only for the horizontal part~$\xi^V$;
{\it as in} (5.11);
\item\textsf{p971, (6.9):} $k\!=\!m$ below; only $k\!=\!1$ is used; 2 can be absorbed into~$\de$;
\item\textsf{p971, below (6.9):} there are no {\it coordinates} in (5.3); 
(4.5) is closer;
\item\textsf{p971, Dfn 6.5:} there is no triple in (6.8);
\item\textsf{p972, top:} $k$ \to $h$; not just {\it Finsler} metric;
\item\textsf{p972, Lemma 6.8(a), line 1:} dist$(f(A(\rho_0)),V)$ 
\to $\max_{z\in A(\rho_0)}$dist$(f(z),V)$;
\item\textsf{p973, lines 3,8:} $p_n\!\in\!C_n$; $p_n\!\in\!C_n\setminus A(\rho_0)$;
\item\textsf{p974, (6.12):}  $|\nu_F\!-\!\nu_f|$ should not be multiplied by $|dF|$, 
similarly to~(6.15);
\item\textsf{p974, (6.14):} $\be$ \to $\be_{\mu}$;
\item\textsf{p974, below (6.14):} $(J_F\!-\!J_f)\!\circ\!dF$ \to $(J_F\!-\!J_f)\!\circ\!df$.\\
\end{enumerate}
\end{rmk}

\noindent
The approximately $J$-holomorphic map~$u_{\mu}$ in~\eref{umudfn_e}
is constructed in the same way at the bottom of page~192 in~\cite{LR}.
Because of the regular nature of the almost complex structures~$J_X$ and~$J_Y$ 
used in~\cite{LR} on neighborhoods of~$V$ in~$X$ and~$Y$,
the gluing approach of~\cite{LR} extends to maps into $\oXYm$ with $m\!\ge\!1$.
As  explained at the end of this section, 
the gluing of the target spaces in~\eref{cZlaLR_e} extends to~$\oXYm$ as~well.
This extension is parametrized by the tuples 
\BE{ufauvt_e}(\ufa,\uvt)\equiv 
\big(\fa_0,\ldots,\fa_m,\vt_0,\ldots,\vt_m\big)
\in (\R^+)^{m+1}\times (\R/2\pi\Z)^{m+1}\EE
so that $\cZ_{\ufa,\uvt}\!=\!\cZ_{|\ufa|,|\uvt|}$ as far as the almost complex structures
are concerned, where 
$$|\ufa|=\fa_0+\ldots+\fa_m\,, \qquad
|\uvt|=\vt_0,\ldots,\vt_m\,.$$
In the next paragraph, we define the space of gluing parameters,
generalizing~\eref{wtMgluedfn_e} from the $m\!=\!0$ case.\\

\noindent
Given $m\!\in\!\Z^{\ge0}$, let $\C_{m+1}$ denote the quotient of~$\C^{m+1}$ 
by the $(\C^*)^m$-action
$$(c_1,\ldots,c_m)\cdot(\la_0,\ldots,\la_m)
=\big(c_1^{-1}\la_0,c_1c_2^{-1}\la_1,\ldots,c_{m-1}c_m^{-1}\la_{m-1},c_m\la_m\big).$$
The map $(\la_0,\ldots,\la_m)\!\lra\!\la_0\!\ldots\!\la_m$ then descends to~$\C_{m+1}$.
For each $\la\!\in\!\C$, let $\C_{m+1;\la}\!\subset\!\C_{m+1}$ be preimage of~$\la$.
Let $u\!:\Si\!\lra\!\XYm$ be a representative of an element of $\ov\M_{g,k}(X\!\cup_V\!Y,A)$ 
and $i\!=\!1,\ldots,\ell$ be an index set for its nodes on the divisors 
$$V\subset X,Y  \qquad\hbox{and}\qquad 
\{r\}\!\times\!\P_{X,0}V,\{r\}\!\times\!\P_{X,\i}V\subset\{r\}\times\P_XV\,.$$
For each such $i$, let $|i|\!=\!0$ if the node lies on $V\!\subset\!X$
and $|i|\!=\!r$ if it lies on $\{r\}\!\times\!\P_{X,0}V$.
Denote by $s_i\!\in\!\Z^+$ the order of contact with the divisor 
of either of the two branches at the $i$-th node,
by $L_{u;i}$ the line of smoothings of this node 
(denoted by $\pi_X^*L_i\!\otimes\!\pi_Y^*L_i$ in~\eref{wtMgluedfn_e}),
and by $\fD_i^{(s)}\!\in\!L_{u;i}^*$ the $s_i$-th derivative
(denoted by $\fD_X^{(s_i)}\!\otimes\!\fD_Y^{(s_i)}$ in~\eref{wtMgluedfn_e}).
The admissible relative smoothing parameters at~$u$ for maps to~$\cZ_{\la}$ 
are the elements of the space
\begin{equation*}\begin{split}
L_{u;\la}=\big\{(\mu_i)_{i=1,\ldots,\ell}\!\in\!
\bigoplus_{i=1}^{\ell}L_{u;i}\!:~&
\exists[\la_0,\ldots,\la_m]\!\in\!\C_{m+1;\la}
~\hbox{s.t.}~\fD_i^{(s_i)}(\mu_i)=\la_{|i|}~\forall\,i\big\}.
\end{split}\end{equation*}
While $\fD_i^{(s_i)}$ depends on the choice of representative~$u$ for 
$[u]\!\in\!\ov\M_{g,k}(X\!\cup_V\!Y,A)$,
$L_{u;\la}$ is determined by~$[u]$ and the choice of ordering of the relative nodes of~$u$,
since the action of~$(\C^*)^m$ on~$\C^{m+1}$ defined above corresponds to the action
of~$(\C^*)^m$ on~$X\!\cup_V^mY$.\\

\noindent
We now define the spaces $\cZ_{\ufa,\uvt}$, with $(\ufa,\uvt)$ as in~\eref{ufauvt_e}
and $\fa_0$ and $\fa_m$ sufficiently large, and identify them with 
$\cZ_{|\ufa|,|\uvt|}$; see~\eref{cZlaLR_e}.
For each $r\!=\!1,\ldots,m$, let
\begin{alignat*}{2}
|\fa|_r^-&=\fa_0\!+\!\ldots\!+\!\fa_{r-1}, &\qquad 
|\fa|_r^+&=\fa_r\!+\!\ldots\!+\!\fa_m\,,\\
|\vt|_r^-&=\vt_0\!+\!\ldots\!+\!\vt_{r-1}, &\qquad 
|\vt|_r^+&=\vt_r\!+\!\ldots\!+\!\vt_m\,.
\end{alignat*}
We assume that $m\!\in\!\Z^+$.
Let
$$\cZ_{\ufa,\uvt}=\bigg(X_{\fa_0}\sqcup
\bigsqcup_{r=1}^m\{r\}\!\times\![-\frac34\fa_r,\frac34\fa_{r-1}]\!\times\!SV
\sqcup Y_{\fa_m}\bigg)\Big/\sim\,,$$
with the equivalence relation defined by
\begin{alignat*}{2}
(1,a,x)&\sim \big(a\!-\!\fa_0,\ne^{-\fI\vt_0}x\big)\subset X_{\fa_0}
&\qquad&\forall~4a\in(\fa_0,3\fa_0),\\
(r,a,x)&\sim \big(r\!+\!1,a\!+\!\fa_r,\ne^{\fI\vt_r}x)
&\qquad&\forall~4a\in(-\fa_r,-3\fa_r),~r\!=\!1,\ldots,m\!-\!1,\\
(m,a,x)&\sim \big(a\!+\!\fa_m,\ne^{\fI\vt_m}x)\subset Y_{\fa_m}
&\qquad&\forall~4a\in(-\fa_m,-3\fa_m)\,.
\end{alignat*}
These identifications respect the almost complex structure $\oJ$
and thus induce an almost complex structure on~$\cZ_{\ufa,\uvt}$.
The bijection $\cZ_{\ufa,\uvt}\!\lra\!\cZ_{|\ufa|,|\uvt|}$ given~by
\begin{gather*}
x\lra \begin{cases}
x\!\in\!X_{|\fa|},&\hbox{if}~x\!\in\!X_{\fa_0};\\
x\!\in\!Y_{|\fa|},&\hbox{if}~x\!\in\!Y_{\fa_m}; 
\end{cases} \quad
(r,a,x)\lra \begin{cases}
(a\!-\!|\fa|_r^-,\ne^{-\fI|\vt|_r^-}x)\in X_{|\fa|},
&\hbox{if}~4a\ge|\fa|_r^-\!-\!3|\fa|_r^+;\\
(a\!+\!|\fa|_r^+,\ne^{\fI|\vt|_r^+}x)\in Y_{|\fa|},
&\hbox{if}~4a\le3|\fa|_r^-\!-\!|\fa|_r^+;
\end{cases} 
\end{gather*}
is well-defined on the overlaps and identifies the two spaces with 
their almost complex structures, as needed for the general gluing construction.
However, the just described construction and identification do not fit 
with the more general almost complex structures of~\cite{IPsum},
as they are not regularized on neighborhoods of~$V$ in~$X$ and~$Y$. 

\begin{rmk}\label{LRglue_rmk}
The only gluing constructions described in \cite{LR} involve smoothing a single node.
In particular, there is no mention of the above identification
$\cZ_{\ufa,\uvt}\!=\!\cZ_{|\ufa|,|\uvt|}$, which is needed to make sense of
the target of the smoothed out maps, or of 
the space~$L_{u;\la}$ of admissible smoothings.
\end{rmk}

\subsection{Uniform estimates: \cite[Sections 7,8]{IPsum}, \cite[Section 4.2]{LR}}
\label{IPunif_subs}

\noindent
Gluing constructions in GW-theory typically require 
defining linearizations~${\bf D}_{u_{\mu}}$ of the $\dbar$-operator 
at the approximately $J$-holomorphic maps~$u_{\mu}$
(these are not unique away from $J$-holomorphic maps)
and establishing  uniform bounds on these linearizations and their right inverses.
Establishing the former is typically fairly straightforward,
with appropriate choices of the linearizations and the Sobolev norms
on their domains and targets.
Uniform bounds on the right inverses can be obtained either 
by bounding the eigenvalues of the Laplacians ${\bf D}_{u_{\mu}}{\bf D}_{u_{\mu}}^*$ 
from below, 
by a direct computation for explicit right inverses, 
or by establishing a uniform elliptic estimate on~${\bf D}_{u_{\mu}}$
with suitable Sobolev norms.
As stated at the beginning of \cite[Section~8]{IPsum}, 
such uniform Fredholm bounds are the key analytic step in the proof.
As we explain below, the argument in~\cite{IPsum} has several material, consecutive errors, 
i.e.~with each sufficient to break~it.\\

\noindent
The approach taken in \cite[Sections~7,8]{IPsum} is to bound 
the eigenvalues of the Laplacians ${\bf D}_{u_{\mu}}{\bf D}_{u_{\mu}}^*$ from below. 
With the definitions at the beginning of \cite[Section~7]{IPsum},
the index of~${\bf D}_u$ (denoted by~${\bf D}_f$ in~\cite{IPsum})
is generally larger than the index of~${\bf D}_{u_{\mu}}$ (denoted by ${\bf D}_F$),
as the former does not see the order of contact.
In particular, ${\bf D}_u$ does not fit into any kind of continuous Fredholm setup, though 
by itself this issue need not be material as far as 
the estimates on~${\bf D}_{u_{\mu}}$ are concerned.\\

\noindent
In the displayed expression above \cite[(7.5)]{IPsum},
$\lr{\ze_1,\ze_2}$ has two different meanings in the same equation.
This equation defines an inner-product only on the first part of 
the domain of ${\bf D}_{u_{\mu}}\!=\!{\bf D}_F$ and so does not define
${\bf D}_{u_{\mu}}^*$.
The explicit formula for the first component of~${\bf D}_F^*$ in \cite[(7.5)]{IPsum}
cannot be correct because it does not satisfy the average value condition 
on the elements of~$L_{1;\bs;0}$ for $F\!=\!u_{\mu}$ and even more conditions for $f\!=\!u$
(the average value condition is described above \cite[(7.1)]{IPsum}).
This formula has to be corrected by an element of the $L^2$-orthogonal complement 
of $L_{1;\bs,0}(u_{\mu}^*T\cZ_{\la})$ in $L_{1;\bs}(u_{\mu}^*T\cZ_{\la})$;
unfortunately, the orthogonal complement does not lie in $L_{1;\bs}(u_{\mu}^*T\cZ_{\la})$. 
Thus, \cite[Proposition~7.3]{IPsum} says {\it nothing} about the uniform boundness 
of ${\bf D}_F^*\!=\!{\bf D}_{u_{\mu}}^*$.
Without taking out the average, the norms of \cite[Definition~6.5]{IPsum}
would not be finite over~$f$, as used in~\cite{IPsum} to obtain uniform bounds over~$F$.

\begin{rmk}\label{IPsumSec7_rmk}
The crucial Sections~7 and~8 in~\cite{IPsum} are written in a confusing way
with the same notation used for different objects, including in the same
equation at times.
With the definition as in \cite[(1.11),(7.2)]{IPsum},
the image of the operator in \cite[(7.4)]{IPsum} would not be in
the $(0,1)$-forms because of the $F_*h$ term (which is not a $(0,1)$-form
if $F$ is not $J$-holomorphic; $F_*h$ needs to be replaced by $\prt F\!\circ\!h$).
Since $F$ is defined on a smooth domain, the operators in \cite[(7.4),(7.6)]{IPsum}
are Fredholm because they differ from real Cauchy-Riemann operators
by finite-dimensional pieces; 
uniform boundness in $\mu$ as in \cite[Proposition~7.3]{IPsum} is a separate issue. 
With a reasonable interpretation of the inner-product above \cite[(7.5)]{IPsum}, 
the last component of~${\bf D}_F^*$ in \cite[(7.5)]{IPsum} is missing~$\frac12$.
The expression for~$A\eta$ in \cite[(7.5)]{IPsum} cannot be correct either,
since it should produce a tuple indexed by the relative marked points, not a sum.
Furthermore, this expression should have more terms, 
as the proof of \cite[Proposition~7.3]{IPsum} suggests,
and should depend on the vertical part of~$\eta$ as well.
However, the exact forms of the second and last components of~${\bf D}_F^*$
do not matter as long as they are uniformly bounded;
this is the case because the restrictions of~${\bf D}_F$
to the second and last components in \cite[(7.4)]{IPsum} are uniformly bounded.
The bound on $\na\nu$ at the beginning of the proof of \cite[Proposition~7.3]{IPsum} is not obvious,
because $\na$ there denotes the Levi-Civita connection with respect to
the metric on~$\cZ_{\la}$, which degenerates as $\la\!\lra\!0$;
this bound depends on the requirement on the second fundamental form
in \cite[Definition~2.2]{IPsum}. 
Other, fairly minor misstatements in \cite[Section~7]{IPsum} include
\begin{enumerate}[label={},topsep=-5pt,itemsep=-5pt,leftmargin=*]
\item\textsf{p975, Section 7, line 2:} there are no Sobolev spaces in Definition 6.5;
\item\textsf{p975, line -5:}  Lemma~7.3 \to Proposition~7.3; same on p976, line~5;
\item\textsf{p976, 2nd paragraph:} there is nothing about generic $\de$ or Fredholm
in Proposition~7.3; there seems to be no connection with Lemma~3.4 at~all;
\item\textsf{p976, line -3:} no such verification in Lemma 3.4;
\item\textsf{p977, lines 1,2:} there is no stabilization in Observation 6.7;
\item\textsf{p977, lines 4,6:} $ev$ \to ev;
\item\textsf{p977, Lemma~7.2:}  $\ze$ should be a vector field along~$F$, not on a chart;
\item\textsf{p978, line 7:} $X$ already denotes a symplectic manifold; 
\item\textsf{p978, line -10:} $L$ \to $L_F$;
\item\textsf{p978, line -9:} no use of Lemma~7.2 in addition to (7.7);
\item\textsf{p978,  line -7:} with this description,  
$\wt\na$ and $\na$ are connections in different spaces;
\item\textsf{p979,  line 5:} there is no $h^v$ or $\ti{x}$ in Definition~6.4.
\end{enumerate}
\end{rmk}

\noindent
There is a {\it crucial} sign error in the proof of \cite[Proposition~8.2]{IPsum}:
the two terms on the second line of \cite[(8.7)]{IPsum},
a Gauss curvature equation written in a rather unusual way, should have 
the opposite signs;
see \cite[Theorem~13.38]{Lee}, which uses
the same (more standard) sign convention for the curvature tensor~$R$ 
(defined at the beginning of \cite[Section~13.2]{Lee}).
Thus, the minus sign in \cite[(8.8)]{IPsum} should be a plus, which destroys the argument.
Conceptually, it seems implausible to have a negative sign in  \cite[(8.8)]{IPsum},
because it should allow to make the right-hand side of  \cite[(8.6)]{IPsum} negative
by taking a local solution of $L_F^*$ and sending~$\mu$ and~$\la$ to~0.\\

\noindent
The proof of \cite[Lemma~8.4]{IPsum} is also incomplete.
At the very bottom of page~984 in~\cite{IPsum}, it is stated that 
${\bf D}_0^*\eta\!=\!{\bf D}_u^*\eta$ lies in the image of 
the map~${\bf D}_0^*$ in \cite[(7.6)]{IPsum}.
However, it had not been shown that the limiting $(0,1)$-form~$\eta$
lies in the domain of~${\bf D}_0^*$, which involves bounding the first derivative
over the entire domain.
The preceding argument shows that the $L^2_1$-norm of~$\eta$ outside of the nodes
of the domain is bounded, but that does not imply that the $L^2_1$-norm of~$\eta$
is bounded everywhere.
Furthermore, since the metrics on the targets~$\cZ_{\la}$ degenerate, 
a proof is needed to show that the elliptic estimate used in 
the proof of \cite[Lemma~8.5]{IPsum} is uniform;
it is not so clear that it~is.

\begin{rmk}\label{IPsumSec8_rmk}
The bound on $\na J$ on line~10 on page~981 of \cite{IPsum} is not obvious,
because $\na$ there denotes the Levi-Civita connection with respect to
the metric on~$\cZ_{\la}$, which degenerates as $\la\!\lra\!0$;
this bound depends on the requirement on the second fundamental form
in \cite[Definition~2.2]{IPsum}. 
Since the metric on the horizontal tangent space of $\cN_{\cZ}V$ varies 
in the normal direction (according to the bottom half on~p951), 
the formula for~$g_{\la}$ on line~-5 on page~982 cannot be precisely correct;
this has an effect on the formulas for Christoffel symbols on the last line on this page
(though this gets absorbed into the error term in the next sentence, which should include
$s_k$ in front of $\tanh$).
There is a similar issue with the statement concerning the independence of~$F^*g_{\la}$. 
Other, fairly minor misstatements in \cite[Section~8]{IPsum} include
\begin{enumerate}[label={},topsep=-5pt,itemsep=-5pt,leftmargin=*]
\item\textsf{p980, line 9:} (1.4) \to (1.5);
\item\textsf{p981, line 15:} $\om$ already denotes a symplectic form;
\item\textsf{p981, (8.6):} $-d(\rho^{\de})\!\wedge\!\om$ is part of the first integrand on RHS;
\item\textsf{p981, line -6:} this has nothing to do with the connection on the domain
(which is also not~flat);
\item\textsf{p981, line -5:} $V$ already denotes the symplectic divisor;
\item\textsf{p982, line 5:} $A_k$ as defined in the proof of Lemma~6.9
is a subset of~$C_{\mu}$, not of~$\cZ_{\la}$; 
\item\textsf{p982, line 7:} $\nu$ already denotes the key $(0,1)$-form; 
missing $\nu$ at the~end;
\item\textsf{p982, line 17:} first inequality does not hold because of $z^s$ in (6.14);
\item\textsf{p982, line 18:} there is no bound on $|\nu^N|$ in the sentence preceding~(6.17);
\item\textsf{p982, line 20:} $U\!-\!JV$ \to $V\!-\!JU$, twice;  
\item\textsf{p982, line 21:} no connection to the preceding statement;
\item\textsf{p982, line -6:} $\th$ \to $\Th$;
\item\textsf{p982, line -3:} $F_*\prt_{\th}$ also involves a $V$-component; 
\item\textsf{p983, lines 15,16:} {\it multiply} and {\it adding} do not help here; 
\item\textsf{p984, top:} $\de$ generic does not appear in this section again;
\item\textsf{p984, lines 13,14:} by definition of $\{F_n\}$, not
{\it Bubble Tree Convergence Theorem}; 
\item\textsf{p984, line 21:} $N=\{\rho\le\de\}$, and 
this $\de$ is different from the $\de$ in the norm;
\item\textsf{p984, bottom third:} $X$ already denotes a symplectic manifold; 
\item\textsf{p984, line -9:} $\be h$ is not in $T_{C_0}\M$.
\end{enumerate}
\end{rmk}

\noindent
In the approach of~\cite{LR}, the metrics on the targets do not collapse.
A family of uniformly bounded right inverses for the linearized 
operators~${\bf D}_{u_{\mu}}$ is constructed in the proof of \cite[Lemma~4.8]{LR}
directly via the approach of \cite[Section~10.5]{MS2}.
Conceptually, the existence of such inverses follows from uniform elliptic estimates
in the metrics of~\cite{LR} on the target;
see the proofs of \cite[Lemmas~3.9,3.10]{LT}.

\subsection{Gluing: \cite[Sections~9,10]{IPsum}, \cite[Sections~4.2,5]{LR}}
\label{glue_sec}

\noindent
The final step in gluing constructions involves showing that every 
approximately $J$-holomorphic map~$u_{\mu}$ can be perturbed to 
an actual $J$-holomorphic map, in a unique way subject to suitable restrictions,
and that every nearby $J$-holomorphic map can be obtained in such a way.
The last part is often established by showing that all nearby maps,
$J$-holomorphic or not, are of the form $\exp_{u_{\mu}}\xi$ with $\xi$ small.
The uniqueness part can be established by showing that each nearby map can 
be written uniquely in
the form $\exp_{u_{\mu}}\xi$, subject to suitable conditions on~$\xi$.
The nearby solutions of the $\dbar$-equations are then determined
by locally trivializing the bundle of~$(0,1)$-forms and expanding 
the $\dbar$-equation~as
\BE{dbarexp_e}\dbar \exp_{u_{\mu}}\xi=\dbar u_{\mu}+{\bf D}_{u_{\mu}}\xi+Q_{u_{\mu}}(\xi)\,,\EE
where ${\bf D}_{u_{\mu}}$ is the linearization of the $\dbar$-operator determined 
by the given trivialization and $Q_{u_{\mu}}(\xi)$ is the error term, quadratic in~$\xi$.
The equation~\eref{dbarexp_e} can be solved for all $\mu$ sufficiently small
if the norm of~$\dbar u_{\mu}$ tends to~$0$ as $\mu\!\lra\!0$,
${\bf D}_{u_{\mu}}$ admits a right inverse which is uniformly bounded in~$\mu$,
and the error term~$Q_{u_{\mu}}$ is also uniformly bounded in~$\mu$.\\

\noindent
The bijectivity of the gluing map is the subject  of \cite[Proposition~9.1]{IPsum},
though its wording is not quite correct.
Based on the proof and the usage, the intended wording is that there exist~$\ve_0,c\!>\!0$
such that the map~$\Phi_{\la}$ is a diffeomorphism as described whenever $\ve,|\la|\!<\!\ve_0$.
The proof of \cite[Proposition~9.1]{IPsum} is incorrect at the end of the injectivity argument, 
even ignoring the problems with the prerequisite statements:
even if $(f_n,C_{0,n},\mu_n)\!=\!(f_n',C_{0,n}',\mu_n')$,
$\eta_n$ and $\eta_n'$ need not lie in the injectivity radius of~$\Phi_{\la_n}$
for $n$ large, as this radius likely collapses as $n\!\lra\!\i$,
because the injectivity radius of the metric~$g_{\la}$ collapses as $\la\!\lra\!0$
and the norms are not scaled to address this.
In order to show that the injectivity radius of~$\Phi_{\la}$ does not collapse,
one needs to show that the vertical part of~$P_F\eta$ on suitable necks
is bounded by something like $|\la|^{\frac{1}{2}}\|P_F\eta\|$.
In light of~\eref{NdExp_e2}, 
this appears plausible for the nearby $J$-holomorphic maps, but less so for
arbitrary nearby maps.
It thus seems quite possible that the injectivity part of the intended  statement 
of \cite[Proposition~9.1]{IPsum} is not correct with the norms 
of \cite[Definition~6.5]{IPsum}, which impose a rather mild weight in the collapsing direction.\\

\noindent
The proof of \cite[Proposition~9.4]{IPsum} is incomplete,
as a justification is required for why the constant~$C$ 
in the bound \cite[(9.11)]{IPsum} on the quadratic error term in~\eref{dbarexp_e}
is uniform in~$\mu$.
This is not obvious in this case, since the metrics on~$\cZ_{\la}$ degenerate
and the constant~$C$ depends on the curvature of the metric;
see \cite[Section~3]{anal}.
Thus, this is also a significant issue in the approach of~\cite{IPsum}.

\begin{rmk}\label{IPsumSec9_rmk}
The proof of  \cite[Lemma~9.2]{IPsum} ignores the regions 
$|\mu_k|^{\frac14}\!\le\!\rho\!\le\!2|\mu_k|^{\frac14}$.
The statement of \cite[Proposition~9.3]{IPsum} is essentially correct, but 
the last part of its proof does not make sense.
For example, since $f_0$ is a map from a wedge of two disks and $f_n$ is a map from
a cylinder, $f_0\!-\!f_n$ is not defined.
Furthermore, the equations $F_n\!-\!f_n\!=\!(\hat\ze_n,\bar\xi_n)$,
$\hat\ze_n\!=\!\ze_n\!+\!(F_n\!-\!f_0)$, and $\ze_n\!=\!f_0\!-\!f_n$ are inconsistent.
Other, fairly minor misstatements in \cite[Section~9]{IPsum} 
and in the first part of \cite[Section~10]{IPsum} include
\begin{enumerate}[label={},topsep=-5pt,itemsep=-5pt,leftmargin=*]
\item\textsf{p986, above (9.2):} determined by \to related to;
\item\textsf{p986, below (9.2):} $\Phi_{\la}$ is defined everywhere and 
is the identity along the zero section;
\item\textsf{p986, (9.3):} it is only an isomorphism, since the first summand on RHS 
is  not a subspace of LHS;
\item\textsf{p986, below (9.3):} Lemma 5.3 is not needed here;
\item\textsf{p986, line -3:} the image of $F_0$ in $TZ_{\la}$ \to $F_0$;
\item\textsf{p986, line -1:} 
RHS describes only the vector field component of LHS and only for $\eta_0\!=\!0$;  
\item\textsf{p987, line 13:} $B$ is the two-dimensional manifold underlying $C_0$ and $C_0'$;
\item\textsf{p987, line -4:} there is no such extension in Section~4;
\item\textsf{p987, bottom:} $h_1$ is a variation of $\mu$, which is basically fixed;
\item\textsf{p988, lines 4,5:} not extended over $Z_{\la}$;
\item\textsf{p988, line 6:} $\xi_0$ has not been defined; 
\item\textsf{p988, (9.5):} second line is missing $\frac12$;
\item\textsf{p988, line 13:} $\rho^{-|s|}$, not  $\rho^{1-|s|}$, 
according to (6.15), which is still good enough;
\item\textsf{p988, after (9.6):} the estimates in the proof of Proposition 7.3;
\item\textsf{p988,  after (9.7):} there is no equation (6.4a);
\item\textsf{p989, line 9:} (9.8) \to (9.6);
\item\textsf{p989, line -3:} Lemma 5.4 does not say this;
\item\textsf{p990, (9.10)} holds only after some identifications;
\item\textsf{p991, below (10.2):} this sentence does not make sense;
\item\textsf{p991, (10.3):} since $s$ is fixed, there should be no $\bigsqcup$;
\item\textsf{p992, lines 3,4:} $\Phi_{\la}^1$  maps into $\M_s^{V,\de}(Z_{\la})$
according~(10.3);
\item\textsf{p992, lines 16,17:} this sentence makes no sense.
\end{enumerate}
\end{rmk}

\noindent
The correspondence between approximately $J$-holomorphic maps and 
actual $J$-holomorphic maps in~\cite{LR} is the subject of
Proposition~4.10.
The expansion~\eref{dbarexp_e} does not even appear in its proof,
with the Implicit Function Theorem applied in an infinite-dimensional setting
without any justification.
On the other hand, the above issues with the collapsing metric do not arise
in the setting of~\cite{LR}, and so the required uniform estimates 
are fairly straightforward to obtain.

\begin{rmk}\label{LR5_rmk}
The approach of \cite[Section~5]{LR} to the symplectic sum formula
involves the existence of a virtual fundamental class for 
$\ov\M_{g,k}(X\!\cup_V\!Y,A)$.
The justification for its existence consists of a few lines after \cite[Lemma~5.4]{LR},
which is far from even mentioning all the required issues.
The comparison of GW-invariants for $X\!\cup_V\!Y$ and $X\!\#_V\!Y$ 
in \cite[Section~5]{LR} again involves integration instead of pseudocycles 
(top of p208 and p209), and does not explain the key multiplicity 
factor~$k$ in \cite[Theorem~5.7]{LR}.
The top of page~209 again suggests an isomorphism between an even and odd-dimensional manifolds. 
The index formula \cite[(5.1)]{LR} cannot possibly follow from the proof 
of \cite[Lemma~4.9]{LR},
as the latter has no numerical expressions for the index.
Since this index also depends on~$\al$ (according to \cite[Remark~4.1]{LR}),
how can there be a natural correspondence between the domains and targets
of the operators~$D_u$ and~$D_{\bar{u}}$ in \cite[Remark~5.2]{LR}?
With the definitions in \cite[Section~4]{LR}, the dimension of $\ker L_{\i}$
is~$2n$, not $2n\!+\!2$, as stated after \cite[(5.1)]{LR}.
Mayer-Vietoris has nothing to do with a pseudoholomorphic map defining
a homology class at the bottom of page~206 in~\cite{LR}.
\cite[Remark~5.9]{LR} is irrelevant, since there had been no assumption that the divisor is connected.
Our Remark~\ref{LR4_rmk} contains additional related comments.
\end{rmk}

\noindent
In general, one has to consider smoothings of nodes that do not map to the junctions
between the smooth pieces of~$X\!\cup_V^mY$.
However, such nodes can be handled in a standard way, such as in \cite[Section~3]{LT},
as mentioned in \cite[Remark~6.3]{IPsum}.

\subsection{The $S$-matrix: \cite[Sections 11,12]{IPsum}}
\label{Smat_subs}

\noindent
The symplectic sum formula of~\cite{IPsum} contains two features
not present in the formulas of~\cite{Jun2} and~\cite{LR}:
a rim tori refinement of relative invariants and the so-called $S$-matrix.
This section explains why the second feature should not appear. 
We also show that in fact the $S$-matrix does not matter because it {\it acts as}
the identity in all cases and not just in the cases considered in
\cite[Sections~14,15]{IPsum}, when the $S$-matrix {\it is} the identity.
The fundamental reason for the latter is the same as for the former:
a  group action is forgotten in~\cite{IPsum}.\\

\noindent
By Gromov's Compactness Theorem \cite[Proposition~3.1]{RT1},
a sequence of $(J_{\cZ},\fj_k)$-holomorphic maps 
$u_k\!:\Si\!\lra\!\cZ_{\la_k}$, with $\la_k\!\in\!\De^*$ and $\la_k\!\lra\!0$,
has a subsequence converging to a $(J_{\cZ},\fj)$-holomorphic map $u\!:\Si'\!\lra\!\cZ_0$.
As explained in Section~\ref{IPconv_subs}, 
$$\Si'=\Si_X'\cup\Si_V'\cup\Si_Y'\,,$$
where $\Si_V'$ is the union of irreducible components of~$\Si'$ mapped into~$V$,
$\Si_X'$ is the union of irreducible components mapped into~$X\!-\!V$ outside
of finitely many points $x_1,\ldots,x_{\ell}$, and
$\Si_Y'$ is the union of irreducible components mapped into~$Y\!-\!V$ outside
of finitely many points $x_1',\ldots,x_{\ell'}'$.
The symplectic sum formulas of \cite{Jun2} and~\cite{LR} arise only from
the limits with~$\Si_V'\!=\!\eset$;
these are also the limits considered in \cite[Sections~6-10]{IPsum}.\\

\noindent
The $S$-matrix arises at the top of page~1003 in~\cite{IPsum} 
from the consideration of limits with $\Si_V'\!\neq\!\eset$.
Such maps are interpreted as maps to the singular spaces~$\XYm$, with $m\!\in\!\Z^+$,
defined in~\eref{XnVYdfn_e}.
This interpretation is obtained by viewing the sequences of maps which give rise to
such limits as having their images inside the total space~$\cZ_m$ of an $(m\!+\!1)$-dimensional 
family of smoothings of~$\XYm$, instead of the total space~$\cZ$
of a one-dimensional family of smoothing of~$X\!\cup_V\!Y$.
However, it is not possible to associate a sequence of maps to~$\cZ_m$
to a sequence of maps to~$\cZ$ in a systematic way which is consistent 
with the aims of \cite[Section~12]{IPsum}.
Contrary to the implicit view in \cite[Section~12]{IPsum}, 
the resulting limiting map to~$\XYm$ is well-defined by the original sequence
of maps to~$\cZ$ not up to a finite number of ambiguities, but 
up to an action of $m$~copies of~$\C^*$ on the target.
Furthermore, the entire setup at the top of page~1003 in~\cite{IPsum} 
is incorrect because the almost complex structure on~$\cZ_{\la}$ viewed as a fiber
of~$\cZ$ is different from what it would have been as a fiber in~$\cZ_m$
(the latter would be effected by $m\!+\!1$ copies of~$V$).
However, these almost complex structure would be the same in the case of
the more restricted almost complex structures of~\cite{LR}.\\

\noindent 
The situation is nearly identical to \cite[Sections~6,7]{IPrel}, 
where limits of sequences of relative maps into~$(X,V)$ are described as
 maps to~$X_m^V$ {\it up to} a natural $(\C^*)^m$-action;
see our Section~\ref{RelInv_subs}. 
The same reasoning implies that 
limits of sequences of  maps into~$\cZ$ correspond to maps to~$\XYm$
{\it up to} a natural $(\C^*)^m$-action.
\\

\noindent
As in the situation in Section~\ref{RelInv_subs}, which reviews \cite[Sections~6,7]{IPrel},
the virtual dimension of the spaces of morphisms into~$\XYm$ is $2m$ less than
the expected dimension of the corresponding spaces of morphisms into~$X\!\cup_V\!Y$
(with the matching conditions imposed) and into~$\cZ_{\la}$.
Thus, the spaces of morphisms into~$\XYm$ with $m\!\ge\!1$ have no effect on
the symplectic sum formula.
The $S$-matrix, which takes such spaces into account, enters at the top of page~1003
in~\cite{IPsum} because the spaces of such morphisms are mistakenly not quotiented out
by~$(\C^*)^m$; this is done in~\cite{Jun2} and in~\cite{LR}.\\

\noindent
While the $S$-matrix is generally not the identity, it acts as the identity 
in the symplectic sum formulas of~\cite{IPsum},
i.e.~in equations~(0.2) and~(12.7) in~\cite{IPsum}, for the following reason.
For all
$$\chi\in\Z, \qquad (A,B)\in H_2(X;\Z)\times_V H_2(Y;\Z)\,,$$
and a generic collection of constraints of appropriate total codimension,
the symplectic sum formula presents the corresponding GT-invariant of $X\!\#_V\!Y$ 
as the sum of weighted cardinalities of finitely many finite sets enumerating
morphisms into~$\XYm$, with $m\!\ge\!0$, meeting the constraints.
The group~$(\C^*)^m$ acts on the set of such morphisms with at most finite stabilizers
(the constraints inside each $\{r\}\!\times\!\P_VX$ are pull-backs from~$V$).  
Thus, the sets with $m\!\ge\!1$ are empty, i.e.~there is no contribution 
to the symplectic sum formula from morphisms to~$\XYm$ with~$m\!\ge\!1$.
Since these are the morphisms that make up the difference between the $S$-matrix 
and the identity, 
the $S$-matrix  acts as the identity in the symplectic sum formulas of~\cite{IPsum}.\\

\noindent
The next observation illustrates one of the problems with the normalizations of
generating functions in \cite[Section~1]{IPsum} and thus
another problem with the symplectic sum formulas of~\cite{IPsum}.
The last statement of \cite[Lemma~11.2(a)]{IPsum} is key to even making
sense of the action of the $S$-matrix. 
However, it does  not hold with the definitions in the paper.
By \cite[(1.24)]{IPsum}, the $\M_{\bI}$-part of $\GW^{\P_V,V_{\i}\sqcup V_0}(1)$ is given~by
\begin{equation*}\begin{split}
\GW^{\P_V,V_{\i}\sqcup V_0}(1)_{\M_{\bI}}
&=\sum_{d=1}^{\i}\big[\wt\ev_1\!\times\!\wt\ev_2\!:\ov\M_{0,2;(d,d)}^{V_{\i}\sqcup V_0}(\P_V,d)
\lra\cH_{\P_V;(d,d)}^{V_{\i}\sqcup V_0} \big]t_{dF}\la^{-2}\\
&=\sum_{d=1}^{\i}\frac{1}{d}\De_{(d,d)}t_{dF}\la^{-2},
\end{split}\end{equation*}
where $\De_{(d,d)}\subset \cH_{\P_V;(d,d)}^{V_{\i}\sqcup V_0}$ is the preimage
of the diagonal in $V_{\i}\!\times\!V_0\!=\!V\!\times\!V$.
The exponential of an element of $H_*(\wt\M\!\times\!\cH_X^V)$ and the product of two
elements of $H_*(\cH_X^V)$ are never defined, but under reasonable definitions
\begin{equation*}\begin{split}
\GT^{\P_V,V_{\i}\sqcup V_0}(1)_{\M_{\bI}}
&\equiv \ne^{\GW^{\P_V,V_{\i}\sqcup V_0}(1)_{\M_{\bI}}}
=\sum_{\ell=0}^{\i}\frac{1}{\ell!}\big(\GW^{\P_V,V_{\i}\sqcup V_0}(1)_{\M_{\bI}}\big)^{\ell}\\
&=1+\sum_{\ell=1}^{\i}\sum_{d_1,\ldots,d_{\ell}>0}\frac{1}{\ell!}\,\frac{1}{d_1\ldots d_{\ell}}
\De_{(d_1,d_1)}\!\times\!\!\ldots\!\times\!\De_{(d_{\ell},d_{\ell})}\,
t_{(d_1+\ldots+d_{\ell})F}\la^{-2\ell}
\,;
\end{split}\end{equation*}
this definition seems to be consistent with \cite[(A.3)]{IPsum} and the description of the coefficients
in the following paragraph.
Let $\eta\!\in\!\cH_{\P_V;(s_1,\ldots,s_m)}^{V_0}t_B$. 
By \cite[(10.6)]{IPsum}, the only nonzero term in 
$\eta\!*\!\GT^{\P_V,V_{\i}\sqcup V_0}(1)_{\M_{\bI}}$ arises from the summand $\ell\!=\!m$ and
$(d_1,\ldots,d_{\ell})\!=\!(s_1,\ldots,s_m)$ and equals
$$\frac{s_1\ldots s_m}{m!}\la^{2m}\eta\cdot\frac{1}{\ell!}\,\frac{1}{d_1\ldots d_{\ell}}\la^{-2\ell}
=\frac{1}{m!m!}\eta\neq\eta\quad\hbox{if}~~m>1.$$
The proof of the symplectic sum formula,  \cite[(12.7)]{IPsum}, makes use of~(11.3); 
otherwise, there would be dependence on~$N$.

\begin{rmk}\label{IPsumSec1112_rmk}
Other, fairly minor misstatements in \cite[Sections~11,12]{IPsum} include
\begin{enumerate}[label={},topsep=-5pt,itemsep=-5pt,leftmargin=*]
\item\textsf{p998, (11.1):} no need for square brackets; the superscripts on~$\cH$ 
should be the same;
\item\textsf{p998, line -6:} before (1.4) \to after (1.5);
\item\textsf{p999, line 1:} the {\it irreducible} $\P_V$-trivial;
\item\textsf{p999, (11.3),(11.4):} LHS missing $*$; $R^{V_{\i},V_0}\!\lra\!R$ in the notation below;
\item\textsf{p999, line -3:} $(J,\nu)$ \to $(A,n,\chi)$;
\item\textsf{p1000, Dfn 11.3:} there is no dependence on $(J,\nu)$;
\item\textsf{p1000, bottom:} this sentence does not make sense;  
\item\textsf{p1001, lines -13,-9:} $2N$ \to $2N\!-\!1$;
\item\textsf{p1001, line -4:} both identities are incorrect;
\item\textsf{p1002, line 6:} there is no $t$ in (2.6);
\item\textsf{p1002, below (12.2):} $\mu$ is on the domain, $\la$ is on the target;
\item\textsf{p1002, line 18:} $\ve=\al_V$;
\item\textsf{p1002, line 22:} {\it nonempty} subset; 
\item\textsf{p1003, Thm 12.3, line 5:} (11.3) \to  of Definition 11.3.
\end{enumerate}
\end{rmk}

\section{Applications}
\label{appl_sec}

\noindent
The purpose of \cite[Sections 14,15]{IPsum} is to give three powerful applications of 
the (standard) symplectic sum formula.
The authors make clear what geometric reasoning should lead to the three main formulas.
Fully implementing their ideas  leads to quick proofs of these formulas,
which had been previously established through significantly more complicated arguments.
Unfortunately, the arguments in~\cite{IPsum} are not completely precise and contain multiple,
sometimes self-canceling, errors (as well as typos), and 
none of the three formulas is stated correctly.
In order to illustrate the beauty of the intended arguments in this part of~\cite{IPsum},
we reproduce two of them, for counts of plane curves and for Hurwitz numbers, 
completely below, but with all the details and without the errors,
and then list the errors and typos made in~\cite{IPsum};
the substance and organization of the proofs come entirely from~\cite{IPsum}.
As noted at the beginning of~\cite{IPsum}, the applications of the symplectic sum formula
in these two cases essentially capture the original proofs.
In Section~\ref{LRapp_subs}, we briefly comment on the applications appearing in~\cite{LR}.
\\

\noindent
The argument in \cite[Sections 14,15]{IPsum} for the third application,
an enumeration of curves on the rational elliptic surface, is fundamentally different
from the original proof in~\cite{BL}.
It also contains the most serious gaps.
We describe and streamline this argument in \cite[Section~6]{GWsumIP}.
In the process, we illustrate some qualitative applications of 
the refinements to the standard relative GW-invariants and the usual symplectic
sum formula suggested in~\cite{IPrel,IPsum};
this is not done in~\cite{IPrel,IPsum}.

\subsection{Invariants of $\P^1$ and $\T^2$: \cite[Section 14.1]{IPsum}}
\label{P1T2_subs}

\noindent
This section computes some relative GW-invariants of $\P^1$ and $\T^2$.
If $V_1,V_2\!\subset\!X$ are two disjoint symplectic divisors,
we will denote by $\GW_{g,A;\bs_1,\bs_2}^{X,V_1\sqcup V_2}$ the relative GW-invariants
of $(X,V_1\!\cup\!V_2)$ with the contacts with~$V_1$ and~$V_2$ described by~$\bs_1$
and~$\bs_2$, respectively.
We will use similar notation for the disconnected GT-invariants and for the moduli spaces.

\begin{lmm}[{\cite[Lemma 14.1]{IPsum}}]
\label{P1_lmm1}
Let $0,\i$ denote two distinct points in~$\P^1$, $V\!=\!\{0,\i\}$, and $d\!\in\!\Z^+$.
The relative degree~$d$ GW-invariants of $(\P^1,V)$ with no constraints from $\P^1$ 
or $\ov\M$ are given~by
$$\GW_{g,d;\bs_0,\bs_{\i}}^{\P^1,\{0,\i\}}()=
\begin{cases} 1/d,&\tn{if}~g\!=\!0,~\bs_0,\bs_{\i}\!=\!(d);\\
0,&\tn{otherwise}.
\end{cases}$$
\end{lmm}

\begin{proof}
By \cite[(1.21)]{IPsum},
\BE{dimM2_e}\begin{split}
\dim_{\C}\ov\M_{g,0;\bs_0,\bs_{\i}}^{0,\i}(\P^1,d)
&=2d+(1\!-\!3)(1\!-\!g)+\ell(\bs_0)+\ell(\bs_{\i})-\deg\,\bs_0-\deg\,\bs_{\i}\\
&=2g-2+\ell(\bs_0)+\ell(\bs_{\i})\ge 2g\ge 0.
\end{split}\EE
This dimension is 0 only if $g\!=\!0$ and $\ell(\bs_0),\ell(\bs_{\i})=1$.
If $g\!=\!0$ and $\bs_0,\bs_{\i}\!=\!(d)$, 
$\ov\M_{g,0;\bs_0,\bs_{\i}}^{0,\i}(\P^1,d)$ consists of a single element,
the map $z\!\lra\!z^d$.
Since the order of the group of automorphisms of this map is~$d$, it contributes
$1/d$ to the GW-invariant.
\end{proof}

\begin{lmm}[{\cite[Lemma 14.2]{IPsum}}]
\label{P1_lmm2}
Let $0,\i,1$ denote three distinct points in~$\P^1$, $V\!=\!\{0,\i,1\}$,  
$d\!\in\!\Z^+$ with $d\!\ge\!2$, and 
\BE{bs1dfn_e}\bs_1\!=\!(2,\underset{d-2}{\underbrace{1,\ldots,1}}).\EE
The relative degree~$d$ GW-invariants of $(\P^1,V)$ enumerating maps 
with simple branching over~1 and no constraints from $\P^1$  or $\ov\M$ are given~by
\begin{equation*}\begin{split}
\GW_{g,d;\bs_0,\bs_{\i}}^{\P^1,\{0,\i\}}(b)
&\equiv \frac{1}{(d\!-\!2)!}\GW_{g,d;\bs_0,\bs_{\i},\bs_1}^{\P^1,\{0,\i,1\}}()\\
&=\begin{cases} 1,&\tn{if}~g\!=\!0,~\{\ell(\bs_0),\ell(\bs_{\i})\}=\{1,2\},~
\deg\,\bs_0,\deg\,\bs_{\i}\!=\!d;\\
0,&\tn{otherwise}.
\end{cases}
\end{split}\end{equation*}
\end{lmm}

\begin{proof}
Similarly to \eref{dimM2_e},
$$\dim_{\C}\ov\M_{g,0;\bs_0,\bs_{\i},\bs_1}^{0,\i,1}(\P^1,d)
=2g-3+\ell(\bs_0)+\ell(\bs_{\i})\ge 2g-1\ge -1.$$
This dimension is 0 only if $g\!=\!0$ and $\ell(\bs_0)\!+\!\ell(\bs_{\i})\!=\!3$.
Every holomorphic function on $\C$ with a pole of order $d$ at $\i$ and zeros at $0$ and
$1$ of orders $a$ and $b$, respectively, with $a\!+\!b\!=\!d$, is of the form
$z\!\lra\!Cz^a(z\!-\!1)^b$.
There is a unique value of~$C$ so that this function sends the remaining critical point,
$z\!=\!a/d$, to~$1$.
Thus, $\ov\M_{g,0;\bs_0,\bs_{\i},\bs_1}^{0,\i,1}(\P^1,d)$ consists 
of $(d\!-\!2)!$ automorphism-free elements 
(corresponding to the orderings of the simple preimages of~1).
\end{proof}

\begin{rmk}\label{14.1_rmk}
\cite[Lemma 14.3]{IPsum} is not used in the rest of the paper. 
Furthermore, its statement is wrong,
as the authors forget to divide by the order of the automorphism group of covers of the torus.
The notation for GW-invariants in \cite[Sections~14.1-14.5]{IPsum} is inconsistent with earlier parts
of the paper, as the first subscript is supposed to indicate the target space.
The notation for the simple branch point invariant of \cite[Lemma~14.2]{IPsum},
which is never formally defined, is even more confusing, since an insertion in parenthesis
is supposed to indicate a class on a product of $\ov\M$ and copies of~$X$.
The conclusion in the proof of Lemma~14.1 about the $S$-matrix does not follow from
the rest of the argument, since it may have contributions from higher genus and classes
coming from~$\ov\M$. 
\end{rmk}

\subsection{Invariants of $\bF_n$: \cite[Section 14.3]{IPsum}}
\label{Fn_subs}

\noindent
This section computes some relative GW-invariants of $(\bF_n,S_0\!\cup\!S_{\i})$, where
$$\bF_n=\P(\cO_{\P^1}(n)\!\oplus\!\cO_{\P^1}),\qquad
S_0=\P(0\!\oplus\!\cO_{\P^1})\subset\bF_n, \qquad
S_{\i}=\P(\cO_{\P^1}(n)\!\oplus\!0)\subset\bF_n.$$
We denote by $s_0$ and $f$ the homology classes of $S_0$ and of a fiber of $\bF_n\lra\P^1$.
For $A\!\in\!H_2(\bF_n)$, ordered partitions $\bs_0$ and $\bs_{\i}$ of $A\!\cdot\!S_0$ and
$A\!\cdot\!S_{\i}$, respectively, and $\al\!\in\!\T^*(\bF_n)$,
$$\GW_{g,A;\bs_0,\bs_{\i}}^{\bF_n,S_0\sqcup S_{\i}}(\al)\in 
H_*(S_0^{\ell(\bs_0)})\otimes H_*(S_{\i}^{\ell(\bs_{\i})}).$$
If $A\!=\!as_0\!+\!bf$, then
$$A\cdot S_{\i}=b, \qquad A\cdot S_0=na\!+\!b, \qquad
\blr{c_1(T\bF_n),A}= (2\!+\!n)a+2b .$$
Thus, by \cite[(1.21)]{IPsum},
\BE{dimFn_e}\begin{split}
\dim_{\C}\GW_{g,A;\bs_0,\bs_{\i}}^{\bF_n,S_0\sqcup S_{\i}}(\al)
&=  (2\!+\!n)a+2b+ (2\!-\!3)(1\!-\!g)+\ell(\al)+\ell(\bs_0)+\ell(\bs_{\i})\\
&\qquad -\deg\,\al-\deg\,\bs_0-\deg\,\bs_{\i}\\
&=g\!-\!1+2a+\ell(\al)+\ell(\bs_0)+\ell(\bs_{\i})-\deg\,\al,
\end{split}\EE
if $\al\!\in\!H^*(\bF_n^{\ell(\al)})$. 
In particular, $\GW_{g,A;\bs_0,\bs_{\i}}^{\bF_n,S_0\sqcup S_{\i}}(\al)\!=\!0$ unless
\BE{dimFn_e2} g+2a \le 1 +\deg\,\al-\ell(\al).\EE

\begin{lmm}[{\cite[Lemma 14.6]{IPsum}}]
\label{Fn_lmm1}
The relative degree~$A$ GW-invariants of $(\bF_n,S_0\!\cup\!S_{\i})$ with no constraints from 
$\bF_n$ or $\ov\M$ are given~by
$$\GW_{g,A;\bs_0,\bs_{\i}}^{\bF_n,S_0\sqcup S_{\i}}()=
\begin{cases} 
\frac1b(S_0\!\otimes\!1+1\!\otimes\!S_{\i}),&\tn{if}~g\!=\!0,~A\!=\!bf,~b\!\in\!\Z^+,~
\bs_0,\bs_{\i}\!=\!(b);\\
0,&\tn{otherwise}.
\end{cases}$$
\end{lmm}

\begin{proof}
By \eref{dimFn_e2}, $\GW_{g,A;\bs_0,\bs_{\i}}^{\bF_n,S_0\sqcup S_{\i}}()\!=\!0$ 
unless $a\!=\!0$ and $g\!=\!0,1$.
Since all elements of $\ov\M_{g,0;\bs_0,\bs_{\i}}^{S_0,S_{\i}}(\bF_n,bf)$ are maps to a fiber,
$\GW_{g,A;\bs_0,\bs_{\i}}^{\bF_n,S_0\sqcup S_{\i}}()$ lies in the image of the homomorphism
$$H_*(\De)\lra H_*(S_0^{\ell(\bs_0)})\otimes H_*(S_{\i}^{\ell(\bs_{\i})}),
\qquad\hbox{where}\quad
\De=\big\{(p,\ldots,p)\!\in\!S_0^{\ell(\bs_0)}\!\times\!S_{\i}^{\ell(\bs_{\i})}\big\},$$
induced by the inclusion.
Since $\dim_{\C}\De\!=\!1$, \eref{dimFn_e} then implies that 
$\GW_{g,bf;\bs_0,\bs_{\i}}^{\bF_n,S_0\sqcup S_{\i}}()\!=\!0$ 
unless $g\!=\!0$ and $\ell(\bs_0),\ell(\bs_{\i})\!=\!1$.
In the case $\bs_0,\bs_{\i}\!=\!(b)$, for every element 
$(p,p)\!\in\!\De\!\subset\!S_0\!\times\!S_{\i}$, there is a unique element $[u,y_1,y_2]$
of $\ov\M_{0,0;\bs_0,\bs_{\i}}^{S_0,S_{\i}}(\bF_n,bf)$ such that $u(y_1)\!=\!p\!\in\!S_0$ and 
$u(y_2)\!=\!p\!\in\!S_{\i}$; this is the map $z\!\lra\!z^b$ onto the fiber of 
$\bF_n\!\lra\!\P^1$ over~$p$.
Since the order of the automorphism group of this map is~$b$, we conclude that
$$\GW_{g,bf;\bs_0,\bs_{\i}}^{\bF_n,S_0\sqcup S_{\i}}() =\frac1b\De
=\frac1b(S_0\!\otimes\!1+1\!\otimes\!S_{\i})\in H_2(S_0\!\times\!S_{\i}),$$
by the Kunneth decomposition of the diagonal.
\end{proof}

\begin{lmm}[{\cite[Lemma 14.7]{IPsum}}]
\label{Fn_lmm2}
The relative degree~$A$ GW-invariants of $(\bF_n,S_0\!\cup\!S_{\i})$ with 
one point insertion from $\bF_n$ and
no other constraints from  $\bF_n$ or $\ov\M$ are given~by
$$\GW_{g,A;\bs_0,\bs_{\i}}^{\bF_n,S_0\sqcup S_{\i}}(p)=
\begin{cases} 
1,&\tn{if}~g\!=\!0,~A\!=\!bf,~b\!\in\!\Z^+,~ \bs_0,\bs_{\i}\!=\!(b);\\
S_0^{\ell(\bs_0)}\!\times\!S_{\i}^{\ell(\bs_{\i})},&
\tn{if}~g\!=\!0,~A\!=\!s_0\!+\!bf,~b\!\in\!\Z^{\ge0},~\deg\,\bs_0\!=\!n\!+b,
~\deg\,\bs_{\i}\!=\!b;\\
0,&\tn{otherwise}.
\end{cases}$$
\end{lmm}

\begin{proof}
By \eref{dimFn_e2}, $\GW_{g,A;\bs_0,\bs_{\i}}^{\bF_n,S_0\sqcup S_{\i}}(p)\!=\!0$ unless 
either $a\!=\!0$ and $g\!=\!0,1,2$ or $a\!=\!1$ and $g\!=\!0$.\\

\noindent
In the first case, all elements of $\ov\M_{g,1;\bs_0,\bs_{\i}}^{S_0,S_{\i}}(\bF_n,bf)$ are 
maps to a fiber and $\GW_{g,A;\bs_0,\bs_{\i}}^{\bF_n,S_0\sqcup S_{\i}}(p)$ lies in
the image of the homomorphism
$$H_*\big(q^{\ell(\bs_0)}\!\times\!q^{\ell(\bs_{\i})}\big)
\lra H_*(S_0^{\ell(\bs_0)})\otimes H_*(S_{\i}^{\ell(\bs_{\i})}),$$
where $q\!=\!\pi(p)\!\in\!\P^1$.
Thus, by \eref{dimFn_e}, $\GW_{g,bf;\bs_0,\bs_{\i}}^{\bF_n,S_0\sqcup S_{\i}}(p)\!=\!0$ unless 
either $b\!=\!0$ and $g\!=\!2$ or $g\!=\!0$ and $\bs_0,\bs_{\i}\!=\!(b)$;
otherwise, this class would not be zero-dimensional.
In the $g\!=\!2$, $b,\ell(\bs_0),\ell(\bs_{\i})\!=\!0$ subcase,
$$\big\{[u,x_1]\!\in\!\ov\M_{g,1;\bs_0,\bs_{\i}}^{S_0,S_{\i}}(\bF_n,bf)\!:~u(x_1)\!=\!p\big\}
\approx\ov\M_{2,1},$$
while the restriction of the obstruction bundle to this subspace is isomorphic to 
$\bE_2^*\!\otimes\!T_p\bF_n$, where $\bE_2\!\lra\!\ov\M_{2,1}$ is the Hodge bundle.
Since $\bE_2$ is the pull-back of the Hodge bundle over $\ov\M_{2,0}$ by the forgetful map,
$$\GW_{2,0;(),()}^{\bF_n,S_0\sqcup S_{\i}}(p)=\lr{c_2(\bE_2^*\!\otimes\!T_p\bF_n),\ov\M_{2,1}}
=\lr{c_2(\bE_2)^2,\ov\M_{2,1}}=0.$$
In the $g\!=\!0$, $\bs_0,\bs_{\i}\!=\!(b)$ subcase,  there is a unique element 
$[u,x_1,y_1,y_2]$
of $\ov\M_{0,1;\bs_0,\bs_{\i}}^{S_0,S_{\i}}(\bF_n,bf)$ such that $u(x_1)\!=\!p$; 
this is the map $z\!\lra\!z^b$ onto the fiber of $\bF_n\!\lra\!\P^1$ containing~$p$.
Unlike the case considered in the proof of Lemma~\ref{Fn_lmm1},
this element is automorphism free, due to the three marked points on its domain;
so the corresponding GW-invariant is~1.\\

\noindent
In the case $g\!=\!0$, $A\!=\!s\!+\!bf$ with $b\!\ge\!0$,
$\deg\,\bs_0\!=\!b\!+\!n$, and $\deg\,\bs_{\i}\!=\!b$, then
$$\dim_{\C}\GW_{g,A;\bs_0,\bs_{\i}}^{\bF_n,S_0\sqcup S_{\i}}(p)=\ell(\bs_0)+\ell(\bs_{\i})$$
by \eref{dimFn_e} and thus $\GW_{g,A;\bs_0,\bs_{\i}}^{\bF_n,S_0\sqcup S_{\i}}(p)$ is a multiple
of the fundamental class of $S_0^{\ell(\bs_0)}\!\times\!S_{\i}^{\ell(\bs_{\i})}$.
This multiple is~1 because $b$ points on~$S_{\i}$ determine poles of a section 
of $\cO(n)\!\lra\!\P^1$
and $b\!+\!n$ points on~$S_0$ determine the unique section of $\cO(n)$ with these poles
that passes through~$p$.
\end{proof}

\begin{rmk}\label{14.3_rmk}
We denote the divisors $S,E\subset\bF_n$ of \cite[Sections~14.3,15.1]{IPsum} by $S_0,S_{\i}$
in order to avoid confusion with the rational elliptic surface of 
\cite[Sections~14.4,15.3]{IPsum}, which is also denoted by~$E$. 
The conclusion in the proof of \cite[Lemma~14.6]{IPsum} about the $S$-matrix does not follow from
the rest of the argument, since it may have contributions from higher genus and classes
coming from~$\ov\M$. 
The proof of \cite[Lemma 14.7]{IPsum} ignores the possibility of $b\!=\!0$ considered above.
In the second case considered in this proof, the dimension of the moduli space
is $\ell(s)\!+\!\ell(s')$ after cutting down by the point constraint.
An irreducible curve representing $S\!+\!bF$ is genus~0 and embedded, because
its projection to~$S$ is of degree~1.
The above argument gives a simpler reason why the multiple is~1.
In the statement of \cite[Lemma 14.7]{IPsum}, the degree conditions on $\bs$ and $\bs'$
are reversed (and are implied by the notation).
Other, minor typos in \cite[Sections~14.1,14.3]{IPsum} include
\begin{enumerate}[label={},topsep=-5pt,itemsep=-5pt,leftmargin=*]
\item \textsf{p1010, lines -2,-1:} $X$ \to $\bF_n$;
\item \textsf{statement and proof of Lemma~14.7:} $SV_{\bs}$ \to $V_{\bs}$;
$SV_{\bs'}$ \to $V_{\bs'}$.
\end{enumerate}
\end{rmk}

\subsection{Enumeration of plane curves: \cite[Section 15.1]{IPsum}} 
\label{CH_subs}

\noindent
This section deduces the Caporaso-Harris formula enumerating curves in~$\P^2$, 
\cite[Theorem~1.1]{CaH}, from the symplectic sum formula.
Fix a line $L\!\subset\!\P^2$.
For tuples 
$$\al\equiv(\al_1,\al_2,\ldots),\be\equiv(\be_1,\be_2,\ldots)\in(\Z^{\ge0})^{\Z^+}$$ 
with finitely many nonzero entries, let
\begin{gather*}
|\al|=\al_1+\al_2+\ldots,\qquad
\al!=\al_1!\cdot\al_2!\cdot\ldots, \qquad
 I\al=\al_1+2\al_2+\ldots, \\
 I^{\al}=1^{\al_1}2^{\al_2}\ldots,\qquad
 \binom{\al}{\be}=\binom{\al_1}{\be_1}\binom{\al_2}{\be_2}\ldots\,, \qquad
\bs_{\al}=\big(\underset{\al_1}{\underbrace{1,\ldots,1}},
\underset{\al_2}{\underbrace{2,\ldots,2}},\ldots\big).
\end{gather*}
For each $k\!\in\!\Z^+$, let $\ve_k\!\in\!(\Z^{\ge0})^{\Z^+}$ be the tuple with 
the $k$-th coordinate equal to~1 and the remaining coordinates equal to~0.\\

\noindent
Given $d\!\in\!\Z^+$, $\de\!\in\!\Z^{\ge0}$,  and $\al,\be\!\in\!(\Z^{\ge0})^{\Z^+}$
such that $I\al\!+\!I\be\!=\!d$, let $N^{d,\de}(\al,\be)$ denote the number
of degree~$d$ curves in~$\P^2$ that have $\de$ nodes, have contact of order~$k$
with $L$ at
$\al_k$ fixed points and $\be_k$ arbitrary points for each $k\!=\!1,2,\ldots$, 
and pass through
\BE{rdfn_e}r=\frac{d(d\!+\!1)}{2}-\de+|\be|\EE
general points in $\P^2$.
Thus,
\BE{CHnums_e}\begin{split}
\be! N^{d,\de}(\al,\be)&=  \GT_{\chi_{\de}(d),dL}^{\P^2,L}(p^r;\bfC_{\al;\be}),\\
&\equiv  \GT_{\chi_{\de}(d),dL;\bs_{\al},\bs_{\be}}^{\P^2,L}
\big(p^r;\underset{|\al|}{\underbrace{p,\ldots,p}},\underset{|\be|}{\underbrace{L,\ldots,L}}\big)
\end{split}\EE
where $\chi_{\de}(d)=2\de-d(d\!-\!3)$ is the geometric euler characteristic of the curves
(the euler characteristic of the normalization)
and  $p,L\!\in\!H^*(L)$ are the Poincare duals of a point in $L$ and 
of the fundamental class of~$L$.
Since a degree~$d$ curve in~$\P^2$ can have at most $d(d\!-\!1)/2$ nodes,
the number~$r$ in~\eref{rdfn_e} is positive whenever $N^{d,\de}(\al,\be)\!\neq\!0$.
This number~$r$ is at least~2 if $N^{d,\de}(\al,\be)\!\neq\!0$ and $(d,\be)\!\neq\!(1,\0)$.

\begin{crl}[{\cite[Theorem~1.1]{CaH}}]\label{CH_crl}
Let $d\!\in\!\Z^+$, $\de\!\in\!\Z^{\ge0}$,  and $\al,\be\!\in\!(\Z^{\ge0})^{\Z^+}$
with $(d,\be)\!\neq\!(1,\0)$.
If $I\al\!+\!I\be\!=\!d$, 
\begin{equation*}\begin{split}
N^{d,\de}(\al,\be)&=\sum_{\begin{subarray}{c}k\in\Z^+\\ \be_k>0\end{subarray}}
\!kN^{d,\de}(\al\!+\!\ve_k,\be\!-\!\ve_k)\\
&\qquad+
\sum_{\begin{subarray}{c}\de'\in\Z^{\ge0},\,\al',\be'\in(\Z^{\ge0})^{\Z^+}\\
I\al'+I\be'=d-1\\
\de-\de'+|\be'-\be|=d-1\end{subarray}}
\!\!\binom{\al}{\al'}\binom{\be'}{\be}I^{\be'-\be}N^{d-1,\de'}(\al',\be').
\end{split}\end{equation*}
\end{crl}

\noindent
As sketched in \cite[Section~15.1]{IPsum}, this formula can be proved by applying 
the natural extension of 
the symplectic sum formula~\eref{SympSumForm_e} to the decomposition
$$(\P^2,L)=(\P^2,L)\underset{L=S_{\i}}{\#}(\bF_1,S_{\i},S_0),$$
with $(\bF_1,S_{\i},S_0)$ as in Section~\ref{Fn_subs}, 
and moving one of the $r$ absolute point insertions to the~$\bF_1$ side.
Since the divisor $L\!=\!S_{\i}$ is simply connected, the connect sum
$$\#\!:H_2(\P^2;\Z)\!\!\!\!\!\!\!\underset{~~L=S_{\i}}{\times}\!\!\!\!\!H_2(\bF_1;\Z)
\lra H_2(\P^2;\Z)$$
is well-defined in this case.
Since $dL\cdot L=(aS_0\!+\!bF)\cdot S_{\i}$ if and only if $d\!=\!b$ and 
$$dL\#(aS_0\!+\!dF)=(d\!+\!a)L,$$ 
the symplectic sum formula~\eref{SympSumForm_e} and \eref{CHnums_e} give 
\BE{CHsplit_e}\begin{split}
&\be! N^{d,\de}(\al,\be)\\
&=
\sum_{\begin{subarray}{c}d'\in\Z^+,d''\in\Z^{\ge0}\\ d'+d''=d\end{subarray}}
\sum_{\begin{subarray}{c}\de'\in\Z^{\ge0},\chi''\in\Z,\al',\be'\in(\Z^{\ge0})^{\Z^+}\\
I\al'+I\be'=d'\\ 
\chi_{\de'}(d')+\chi''=\chi_{\de}(d)+2|\al'|+2|\be'|\\
\frac{d'(d'+1)}{2}-\de'+|\be'|=r-1
\end{subarray}}\!\!\!\!\!\!\!\!
\frac{I^{\al'}I^{\be'}}{\al'!}
N^{d',\de'}(\al',\be')\cdot
\GT_{\chi'',d''S_0+d'F}^{\bF_1,S_{\i}\sqcup S_0}(p;\bfC_{\be';\al'},\bfC_{\al;\be}),
\end{split}\EE
with the GT-invariant defined analogously to~\eref{CHnums_e} for each component 
of the relative divisor.\\

\noindent
By Lemmas~\ref{Fn_lmm1} and~\ref{Fn_lmm2}, 
there are two types of configurations that contribute to the GT-invariant in~\eref{CHsplit_e}:
\begin{enumerate}[label=(\arabic*),leftmargin=*]
\item genus 0 multiple covers of fibers, each with a single point of contact with $S_{\i}$
and a single point of contact with~$S_0$, with one of these fiber maps passing through the constraint
point in~$\bF_1$;
\item genus 0 multiple covers of fibers, each with a single point of contact with $S_{\i}$
and a single point of contact with~$S_0$, and one genus~0 degree $S_0\!+\!d'F$ map passing
through the constraint point in~$\bF_1$.
\end{enumerate}
By Lemma~\ref{Fn_lmm1}, a genus 0 multiple cover of a fiber not passing through the constraint
point passes through either a fixed point on~$S_0$ (i.e.~a point with contact specified by~$\al$) 
and an arbitrary point on $S_{\i}$ (i.e.~a point with contact encoded by~$\al'$) 
or an arbitrary point on $S_0$ (i.e.~a point with contact encoded by~$\be$) and
a fixed point on~$S_{\i}$ (i.e.~a point with contact specified by~$\be'$).
The orders of contact on the two ends are the same number~$b$, which is
the degree of the cover. Such a cover contributes a factor of $1/b$ to the
GT-invariant in \eref{CHsplit_e}.\\ 

\noindent
In the first case above, $d'\!=\!d$, $\de'\!=\!\de$, and $\al'\!=\!\al\!+\!\ve_k$ and 
$\be'\!=\!\be\!-\!\ve_k$ for some $k\!\in\!\Z^+$, as both relative conditions 
on the distinguished fiber map into~$\bF_1$ must be single
arbitrary points by the first statement in Lemma~\ref{Fn_lmm2}.
For each $k\!\in\!\Z^+$ with $\be_k\!>\!0$, there are 
\begin{enumerate}[label=(\alph*),leftmargin=*]
\item $\be_k$ choices for the relative marked point on the $S_0$ end 
of the distinguished fiber map into~$\bF_1$
and $\al'_k$ choices on the $S_{\i}$ end of this map,
\item $\be'!=\be!/\be_k$ choices of  ordering the ``arbitrary" points on the~$S_0$ 
end of the other fiber maps (the ordering on the $S_{\i}$ end of these maps can be fixed
by the fixed points; along with~(a), this contributes a factor of $\be!$ to the right-hand side
of~\eref{CHsplit_e}),
\item $\al!=\al'!/\al_k'$ choices of  ordering the ``arbitrary" points on the $S_{\i}$ end 
of the other fiber maps (the ordering on the $S_0$ end of these maps can be fixed
by the fixed points; 
along with~(a), this eliminates $1/\al'!$ from  the right-hand side
of~\eref{CHsplit_e}).
\end{enumerate}
Furthermore, the non-distinguished fiber maps in a single configuration 
contribute 
$$\frac{1}{I^{\al}I^{\be'}}=\frac{k}{I^{\al'}I^{\be'}}$$ 
to the invariant.
Thus, Case~1 contributes $\be!\cdot kN^{d,\de}(\al\!+\!\ve_k,\be\!-\!\ve_k)$
to the right-hand side of~\eref{CHsplit_e}.\\

\noindent
In the second case above, $d'\!=\!d\!-\!1$, 
$\al'\!=\!\al\!-\!\al_0$ for some $\al_0\!\in\!(\Z^{\ge0})^{\Z^+}$, and 
$\be\!=\!\be'\!-\!\be_0'$ for some $\be_0'\!\in\!(\Z^{\ge0})^{\Z^+}$,
as both relative conditions on the distinguished map into~$\bF_1$ must be 
fixed points by the second statement in Lemma~\ref{Fn_lmm2}.
For each pair $(\al_0,\be_0')$, there are 
\begin{enumerate}[label=(\alph*),leftmargin=*]
\item $\binom{\al}{\al_0}$ choices of fixed points on the $S_0$ end of $\bF_1$
and $\binom{\be'}{\be_0'}$ choices of fixed points on the $S_{\i}$ end of $\bF_1$
(these go on the non-fiber curve),
\item $\be!=\!(\be'\!-\!\be_0')!$ choices of  ordering the ``arbitrary" points on the~$S_0$ 
end of the fiber maps (the ordering on the $S_{\i}$ end of these maps can be fixed
by the fixed points; this contributes a factor of $\be!$ to the right-hand side
of~\eref{CHsplit_e}),
\item $\al'!\!=\!(\al\!-\!\al_0)!$ choices of  ordering the ``arbitrary" points on 
the $S_{\i}$ end 
of the fiber maps (the ordering on the $S_0$ end of these maps can be fixed
by the fixed points; this eliminates $1/\al'!$ from the formula).
\end{enumerate}
Furthermore, the fiber maps in a single configuration contribute 
$$\frac{1}{I^{\al'}I^{\be}}=\frac{I^{\be'-\be}}{I^{\al'}I^{\be'}}$$
to the invariant.
Thus, Case~2 contributes
$\be!\cdot\binom{\al}{\al'}\binom{\be'}{\be}I^{\be'-\be}N^{d-1,\de'}(\al',\be')$
to the right-hand side of~\eref{CHsplit_e}.
This establishes Corollary~\ref{CH_crl}.

\begin{rmk}\label{15.1_rmk}
Throughout \cite[Section~15.1]{IPsum}, $\P$ and $\P_1$ denote the surface $\bF_1$ of 
\cite[Section~14.3]{IPsum}.
The first identity on page~1015 cannot possibly be true, since $\GT_{\chi,dL}^{\P^2,L}$,
however its definition is interpreted, groups the relative constraints
of the same type together and treats the resulting sets in the same way,
 while $N^{d,\de}(\al,\be)$ treats the $\al$ and $\be$ constraints differently
(the $\al$-contacts are fixed and so the corresponding contact points of the domain
can be ordered).
The definition in the second displayed expression 
and the symplectic sum formula in the third displayed expression have 
the same issue.
The former is unnecessary,
since the symplectic sum formula involves GW/GT-invariants and these
are also the numbers computed in \cite[Lemma~14.7]{IPsum}.
Finally:
\begin{enumerate}[label={},topsep=-5pt,itemsep=-5pt,leftmargin=*]
\item\textsf{p1014, -3:} $\P$ \to $\P^2$;
\item\textsf{p1015, lines 3,18,19,21,29; p1016, line 8:} $\P$ \to $\P_1$;
\item\textsf{p1015, line 9:} $\ga^1$ \to $\ga_1$;
\item\textsf{line 12:} $m!$ and $|m|$ correspond to $\al!\be!$ and $|\al|\cdot|\be|$;
$\prod_i\al_i$ \to $\prod_i\al_i!$;
\item\textsf{line 14:} $\GT_{\chi,dL,\P^2}^L$ \to $\GT_{\P^2,dL,\chi}^L$; $C_m$ \to~$\bfC_m$;
\item\textsf{line 15:} $\chi$ is geometric euler characteristic;
\item\textsf{line 18:} $\GT_{\chi,aL+bF,\P}^{E,L}$ \to $\GT_{\P,aL+bF,\chi}^{E,L}$;
$(C_m;p;C_{m'})$ \to $(\bfC_{m'};p;\bfC_m)$
\item\textsf{line 23:} the $S$-matrix is the identity;
\item\textsf{bottom:} $\al'=\al+\ve_k$, $\be'=\be-\ve_k$, $\chi'=\chi$;
\item\textsf{p1016, lines 7-9:} it is unclear what this sentence is saying;
\item\textsf{line 11:} 
$N^{d,\de'}(\al\!-\!\ve_k,\be\!+\!\ve_k)$ \to $N^{d,\de}(\al\!+\!\ve_k,\be\!-\!\ve_k)$.
\end{enumerate}
\end{rmk}

\subsection{Hurwitz numbers: \cite[Section 15.2]{IPsum}} 
\label{Hurwitz_subs}

\noindent
This section deduces a cut and paste formula for branched covers of~$\P^1$,
\cite[Lemma~3.1]{GJV}, from the natural extension of
the symplectic sum formula~\eref{SympSumForm_e} to relative invariants.
We continue with the combinatorial notation introduced at the beginning of 
Section~\ref{CH_subs}.\\

\noindent
Fix a point $p\!\in\!\P^1$.
Given $d\!\in\!\Z^+$, $g\!\in\!\Z^{\ge0}$,  and $\al\!\in\!(\Z^{\ge0})^{\Z^+}$
such that $I\al\!=\!d$, let $N_{d,g}(\al)$ denote the number of 
genus~$g$ degree~$d$ branched covers of~$\P^1$ with $\al_k$ branch points of order~$k$
over~$p$ for each $k\!\in\!\Z^+$ and simple branching over
\BE{rdfn_e2}r=d+|\al|+2g-2\EE
other fixed points $p_1,\ldots,p_r$ in $\P^1$.
Thus,
\BE{Hurnums_e}
N_{d,g}(\al)=
\frac{1}{\al!\,(d\!-\!2)!^r}\ \deg\,\GW^{\P^1,V_r}_{g,d;\bs_1^r,\bs_{\al}}()\,,\EE
where $V_r\!=\!\{p_1,...,p_r,p\}$ and $\bs_1^r$ denotes $r$ copies of the tuple $\bs_1$
defined in~\eref{bs1dfn_e}.
Since $d,|\al|\!\ge\!1$ and $g\!\ge\!0$, the number~$r$ in~\eref{Hurnums_e} is positive 
unless $(d,g,\al)=(1,0,(1))$; in this exceptional case, $N_{d,g}(\al)\!=\!1$.

\begin{crl}[{\cite[Lemma~3.1]{GJV}}]\label{Hur_crl}
The generating function
\BE{HurFdfn_e} 
F(\la,u,z_1,z_2,\ldots)=
\sum_{g=0}^{\i}\sum_{d=1}^{\i}
\sum_{\begin{subarray}{c}\al\in(\Z^{\ge0})^{\i}\\ I\al=d\end{subarray}}
\!\!\!\!\!\! N_{d,g}(\al) \bigg(\!\prod_{k=1}^{\i}z_k^{\al_k}\!\!\bigg)
\frac{u^{d+|\al|+2g-2}}{(d\!+\!|\al|\!+\!2g\!-\!2)!}\la^{2g-2}\EE
satisfies the PDE
\BE{HurPDE_e}\partial_uF= \frac12 \sum_{i,j\ge1}
\big(ijz_{i+j}\la^2[\partial_{z_i}\partial_{z_j}F+
\partial_{z_i}F\cdot\partial_{z_j}F]+
(i\!+\!j)z_iz_j\partial_{z_{i+j}}F\big).\EE
\end{crl}

\noindent
As sketched in \cite[Section~15.2]{IPsum}, this statement can be proved by applying 
the symplectic sum formula to the decomposition
$$(\P^1,p_1,\ldots,p_r,p)=(\P^1,p_1,\ldots,p_{r-1},x)\underset{x=y}{\#}
(\P^1,y,p_r,p),$$
i.e.~by separating off the distinguished branch point and one of the simple branch points
onto a second copy of~$\P^1$.
In this case, the connect sum is well-defined on~$H_2$ and is given~by
\begin{gather*}
\#\!:H_2(\P^1;\Z)\!\!\underset{x=y}\times\!\!H_2(\P^1;\Z)
=\big\{(d\P^1,d\P^1)\!:~d\!\in\!\Z\big\}\lra H_2(\P^1;\Z), \\
(d\P^1,d\P^1)\lra d\P^1\,.
\end{gather*}
Thus, the symplectic sum formula and \eref{Hurnums_e} give 
\BE{Hursplit_e}\begin{split}
&\al!\,(d\!-\!2)!^rN_{d,g}(\al)\\
&\hspace{.3in}=
\sum_{\begin{subarray}{c}\Ga=(\Ga_1,\Ga_2)\\  
g(\Ga)=g+|\tV_{\Ga}|-|\tE_{\Ga}|-1\end{subarray}}
\sum_{\begin{subarray}{c}\al'\in(\Z^{\ge0})^{\Z^+}\\ |\al'|=|\tE_{\Ga}|,~ I\al'=d\end{subarray}}
\!\!\!\!\!\!\frac{I^{\al'}}{\al'!}
\Big(\deg\,\GT_{\Ga';\bs_1^{r-1},\bs_{\al'}}^{\P^1,V_{r-1}}\Big)
\Big(\deg\,\GT_{\Ga'';\bs_{\al'},\bs_1,\bs_{\al}}^{\P^1,\{y,p_r,p\}}\Big),\\
\end{split}\EE
with the outer sum taken over all bipartite connected graphs $\Ga\!=\!(\Ga',\Ga'')$
with vertices~$\tV_{\Ga}$ decorated by nonnegative integers, as in Figure~\ref{Hur_fig}; 
we denote the sum of these numbers by~$|V_{\Ga}|$.
Each vertex of $\Ga'$ (resp.~$\Ga'')$ corresponds to a map from a connected curve
of the genus given by the vertex label into the first (resp.~second)~$\P^1$. 
Each edge in~$\Ga$ represents paired relative marked points of the domains
mapped into the two copies of~$\P^1$.\\

\noindent
By Lemmas~\ref{P1_lmm1} and~\ref{P1_lmm2}, 
there are two types of configurations that contribute to the right-hand side of~\eref{Hursplit_e}:
\begin{enumerate}[label=(\arabic*),leftmargin=*]
\item genus 0 branched covers of the second $\P^1$, each with a single preimage of~$y$ 
and a single preimage of~$p$, and a genus~0 branched cover of the second $\P^1$
with a single preimage of~$y$, two preimages of~$p$, and a simple branching over~$p_r$;
\item  genus 0 branched covers of the second $\P^1$, each with a single preimage of~$y$ 
and a single preimage of~$p$, and a genus~0 branched cover of the second $\P^1$
with a single preimage of~$p$, two preimages of~$y$, and a simple branching over~$p_r$.
\end{enumerate}
By Lemma~\ref{P1_lmm1}, a degree~$k$ branched cover of the second $\P^1$ without 
the branching condition over~$p_r$ contributes a factor of $1/k$ to the
last GT-invariant in \eref{Hursplit_e}.
Such a cover  has contact of order~$k$ with $p$  (i.e.~a point with contact specified by~$\al$)
and~$y$ (i.e.~a point with contact encoded by~$\al'$).
For all curve types~$\Ga$, there are $(d\!-\!2)!$ choices of ordering the non-branched preimages
of each of the $r$ simple branch points, which together contribute a factor of $(d\!-\!2)!^r$ 
to the right-hand side of~\eref{Hursplit_e}.\\

\begin{figure}
\begin{pspicture}(-2,-2)(10,0)
\psset{unit=.3cm}
\pscircle*(5,-6){.2}\rput(5,-7){\smsize{$g$}}
\psline[linewidth=.05](5,-6)(8,-2)\pscircle*(8,-2){.2}\rput(8,-1.2){\smsize{0}}
\psline[linewidth=.05](5,-6)(6,-2)\pscircle*(6,-2){.2}\rput(6,-1.2){\smsize{0}}
\psline[linewidth=.05](5,-6)(4,-2)\pscircle*(4,-2){.2}\rput(4,-1.2){\smsize{0}}
\psline[linewidth=.05](5,-6)(2,-2)\pscircle*(2,-2){.2}\rput(2,-1.2){\smsize{0}}
\pscircle*(25,-6){.2}\rput(25,-7){\smsize{$g-1$}}
\psline[linewidth=.05](25,-6)(28,-2)\pscircle*(28,-2){.2}\rput(28,-1.2){\smsize{0}}
\psline[linewidth=.05](25,-6)(22,-2)\pscircle*(22,-2){.2}\rput(22,-1.2){\smsize{0}}
\psarc[linewidth=.04](29,-4){4.47}{153.43}{206.57}
\psarc[linewidth=.04](21,-4){4.47}{-26.57}{26.57}
\pscircle*(25,-2){.2}\rput(25,-1.2){\smsize{0}}
\pscircle*(33.5,-6){.2}\rput(33.5,-7){\smsize{$g_1$}}
\pscircle*(36.5,-6){.2}\rput(36.5,-7){\smsize{$g_2$}}
\psline[linewidth=.05](36.5,-6)(38,-2)\pscircle*(38,-2){.2}\rput(38,-1.2){\smsize{0}}
\psline[linewidth=.05](33.5,-6)(32,-2)\pscircle*(32,-2){.2}\rput(32,-1.2){\smsize{0}}
\psline[linewidth=.05](36.5,-6)(35,-2)
\psline[linewidth=.05](33.5,-6)(35,-2)
\pscircle*(35,-2){.2}\rput(35,-1.2){\smsize{0}}
\end{pspicture}
\caption{Graph types $\Ga$ contributing to the right-hand side of \eref{Hursplit_e}.}
\label{Hur_fig}
\end{figure}
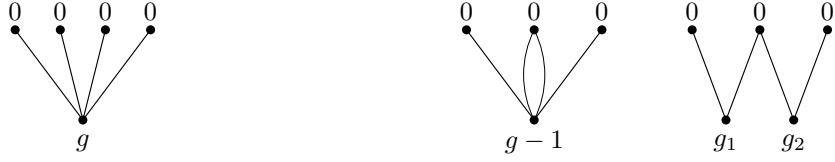

\noindent
In the first case above, $\Ga'$ consists of a single vertex with label~$g$
and 
$$\al'=\al-\ve_i-\ve_j+\ve_{i+j}$$
for some $i,j\!\in\!\Z^+$, as there are two contact conditions on the $p$ end 
of a branched cover of the second~$\P^1$ (corresponding to~$\al$)
and only one on the~$y$ end (corresponding to~$\al'$).
Whenever $i\!\neq\!j$ and $\al_i,\al_j\!>\!0$, there are 
\begin{enumerate}[label=(\alph*),leftmargin=*]
\item $\al_i\al_j$ choices for the relative marked points on the $p$ end 
of the distinguished map into the second~$\P^1$,
\item $\al'!$ choices for ordering the points on the $y$ end of the second~$\P^1$ 
for a given ordering of the points on the $p$ end 
(this eliminates $1/\al'!$ from the formula).
\end{enumerate}
Furthermore, the non-distinguished fiber maps in a single configuration 
contribute 
$$\frac{1}{I^{\al}/(ij)}=\frac{1}{I^{\al'}/(i\!+\!j)}.$$ 
Thus, the contribution from this case is
$$\al_i\al_j\cdot(i\!+\!j)\cdot\al'!N_{g,d}(\al')
=\al!\cdot(i\!+\!j)(\al_{i+j}\!+\!1)N_{g,d}(\al').$$
In the $i\!=\!j$ case, there are $\al_i(\al_i\!-\!1)/2$ choices in (a) above and
the same number of choices in~(b).
So, the contribution now is
$$\frac{\al_i(\al_i\!-\!1)}{2}\cdot(i\!+\!i)\cdot\al'!N_{g,d}(\al')
=\al!\cdot\frac12(i\!+\!j)(\al_{i+j}\!+\!1)N_{g,d}(\al').$$
Both cases correspond to the last term in~\eref{HurPDE_e}, since $\al'$ is obtained from $\al$
by reducing $\al_i$ and $\al_j$ and increasing $\al_{i+j}$ by~1
(thus, $N_{g,d}(\al')$ is the coefficient of the product of $z_1,\ldots$ with one smaller
power of $z_i$ and $z_j$ and one larger power of $z_{i+j}$;
the factor of $\al_{i+j}\!+\!1$ above corresponds to differentiating $z_{i+j}^{\al_{i+j}+1}$).\\

\noindent
In the second case above, $\Ga'$ consists either of a single vertex with label $g\!-\!1$
or two vertices with labels adding up to~$g$ and 
$$\al'=\al+\ve_i+\ve_j-\ve_{i+j}$$
for some $i,j\!\in\!\Z^+$, as there is one contact condition on the $p$ end 
of a branched cover of the second~$\P^1$ (corresponding to~$\al$)
and two on the~$y$ end (corresponding to~$\al'$).
Whenever $i\!\neq\!j$ and $\al_i',\al_j'\!>\!0$, there are 
\begin{enumerate}[label=(\alph*),leftmargin=*]
\item $\al_i'\al_j'$ choices for the relative marked points on the $y$ end 
of the distinguished map into the second~$\P^1$,
\item $\al!$ choices for ordering the points on the $p$ end of the second~$\P^1$ 
for a given ordering of the points on the $y$ end 
(this contributes a factor of $\al!$ to the right-hand side of~\eref{Hursplit_e}).
\end{enumerate}
Furthermore, the non-distinguished fiber maps in a single configuration 
contribute 
$$\frac{1}{I^{\al}/(i+j)}=\frac{1}{I^{\al'}/(ij)}.$$ 
Thus, the contribution from this case is
\BE{Hur2a_e} 
\al!\cdot \al_i'\al_j'\cdot\frac{ij}{\al'!}\cdot\al'!N_{g,d}'(\al')
=\al!\cdot ij(\al_i\!+\!1)(\al_j\!+\!1)N_{g,d}'(\al'),\EE
where $N_{g,d}'(\al')$ denotes the sum of the contribution from the two possible configurations
into the first~$\P^1$, divided by $(d\!-\!2)!^{r-1}$ and $\al'!$.
In the $i\!=\!j$ case, there are $\al_i'(\al_i'\!-\!1)/2$ choices in~(a) above and
the same number of choices in~(b).
So, the contribution now is
\BE{Hur2b_e}
\al!\cdot\frac12\al_i'(\al_i'\!-\!1)\frac{ii}{\al'!}\cdot\al'!N_{g,d}'(\al')
=\al!\cdot\frac12ij (\al_i\!+\!2)(\al_i\!+\!1)N_{g,d}'(\al').\EE
The connected configuration into the first $\P^1$ contributes $N_{g-1,d}(\al')$
to the number $N_{g,d}'(\al')$.
Combined with the factors \eref{Hur2a_e} and~\eref{Hur2b_e}, 
this corresponds to the first term on the right-hand side of~\eref{HurPDE_e}, 
since $\al'$ is obtained from $\al$ by increasing $\al_i$ and $\al_j$ and reducing 
$\al_{i+j}$ by~1
(thus, $N_{g-1,d}(\al')$ is the coefficient of the product of $z_1,\ldots$ with one larger
power of $z_i$ and $z_j$ and one smaller power of $z_{i+j}$ and~$\la^2$).\\

\noindent
Finally, the contribution of the two-component configuration to~$N_{g,d}'(\al')$ is
\begin{equation*}\begin{split}
&\frac{1}{\al'!}\sum_{r_1+r_2=r-1}\sum_{g_1+g_2=g}
\sum_{\begin{subarray}{c}\al_1'+\al_2'=\al'\\ \al_{1;i}',\al_{2;j}'>0\end{subarray}}
\binom{r\!-\!1}{r_1}\binom{\al'\!-\!\ve_i\!-\!\ve_j}{\al_1'\!-\!\ve_i}
\al_1'!N_{d_1,g_1}(\al_1')\al_2'!N_{d_2,g_2}(\al_2')\\
&\qquad=(r\!-\!1)!\frac{(\al'\!-\!\ve_i\!-\!\ve_j)!}{\al'!}\sum_{r_1+r_2=r-1}\sum_{g_1+g_2=g}
\sum_{\al_1'+\al_2'=\al'}
\frac{\al_{1;i}'N_{d_1,g_1}(\al_1')}{r_1!}\frac{\al_{2;j}'N_{d_2,g_2}(\al_2')}{r_2!}\,,
\end{split}\end{equation*}
where $d_i$ is determined by $g_i$, $\al_i'$, and $r_i$.
Combined with  the factors \eref{Hur2a_e} and~\eref{Hur2b_e} and summed over 
ordered pairs $(i,j)$, this contributes
$$\al!\cdot (r\!-\!1)!\cdot\frac12\sum_{i,j}ij
\sum_{r_1+r_2=r-1}\sum_{g_1+g_2=g}
\sum_{\al_1'+\al_2'=\al'}
\frac{\al_{1;i}'N_{d_1,g_1}(\al_1')}{r_1!}\frac{\al_{2;j}'N_{d_2,g_2}(\al_2')}{r_2!}$$
to the right-hand side of \eref{Hursplit_e}.
This corresponds to the middle term on the right-hand side of~\eref{HurPDE_e}
(the factorials in the above expression precisely correspond to $u^r/r!$ in
the definition of~$F$).

\begin{rmk}\label{15.2_rmk}
Our notation in this section differs from that of \cite[Section~15.2]{IPsum} and \cite{GJV}.
The $k$-th component of our tuple~$\al$ is the number of entries in the tuple~$\al$
of \cite[Section~15.2]{IPsum} and \cite{GJV} that equal~$k$.
Thus, our usage of~$\al$ is consistent with Section~\ref{CH_subs},
which is essentially \cite[Section~15.1]{IPsum}, while 
the tuples~$\al$ of \cite[Section~15.2]{IPsum} and \cite{GJV} are denoted by~$\bs$
in the rest of~\cite{IPsum}.
Similarly to the situation with \cite[Lemma~14.2]{IPsum}, 
$\GW^{\P^1,p}_{g,d}(b^r;\bfC_m)$ is not defined in \cite[Section~15.2]{IPsum};
its intended meaning is inconsistent with the notation used in the rest of the paper.
The generating function on line~3 on page~1017 in~\cite{IPsum} is not \cite[(A.6)]{IPsum}.
More significantly, it is also not the generating function of~\cite[(3.1)]{GJV}.
Dropping $t^d$ from this definition is not material, since 
$t$ does not appear in the PDE for~$F$ and $d$ is encoded by~$m$, 
but dropping $m_a!$ is material; 
otherwise, $F$ would not satisfy the PDE.
It is not immediately clear whether the sum in \cite[(3.1)]{GJV}
is over ordered or unordered partitions~$\al$ of $n\!=\!d$
(neither of which would correspond to the generating function in~\cite{IPsum}),
but summing over the unordered partitions~$\al$ 
(which is the same as summing over our tuples~$\al$) gives 
a solution of the desired PDE.
As stated, \cite[(15.3)]{IPsum} is incorrect, 
since the sum is only over certain configurations of curves 
(as explained after this formula).
The last term in \cite[(15.3)]{IPsum} is not even defined in the paper 
(though its meaning could be guessed);
it is also unnecessary, since the symplectic sum formula involves GW/GT-invariants and 
these are also the numbers computed in \cite[Lemmas~14.1,14.2]{IPsum}.
Other, fairly minor misstatements in \cite[Section~6]{LR} include
\begin{enumerate}[label={},topsep=-5pt,itemsep=-5pt,leftmargin=*]
\item\textsf{p1016, Section 15.2, line 17,20; p1017, line 3,7:} $C_m$ \to $\bfC_m$;
\item\textsf{p1016, line -1:} there should be only one -2 in this formula;
\item\textsf{p1017, line 6:} the $S$-matrix is the identity;
\item\textsf{p1017, line 10:} it is not clear what $\GT=\exp\GW$ means here or why it is relevant;
\item\textsf{p1017, 2.:} $-\chi_1=2g-4$, $g_1+g_2=g$, $d_1+d_2=d$.
\end{enumerate}
\end{rmk}

\subsection{GW-invariants and birational geometry: \cite[Section~1,2,6]{LR}}
\label{LRapp_subs}

\noindent
This section summarizes comments on the remainder of~\cite{LR}.\\

\noindent
Connections with birational geometry are described extensively 
on page~152 of~\cite{LR}, at the beginning of the introduction.
There are many other instances of the discussion diverging in this direction
which have little to do with the content of the paper.
These include the entire page~153, the paragraph preceding Definition~1.1,
the last three sentences on page~155, the sentence after~(1.8),
the short and long paragraphs on page~159, the sentence before Definition~2.4,
and the three paragraphs of Remarks~2.15 and~2.16.\\

\noindent
There are many statements that come with no citations or imprecise citations.
These~include
\begin{enumerate}[label=(\arabic*),leftmargin=*]
\item the sentence before Corollary~A.3;
\item bottom of page of 160 (Gray's Theorem);
\item the paragraph below (2.24);
\item some statements in Remarks~2.14 and~2.15;
\item citations of [H] above Lemma 3.5 on p174 and of [HWZ1] at the top of p177;
\item statement above Remark~4.1 on p188;
\item reference to Siebert's construction at the bottom of p189;
\item reference to Donaldson's book on at the top of p198.
\end{enumerate}

\noindent
The statements of Theorems~A and~B do not make the assumptions on the manifold~$M$
clear. Based on the proofs, $M$ is a threefold in both cases.
Symplectic sum formulas are not necessary to establish these formulas;
a nearly complete geometric argument for them is already given in the paper.\\

\noindent
Theorems~A and~B are deduced from the symplectic sum formula in \cite[Section~6]{LR}.
However, the latter is barely used in their proofs and the arguments indicate how to 
avoid it entirely.
Let $\wh{X}$ be the symplectic blowup of a threefold~$X$ along an embedded curve~$C$, 
$A\!\in\!H_2(X;\Z)$ be such that $\lr{c_1(X),A}\!>\!0$,
and $\al_i\!\in\!H^4(X)\!\cup\!H^6(X)$ be a collection of classes
of total codimension corresponding to GW-invariants of degree~$A$.
These invariants are then counts of $(J,\nu)$-holomorphic curves, 
for a generic~$(J,\nu)$, passing through representatives of $\PD_X(\al_i)$.
The latter can be chosen to be disjoint from~$C$;
$J$ can be chosen to be Kahler along~$C$ and so that all curves of degree~$A$
passing through the constraints are disjoint from~$C$.
These representatives and curves then lift to~$\wh{X}$, contributing to the corresponding
GW-invariants of~$\wh{X}$;
any other curve in~$\wh{X}$ contributing to this count would descend to~$X$
and thus would be disjoint from~$C$.

\begin{rmk}\label{LR6_rmk}
The fourth sentence of the long paragraph on page~163 in~\cite{LR}
makes it sound that every two smooth CY 3-folds
are related by a sequence of flops, but it is apparently meant to apply to 
every pair of birational smooth CY 3-folds.
The flop construction is never formally defined, but apparently the sentence 
after \cite[(2.15)]{LR} is part of the definition.
Other, fairly minor misstatements in \cite[Sections~1,6]{LR} include
\begin{enumerate}[label={},topsep=-5pt,itemsep=-5pt,leftmargin=*]
\item\textsf{p155, line -7:} the above corollary \to Corollary~A.2;
\item\textsf{p156, Crl B:} there exists such a $\vph$; 
\item\textsf{p157, line 12:} this equality does not hold, as LHS is degenerate along $\pi^{-1}(Z)$;
\item\textsf{p157, lines 14-16:} this sentence makes no sense;
\item\textsf{p157, line 19:} positive \to nonnegative;
\item\textsf{p157, line 23:} is tangent to \to has contact with;
\item\textsf{p157, bottom:} no connection to justification;
\item\textsf{p158, line 1:} Theorem 5.3 \to Theorem 4.14;
\item\textsf{p158, Theorem C(iii):} need an almost complex structure on~$M$;
\item\textsf{p158, line 1:} Theorem 5.6,5.7 \to Theorems 5.6,5.8;
\item\textsf{p212, above (6.3):} there is no $\tn{Ind}\,D_u$ in (5.2);
\item\textsf{p212, below (6.4):} $\bar{M}^+$ just defined on the previous page;
\item\textsf{p212, above (6.6):} Remark 3.24 \to Remark 5.2;
\item\textsf{p213, above Pf of Crl A.2:} Corollaries A.1 and A.3 are immediate consequences;
\item\textsf{p215, line 6:} $Y$ \to $M^+$;
\item\textsf{p215, above Pf of Crl B.2:} Corollary B.1 is an immediate consequence. 
\end{enumerate}
\end{rmk}


\vspace{.2in}

\noindent
{\it Simons Center for Geometry and Physics, SUNY Stony Brook, NY 11794\\
mtehrani@scgp.stonybrook.edu}\\

\noindent
{\it Department of Mathematics, SUNY Stony Brook, Stony Brook, NY 11794\\
azinger@math.sunysb.edu}\\


\begin{thebibliography}{99}


\bibitem[AF]{AF} D.~Abramovich and B.~Fantechi,
{\it Orbifold techniques in degeneration formulas}, math/1103.5132.

\bibitem[BF]{BF} K.~Behrend and B.~Fantechi, {\it The intrinsic normal cone},
Invent.~Math.~128 (1997), no.~1, 45--88.

\bibitem[BL]{BL} J.~Bryan and N.-C.~Leung, 
{\it The enumerative geometry of K3 surfaces and modular forms}, 
J.~Amer.~Math.~Soc.~13 (2000), no.~2, 371--410.

\bibitem[CH]{CaH} L.~Caporaso and J.~Harris,
{\it Counting plane curves in any genus}, Invent.~Math.~131 (1998), no.~2, 345--392.

\bibitem[Chen]{Chen} Q.~Chen, 
{\it Stable logarithmic maps to Deligne-Faltings pairs~I},
Ann.~of Math. (2) 180 (2014), no.~2, 455-–521.


\bibitem[CM]{CM} K.~Cieliebak and K.~Mohnke, 
{\it Symplectic hypersurfaces and transversality in Gromov-Witten theory}, 
J.~Symplectic Geom.~5 (2007), 281--356.

\bibitem[FZ1]{GWrelIP} M.~Farajzadeh Tehrani and A.~Zinger,
{\it On the rim tori refinement of relative Gromov-Witten invariants}, pre-print.


\bibitem[FZ2]{GWsumIP} M.~Farajzadeh Tehrani and A.~Zinger,
{\it On the refined symplectic sum formula  for Gromov-Witten invariants},
pre-print.

\bibitem[FHS]{FHS} A.~Floer, H.~Hofer, and D.~Salamon, 
{\it Transversality in elliptic Morse theory for the symplectic action},
Duke Math.~J.~80 (1996), no.~1, 251--292.

\bibitem[FO]{FO} K.~Fukaya and K.~Ono, {\it Arnold conjecture and Gromov-Witten invariant}, 
Topology 38 (1999), no.~5, 933--1048.

\bibitem[Gf0]{Gf0} R.~Gompf, {\it Some new symplectic 4-manifolds}, 
Turk.~J.~Math 18 (1994), no.~1, 7--15.

\bibitem[Gf]{Gf} R.~Gompf, {\it A new construction of symplectic manifolds}, 
Ann.~of Math.~142 (1995), no.~3, 527--595.

\bibitem[GH]{GH} P.~Griffiths and J.~Harris,
{\it Principle of Algebraic Geometry}, Wiley, 1978.

\bibitem[GJV]{GJV} I.~Goulden, D.~Jackson, and A.~Vainshtein,
{\it The number of ramified coverings of the sphere by the torus and surfaces of higher genera}, 
Ann.~Comb.~4 (2000), no.~1, 27--46.


\bibitem[GS]{GS} M.~Gross and B.~Siebert, 
{\it Logarithmic Gromov-Witten invariants}, J.~Amer.~Math.~Soc.~26 (2013), no.~2, 451--510. 

\bibitem[H]{H} H.~Hofer, 
{\it Pseudoholomorphic curves in symplectizations with applications to the
Weinstein conjecture in dimension three}, Invent.~Math.~114 (1993), 515--563.

\bibitem[HWZ1]{HWZ1} H.~Hofer, K.~Wysocki, and E.~Zehnder, 
{\it Properties of pseudo-holomorphic curves in symplectizations 1: asymptotics}, 
Ann.~Inst.~H.~Poincar\'e 13 (1996), 337--371.

\bibitem[HLR]{HLR} J.~Hu, T.-J.~Li, and Y.~Ruan, 
{\it Birational cobordism invariance of uniruled symplectic manifolds}, 
Invent.~Math.~172 (2008), no.~2, 231--275. 

\bibitem[IP3]{IPmrl} E.~Ionel and T.~Parker, 
{\it Gromov-Witten invariants of symplectic sums}, Math.~Res.~Lett.~5 (1998),
563--576.

\bibitem[IP4]{IPrel} E.~Ionel and T.~Parker, 
{\it Relative Gromov-Witten invariants},
Ann.~of Math.~157 (2003), no.~1, 45--96.

\bibitem[IP5]{IPsum} E.~Ionel and T.~Parker, 
{\it The symplectic sum formula for Gromov-Witten invariants},
Ann.~of Math.~159 (2004), no.~3, 935--1025.

\bibitem[IP6]{IPvfc} E.~Ionel and T.~Parker, 
{\it A natural Gromov-Witten virtual fundamental class},
1302.3472. 

\bibitem[IP7]{IPgv} E.~Ionel and T.~Parker, 
{\it The Gopakumar-Vafa formula for symplectic manifolds},  
1306.1516  
 

\bibitem[L]{Lee}  J.~Lee, {\it Manifolds and Differential Geometry},
Graduate Studies in Mathematics~107, AMS~2009.


\bibitem[Ler]{Ler} E.~Lerman, 
{\it Symplectic cuts}, Math.~Res.~Lett.~2 (1995), no.~3, 247--258. 

\bibitem[Lj1]{Jun1} J.~Li,
{\it Stable morphisms to singular schemes and relative stable morphisms}, 
J.~Diff.~Geom.~57 (2001), no.~3, 509--578.

\bibitem[Lj2]{Jun2} J.~Li,
{\it A degeneration formula for GW-invariants}, 
J.~Diff.~Geom.~60 (2002), no.~1, 199--293.

\bibitem[Li]{AMLi} A.-M.~Li,
{\it Remark on symplectic relative Gromov-Witten invariants and degeneration formula},
math/1405.3825.

\bibitem[LR]{LR} A.-M. Li and Y.~Ruan, 
{\it Symplectic surgery and Gromov-Witten invariants of Calabi-Yau 3-folds},
Invent.~Math.~145 (2001), no.~1, 151--218. 


\bibitem[LT]{LT} J.~Li and G.~Tian,
{\it Virtual moduli cycles and Gromov-Witten invariants of general symplectic manifolds},
in {\it Topics in Symplectic 4-Manifolds}, 47--83, Internat.~Press 1998.

\bibitem[Lo]{Lo} E.~Looijenga, {\it Smooth Deligne-Mumford compactifications 
by means of Prym level structures}, J.~Algebraic Geom.~3 (1994), 283--293.


\bibitem[MS1]{MS1} D.~McDuff and D.~Salamon, 
{\it Symplectic Topology}, 2nd Ed., Oxford University Press, 1998.

\bibitem[MS2]{MS2} D.~McDuff and D.~Salamon, 
{\it $J$-Holomorphic Curves and Symplectic Topology}, 
AMS Colloquium Publications~52, 2004 (expanded version of~[MS]).

\bibitem[Mu1]{Mu}  J.~Munkres, {\it Topology: a first course},
2nd Ed., Pearson, 2000.

\bibitem[Mu2]{Mu2}  J.~Munkres, {\it Elements of Algebraic Topology},
Addison-Wesley 1984.

\bibitem[MW]{MW} J.~McCarthy and J.~Wolfson,  {\it Symplectic normal connect sum}, 
Topology 33 (1994), no.~4, 729--764.

\bibitem[R5]{R5} Y.~Ruan,
{\it Virtual neighborhoods and pseudo-holomorphic curves}, 
Turkish J.~Math.~23 (1999), no.~1, 161--231.

\bibitem[RT1]{RT1} Y.~Ruan and G.~Tian,
{\it A mathematical theory of quantum cohomology}, 
J.~Diff.~Geom.~42 (1995), no.~2, 259--367.

\bibitem[RT2]{RT2} Y.~Ruan and G.~Tian,
{\it Higher genus symplectic invariants and sigma models coupled with gravity},
Invent.~Math.~130 (1997), no.~3, 455--516.


\bibitem[Sh]{Sh} V.~Shevchishin, 
{\it Pseudoholomorphic curves and the symplectic isotopy problem},
math/0010262.




\bibitem[T]{T} G.~Tian, 
{\it The quantum cohomology and its associativity}, 
Current Developments in Mathematics~1995, 361--401, Inter.~Press.


\bibitem[Z1]{Z} A.~Zinger, {\it Pseudocycles and integral homology},
Trans.~Amer.~Math.~Soc.~360 (2008), no.~5, 2741--2765. 

\bibitem[Z2]{anal} A.~Zinger, {\it Basic Riemannian geometry and Sobolev estimates
used in symplectic topology}, math/1012.3980.

\end{thebibliography}
\end{document}